\documentclass[11pt]{amsart}
\usepackage{geometry}                
\usepackage[latin1]{inputenc}   
\usepackage{mathpazo}  
\usepackage[english]{babel}
\usepackage[usenames,dvipsnames,cmyk]{xcolor}
\usepackage{url}
\usepackage{amsmath,amssymb,latexsym,mathrsfs,amssymb,amsthm}
\usepackage{enumerate}
\usepackage{mathtools} 
\usepackage[all]{xy}
\usepackage{setspace}
\usepackage{tikz, tkz-euclide}

\numberwithin{equation}{section}

\usepackage[colorlinks=true, linkcolor=MidnightBlue, citecolor=MidnightBlue, urlcolor=MidnightBlue]{hyperref}

\newtheorem{theorem}{Theorem}[section]
\newtheorem{lemma}[theorem]{Lemma}
\newtheorem{proposition}[theorem]{Proposition}
\newtheorem{corollary}[theorem]{Corollary}

\theoremstyle{definition}

\theoremstyle{remark}
\newtheorem{remark}[theorem]{Remark}
\newtheorem{rhp}{Riemann-Hilbert Problem}
\newcommand{\rhref}[1]{Riemann-Hilbert Problem~\ref{#1}}

\let\Re=\undefined\DeclareMathOperator{\Re}{Re}
\let\Im=\undefined\DeclareMathOperator{\Im}{Im}

\newcommand{\ii}{\ensuremath{\mathrm{i}}}
\newcommand{\ee}{\ensuremath{\mathrm{e}}}
\newcommand*\dd{\mathop{}\!\mathrm{d}}

\newcommand{\bg}{B}

\newcommand{\critpoly}{F}
\newcommand{\phase}{\vartheta}
\newcommand{\JacobiSigma}{\sigma}
\newcommand{\Jacobiu}{\mathfrak{u}}
\newcommand{\Jacobiv}{\mathfrak{v}}
\newcommand{\critpt}{\xi}
\newcommand{\solmat}{W}
\newcommand{\solitonphase}{\Omega}
\newcommand{\conformalprime}{\eta}
\newcommand{\compact}{\mathscr{K}}
\newcommand{\holomat}{Y}
\newcommand{\Qpoly}{Q}

\makeatletter
\renewcommand*\env@matrix[1][\arraystretch]{%
  \edef\arraystretch{#1}%
  \hskip -\arraycolsep
  \let\@ifnextchar\new@ifnextchar
  \array{*\c@MaxMatrixCols c}}
\makeatother

\let\originalleft\left
\let\originalright\right
\renewcommand{\left}{\mathopen{}\mathclose\bgroup\originalleft}
\renewcommand{\right}{\aftergroup\egroup\originalright}

\title{Infinite-order rogue waves that are small (but not small in $L^2$)}

\author{Deniz Bilman}
\address{Deniz Bilman:  Department of Mathematical Sciences, University of Cincinnati, Cincinnati, OH, USA}
\email{bilman@uc.edu}
\author{Peter D.~Miller}
\address{Peter D. Miller:  Department of Mathematics, University of Michigan, Ann Arbor, MI, USA}
\email{millerpd@umich.edu}

\date{\today}

\begin{document}
\begin{abstract}
General rogue waves of infinite order constitute a family of solutions of the focusing nonlinear Schr\"odinger equation that have recently been identified in a variety of asymptotic limits such as high-order iteration of B\"acklund transformations and semiclassical focusing of pulses with specific amplitude profiles.  These solutions have compelling properties such as finite $L^2$-norm contrasted with anomalously slow temporal decay in the absence of coherent structures.  In this paper we investigate the asymptotic behavior of general rogue waves of infinite order in a parametric limit in which the solution becomes small uniformly on compact sets while the $L^2$-norm remains fixed.  We show that the solution is primarily concentrated on one side of a specific curve in logarithmically rescaled space-time coordinates, and we obtain the leading-order asymptotic behavior of the solution in this region in terms of elliptic functions as well as near the boundary curve in terms of modulated solitons.  
The asymptotic formula captures the fixed $L^2$-norm even as the solution becomes uniformly small.
\end{abstract}
\maketitle

\tableofcontents

\section{Introduction}
\label{s:introduction}
\subsection{General rogue waves of infinite order}
(General) rogue waves of infinite order $\Psi(X,T;\mathbf{G},B)$ are particular solutions of the focusing nonlinear Schr\"odinger equation
\begin{equation}
\ii Q_T + \frac{1}{2}Q_{XX} + |Q|^2Q = 0
\label{nls-general}
\end{equation}
that have recently appeared in numerous applications.  In \cite{Suleimanov17}, one of these solutions was conjectured to describe the weakly dispersive regularization of an infinite-amplitude collapse solution of the dispersionless focusing nonlinear Schr\"odinger equation, and a version of this result has recently been proved rigorously \cite{BuckinghamJM25}.  The same family of solutions arises in the fixed-dispersion setting when B\"acklund transformations are repeatedly iterated.  This was first observed in the case of high-order fundamental rogue-wave solutions \cite{BilmanLM2020} (the application for which the solutions are named) and also high-order soliton solutions \cite{BilmanB2019}.  Later it was shown that the same limiting objects describe a family of solutions continuously interpolating between high-order rogue waves and solitons \cite{BilmanM2021}, and in fact one can obtain the same limit via repeated B\"acklund iterations starting with an arbitrary seed solution \cite{BilmanM22}.  Evaluating at $T=0$ one obtains functions of $X$ that play a role in boundary-layer theory for the sharp-line Maxwell-Bloch system \cite{LiM24} and the large-degree asymptotics of rational solutions of the third Painlev\'e equation \cite{BarhoumiLMP24}.

The solutions $\Psi(X,T;\mathbf{G},\bg)$ are parametrized by a scalar $B>0$ and certain $2\times 2$ constant matrices $\mathbf{G}$.  They are defined in terms of the following Riemann-Hilbert problem.

\begin{rhp}
Let $(X,T)\in\mathbb{R}^2$ and $\bg>0$ be fixed  
and let $\mathbf{G}$ be a $2\times 2$ matrix satisfying $\det(\mathbf{G})=1$ and $\mathbf{G}^*=\sigma_2\mathbf{G}\sigma_2$.
 Find a $2\times 2$ matrix $\mathbf{P}(\Lambda;X,T,\mathbf{G},\bg)$ with the following properties:
\begin{itemize}
\item[]\textbf{Analyticity:}  $\mathbf{P}(\Lambda;X,T,\mathbf{G},\bg)$ is analytic in $\Lambda$ for $|\Lambda|\neq 1$, and it takes continuous boundary values on the clockwise-oriented unit circle from the interior and exterior.
\item[]\textbf{Jump condition:} The boundary values\footnote{We use the standard convention that a subscript $+$ (resp., $-$) denotes a boundary value taken on an oriented contour arc from the left (resp., right).} on the unit circle are related as follows:
\begin{equation}
\mathbf{P}_+(\Lambda;X,T,\mathbf{G},\bg)=\mathbf{P}_-(\Lambda;X,T,\mathbf{G},\bg)
\ee^{-\ii(\Lambda X+\Lambda^2T+2 \bg \Lambda^{-1})\sigma_3}\mathbf{G}
\ee^{\ii(\Lambda X+\Lambda^2T+2\bg \Lambda^{-1})\sigma_3},\quad |\Lambda|=1.
\label{P-jump}
\end{equation}
\item[]\textbf{Normalization:}  $\mathbf{P}(\Lambda;X,T,\mathbf{G},\bg)\to\mathbb{I}$ as $\Lambda\to\infty$.
\end{itemize}
\label{rhp:near-field}
\end{rhp}
Here $\sigma_2$ and $\sigma_3$ denote standard Pauli matrices
\begin{equation*}
\sigma_1 := \begin{bmatrix} 0 & 1 \\ 1 & 0\end{bmatrix},\quad \sigma_2 := \begin{bmatrix} 0 & -\ii \\ \ii & 0\end{bmatrix},\quad \sigma_3:=\begin{bmatrix} 1 & 0 \\ 0 & -1\end{bmatrix},
\end{equation*}
and $\mathbf{G}^*$ means the complex-conjugate matrix (no transpose).
In general any matrix $\mathbf{G}$ satisfying $\det(\mathbf{G})=1$ and $\mathbf{G}^*=\sigma_2 \mathbf{G} \sigma_2 $ as in \rhref{rhp:near-field} can be written as
\begin{equation}
\mathbf{G}=\mathbf{G}(a,b)=\frac{1}{\sqrt{|a|^{2}+|b|^{2}}}\begin{bmatrix}
a & b^{*} \\
-b & a^{*}
\end{bmatrix},\quad (a,b)\in\mathbb{C}^2\setminus\{(0,0)\}.
\label{G-form}
\end{equation}
\rhref{rhp:near-field} has a 
unique solution 
that
depends real-analytically on $(X,T)\in\mathbb{R}^2$ \cite[Proposition 1.1]{BilmanM2024}, and the special solution $\Psi(X,T;\mathbf{G},\bg)$ is defined in terms of $\mathbf{P}(\Lambda;X,T,\mathbf{G},\bg)$ via
\begin{equation}
\Psi(X,T;\mathbf{G},\bg) := 2\ii \lim_{\Lambda\to\infty} \Lambda P_{12}(X,T;\mathbf{G},\bg).
\label{}
\end{equation}
Without loss of generality we take $B=1$ thanks to the scaling symmetry \cite[Proposition 1.2]{BilmanM2024}
\begin{equation}
\Psi(X, T ; \mathbf{G}, \bg)=\bg \Psi(\bg X, \bg^2 T ; \mathbf{G}, 1),
\label{scaling-sym}
\end{equation}
so we will write $\Psi(X,T;\mathbf{G}):=\Psi(X,T;\mathbf{G},1)$ and omit $B=1$ from the argument list of functions related to $\mathbf{P}$.
On the other hand, the dependence on $\mathbf{G}=\mathbf{G}(a,b)$ is more consequential. 

\subsection{Dependence on parameters}
In recent work \cite{BilmanM2024}, numerous properties of the special solution $\Psi(X, T ; \mathbf{G})$ were established for fixed values of the parameters $(a,b)\in\mathbb{C}^2\setminus\{(0,0)\}$ including a Fredholm determinant representation and the asymptotic behavior in spacetime as $(X,T)\to\infty$ in the entire plane.  
As part of the same work, a software package 
\cite{RogueWaveInfiniteNLS} was developed to compute $\Psi(X, T ; \mathbf{G})$ at virtually any point in the $(X,T)$-plane based on numerical solution of suitable Riemann-Hilbert problem representations (modifications of \rhref{rhp:near-field}). 
One key result is:
\begin{theorem}[\protect{\cite[Theorem 1.9]{BilmanM2024}}]
Let $\mathbf{G}=\mathbf{G}(a,b)$ be as in \eqref{G-form}.  If also $ab \neq 0$, then $\Psi(\diamond, T; \mathbf{G})\in L^2(\mathbb{R})$ for all $T\in\mathbb{R}$ with $\| \Psi(\diamond, T;\mathbf{G} )\|_{L^2(\mathbb{R})}=\sqrt{8}$.
\label{t:L2-norm}
\end{theorem}
$L^2(\mathbb{R})$ is a natural space on which the Cauchy problem for \eqref{nls-general} is globally well posed \cite{Tsutsumi87}.  Most solutions in this space obey soliton resolution, in that either they undergo modified scattering with dispersive decay in $L^\infty(\mathbb{R})$ proportional to $|T|^{-\frac{1}{2}}$ or they have a soliton component that does not decay at all in $L^\infty(\mathbb{R})$ 
\cite{BorgheseJM2018}.
However, the general rogue waves of infinite order with $ab\neq 0$ are exceptional solutions that exhibit anomalously slow temporal decay \cite[Theorem 1.22]{BilmanM2024}: $\Psi(X,T;\mathbf{G})\sim |T|^{-\frac{1}{3}}$ as $|T|\to+\infty$ (upper and lower bounds).  This property makes these solutions very interesting from the point of view of mathematical analysis as well as physical applications.

The condition $ab\neq 0$ is essential for Theorem~\ref{t:L2-norm} because otherwise the solution is trivial:
\begin{proposition}[\protect{\cite[Proposition 1.8]{BilmanM2024}}] If $ab=0$, then $\Psi(X,T;\mathbf{G})=0$ for all $(X,T)\in\mathbb{R}^2$.
\label{prop:degeneration}
\end{proposition}
This result relies on the fact that \rhref{rhp:near-field} can be solved explicitly if either $a=0$ or $b=0$ (see, for instance \eqref{eq:Pfor(a,b)=(0,1)} below). 
The solution $\Psi(X,T;\mathbf{G}(a,b))$ depends continuously on $(a,b)$ for fixed $(X,T)$, so Theorem~\ref{t:L2-norm} and Proposition~\ref{prop:degeneration} together imply that
\begin{equation}
\lim_{\substack{ab \to 0 \\ ab \neq 0}}\|\Psi(\diamond,T;\mathbf{G}(a,b)\|_{L^2(\mathbb{R})}=\sqrt{8}\quad\text{but}\quad\forall (X,T)\in\mathbb{R}^2:  \lim_{\substack{ab \to 0 \\ ab \neq 0}}\Psi(X,T;\mathbf{G}(a,b))=0.
\label{eq:conundrum}
\end{equation}
While this is not a contradiction, it deserves a proper explanation in terms of the asymptotic behavior of the solution as $ab\to 0$.  Does the solution $\Psi(\diamond,T;\mathbf{G}(a,b))$ decay in $L^\infty(\mathbb{R})$ as $ab\to 0$ for fixed $T$ similar to dispersive decay conserving the $L^2$-norm?  Or perhaps the solution undergoes translation to $X=\infty$ as $ab\to 0$ for fixed $T$ like a traveling pulse?  

Numerical computations done with the software package \cite{RogueWaveInfiniteNLS} described in the paper \cite{BilmanM2024} suggest that it is a combination of these two scenarios that characterizes the limit $ab\to 0$.
Figure~\ref{f:surface-plots-decreasing-a} shows the behavior of $\Psi(X,T;\mathbf{G}(a,b))$ as $a$ decreases from $a=1$ to $a=\frac{1}{4}$ for fixed $b=1$.
\begin{figure}[h]
\includegraphics[height=1.85in]{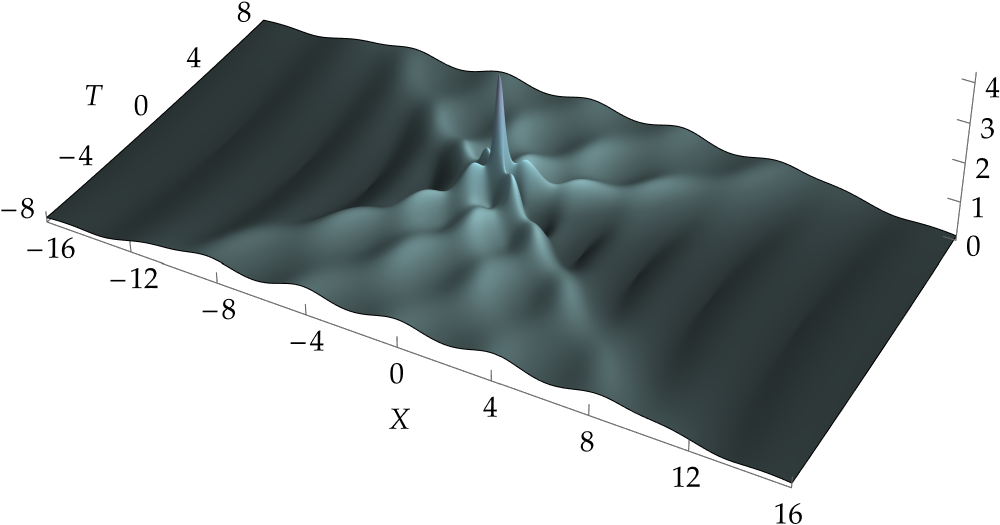}\includegraphics[height=1.85in]{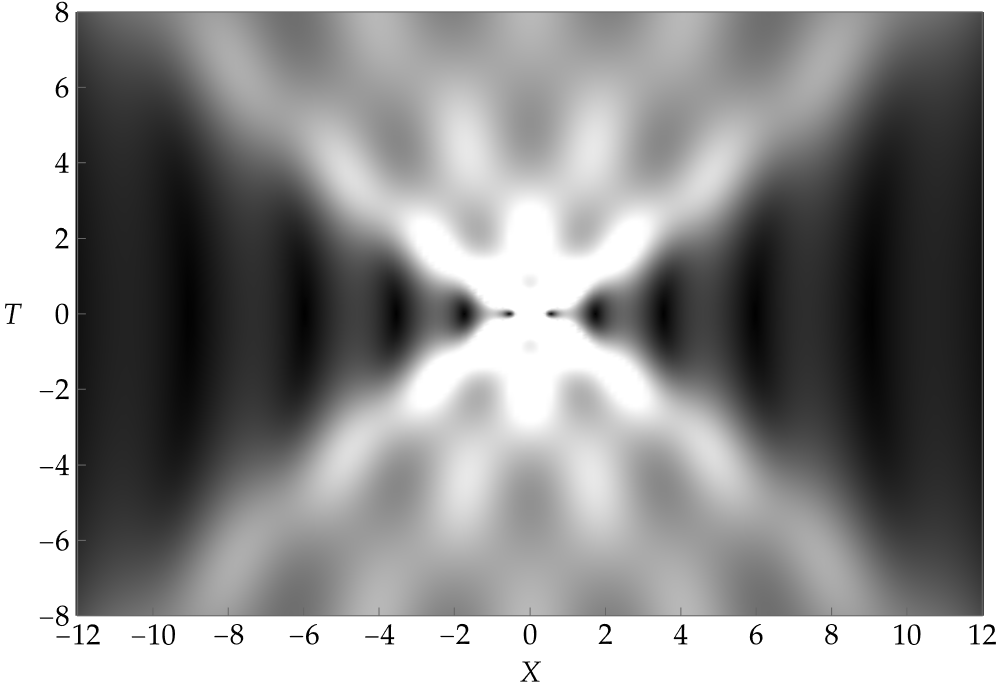}\\
\includegraphics[height=1.85in]{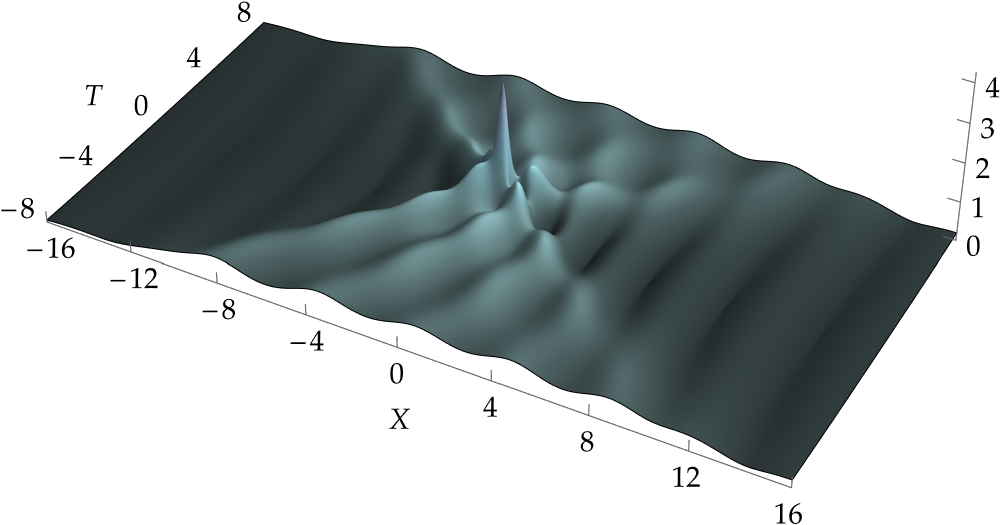}\includegraphics[height=1.85in]{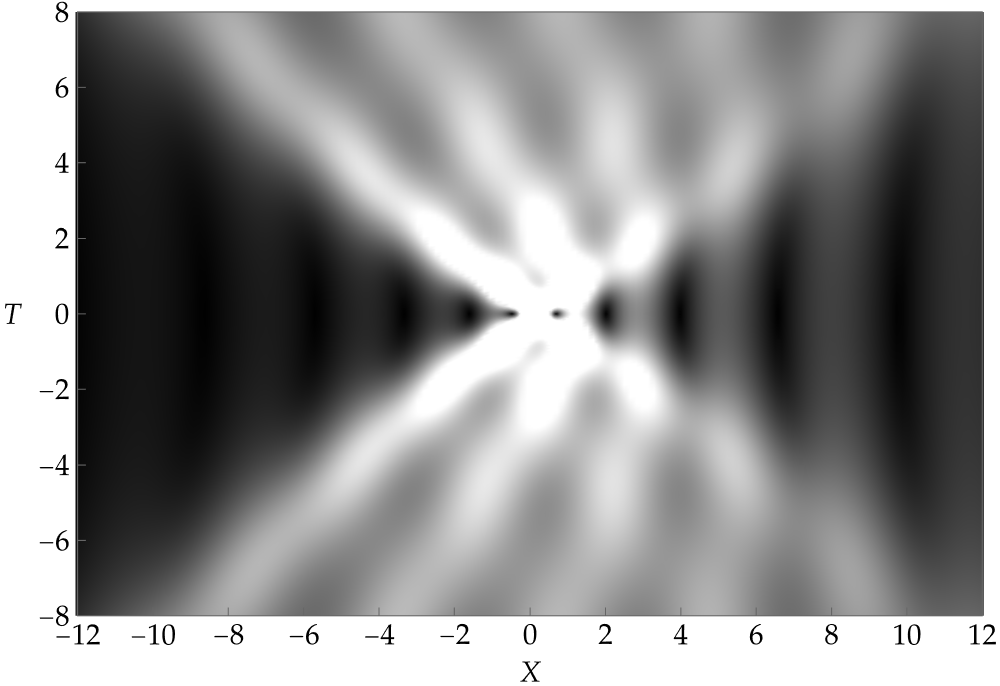}\\
\includegraphics[height=1.85in]{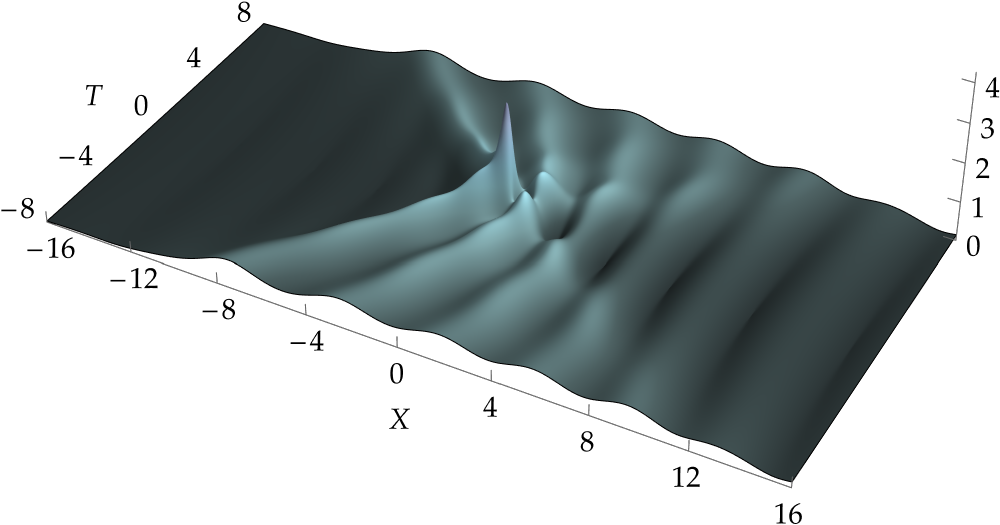}\includegraphics[height=1.85in]{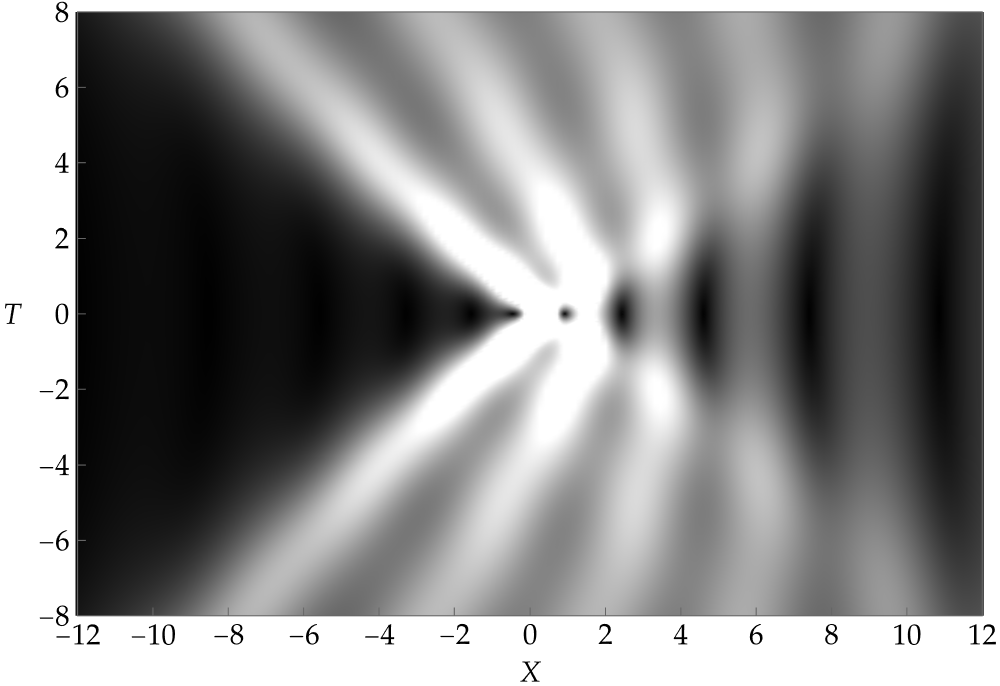}\\
\caption{Surface plots (left) and density plots (right) of $|\Psi(X,T;\mathbf{G}(a,b))|$ computed with the software package \texttt{RogueWaveInfinite.jl}, for $b=1$. From top to bottom: $a=1$, $a=\frac{1}{2}$, and $a=\frac{1}{4}$.}
\label{f:surface-plots-decreasing-a-v2}
\end{figure}
There evidently emerges a quiescent unbounded region of the $(X,T)$-plane extending to $X=-\infty$ in which the solution becomes small as $a\to 0$.  Indeed, the small-amplitude oscillations visible in the wedge-shaped region containing the the negative $X$-axis when $a=b=1$ appear to decay as $a\to 0$, while those on the opposite side become amplified as $a\to 0$. In fact, these amplified oscillations seem to merge with the larger-amplitude waves present for $a=b=1$ in wedge-shaped regions surrounding the positive and negative $T$-axes.
Figure~\ref{f:surface-plots-decreasing-a} shows the behavior of the solution close to the origin $(X,T)=(0,0)$ (the peak location of the solution if $a=b=1$) as $a$ decreases from $a=1$ to $a=\frac{1}{256}$. 
\begin{figure}[h]
\includegraphics[width=0.24\textwidth]{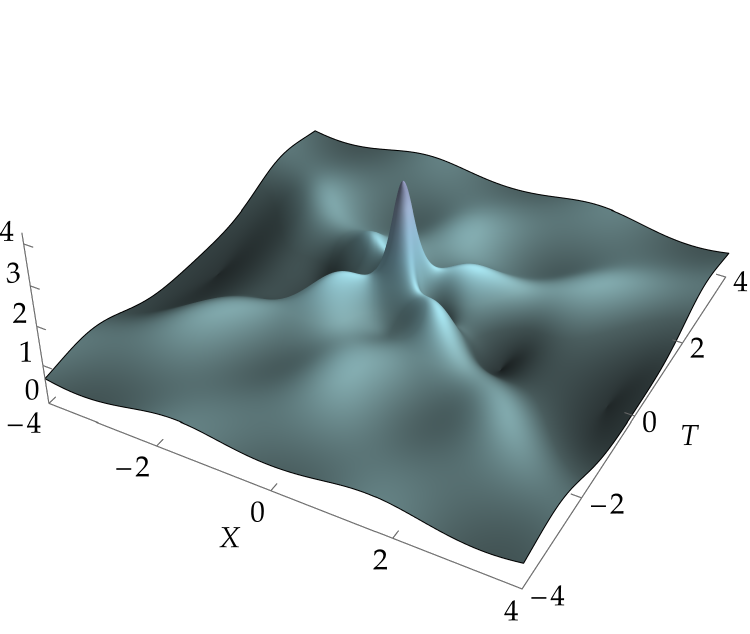}
\includegraphics[width=0.24\textwidth]{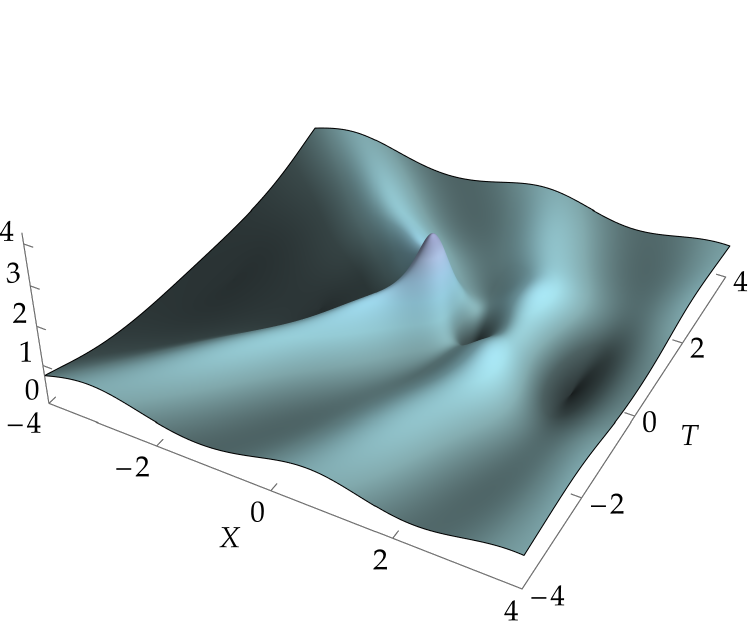}
\includegraphics[width=0.24\textwidth]{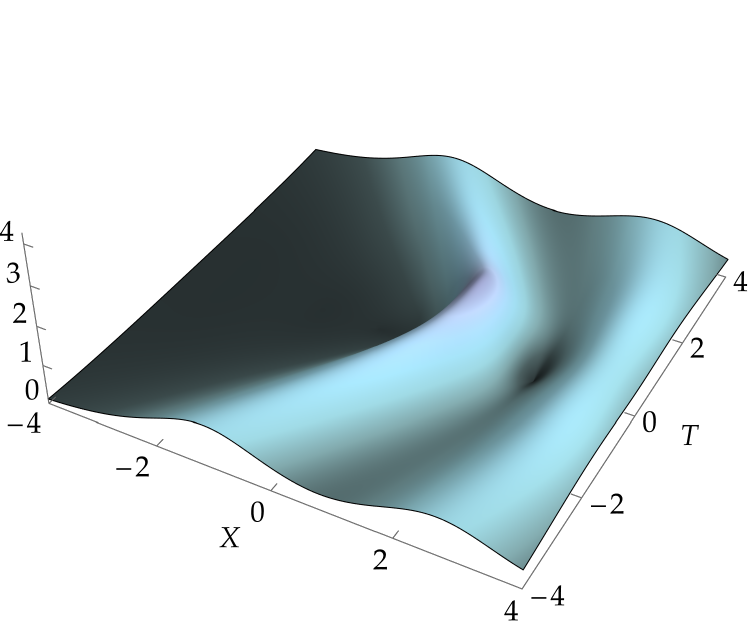}
\includegraphics[width=0.24\textwidth]{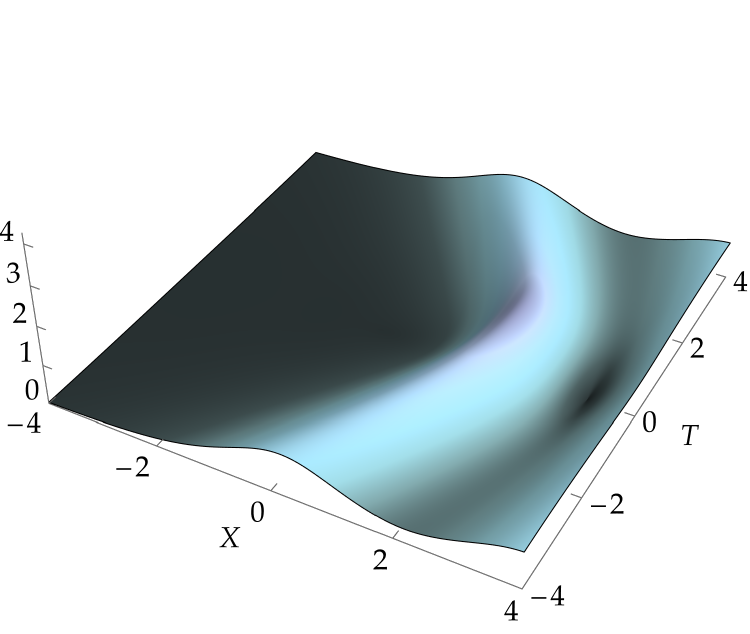}
\caption{Surface plots of $|\Psi(X,T;\mathbf{G}(a,b))|$ computed with the software package \texttt{RogueWaveInfinite.jl}, for $b=1$. From left to right: $a=1$, $a=\frac{1}{8}$, $a=\frac{1}{64}$, and $a=\frac{1}{256}$.}
\label{f:surface-plots-decreasing-a}
\end{figure}
Evidently, the peak of the solution both becomes smaller and shifts somewhat toward the positive $X$ direction as $a\to 0$. Plots of the time-slices $\Psi(X,T=0;\mathbf{G}(a,b))$ as $a$ becomes small are shown in Figure~\ref{f:Psi-on-a-b}.
\begin{figure}[h]
\includegraphics[width=0.24\linewidth]{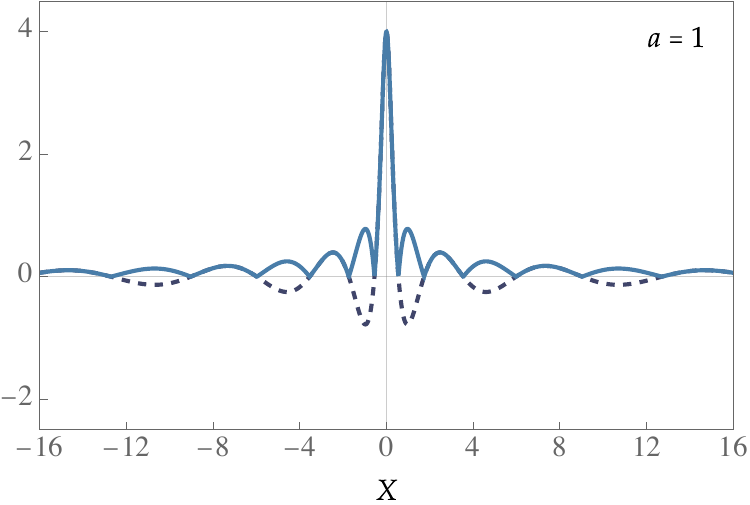}
\includegraphics[width=0.24\linewidth]{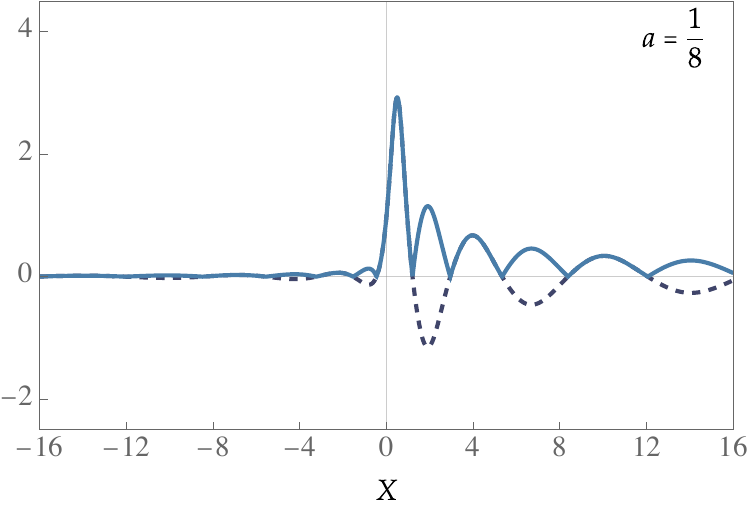}
\includegraphics[width=0.24\linewidth]{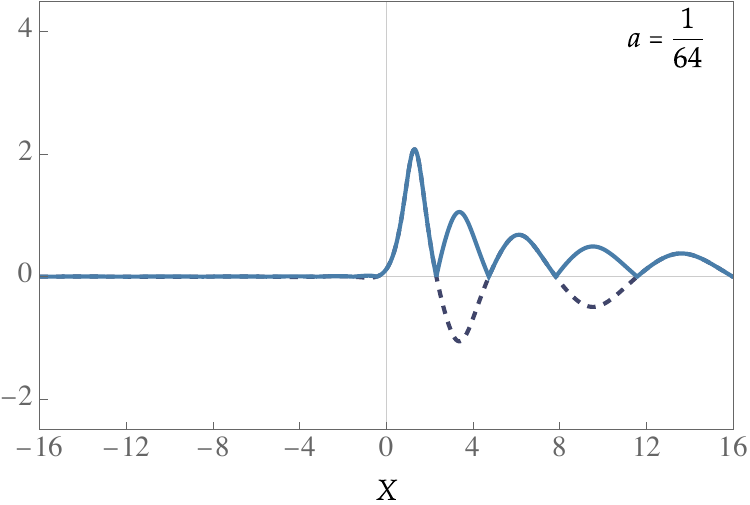}
\includegraphics[width=0.24\linewidth]{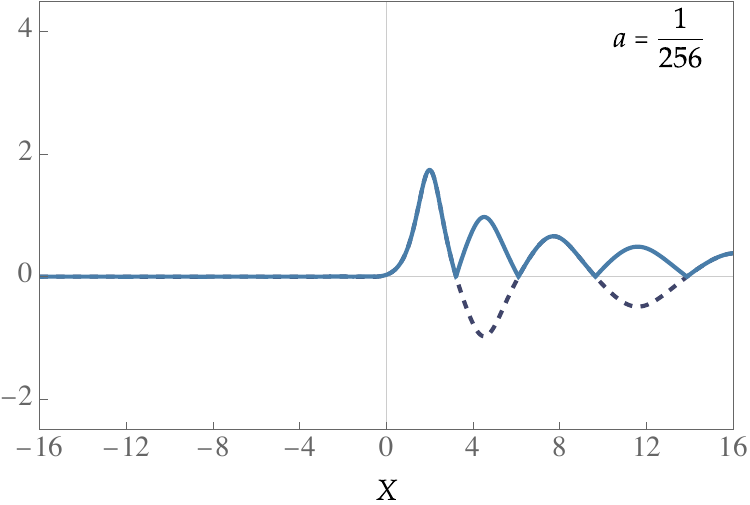}
\caption{Plots of $\Psi(X,0;\mathbf{G}(a,b))$ computed with the software package \texttt{RogueWaveInfinite.jl}, for $b=1$. From left to right: $a=1$, $a=\frac{1}{8}$, $a=\frac{1}{64}$, and $a=\frac{1}{256}$. Solid lines:  $|\Psi(X,0;\mathbf{G}(a,b))|$; dashed lines: $\Re(\Psi(X,0;\mathbf{G}(a,b)))$. The solution is real valued for $T=0$ since $ab\in\mathbb{R}$.} 
\label{f:Psi-on-a-b}
\end{figure}

Similar behavior of the solution but with $X$ replaced by $-X$ occurs if instead $b\to 0$ with $a$ fixed, due to the following symmetry:
\begin{proposition}[\protect{\cite[Proposition 1.3]{BilmanM2024}}]
$\Psi(X,T;\mathbf{G}(a,b)) = \Psi(-X,T;\mathbf{G}(b,a))$.
\label{prop:X-sym}
\end{proposition}

\subsection{Main results}
The purpose of this paper is to provide a rigorous explanation for the statement \eqref{eq:conundrum}  by describing the wave profile in the asymptotic limit $ab\to 0$. We find that in this limit, $\Psi(X,T;\mathbf{G}(a,b))$ is modeled by a gradually dispersing elliptic wavetrain with slowly decaying amplitude.

Using Proposition~\ref{prop:X-sym}, we focus on the case that $a\to 0$ for fixed $b\neq 0$. 
The smallness of $a$ is measured by the quantity
\begin{equation}
M:=-\frac{1}{2}\ln\left(\frac{|a|}{\sqrt{|a|^2+|b|^2}}\right)>0
\label{M-def}
\end{equation}
and we introduce logarithmically rescaled coordinates $(\chi,\tau)$ defined by
\begin{equation}
X = M^2 \chi\qquad\text{and} \quad T = M^3\tau.
\label{chi-tau-def}
\end{equation}
Thus the limit $a\to 0$ corresponds to $M\to+\infty$.
Also introducing the rescaled complex variable $z$ by setting
\begin{equation}
\Lambda = M^{-1} z,
\label{eq:z-Lambda}
\end{equation}
we see that with $\bg=1$ the exponent in \rhref{rhp:near-field} is expressed in the rescaled variables as:
\begin{equation}
\left.\vphantom{\int} (\Lambda X + \Lambda^2 T - 2 \Lambda^{-1})\right\vert_{X=M^2 \chi, T= M^3\tau, \Lambda=M^{-1}z} = M\phase(z;\chi,\tau),
\end{equation}
where
\begin{equation}
\phase(z)=\phase(z;\chi,\tau):= \chi z+ \tau z^2 - 2z^{-1}.
\label{eq:DS-tildevartheta}
\end{equation}
When $\chi>-(54\tau^2)^{\frac{1}{3}}$, the function $z\mapsto\phase(z;\chi,\tau)$ has a complex-conjugate pair of critical points $z=\critpt,\critpt^*$ depending real-analytically on $(\chi,\tau)$ with $\Im(\critpt)>0$, and in Section~\ref{s:boundary-curve} we show that the condition $\Re(\ii\phase(\critpt;\chi,\tau))=-1$ determines $\chi$ as an even real-analytic function of $\tau$ denoted $\chi_\mathrm{c}(\tau)$.  
This function satisfies $\chi_\mathrm{c}(\tau)>-(54\tau^2)^\frac{1}{3}$ and $\chi_\mathrm{c}(0)=\frac{1}{8}$, and $\chi=\chi_\mathrm{c}(\tau)$ satisfies the polynomial equation $\critpoly(\chi,\tau)=0$, where
\begin{multline}
\critpoly(\chi,\tau):=-4096\tau^4 + 1259712\tau^6+55296\tau^4\chi-186624\tau^4\chi^2\\
{}+69984\tau^4\chi^3-128\tau^2\chi^4+864\tau^2\chi^5+1296\tau^2\chi^6 -\chi^8+8\chi^9.
\label{eq:boundary-curve-exact}
\end{multline}
In the right-hand pane of Figure~\ref{f:Psi-elliptic}, the curve $(\chi,\tau)=(\chi_{\mathrm{c}}(\tau),\tau)$ is superimposed on the density plot of the solution $|\Psi(X,T;\mathbf{G}(a,b))|$ with $a=10^{-3}$ and $b=\sqrt{1-a^2}$, corresponding to $M\approx 3.45$.
\begin{figure}[h]
\includegraphics[width=0.48\linewidth]{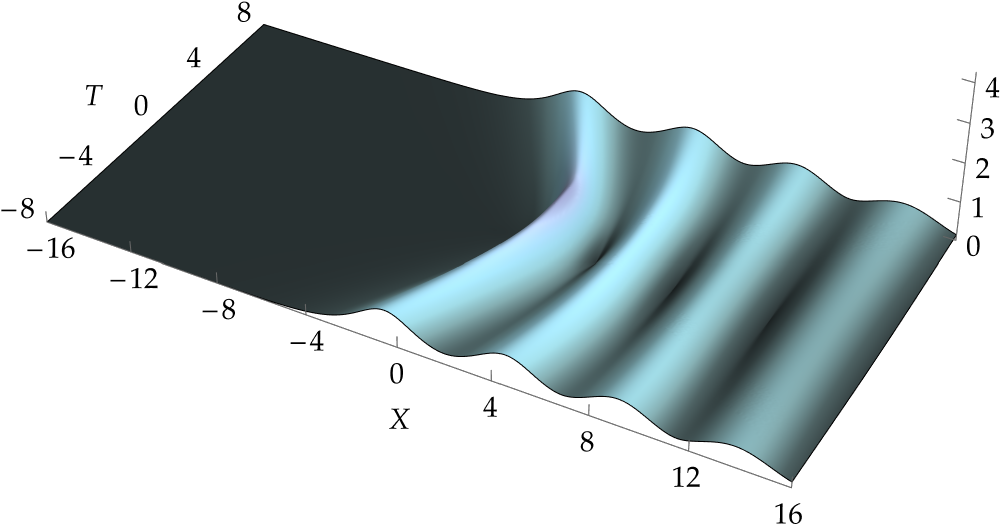}
\includegraphics[width=0.48\linewidth]{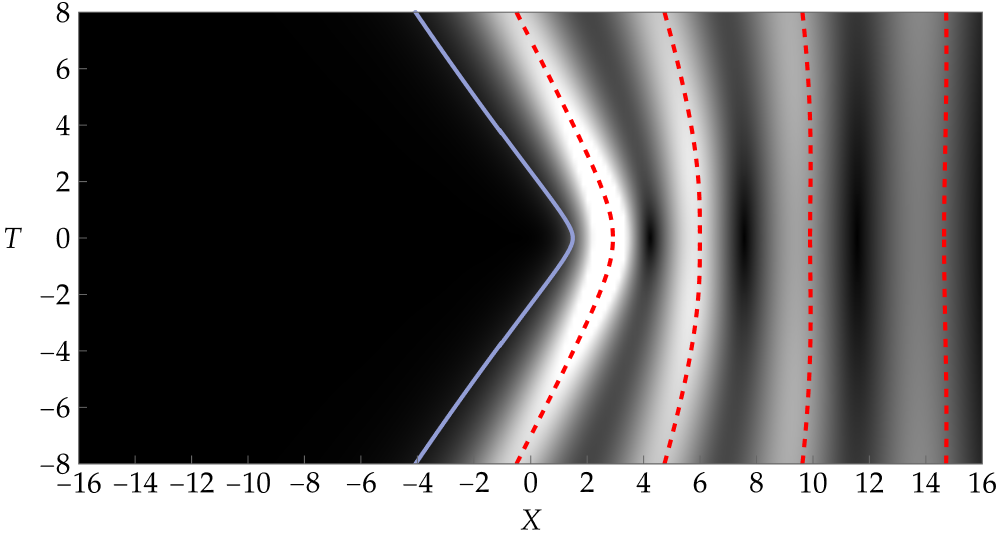}
\caption{Plots of $|\Psi(X,T;\mathbf{G}(a,b))|$ computed with the software package \texttt{RogueWaveInfinite.jl}, for $a=10^{-3}$ and $b=\sqrt{1-a^2}$ so $M=\frac{3}{2}\ln(10)\approx 3.45$. Left-pane: surface plot, right-pane: density plot with the critical curve $(\chi,\tau)=(\chi_{\mathrm{c}}(\tau),\tau)$ plotted in light blue in the $(X,T)$ plane. Also shown with red dashed lines are the curves $\Phi_n(\chi,\tau;M)=0$ for $n=0,1,2,3$ (see Theorem~\ref{t:edge}).}
\label{f:Psi-elliptic}
\end{figure}

In Section~\ref{s:DS-Integral-Condition} below we prove that there is a well-defined continuous and real-valued function $\lambda=\lambda(\chi,\tau)$ defined for $\chi>\chi_\mathrm{c}(\tau)$ with the following properties.  Firstly, when $\chi=\chi_\mathrm{c}(\tau)$, $-\chi/(2\tau)+\tau\lambda(\chi,\tau)/2$ is the unique real critical point of $z\mapsto\phase(z;\chi,\tau)$.  Secondly, for $\chi>\chi_\mathrm{c}(\tau)$, $\lambda$ is a smooth function of $(\chi,\tau)$ for which the quartic polynomial
\begin{equation}
R(z)^2 := z^4 + \tau\lambda z^3 + \left[\frac{3}{4}\tau^2\lambda^2-\frac{1}{2}\chi\lambda\right]z^2 +\frac{16\tau}{(\tau^2\lambda-\chi)^3}z + \frac{4}{(\tau^2\lambda-\chi)^2}
\label{eq:Intro-Rsquared}
\end{equation}
has a Schwarz-symmetric quartet of distinct complex roots $z=\alpha,\beta,\alpha^*,\beta^*$ labeled so that $\mathrm{Im}(\alpha)>0$, $\mathrm{Im}(\beta)>0$,  $\mathrm{Re}(\alpha)\le\mathrm{Re}(\beta)$ and if $\mathrm{Re}(\alpha)=\mathrm{Re}(\beta)$, then $\mathrm{Im}(\alpha)>\mathrm{Im}(\beta)$.  Thirdly, taking the branch cuts of $R(z)$ to be a Schwarz-symmetric pair of arcs, one of which joins $\alpha$ to $\beta$ in the upper half-plane, with the branch chosen so that $R(z)=z^2+O(z)$ as $z\to\infty$, the condition 
\begin{equation}
\int_{\alpha^*}^\alpha\frac{2\tau z +\chi-\tau^2\lambda}{z^2}R(z)\,\dd z = 2\ii
\label{eq:Int-int-hprime}
\end{equation}
holds (the integrand has zero residue at $z=0$ so the integral is independent of any Schwarz-symmetric path avoiding the origin).  Then, since the roots of $R(z)^2$ are distinct for $\chi>\chi_\mathrm{c}(\tau)$, $\alpha$ and $\beta$ are also smooth functions of $(\chi,\tau)$ on this region, and there exists a well-defined antiderivative $h(z)=h(z;\chi,\tau)$ of the integrand in \eqref{eq:Int-int-hprime} in the neighborhood of $z=\infty$ with the integration constant determined so that $h(z)=\phase(z;\chi,\tau)+O(z^{-1})$ as $z\to\infty$.  Although $z\mapsto h'(z)$ is meromorphic on $\mathbb{R}$ for $\chi>\chi_\mathrm{c}(\tau)$ with zero residue at the pole $z=0$, the continuation of $h(z)$ from a neighborhood of $z=\infty$ along $\mathbb{R}$ from $z>0$ and $z<0$ avoiding the pole results respectively in two different values $h_+(z)$ and $h_-(z)$ differing by a constant that we denote by $\Delta(\chi,\tau)$:
\begin{equation}
\Delta(\chi,\tau):=h_+(z;\chi,\tau)-h_-(z;\chi,\tau),\quad\chi>\chi_\mathrm{c}(\tau).
\label{eq:Int-Delta-def}
\end{equation}
This is a real-valued smooth function of $(\chi,\tau)$ on the indicated domain, and an explicit formula for it is given in \eqref{eq:DS-Delta-equation} below.  Its partial derivatives are explicitly given by (see \eqref{eq:DS-r3-r4} and \eqref{eq:DS-Ip}--\eqref{eq:DS-tau-derivs} below)
\begin{equation}
\frac{\partial\Delta}{\partial\chi}=-\left[\frac{1}{2\pi\ii}\int_{\beta}^{\beta^*}\frac{\dd z}{R(z)}\right]^{-1},\quad \frac{\partial\Delta}{\partial\tau}=\frac{1}{2}(\alpha+\beta+\alpha^*+\beta^*)\frac{\partial\Delta}{\partial\chi}.
\end{equation}
Assuming that $\mathrm{Re}(\alpha)<\mathrm{Re}(\beta)$, all four roots lie on a single circle with center $x\in\mathbb{R}$, and we may therefore define angles 
$\theta_\alpha:=\arg(\alpha(\chi,\tau)-x)$ and $\theta_\beta:=\arg(\beta(\chi,\tau)-x)$ with  $0<\theta_\beta<\theta_\alpha<\pi$.  From these, we define
\begin{equation}
m_1(\chi,\tau):=\frac{\sin(\theta_\alpha)\sin(\theta_\beta)}{\sin^2(\tfrac{1}{2}(\theta_\alpha+\theta_\beta))},
\label{eq:Intro-m1}
\end{equation}
and note that the definition extends by continuity to the case that $\mathrm{Re}(\alpha)=\mathrm{Re}(\beta)$ with the limiting value being given by $m_1=4\mathrm{Im}(\alpha)\mathrm{Im}(\beta)/(\mathrm{Im}(\alpha)+\mathrm{Im}(\beta))^2$.

With these ingredients, our main result describing $\Psi(X,T;\mathbf{G})$ for large $M$ is then the following.

\begin{theorem}[Asymptotic behavior of $\Psi(X,T;\mathbf{G}(a,b))$ for small $a$ / large $M$]
Assume that $a$ and $M>0$ are related by \eqref{M-def}.  Let $\compact$ be a compact subset of $\mathbb{R}^2$ with coordinates $(\chi,\tau)$ satisfying $\chi<\chi_\mathrm{c}(\tau)$.  Then in the limit $M\to+\infty$, 
$(\chi,\tau)\mapsto M\Psi(M^2\chi,M^3\tau;\mathbf{G}(a,b))$ is 
uniformly exponentially small on $\compact$. 
On the other hand, if $\compact$ is a compact subset of $(\chi,\tau)\in\mathbb{R}^2$ for which $\chi>\chi_\mathrm{c}(\tau)$, then in the same limit, 
\begin{equation}
M\Psi(M^2\chi,M^3\tau;\mathbf{G}(a,b))=\ee^{-\ii\arg(ab)}\breve{\Psi}(\chi,\tau;M) + O(M^{-1})
\label{eq:DS-locally-uniform}
\end{equation}
holds uniformly for $(\chi,\tau)\in \compact$.
Here, $\breve{\Psi}(\chi,\tau;M)$ is given explicitly by \eqref{eq:DS-Psi-breve-define}--\eqref{eq:DS-q-ratio} below, 
with modulus satisfying 
\begin{multline}
|\breve{\Psi}(\chi,\tau;M)|^2 = (\mathrm{Im}(\alpha(\chi,\tau))+\mathrm{Im}(\beta(\chi,\tau)))^2\\
{}-4\mathrm{Im}(\alpha(\chi,\tau))\mathrm{Im}(\beta(\chi,\tau))\mathrm{sn}^2\left(\frac{\mathbb{K}(m_1(\chi,\tau))}{\pi}(M\Delta(\chi,\tau)+\pi);m_1(\chi,\tau)\right),
\label{eq:Intro-square-modulus}
\end{multline}
where the elliptic parameter $m_1(\chi,\tau)\in (0,1)$ 
ranges from $m_1(\chi_\mathrm{c}(\tau),\tau)=1$ to $m_1(+\infty,\tau)=0$, 
$\mathbb{K}=\mathbb{K}(m)$ denotes the complete elliptic integral of the first kind, and $\alpha(\chi,\tau)$, $\beta(\chi,\tau)$ along with their complex conjugates constitute an exact solution of the elliptic quasilinear system \eqref{eq:DS-Whitham} of Whitham modulation equations in Riemann-invariant form  obtained by Forest and Lee \cite{ForestL86} for the rescaled focusing nonlinear Schr\"odinger equation in the form
\begin{equation}
\ii M^{-1}\frac{\partial q}{\partial \tau} +\frac{1}{2}M^{-2}\frac{\partial^2 q}{\partial\chi^2} +|q|^2q=0.
\label{eq:semiclassicalNLS}
\end{equation}
 Also, $\breve{\Psi}(\diamond,\tau;M)\in L^2((\chi_\mathrm{c}(\tau),+\infty))$ and 
\begin{equation}
\int_{\chi_\mathrm{c}(\tau)}^{+\infty}|\breve{\Psi}(\chi,\tau;M)|^2\,\dd\chi  = 8 + O(M^{-1}),\quad M\to+\infty.
\label{eq:Intro-approx-L2}
\end{equation}
\label{t:DS}
\end{theorem}
Note that according to \eqref{eq:Intro-square-modulus}, the amplitude $|\breve{\Psi}(\chi,\tau;M)|^2$ of the approximation valid for $\chi>\chi_\mathrm{c}(\tau)$ is oscillatory (wavelength proportional to $M^{-1}$ as a function of $\chi$, or to $M$ as a function of $X$) and varies between the lower bound of $|\mathrm{Im}(\alpha(\chi,\tau)-\beta(\chi,\tau))|$ and the upper bound of $\mathrm{Im}(\alpha(\chi,\tau)+\beta(\chi,\tau))$, both of which are $M$-independent 
functions of $(\chi,\tau)$
and hence have characteristic length scales of $M^2$ as functions of $X$.  It turns out that when $\tau=0$, one has $\beta(\chi,\tau)=-\alpha(\chi,\tau)^*$, so the lower bound vanishes.  The upper bound expressed in terms of $X=M^2\chi$  for $\chi>\chi_\mathrm{c}(0)=\frac{1}{8}$ is compared with plots of $\Psi(X,0;\mathbf{G})$ for two different small values of $a$ in Figure~\ref{f:Psi-a-small}.
\begin{figure}[h]
\includegraphics[width=\linewidth]{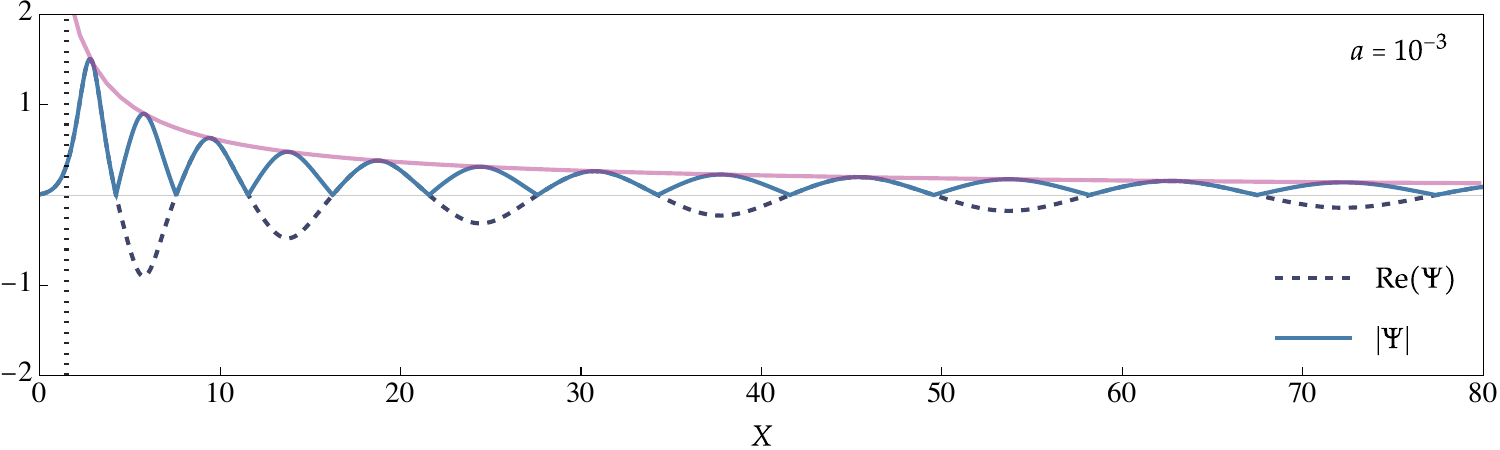}\\
\includegraphics[width=\linewidth]{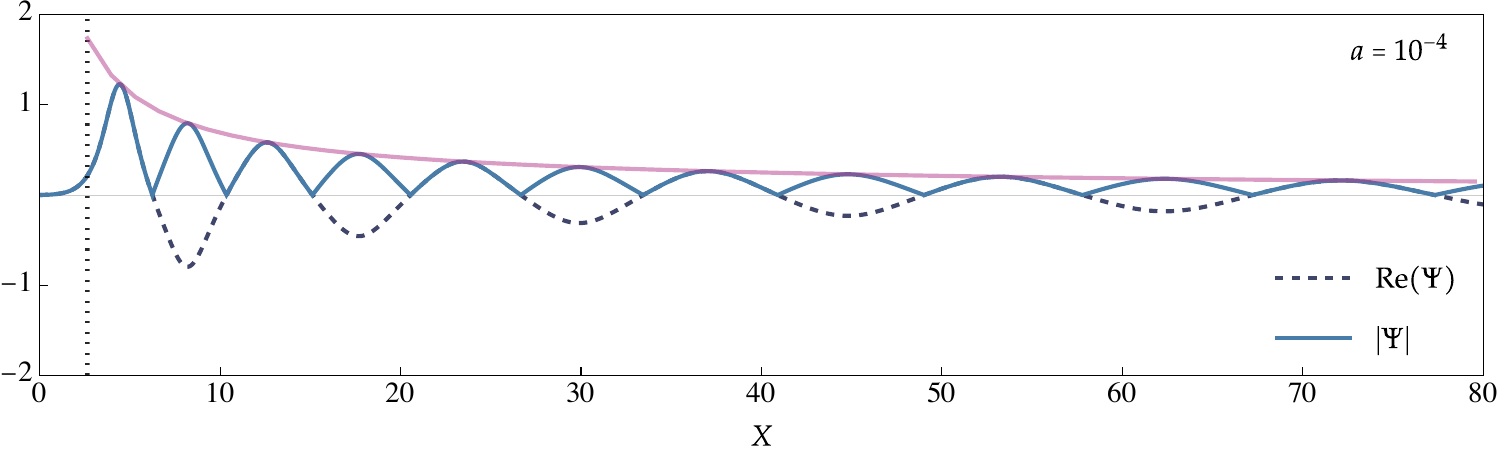}
\caption{Plots of $|\Psi(X,0;\mathbf{G})|$ (solid) and $\Re(\Psi(X,0;\mathbf{G}))$ (dashed) computed with \texttt{RogueWaveInfinite.jl} with $\mathbf{G}=\mathbf{G}(a,b)$, $b=\sqrt{1-a^2}$. Top: $a=10^{-3}$, bottom $a=10^{-4}$. Because $ab>0$, $\Psi(X,0;\mathbf{G})$ is real valued.  In both plots, the amplitude upper bound of $\mathrm{Im}(\alpha(M^{-2}X,0)+\beta(M^{-2}X,0))$ is plotted for $X>\frac{1}{8}M^2$ in pink.}
\label{f:Psi-a-small}
\end{figure}

One implication of this theorem is that, since each compact subset $\compact'$ of the $(X,T)$-plane is eventually contained as $M\to+\infty$ within a fixed closed disk $\compact$ centered at $(\chi,\tau)=(0,0)$ of sufficiently small radius in the $(\chi,\tau)$-plane that $\chi<\chi_\mathrm{c}(\tau)$ holds on $\compact$, we see that $(X,T)\mapsto\Psi(X,T;\mathbf{G}(a,b))$ tends uniformly to zero on compact sets $\compact'$ as $M\to+\infty$.  In this sense, general rogue waves of infinite order are small locally uniformly when $ab$ is small.  However, 
in light of Theorem~\ref{t:L2-norm} they cannot be small in $L^2(\mathbb{R})$, and \eqref{eq:Intro-approx-L2} shows that the leading approximation of $\Psi(X,T;\mathbf{G}(a,b))$ asserted in Theorem~\ref{t:DS} correctly captures the $L^2$-norm as the solution spreads and decays in the limit $a\to 0$.  Indeed, combining
$|\Psi(M^2\chi,M^3\tau;\mathbf{G}(a,b))|^2 = M^{-2}|\breve{\Psi}(\chi,\tau;M)|^2 + O(M^{-3})$  with $X=M^2\chi$ and hence $\dd X=M^2\dd\chi$, and neglecting contributions for $\chi<\chi_\mathrm{c}(\tau)$, we can approximate
\begin{equation}
\int_\mathbb{R}|\Psi(X,M^3\tau;\mathbf{G}(a,b))|^2\,\dd X \approx \int_{\chi_\mathrm{c}(\tau)}^{+\infty} |\breve{\Psi}(\chi,\tau;M)|^2\,\dd\chi=8,
\end{equation}
so the $L^2$-norm is not actually lost in the limit that $a\to 0$.  A more precise statement is:
\begin{corollary}[Convergence in $L^2(\mathbb{R})$]
We have
\begin{equation}
\lim_{M\to+\infty}\|M\Psi(M^2\diamond,M^3\tau;\mathbf{G}(a,b))-\ee^{-\ii\arg(ab)}\breve{\Psi}(\diamond,\tau;M)\|_{L^2(\mathbb{R})}=0,
\label{eq:L2-convergence}
\end{equation}
where $\breve{\Psi}(\chi,\tau;M)$ is extended by zero to $\chi<\chi_\mathrm{c}(\tau)$.
\label{cor:L2-convergence}
\end{corollary}
The proof will given in Section~\ref{sec:L2-convergence} below.
The bulk of Section~\ref{s:escape} is occupied with the proof of Theorem~\ref{t:DS}.

The small-$a$ asymptotic behavior described in Theorem~\ref{t:DS} is not uniform for $(\chi,\tau)$ on any neighborhood of a point $(\chi_\mathrm{c}(\tau),\tau)$ on the critical curve.
Our next result concerns the transitional regime where $(\chi,\tau)$ is allowed to penetrate into the region $\chi>\chi_{\mathrm{c}}(\tau)$ from the region $\chi<\chi_{\mathrm{c}}(\tau)$. To quantify the extent of this transition region in the $(\chi,\tau)$-plane, we recall the complex critical point $z=\critpt(\chi,\tau)$ of the phase function $z\mapsto \phase(z;\chi,\tau)=\chi z+\tau z^2-2 z^{-1}$ for $\chi>-\left(54 \tau^2\right)^{\frac{1}{3}}$ with $\Im(\critpt)>0$ and introduce the quantity
\begin{equation}
d(\chi,\tau):=-\ii(\phase(\critpt(\chi,\tau) ; \chi, \tau)-\ii).
\label{d-def-intro}
\end{equation}
The curve $(\chi,\tau)=(\chi_{\mathrm{c}}(\tau),\tau)$ coincides with the zero level $\Re(d(\chi,\tau))=0$ and the domain $\chi>\chi_{\mathrm{c}}(\tau)$ corresponds to the region $\Re(d(\chi,\tau))>0$.  Because $\critpt(\chi,\tau)$ is a simple critical point of $z\mapsto\phase(z;\chi,\tau)$, $\phase''(\critpt(\chi,\tau);\chi,\tau)\neq 0$, and in Section~\ref{s:local-param} below a certain complex-valued continuous function $\conformalprime(\chi,\tau)$ is defined with the property that $\conformalprime(\chi,\tau)^2=\ii\phase''(\critpt(\chi,\tau);\chi,\tau)$.  
Our theorem is then the following.

\begin{theorem}[Behavior of $\Psi(X,T;\mathbf{G}(a,b))$ at the edge for small $a$]
Assume that $a$ and $M>0$ are related by \eqref{M-def}.
Fix constants $K_->0$ sufficiently small and  $K_+>0$ arbitrary, and consider the region $\mathcal{S}$
of the $(\chi,\tau)$-plane defined by the inequalities
\begin{equation}
\mathcal{S}:=\left\{(\chi,\tau)\in\mathbb{R}^2: -K_-\le \Re(2 d (\chi,\tau))\le K_+\frac{\ln(M)}{M}\right\}.
\label{d-size}
\end{equation}
Also fix an arbitrary compact subset $\compact$ of the domain $\chi<\chi_\mathrm{c}(\tau)$.  For $n\in\mathbb{Z}_{\ge 0}$, set
\begin{equation}
D_n(\chi,\tau;M):=(-1)^n\frac{\ee^{2Md(\chi,\tau)}M^{-(n+\frac{1}{2})}}{2\pi\gamma_n^2(2\Im(\critpt(\chi,\tau))\conformalprime(\chi,\tau))^{2n+1}},
\label{eq:Dn-def}
\end{equation}
where
\begin{equation}
\gamma_n=\sqrt{\frac{2^n}{\sqrt{\pi}n!}},
\label{gamma-n-intro}
\end{equation}
and set 
\begin{equation}
\Phi_n(\chi,\tau;M):=\ln(|D_n(\chi,\tau;M)|),\quad \solitonphase_n(\chi,\tau;M):=\arg(D_n(\chi,\tau;M)).
\label{eq:intro-Phi-n-theta-n}
\end{equation}
Then defining
\begin{equation}
\psi_n(\chi,\tau;M):=2\mathrm{Im}(\critpt(\chi,\tau))\mathrm{sech}(\Phi_n(\chi,\tau;M))\ee^{\ii\solitonphase_n(\chi,\tau;M)},
\label{eq:modulated-soliton}
\end{equation}
the asymptotic formula
\begin{equation}
M\Psi(M^2\chi,M^3\tau;\mathbf{G}(a,b))=\ee^{-\ii\arg(ab)}\sum_{n=0}^{\lfloor K_+\rceil}\psi_n(\chi,\tau;M) + O(M^{-\frac{1}{2}}),\quad M\to+\infty
\end{equation}
holds uniformly for $(\chi,\tau)\in \mathcal{S}\cup\compact$.
Here, $\lfloor\diamond\rceil$ denotes the nearest integer function that rounds up at the half-integers.
\label{t:edge}
\end{theorem}

The proof is given in Section~\ref{s:edge}. We also have the following corollary to Theorem~\ref{t:edge} providing more information about the function $\psi_n(\chi,\tau;M)$ given by \eqref{eq:modulated-soliton}.
\begin{corollary}[Soliton approximation]
Let $(\chi_0,\tau_0)$ be a point maximizing $|\psi_n(\chi,\tau;M)|$, i.e., satisfying $\Phi_n(\chi_0,\tau_0;M)=0$.  Set $\critpt_0:=\critpt(\chi_0,\tau_0)$ and $\solitonphase_n^0:=\solitonphase_n(\chi_0,\tau_0;M)$.  Then
\begin{equation}
\psi_n(\chi_0+\chi,\tau_0+\tau;M)=q_n(\chi,\tau;M) + O(M^{-1}),\quad M\to+\infty
\end{equation}
holds uniformly for $\chi=O(M^{-1})$ and $\tau=O(M^{-1})$, where
\begin{equation}
q_n(\chi,\tau;M):=2\Im(\critpt_0)\mathrm{sech}(2M\Im(\critpt_0)[\chi +2\Re(\critpt_0)\tau])\ee^{\ii\solitonphase^0_n}\ee^{-2\ii M(\Re(\critpt_0)\chi+[\Re(\critpt_0)^2-\Im(\critpt_0)^2]\tau)}
\label{eq:q-define}
\end{equation}
is an exact soliton solution of the focusing nonlinear Schr\"odinger equation in the form \eqref{eq:semiclassicalNLS} corresponding to a simple eigenvalue $\critpt_0\in\mathbb{C}_+$ of the Zakharov-Shabat eigenvalue problem.
\label{cor:soliton}
\end{corollary}
The proof is given in Section~\ref{sec:local-behavior}.  These results show that the elliptic wave that occupies the domain $\chi>\chi_\mathrm{c}(\tau)$ according to Theorem~\ref{t:DS} degenerates at the boundary to a train of sech-shaped pulses each of width proportional to $M^{-1}$ and separated from its neighbors by a distance proportional to $M^{-1}\ln(M)$.  The pulses are modulated solitons confined to curvilinear trajectories $\Phi_n(\chi,\tau;M)=0$ (plotted in Figure~\ref{f:Psi-elliptic}) roughly parallel to the boundary curve $\chi=\chi_\mathrm{c}(\tau)$, and upon zooming in on a point on one of the trajectories at a scale proportional to the width of the corresponding pulse one sees an exact soliton solution of \eqref{eq:semiclassicalNLS} with the correct relationships between amplitude, wavelength, velocity, wavenumber, and frequency.
Note that the apparent drift of the curves $\Phi_n(\chi,\tau;M)=0$ for larger $n$ from the amplitude peaks in Figure~\ref{f:Psi-elliptic} is a finite-$M$ effect.

The remainder of the article is devoted to the proofs of these results, making use of the following proposition to work with renormalized versions of the parameters $(a,b)$.

\begin{proposition}[\protect{\cite[Proposition 1.5]{BilmanM2024}}]
For all $(X, T) \in \mathbb{R}^2$ and $(a,b) \in \mathbb{C}^2$ with $a b \neq 0$,
\begin{equation}
\Psi(X, T ; \mathbf{G}(a, b))=\ee^{-\ii \arg (a b)} \Psi(X, T ; \mathbf{G}(\mathfrak{a}, \mathfrak{b})),
\end{equation}
where normalized parameters defined by
\begin{equation}
\mathfrak{a}:=\frac{|a|}{\sqrt{|a|^2+|b|^2}} \quad \text { and } \quad \mathfrak{b}:=\frac{|b|}{\sqrt{|a|^2+|b|^2}}
\label{eq:frakafrakb}
\end{equation}
satisfy $\mathfrak{a}, \mathfrak{b}>0$ with $\mathfrak{a}^2+\mathfrak{b}^2=1$.
\label{p:a-b-scaling}
\end{proposition}

\subsection{Acknowledgements} D. Bilman was supported by the National Science Foundation on grant number DMS-2108029; P. D. Miller was supported by the National Science Foundation on grant numbers DMS-1812625, DMS-2204896, and DMS-2508694.
The authors would like to thank the Isaac Newton Institute for Mathematical Sciences for support and hospitality during the programme Dispersive Hydrodynamics (EPSRC Grant Number EP/R014604/1), as well as the University of Bristol where part of the work on this paper was undertaken.
The computations in this work were facilitated through the use of the advanced computational, storage, and networking infrastructure provided by the Ohio Supercomputer Center (48-core Pitzer nodes) \cite{OhioSupercomputerCenter1987}.

\section{Asymptotic behavior of $\Psi(X,T;\mathbf{G}(a,b))$ in the limit $a\to 0$:  the case $\chi\neq\chi_\mathrm{c}(\tau)$}
\label{s:escape}

To study the limit $a\to 0$, we assume without loss of generality (see Proposition~\ref{p:a-b-scaling}) that the parameters $(a,b)$ are replaced by normalized parameters $(\mathfrak{a},\mathfrak{b})$ with $\mathfrak{a}>0$ and $\mathfrak{b}>0$ and $\mathfrak{a}^2+\mathfrak{b}^2=1$, provided we restore a phase factor of $\ee^{-\ii\arg(ab)}$ at the end.
Since $\mathfrak{a}$ is small, it is useful to compare the solution $\mathbf{P}(\Lambda;X,T,\mathbf{G}(\mathfrak{a},\sqrt{1-\mathfrak{a}^2}))$ of Riemann-Hilbert Problem~\ref{rhp:near-field} with the limiting solution $\mathbf{P}(\Lambda;X,T,\mathbf{G}(0,1))$ which is given explicitly by (see \cite[Appendix A]{BilmanM2024}):
\begin{equation}
\mathbf{P}(\Lambda;X,T,\mathbf{G}(0,1))=\begin{cases}\displaystyle\begin{bmatrix}0 & -\ee^{-2\ii(\Lambda X+\Lambda^2 T)}\\\ee^{2\ii (\Lambda X+\Lambda^2T)} & 0\end{bmatrix}, &|\Lambda|<1,\\
\ee^{4\ii\Lambda^{-1}\sigma_3},&|\Lambda|>1.
\end{cases}
\label{eq:Pfor(a,b)=(0,1)}
\end{equation}
We therefore set
\begin{equation}
\mathbf{N}(\Lambda;X,T,\mathfrak{a}):=\mathbf{P}(\Lambda;X,T,\mathbf{G}(\mathfrak{a},\sqrt{1-\mathfrak{a}^2}))\mathbf{P}(\Lambda;X,T,\mathbf{G}(0,1))^{-1},\quad |\Lambda|\neq 1,\quad \mathfrak{a}\in [0,1].
\end{equation}
Then, $\Psi(X,T;\mathbf{G}(\mathfrak{a},\sqrt{1-\mathfrak{a}^2}))=2\ii\lim_{\Lambda\to\infty}\Lambda N_{12}(\Lambda;X,T,\mathfrak{a})$, and $\mathbf{N}(\Lambda;X,T,\mathfrak{a})$ satisfies the conditions of a related Riemann-Hilbert Problem:
\begin{rhp}
Let $(X,T)\in\mathbb{R}^2$, and let $\mathfrak{a}\in[0,1]$.  Find a $2\times 2$ matrix $\mathbf{N}(\Lambda)=\mathbf{N}(\Lambda;X,T,\mathfrak{a})$ with the following properties:
\begin{itemize}
\item[]\textbf{Analyticity:}  $\mathbf{N}(\Lambda)$ is analytic in $\Lambda$ for $|\Lambda|\neq 1$, and it takes continuous boundary values on the clockwise-oriented unit circle from the interior and exterior.
\item[]\textbf{Jump condition:}  The boundary values on the unit circle are related as follows:
\begin{equation}
\mathbf{N}_+(\Lambda)=\mathbf{N}_-(\Lambda)\ee^{-\ii(\Lambda X+\Lambda^2 T -2\Lambda^{-1})\sigma_3}\begin{bmatrix} \sqrt{1-\mathfrak{a}^2} & -\mathfrak{a}\\\mathfrak{a} & \sqrt{1-\mathfrak{a}^2}\end{bmatrix}\ee^{\ii(\Lambda X+\Lambda^2 T-2\Lambda^{-1})\sigma_3},\quad|\Lambda|=1.
\end{equation}
\item[]\textbf{Normalization:}  $\mathbf{N}(\Lambda)\to\mathbb{I}$ as $\Lambda\to\infty$.
\end{itemize}
\label{rhp:N}
\end{rhp}

In terms of $M>0$ defined by \eqref{M-def}, according to \eqref{eq:frakafrakb} we have
$\mathfrak{a}=\ee^{-2M}$.   Recall the scalings $(X,T,\Lambda)\mapsto (\chi,\tau,z)$ given by \eqref{chi-tau-def}--\eqref{eq:z-Lambda}.   Without changing the problem in any essential way, we may assume that the jump contour is any simple closed curve $\Gamma$ in the $z$-plane that surrounds the origin in the clockwise sense.  The equivalent Riemann-Hilbert problem for $\mathbf{S}(z;\chi,\tau,M):=\mathbf{N}(\Lambda,X,T,\mathfrak{a})=\mathbf{N}(M^{-1}z;M^2\chi,M^3\tau,\ee^{-2M})$ is then the following.  
\begin{rhp}
Let $(\chi,\tau)\in\mathbb{R}^2$ and let $M\in\mathbb{R}$.  Find a $2\times 2$ matrix $\mathbf{S}(z)=\mathbf{S}(z;\chi,\tau,M)$ with the following properties:
\begin{itemize}
\item[]\textbf{Analyticity:}  $\mathbf{S}(z)$ is analytic in $z$ for $z\in\mathbb{C}\setminus\Gamma$, and it takes continuous boundary values on $\Gamma$ from the interior and exterior.
\item[]\textbf{Jump condition:}  The boundary values are related by $\mathbf{S}_+(z)=\mathbf{S}_-(z)\mathbf{V}^\mathbf{S}(z)$ for $z\in\Gamma$, where
\begin{equation}
\mathbf{V}^\mathbf{S}(z)=\mathbf{V}^\mathbf{S}(z;\chi,\tau,M):=\begin{bmatrix}\sqrt{1-\ee^{-4M}} & -\ee^{-2\ii M(\phase(z;\chi,\tau)-\ii)}\\
\ee^{2\ii M(\phase(z;\chi,\tau)+\ii)} & \sqrt{1-\ee^{-4M}}\end{bmatrix},
\label{eq:S-jump}
\end{equation}
in which $\phase(z;\chi,\tau)$ is defined by \eqref{eq:DS-tildevartheta}.
\item[]\textbf{Normalization:}  $\mathbf{S}(z)\to\mathbb{I}$ as $z\to\infty$.
\end{itemize}
\label{rhp:S}
\end{rhp}
We assume that $\Gamma$ is Schwarz-symmetric (with reflection reversing the orientation) and independent of $M$.  Then, it is easy to see that whenever $\mathbf{S}(z;\chi,\tau,M)$ is a solution of \rhref{rhp:S}, so is $\sigma_2\mathbf{S}(z^*;\chi,\tau,M)^*\sigma_2$.  Since $\mathbf{P}(\Lambda;X,T,\mathbf{G})$ is uniquely determined from the conditions of \rhref{rhp:near-field}, it follows that also $\mathbf{S}(z;\chi,\tau,M)$ is uniquely determined from the conditions of \rhref{rhp:S}, and so the solution is necessarily Schwarz symmetric in the sense that
\begin{equation}
\mathbf{S}(z;\chi,\tau,M)=\sigma_2\mathbf{S}(z^*;\chi,\tau,M)^*\sigma_2
\label{eq:S-Schwarz}
\end{equation}
Since $\Psi(X,T;\mathbf{G}(\mathfrak{a},\sqrt{1-\mathfrak{a}^2}))=2\ii\lim_{\Lambda\to\infty}\Lambda N_{12}(\Lambda;X,T,\mathfrak{a})$, we have
\begin{equation}
M\Psi(M^2\chi,M^3\tau;\mathbf{G}(\ee^{-2M},\sqrt{1-\ee^{-4M}}))=2\ii\lim_{z\to\infty}zS_{12}(z;\chi,\tau,M).
\label{eq:DS-Psi-from-S}
\end{equation}

\subsection{Definition of $\chi_\mathrm{c}(\tau)$ and proof of Theorem~\ref{t:DS} for $\chi<\chi_\mathrm{c}(\tau)$}
\label{s:boundary-curve}
Noting the rapid convergence to $1$ of the diagonal elements of the jump matrix in \eqref{eq:S-jump}, we will have a small-norm problem if the contour $\Gamma$ can be placed such that  the rescaled exponents $\pm \ii (\phase(z;\chi,\tau)\pm\ii)$ of both off-diagonal entries have strictly negative real parts uniformly on $\Gamma$.  This will be the case if the region of the complex $z$-plane defined by the inequalities $-1<\mathrm{Re}(\ii\phase(z;\chi,\tau))<1$  contains a Jordan curve $\Gamma$ surrounding the origin, a condition that depends on the coordinates $(\chi,\tau)$.

Firstly, suppose that $\chi^3<-54\tau^2$.  Then a discriminant calculation shows that $\phase(z;\chi,\tau)$ has distinct real critical points (two if $\tau=0$ and three more generally).  One can show that in this situation, (i) there are two critical points $z=z_\pm$ having opposite signs:  $z_-<0<z_+$ such that the third real critical point lies outside of the interval $[z_-,z_+]$ and (ii) there is a zero level curve $\mathrm{Re}(\ii\phase(z;\chi,\tau))=0$ in the upper half-plane connecting the critical points $z_-,z_+$.   Taking $\Gamma$ to be the union of the latter curve with its Schwarz reflection we therefore have the desired strict inequality $-1<\mathrm{Re}(\ii\phase(z;\chi,\tau))<1$ holding uniformly for $z\in\Gamma$. See the first pane of Figure~\ref{f:vartheta-composite}.

When $\chi$ increases through the value $-(54\tau^2)^{\frac{1}{3}}$ and $\tau\neq 0$, the third real critical point coalesces with either $z_-$ or $z_+$ (if $\tau=0$ instead $z_-$ and $z_+$ coalesce at $z=0$) and the double real critical point bifurcates into a complex-conjugate pair denoted $\critpt=\critpt(\chi,\tau),\critpt^*$ with $\mathrm{Im}(\critpt)>0$.  Because the double real critical point for $\chi=-(54\tau^2)^\frac{1}{3}$ lies on the set $\Re(\ii\phase(z;\chi,\tau))=0\in (-1,1)$, there is a neighborhood of $(-(54\tau^2)^\frac{1}{3},\tau)$ in $\mathbb{R}^2$ on which it is still possible to place $\Gamma$ surrounding the origin within the region on which $-1<\mathrm{Re}(\ii\phase(z;\chi,\tau))<1$ holds, although it is generally no longer possible to choose $\Gamma$ to be a zero level curve of $\mathrm{Re}(\ii\phase(z;\chi,\tau))$.   Indeed, if $\chi>-(54\tau^2)^\frac{1}{3}$, the branch of the latter curve emanating into the upper half-plane from the remaining real root no longer returns to the real line. See the second pane of Figure~\ref{f:vartheta-composite}.
\begin{figure}[h]
\includegraphics[width=0.24\textwidth]{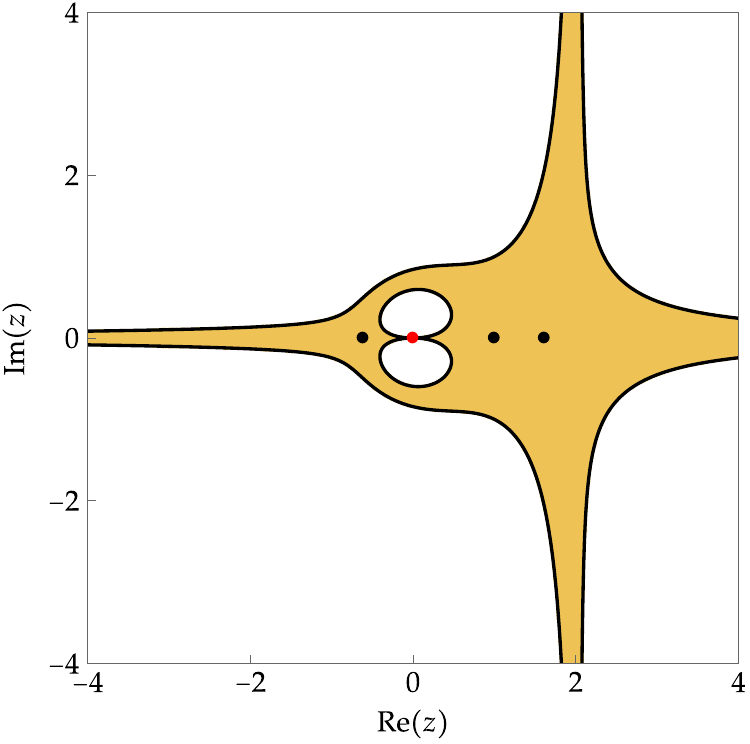}
\includegraphics[width=0.24\textwidth]{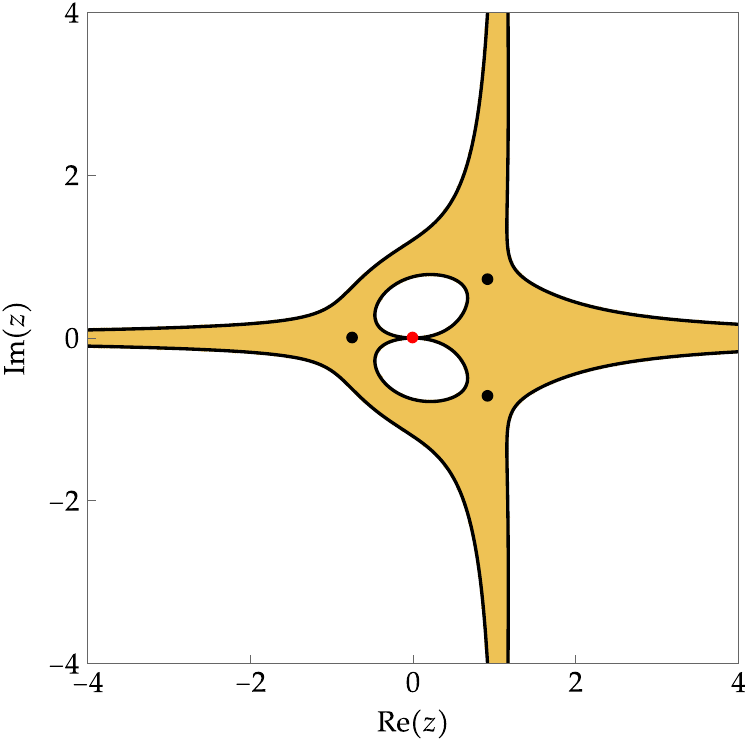}
\includegraphics[width=0.24\textwidth]{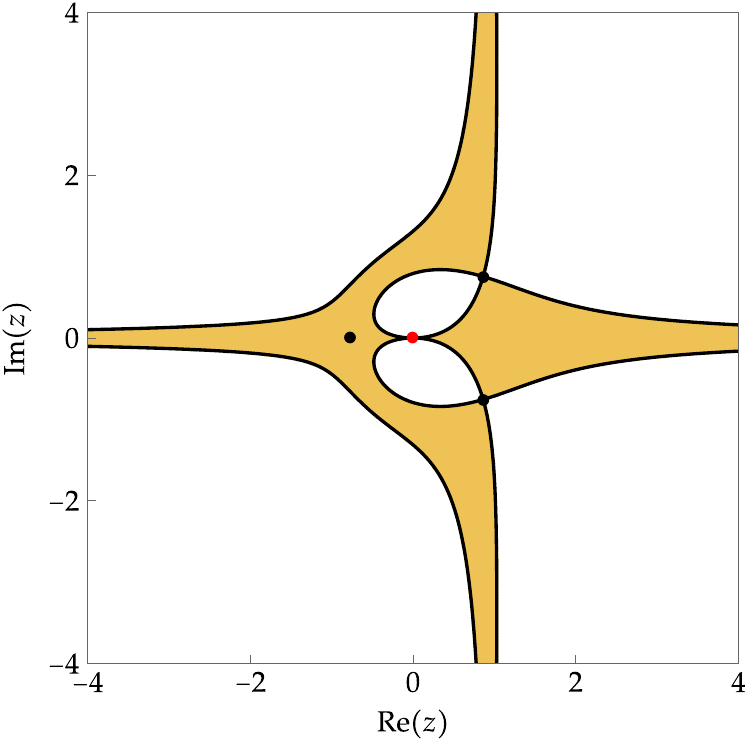}
\includegraphics[width=0.24\textwidth]{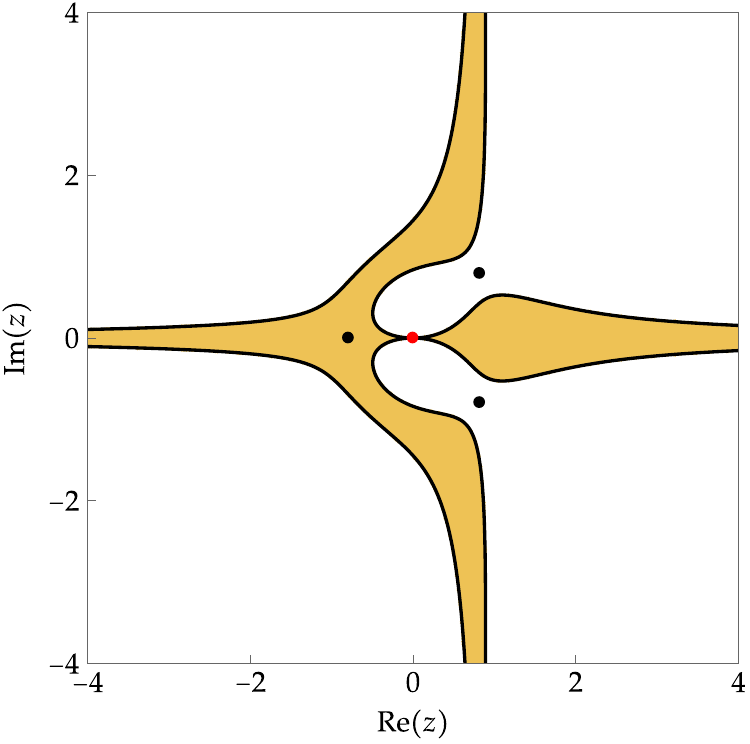}
\caption{The region $-1<\mathrm{Re}(\ii\phase(z;\chi,\tau))<1$ (shaded) for $\tau=1$. The critical points of $\phase(z;\chi,\tau)$ are marked with black dots. The red dot marks the singularity of $\phase(z;\chi,\tau)$ at $z=0$. From left to right, first pane: $\chi<-(54\tau^2)^\frac{1}{3}$, second pane $-(54\tau^2)^\frac{1}{3}<\chi<\chi_{\mathrm{c}}(\tau)$, third pane: $\chi=\chi_{\mathrm{c}}(\tau)$, and fourth pane: $\chi>\chi_{\mathrm{c}}(\tau)$.}
\label{f:vartheta-composite}
\end{figure}

However, given $\tau\in\mathbb{R}$, there is a maximum value of $\chi$ denoted $\chi_\mathrm{c}(\tau)$ with $\chi_\mathrm{c}(\tau)>-(54\tau^2)^{\frac{1}{3}}$ such that if $\chi>\chi_\mathrm{c}(\tau)$, the region $-1<\mathrm{Re}(\ii\phase(z;\chi,\tau))<1$ no longer contains a Jordan curve $\Gamma$ surrounding the origin.  
See the third and fourth panes of Figure~\ref{f:vartheta-composite}.
The transition occurs because a critical point of $\phase(z;\chi,\tau)$ moves onto one of the two level curves $\mathrm{Re}(\ii\phase(z;\chi,\tau))=\pm 1$, changing the topology of the indicated region in the $z$-plane so that it no longer contains a suitable contour $\Gamma$.  In fact, since $\phase(z;\chi,\tau)^*=\phase(z^*;\chi,\tau)$, the two conditions $\mathrm{Re}(\ii\phase(\critpt;\chi,\tau))=-1$ and $\mathrm{Re}(\ii\phase(\critpt^*;\chi,\tau))=1$ occur simultaneously for $\chi=\chi_\mathrm{c}(\tau)$.  These two equivalent conditions can be written in integral form as
\begin{equation}
\mathrm{Im}\left(\int_{\critpt^*}^\critpt\phase'(z;\chi,\tau)\,\dd z\right) = 2.
\label{eq:BreakingCurve}
\end{equation}
Using $\phase'(\critpt;\chi,\tau)=0$, it is easy to see that for each fixed $\tau\in\mathbb{R}$,
\begin{equation}
\frac{\partial}{\partial \chi}\mathrm{Im}\left(\int_{\critpt^*}^\critpt\phase'(z;\chi,\tau)\,\dd z\right)=\mathrm{Im}\left(\int_{\critpt^*}^\critpt\frac{\partial\phase'}{\partial \chi}(z;\chi,\tau)\,\dd z\right) = \mathrm{Im}\left(\int_{\critpt^*}^\critpt\,\dd z\right)=2\mathrm{Im}(\critpt)
>0.
\label{eq:k-chi-positive}
\end{equation}
The left-hand side of \eqref{eq:BreakingCurve} can be written as $k(\chi,\tau):=\Im(\phase(\critpt;\chi,\tau)-\phase(\critpt^*;\chi,\tau))=-\ii(\phase(\critpt;\chi,\tau)-\phase(\critpt^*;\chi,\tau))$.  Of course, when $\chi=-(54\tau^2)^{1/3}$, we have $k(\chi,\tau)=0<2$ since $\critpt=\critpt^*$.  For each fixed $\tau\in\mathbb{R}$, it is easy to calculate the asymptotic behavior of the complex critical point $\critpt$ as $\chi\to+\infty$; namely $\critpt=\ii\sqrt{2\chi^{-1}}+O(\chi^{-3/2})$.    Therefore $k(\chi,\tau)=4\sqrt{2\chi} + O(\chi^{-1/2})>2$ as $\chi\to+\infty$.  Since $\chi\mapsto k(\chi,\tau)$ is strictly increasing, it follows that there is a unique value $\chi=\chi_\mathrm{c}(\tau)>-(54\tau^2)^{1/3}$ consistent with \eqref{eq:BreakingCurve}.  

The function $\tau\mapsto\chi_\mathrm{c}(\tau)$ is differentiable.  Indeed, 
implicit differentiation of \eqref{eq:BreakingCurve} 
for $\chi=\chi_\mathrm{c}(\tau)$
with respect to $\tau$ gives
\begin{equation}
\chi_\mathrm{c}'(\tau)=-\dfrac{\displaystyle\int_{\critpt^*}^\critpt\phase_\tau'(z;\chi,\tau)\,\dd z}{\displaystyle\int_{\critpt^*}^\critpt\phase_\chi'(z;\chi,\tau)\,\dd z}=
-\dfrac{\displaystyle\int_{\critpt^*}^\critpt 2z\,\dd z}{\displaystyle\int_{\critpt^*}^\critpt\,\dd z}=-\frac{\critpt^2-\critpt^{*2}}{\critpt-\critpt^*} = -(\critpt+\critpt^*)=-2\mathrm{Re}(\critpt).
\label{chi-crit-prime}
\end{equation}

When $\tau=0$ and $\chi>0$, the complex critical point is explicit:  $\critpt(\chi,0)=\ii\sqrt{2\chi^{-1}}$.  Using $\phase(z;\chi,0)=\chi z -2z^{-1}$ shows that for $\chi>0$, $\phase(\critpt(\chi,0);\chi,0)=\ii\sqrt{8\chi}$.  Therefore, putting $\tau=0$ in \eqref{eq:BreakingCurve} gives $\sqrt{8\chi}=1$, that is, $\chi_\mathrm{c}(0)=\frac{1}{8}$.
It is easy to show using $\phase(z;\chi,-\tau)=-\phase(-z;\chi,\tau)$ that 
$\chi_\mathrm{c}(\diamond)$ is an even function and therefore also $\chi_\mathrm{c}'(0)=0$.
On the other hand,
the condition $\phase'(\critpt;\chi,\tau)=0$ is equivalent to the cubic equation $2\tau \critpt^3 + \chi \critpt^2 + 2=0$.  Taking the imaginary part shows that $\chi_\mathrm{c}'(\tau)=-2\Re(\critpt)$ cannot vanish if $\tau\neq 0$ and $\Im(\critpt)>0$.  
For $\tau\neq 0$
we can determine the well-defined sign of the continuous function $\tau\mapsto\chi_\mathrm{c}'(\tau)$ by considering the limits $\tau\to\pm\infty$.  The relevant dominant balance in the critical point equation leads to $2\tau\critpt^3+2\approx 0$ which has two complex roots satisfying $\mathrm{sgn}(\mathrm{Re}(\critpt))=\mathrm{sgn}(\tau)$.  Therefore also $\mathrm{sgn}(\chi_\mathrm{c}'(\tau))=-\mathrm{sgn}(\tau)$ for all $\tau\neq 0$.  
It follows that 
for all $\tau\neq 0$ we have the strict inequality $-(54\tau^2)^{1/3}<\chi_\mathrm{c}(\tau)<\frac{1}{8}$. 

In fact, it is possible to express $\chi_\mathrm{c}(\tau)$ 
in algebraic form as follows.  The condition $\phase(\critpt;\chi,\tau)-\phase(\critpt^*;\chi,\tau)=2\ii$ can be written explicitly in terms of $\critpt=u+\ii v$ as
\begin{equation}
v\cdot\left(\chi+2\tau u +\frac{2}{u^2+v^2}\right)=1.
\label{eq:integralcondition-uv}
\end{equation}
Likewise the statement that $\phase'(u+\ii v;\chi,\tau)=0$ can be split into its real part
\begin{equation}
\chi(u^2-v^2)+2\tau(u^3-3uv^2)+2=0
\label{eq:criticalpoint-real}
\end{equation}
and its imaginary part (cancelling $v\neq 0$):
\begin{equation}
2\chi u + 6\tau u^2 -2\tau v^2=0.
\label{eq:criticalpoint-imaginary}
\end{equation}
From \eqref{eq:criticalpoint-imaginary} we may explicitly eliminate $u^2$ in favor of $u$:
\begin{equation}
u^2=\frac{1}{3\tau}(\tau v^2-\chi u).
\end{equation}
This formula allows one to systematically reduce higher powers of $u$ to linear terms.  Indeed, substituting two consecutive times into \eqref{eq:criticalpoint-real} results in a linear equation for $u$ that is solved by
\begin{equation}
u=\frac{18\tau-8\chi\tau v^2}{\chi^2+48\tau^2v^2}.
\end{equation}
Substituting this expression for $u$ into \eqref{eq:integralcondition-uv} and \eqref{eq:criticalpoint-imaginary} and finding  common denominators results in
a $9^\mathrm{th}$ degree and $6^\mathrm{th}$ degree polynomial equation in $v$ respectively.  The resultant between these two polynomials is the condition on $(\chi,\tau)$ that they are simultaneously solvable; this resultant has factors $\tau^{37}\neq 0$, $(108\tau^2+\chi^3)^{12}$ which vanishes for $\chi=-(108\tau^2)^{\frac{1}{3}}\le -(54\tau^2)^{\frac{1}{3}}$, and finally $\critpoly(\chi,\tau)$ given by \eqref{eq:boundary-curve-exact} in Section~\ref{s:introduction}. 
It is straightforward to check by explicit calculation that $\critpoly(\frac{1}{8},0)=0$, while
$\critpoly_\chi(\frac{1}{8},0)\neq 0$, so locally the implicit function theorem applies to yield $\chi=\chi_\mathrm{c}(\tau)$. 
From this representation it is easy to see that for large $\tau$, 
$\chi_\mathrm{c}(\tau)\sim -(54\tau^2)^\frac{1}{3}$, 
even though the strict inequality $\chi_\mathrm{c}(\tau)>-(54\tau^2)^\frac{1}{3}$ holds for all $\tau\in\mathbb{R}$.
The curves $\chi=\chi_\mathrm{c}(\tau)$ and $\chi=-(54\tau^2)^\frac{1}{3}$ 
are plotted in Figure~\ref{f:critical-curves}.
\begin{figure}[h]
\includegraphics[width=0.3\textwidth]{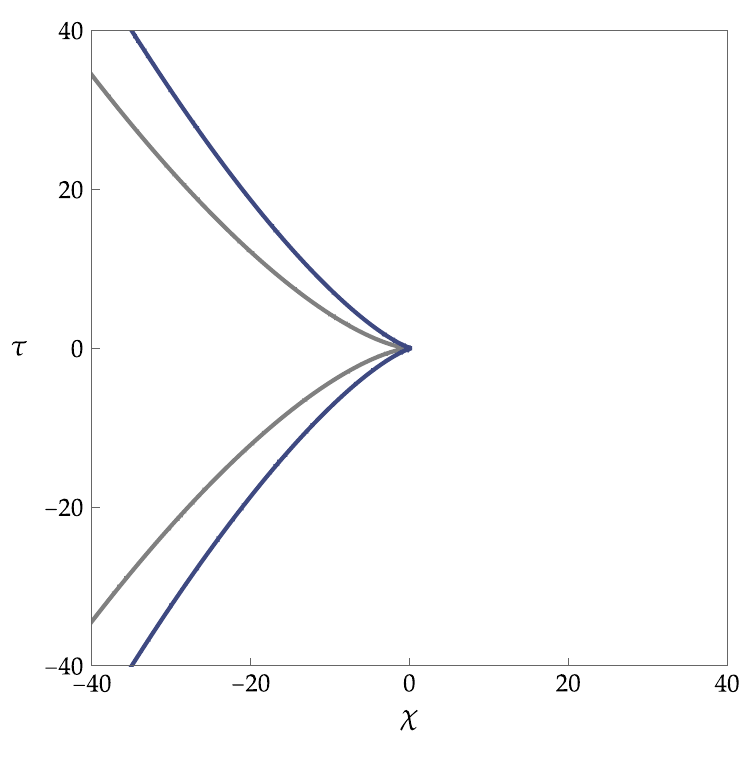}\quad\includegraphics[width=0.3\textwidth]{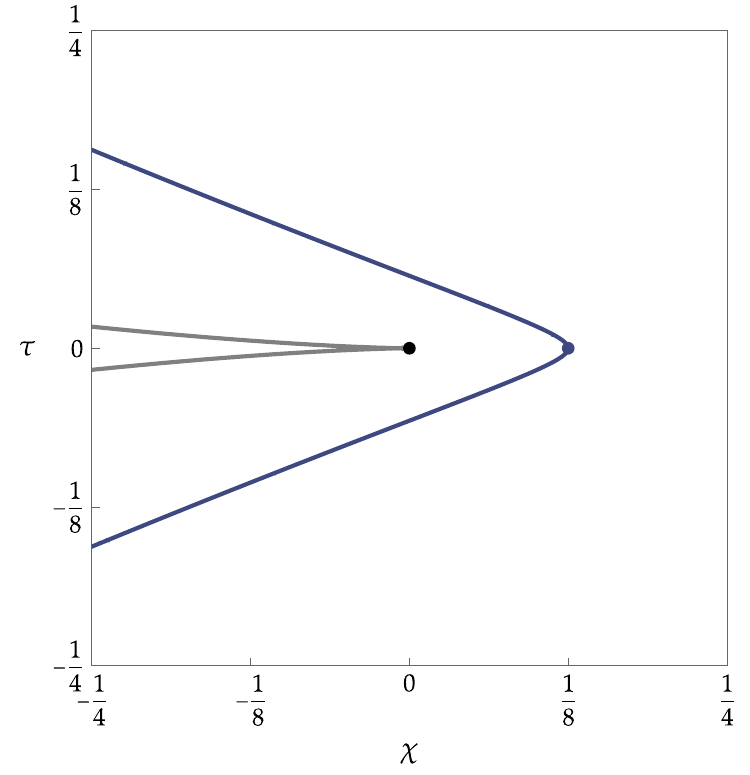}\quad\includegraphics[width=0.3\textwidth]{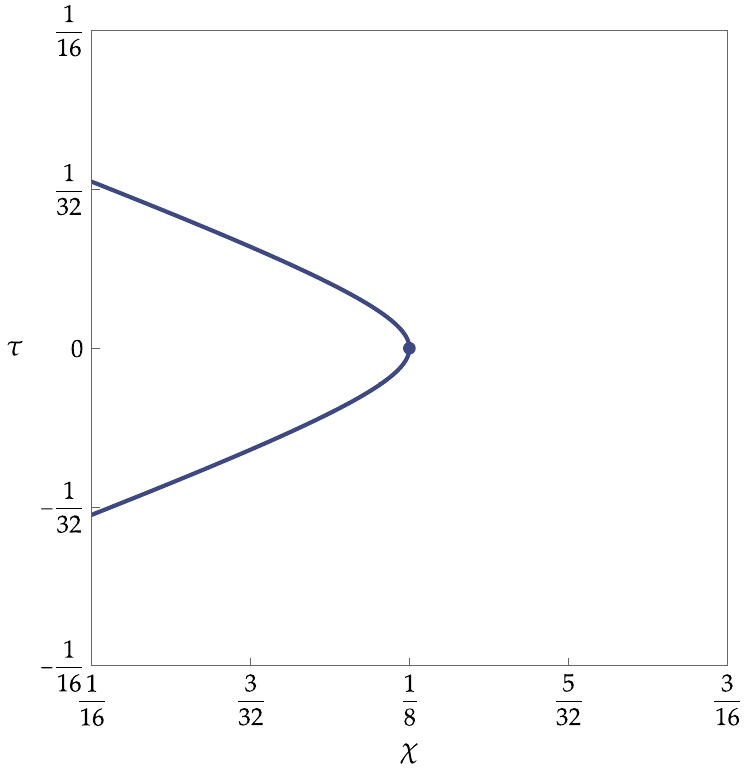}
\caption{Plots showing the curve 
$\chi=\chi_\mathrm{c}(\tau)$
(dark blue) and the curve $\chi=-(54\tau^2)^\frac{1}{3}$ (gray), to the left of which $z\mapsto\phase(z;\chi,\tau)$ has real critical points, at different scales. 
}
\label{f:critical-curves}
\end{figure}

The coordinates $(\chi,\tau)$ of the point on the graph $\chi=\chi_\mathrm{c}(\tau)$ 
can be expressed in terms of the corresponding critical point $\critpt=u+\ii v$ with the use of \eqref{eq:criticalpoint-real}--\eqref{eq:criticalpoint-imaginary}:
\begin{equation}
\chi = \frac{2v^2-6u^2}{(u^2+v^2)^2}\quad\text{and}\quad\tau=\frac{2u}{(u^2+v^2)^2}.
\label{eq:chi-tau-sigma}
\end{equation}
Finally, when $\chi=\chi_\mathrm{c}(\tau)$ and $\tau\neq 0$, there is a third real critical point of $z\mapsto\phase(z;\chi,\tau)$ that can be written in the form $z=-\chi/(2\tau)+\tau\lambda/2$, so since the product of the critical points is $-\tau^{-1}$ we obtain a parametrization of $\lambda$ in the form
\begin{equation}
\lambda=\frac{1}{\tau^2}\left(\chi-\frac{2}{|\critpt|^2}\right) = -2(u^2+v^2)^2
\label{eq:gamma-negative}
\end{equation}
where we have used \eqref{eq:chi-tau-sigma}.

Now we give the proof of Theorem~\ref{t:DS} in the case $\chi<\chi_\mathrm{c}(\tau)$.  Under this condition,
a Jordan curve $\Gamma$ exists enclosing the origin on which $-1<\mathrm{Re}(\ii\phase(z;\chi,\tau))<1$ holds and hence Riemann-Hilbert Problem~\ref{rhp:S} is of exponentially small-norm type in the limit $M\to+\infty$, locally uniformly with respect to $(\chi,\tau)$.  The statement in Theorem~\ref{t:DS} that $M\Psi(M^2\chi,M^3\tau;\mathbf{G}(a,b))$ is uniformly exponentially small as $M\to+\infty$ on compact subsets of $\chi<\chi_\mathrm{c}(\tau)$ then follows from \eqref{eq:DS-Psi-from-S}.

\subsection{The spectral curve}
\label{s:DS-spectral-curve}
When $\chi>\chi_\mathrm{c}(\tau)$, to control the exponential factors on the off-diagonal in the jump matrix in \eqref{eq:S-jump} we use the technique of multiplying $\mathbf{S}(z;\chi,\tau,M)$ on the right by a diagonal matrix factor $\ee^{\ii Mg(z)\sigma_3}$ where $g(z)=g(z;\chi,\tau)=g(z^*;\chi,\tau)^*$ is to be analytic except for certain arcs of $\Gamma$ along which either $g_+(z)-g_-(z)$ or $g_+(z)+g_-(z)+2\phase(z)$ is constant, and we require $g(z)\to 0$ as $z\to\infty$.  For the method to be effective, $g(z)$ cannot vanish identically.  Setting 
\begin{equation}
h(z)=h(z;\chi,\tau):=g(z;\chi,\tau)+\phase(z;\chi,\tau), 
\label{eq:DS-h-g-tildevartheta}
\end{equation}
we then see that $h'(z)^2$ is analytic except at $z=0$, and it has the asymptotic behavior
\begin{equation}
h'(z)^2 = 4z^{-4}+O(z^{-2}),\quad z\to 0,\quad\text{and}
\end{equation}
\begin{equation}
h'(z)^2 = 4\tau^2z^2+4\tau\chi z + \chi^2 + O(\tau z^{-1}) + O(\chi z^{-2}) + O(z^{-4}),\quad z\to\infty,
\end{equation}
in which the error terms cannot be made more precise without further knowledge of $g'(z)$.
Therefore, by Liouville's theorem
\begin{equation}
h'(z)^2 = z^{-4}P(z),\quad P(z):=4\tau^2z^6 + 4\tau\chi z^5 + \chi^2 z^4 + C_3z^3 + C_2z^2 + 4
\label{eq:P-sextic}
\end{equation}
where $C_2$ and $C_3$ are real coefficients independent of $z$, and where $C_3=0$ when $\tau=0$.  The relation $y^2=z^{-4}P(z)$ is said to define the \emph{spectral curve} for the problem.

Note that $g(z)$ vanishes identically if and only if $P(z)=(2\tau z^3+\chi z^2+2)^2$ is a perfect square.  Therefore we are only interested in the case that some of the roots of $P(z)$ are simple.  Also in the special case that $\tau=0$, $P$ is quartic instead of sextic:  $P(z)=\chi^2z^4+C_2z^2+4$.

\subsection{Integral condition}
\label{s:DS-Integral-Condition}
In fact, for $\chi>\chi_\mathrm{c}(\tau)$, we will 
assume that $P(z)$ has four simple roots forming a complex quartet and denoted $z=\alpha,\beta,\alpha^*,\beta^*$ with $\alpha\neq \beta$ and $\mathrm{Im}(\alpha)>0$ and $\mathrm{Im}(\beta)>0$, and that if $\tau\neq 0$ the remaining two roots are repeated and real, say $z=\gamma/(2\tau)$.  In the latter case we also require that $C_3$ and $C_2$ are related so that the discriminant of $P(z)$ vanishes.  So, whether or not $\tau=0$ it remains to determine $C_2\in\mathbb{R}$ as a function of $(\chi,\tau)$.  To this end, we impose the integral condition
\begin{equation}
\mathrm{Im}\left(\int_{\alpha^*}^\alpha h'(z)\,\dd z\right)=2,
\label{eq:integralcondition}
\end{equation}
which is the correct analogue in this setting of \eqref{eq:BreakingCurve}.
We note that since $h(z;\chi,\tau)=g(z;\chi,\tau)+\phase(z;\chi,\tau)$ and $g$ is analytic at $z=0,\infty$ with $g(z)\to 0$ as $z\to\infty$ it follows immediately that 
\begin{equation}
\mathop{\mathrm{Res}}_{z=0}h'(z;\chi,\tau) =\mathop{\mathrm{Res}}_{z=\infty}h'(z;\chi,\tau)=0.
\end{equation}
This along with the Schwarz symmetry $h'(z^*;\chi,\tau)=h'(z;\chi,\tau)^*$ implies that if we take a contour of integration from $z=\alpha$ to $z=\beta$ that avoids the branch cuts of $h'(z;\chi,\tau)$ except at the endpoints, then without further conditions,
\begin{equation}
\mathrm{Im}\left(\int_\alpha^\beta h'(z)\,\dd z\right)=0.
\end{equation}
Consequently, it makes no difference if $(\alpha,\alpha^*)$ are replaced with $(\beta,\beta^*)$ in \eqref{eq:integralcondition}.

If $\tau\neq 0$, we assume that $P(z)$ as given by \eqref{eq:P-sextic} has a double real root that we write in the form $z=\gamma/(2\tau)$ for some $\gamma\in\mathbb{R}$ to be determined, and no further real roots.  
We then match $P(z)=(2\tau z-\gamma)^2(z^4-S_1z^3+S_2z^2-S_3z+S_4)$ with the form given in \eqref{eq:P-sextic} to equate the symmetric homogeneous polynomials $S_p$ of degree $p$ in the four generically simple complex roots $z=\alpha,\beta,\alpha^*,\beta^*$ to explicit expressions in $(\chi,\tau,\gamma)$:
\begin{equation}
\begin{split}
S_1:=\alpha+\alpha^*+\beta+\beta^* &= -\frac{\chi+\gamma}{\tau}
\\
S_2:=\alpha\alpha^* + \alpha\beta + \alpha\beta^* + \alpha^*\beta+\alpha^*\beta^*+\beta\beta^*&=\frac{(\chi+\gamma)(\chi+3\gamma)}{4\tau^2}
\\
S_3:=\alpha\alpha^*\beta+\alpha\alpha^*\beta^*+\alpha\beta\beta^*+\alpha^*\beta\beta^*&=-\frac{16\tau}{\gamma^3}
\\
S_4:=\alpha\alpha^*\beta\beta^*&=\frac{4}{\gamma^2}
.
\end{split}
\label{eq:Ss-gamma-chi-tau}
\end{equation}
To match the asymptotic that $h'(z)=2\tau z + O(1)$ as $z\to\infty$ we write 
\begin{equation}
h'(z)=\frac{2\tau z-\gamma}{z^2}R(z),\quad R(z)^2 = z^4-S_1z^3+S_2z^2-S_3z+S_4,\quad S_p=S_p(\gamma,\chi,\tau),
\label{eq:hprime-Rsquared}
\end{equation}
(cf., \eqref{eq:Intro-Rsquared}) where $R(z)$ is analytic except for a branch cut connecting $z=\alpha$ and $z=\beta$ in the upper half-plane and its reflection in the real axis, and where $R(z)=z^2+O(z)$ as $z\to\infty$.  

At this point, to be able to include the $\tau=0$ case in the same framework it is convenient to parametrize $\gamma$ by a new parameter $\lambda$ so that
$\gamma=-\chi+\tau^2\lambda$.  Then the coefficients of $R(z)^2$ become
\begin{equation}
S_1=-\tau\lambda,\quad S_2=\frac{3}{4}\tau^2\lambda^2-\frac{1}{2}\chi\lambda ,\quad S_3=-\frac{16\tau}{(\tau^2\lambda-\chi)^3},\quad S_4=\frac{4}{(\tau^2\lambda-\chi)^2}.
\label{eq:SymmetricPolynomials-lambda}
\end{equation}
Indeed, taking the limit $\tau\to 0$ for fixed $\chi$, $\lambda$, and $z\neq 0$, $R(z)^2\to z^4+\chi^{-2}C_2z^2 + 4\chi^{-2}$ and $h'(z)^2\to \chi^2(z^4+\chi^{-2}C_2z^2 + 4\chi^{-2})z^{-4}$ where $C_2=-\frac{1}{2}\chi^3\lambda$.

\begin{lemma}
Let $\tau\in\mathbb{R}$ and suppose that $\chi\ge\chi_\mathrm{c}(\tau)$.  Then $\gamma=-\chi+\tau^2\lambda\neq 0$.
\label{lem:gamma-nonzero}
\end{lemma}
\begin{proof}
If $\tau\neq 0$, expanding out $(2\tau z-\gamma)^2R(z)^2$ and comparing with $P(z)$ in the form \eqref{eq:P-sextic} shows that for all finite values of $C_2,C_3$ one has $\gamma\neq 0$.  If $\tau=0$, then $\gamma=-\chi<-\chi_\mathrm{c}(0)=-\frac{1}{8}<0$.
\end{proof}

Using $h'(z^*)=h'(z)^*$, the integral condition \eqref{eq:integralcondition} then can be written as $f(\lambda;\chi,\tau)=0$, where
\begin{equation}
f(\lambda;\chi,\tau):=-\ii\left(\int_{\alpha^*}^\alpha h'(z)\,\dd z-2\ii\right)=-\ii\left(\int_{\alpha^*}^\alpha\frac{2\tau z+\chi-\tau^2\lambda}{z^2}R(z)\,\dd z-2\ii\right).
\label{eq:integralcondition-rewrite}
\end{equation}
We think of this as an equation to be solved for $\lambda$ given $(\chi,\tau)\in\mathbb{R}^2$ with $\tau\neq 0$ and $\chi>\chi_\mathrm{c}(\tau)$. 
A computation using \eqref{eq:Ss-gamma-chi-tau}--\eqref{eq:hprime-Rsquared} shows that 
\begin{equation}
\frac{\partial R}{\partial\lambda}(z)=\frac{2\tau(\tau^2\lambda-\chi)^4 z^3+(\tau^2\lambda-\chi)^4(3\tau^2\lambda-\chi)z^2-96\tau^3z-16\tau^2(\tau^2\lambda-\chi)}{4(\tau^2\lambda-\chi)^4R(z)}.
\label{eq:gamma-deriv-1}
\end{equation}
From this it follows that
\begin{equation}
\frac{\partial h'}{\partial\lambda}(z)=-\frac{\Qpoly(\lambda;\chi,\tau)}{4(\tau^2\lambda-\chi)^4R(z)},\quad
\Qpoly(\lambda;\chi,\tau):=6(\tau^2\lambda-\chi)^6+6\chi(\tau^2\lambda-\chi)^5+\chi^2(\tau^2\lambda-\chi)^4+192\tau^4.
\label{eq:gamma-deriv-2}
\end{equation}
Therefore, using the fact that $h'(\alpha)=h'(\alpha^*)=0$,
\begin{equation}
\frac{\partial f}{\partial\lambda}(\lambda;\chi,\tau)=\ii\frac{\Qpoly(\lambda;\chi,\tau)}{4(\tau^2\lambda-\chi)^4}\int_{\alpha^*}^\alpha\frac{\dd z}{R(z)}.
\label{eq:gamma-deriv-3}
\end{equation}
The integral factor is a complete elliptic integral of the first kind.  It cannot vanish, and by Schwarz symmetry of $R(z)$, it is purely imaginary.  By artificially deforming the branch points $\alpha,\alpha^*,\beta,\beta^*$ to a configuration symmetric about the imaginary axis (so $\beta=-\alpha^*$), one can show that the integral is purely positive imaginary.  Consequently, $f_\lambda(\lambda;\chi,\tau)=0$ if and only if $\Qpoly(\lambda;\chi,\tau)=0$, and $\mathrm{sgn}(f_\lambda(\lambda;\chi,\tau))=-\mathrm{sgn}(\Qpoly(\lambda;\chi,\tau))$.

To determine how the configuration of roots of $R(z)^2$ depends on $\lambda$ and $(\chi,\tau)$ with $\chi>\chi_\mathrm{c}(\tau)$, we now analyze the
\emph{discriminant locus}, i.e., the points in the $(\chi,\lambda)$-plane for fixed $\tau$ at which the discriminant $\mathscr{D}(\chi,\tau,\lambda)$ of $R(z)^2$  
vanishes.  
The discriminant is given explicitly by
\begin{equation}
\mathscr{D}(\chi,\tau,\lambda)=\frac{\mathscr{D}_2^+(\chi,\tau,\lambda)^2\mathscr{D}_2^-(\chi,\tau,\lambda)^2\mathscr{D}_1(\chi,\tau,\lambda)}{2(\chi-\tau^2\lambda)^{12}}
\label{eq:DS-discriminant-factors}
\end{equation}
in which
\begin{equation}
\begin{split}
\mathscr{D}_2^\pm(\chi,\tau,\lambda)&:=(\chi-\tau^2\lambda)^2\lambda\pm 8\\
\mathscr{D}_1(\chi,\tau,\lambda)&:=27(\tau^2\lambda-\chi)^5\tau^2\lambda+9\chi^2(\tau^2\lambda-\chi)^4+\chi^3(\tau^2\lambda-\chi)^3-864\tau^4.
\end{split}
\label{eq:DS-discriminant-factors-def}
\end{equation}
Neither of the double factors $\mathscr{D}_2^\pm(\chi,\tau,\lambda)$ nor the simple factor $\mathscr{D}_1(\chi,\tau,\lambda)$ can vanish for $\lambda=\tau^{-2}\chi$.  (Neither of the double factors can vanish for $\lambda=0$ either.)  Hence given $\tau\neq 0$, equating the three factors to zero gives a system of curves in the $(\chi,\lambda)$-plane, none of which can intersect the line $\lambda=\tau^{-2}\chi$.  In the limit $\chi\to+\infty$, it is straightforward to determine all of the real solutions $\lambda=\lambda_k(\chi,\tau)$ of $\mathscr{D}=0$ at fixed $\tau\neq 0$; in order of increasing $\lambda$ (i.e., $\lambda_j(\chi,\tau)<\lambda_k(\chi,\tau)$) these are:
\begin{itemize}
\item
$\lambda_1=-8\chi^{-2}+128\tau^2\chi^{-5}+O(\chi^{-8})$, a root of $\mathscr{D}_2^+(\chi,\tau,\lambda)$;
\item
$\lambda_2=8\chi^{-2}+128\tau^2\chi^{-5}+O(\chi^{-8})$, a root of $\mathscr{D}_2^-(\chi,\tau,\lambda)$;
\item 
$\lambda_3=\frac{2}{3}\chi\tau^{-2}-(864\tau^{-2})^\frac{1}{3}\chi^{-1}+O(\chi^{-3})$, a root of $\mathscr{D}_1(\chi,\tau,\lambda)$;
\item
$\lambda_4=\chi\tau^{-2}-\sqrt{8}|\tau|^{-1}\chi^{-\frac{1}{2}}+O(\chi^{-2})$, a root of $\mathscr{D}_2^-(\chi,\tau,\lambda)$;
\item
$\lambda_5=\chi\tau^{-2}+(864\tau^{-2})^\frac{1}{3}\chi^{-1}+O(\chi^{-3})$, a root of $\mathscr{D}_1(\chi,\tau,\lambda)$;
\item
$\lambda_6=\chi\tau^{-2}+\sqrt{8}|\tau|^{-1}\chi^{-\frac{1}{2}}+O(\chi^{-2})$, a root of $\mathscr{D}_2^-(\chi,\tau,\lambda)$.
\end{itemize}
The other two roots of the double factor $\mathscr{D}_2^+(\chi,\tau,\lambda)$ and the remaining four roots of the simple factor $\mathscr{D}_1(\chi,\tau,\lambda)$ are non-real for $\chi>0$ sufficiently large.   Now considering finite $\chi$, we note:
\begin{itemize}
\item
The discriminant of $\mathscr{D}_2^+(\chi,\tau,\lambda)$ with respect to $\lambda$ is $-32\tau^6(54\tau^2+\chi^3)$.  This is strictly negative for $\tau\neq 0$ and $\chi>\chi_\mathrm{c}(\tau)$, and hence $\lambda_1(\chi,\tau)$ remains the only real root of this factor on the whole interval $\chi>\chi_\mathrm{c}(\tau)$.  
\item
The discriminant of $\mathscr{D}_2^-(\chi,\tau,\lambda)$ with respect to $\lambda$ is $-32\tau^6(54\tau^2-\chi^3)$.  This will change sign from positive to negative as $\chi$ decreases through the value 
\begin{equation}
\chi=\chi_0(\tau):=(54\tau^2)^{\frac{1}{3}}>0,
\end{equation}
which lies in the interval $\chi>\chi_\mathrm{c}(\tau)$ provided $|\tau|$ is sufficiently large\footnote{Computing the resultant of $\critpoly(\chi,\tau)$ given by \eqref{eq:boundary-curve-exact} (whose zero locus contains the curve $\chi=\chi_\mathrm{c}(\tau)$) and $\chi^3-54\tau^2$ with respect to $\chi$ gives a cubic polynomial in $\tau^2$ with a unique positive root that determines the threshold value of $|\tau|$ beyond which $(54\tau^2)^{\frac{1}{3}}>\chi_\mathrm{c}(\tau)$, namely $|\tau|=|\widehat{\tau}|\approx 0.00573703$.}.  In this case, as $\chi$ decreases from $+\infty$ through the value $\chi_0(\tau)>\chi_\mathrm{c}(\tau)$, two of the three real roots $\lambda_2(\chi,\tau)$, $\lambda_4(\chi,\tau)$, and $\lambda_6(\chi,\tau)$ will coalesce and disappear.  Since the graphs of $\lambda=\lambda_4(\chi,\tau)$ and $\lambda=\lambda_6(\chi,\tau)$ lie on opposite sides of the line $\lambda=\tau^{-2}\chi$ for large $\chi>0$ and neither can cross this line, the two roots $\lambda_2(\chi,\tau)$ and $\lambda_4(\chi,\tau)$ lying below this line are the ones that coalesce and disappear with the common value of $\lambda=18\chi_0(\tau)^{-2}$.  
\item
The discriminant of $\mathscr{D}_1(\chi,\tau,\lambda)$ with respect to $\lambda$ is proportional by a large positive integer to $\tau^{76}(1492992\tau^4+\chi^6)$.  This is strictly positive for $\tau\neq 0$ and $\chi>\chi_\mathrm{c}(\tau)$, and hence $\lambda_3(\chi,\tau)<\lambda_5(\chi,\tau)$ remain the only real roots of this factor for all $\chi>\chi_\mathrm{c}(\tau)$.  
\end{itemize}
Next we discuss the possibility of intersections of branches $\lambda=\lambda_k(\chi,\tau)$ belonging to different factors in the discriminant:
\begin{itemize}
\item
Obviously, the double factors $\mathscr{D}_2^\pm(\chi,\tau,\lambda)$ can have no roots $\lambda$ in common.
\item
The resultant of the double factor $\mathscr{D}_2^+(\chi,\tau,\lambda)$ and the simple factor $\mathscr{D}_1(\chi,\tau,\lambda)$ with respect to $\lambda$ is $4096\tau^{30}(54\tau^2+\chi^3)^3$. Since this is nonzero for $\tau\neq 0$ and $\chi\ge\chi_\mathrm{c}(\tau)$, $\lambda_1(\chi,\tau)<\lambda_3(\chi,\tau)$ holds on the whole interval $\chi\ge\chi_\mathrm{c}(\tau)$.
\item
The resultant of the double factor $\mathscr{D}_2^-(\chi,\tau,\lambda)$ and the simple factor $\mathscr{D}_1(\chi,\tau,\lambda)$ with respect to $\lambda$ is $4096\tau^{30}(54\tau^2-\chi^3)^3$. If $\tau^2>\widehat{\tau}^2$, as $\chi$ decreases from $+\infty$, this will vanish at exactly the same value $\chi=\chi_0(\tau)>\chi_\mathrm{c}(\tau)$ at which point the real roots $\lambda_2(\chi,\tau)$ and $\lambda_4(\chi,\tau)$ of $\mathscr{D}_2^-(\chi,\tau,\lambda)$ coalesce and disappear.  This is consistent with the fact that between each pair of consecutive real roots of $\mathscr{D}_2^-(\chi,\tau,\lambda)$ for $\chi>0$ sufficiently large there lies a real root of $\mathscr{D}_1(\chi,\tau,\lambda)$. Moreover, when $\chi=\chi_0(\tau)$, it is easy to check that $\lambda_2(\chi,\tau)=\lambda_3(\chi,\tau)=\lambda_4(\chi,\tau)=18\chi^{-2}$ and $\lambda_5(\chi,\tau)=\lambda_6(\chi,\tau)=72\chi^{-2}$. Expanding 
$\mathscr{D}_2^-(\chi,\tau,\lambda)$ and $\mathscr{D}_1(\chi,\tau,\lambda)$ about $\lambda=72\chi_0(\tau)^{-2}$ and $\chi=\chi_0(\tau)$ shows that 
$\lambda_5(\chi,\tau)<\lambda_6(\chi,\tau)$ for $\chi\neq \chi_0(\tau)$.
\end{itemize}
Next, we determine the components of the complement of the discriminant locus for which $R(z)^2$ has two distinct complex-conjugate pairs of non-real roots.  Based on the above analysis, the complement of the discriminant locus and the singular line $\lambda=\tau^{-2}\chi$ consists of eight pairwise disjoint components if $\tau^2\le\widehat{\tau}^2$, and a ninth component (denoted $R_{56}^<$ below) appears for $\tau^2>\widehat{\tau}^2$:  
\begin{itemize}
\item
The region $R_-$ is defined by the inequality $\lambda<\lambda_1(\chi,\tau)$ for $\chi>\chi_\mathrm{c}(\tau)$.
\item
The region $R_{12}$ is defined by the inequalities $\lambda_1(\chi,\tau)<\lambda<\lambda_2(\chi,\tau)$ for $\chi>\chi_\mathrm{c}(\tau)$ if $\tau^2\le\widehat{\tau}^2$. If $\tau^2>\widehat{\tau}^2$, then $R_{12}$ is defined by $\lambda_1(\chi,\tau)<\lambda<\lambda_2(\chi,\tau)$ for $\chi\ge \chi_0(\tau)$ and by $\lambda_1(\chi,\tau)<\lambda<\lambda_3(\chi,\tau)$ for $\chi_\mathrm{c}(\tau)<\chi<\chi_0(\tau)$.
\item 
The region $R_{23}$ is defined by the inequalities $\lambda_2(\chi,\tau)<\lambda<\lambda_3(\chi,\tau)$ for $\chi>\chi_\mathrm{c}(\tau)$ if $\tau^2\le\widehat{\tau}^2$.  If $\tau^2>\widehat{\tau}^2$, then $R_{23}$ is defined by $\lambda_2(\chi,\tau)<\lambda<\lambda_3(\chi,\tau)$ for $\chi>\chi_0(\tau)$.
\item 
The region $R_{34}$ is defined by the inequalities $\lambda_3(\chi,\tau)<\lambda<\lambda_4(\chi,\tau)$ for $\chi>\chi_\mathrm{c}(\tau)$ if $\tau^2\le\widehat{\tau}^2$.  If $\tau^2>\widehat{\tau}^2$, then $R_{34}$ is defined by $\lambda_3(\chi,\tau)<\lambda<\lambda_4(\chi,\tau)$ for $\chi>\chi_0(\tau)$.
\item 
The region $R_{40}$ is defined by the inequalities $\lambda_4(\chi,\tau)<\lambda<\tau^{-2}\chi$ for $\chi>\chi_\mathrm{c}(\tau)$ if $\tau^2\le\widehat{\tau}^2$.  If $\tau^2>\widehat{\tau}^2$, then $R_{40}$ is defined by $\lambda_4(\chi,\tau)<\lambda<\tau^{-2}\chi$ for $\chi\ge \chi_0(\tau)$ and by $\lambda_3(\chi,\tau)<\lambda<\tau^{-2}\chi$ for $\chi_\mathrm{c}(\tau)<\chi<\chi_0(\tau)$.
\item
The region $R_{05}$ is defined by the inequalities $\tau^{-2}\chi<\lambda<\lambda_5(\chi,\tau)$ for $\chi>\chi_\mathrm{c}(\tau)$ (regardless of the sign of $\tau^2-\widehat{\tau}^2$).
\item
The region $R_{56}^>$ is defined by the inequalities $\lambda_{5}(\chi,\tau)<\lambda<\lambda_6(\chi,\tau)$ for $\chi>\chi_\mathrm{c}(\tau)$ if $\tau^2\le\widehat{\tau}^2$.  If $\tau^2>\widehat{\tau}^2$, then $R_{56}^+$ is defined by $\lambda_5(\chi,\tau)<\lambda<\lambda_6(\chi,\tau)$ for $\chi>\chi_0(\tau)$.
\item
The region $R_{56}^<$ is only defined for $\tau^2>\widehat{\tau}^2$, and it is given by the inequalities $\lambda_5(\chi,\tau)<\lambda<\lambda_6(\chi,\tau)$ for $\chi_\mathrm{c}(\tau)<\chi<\chi_0(\tau)$.
\item
The region $R_+$ is defined by the inequality $\lambda>\lambda_6(\chi,\tau)$ for $\chi>\chi_\mathrm{c}(\tau)$. 
\end{itemize}
See Figure~\ref{fig:DiscriminantLocus}.  
\begin{figure}[h]
\begin{center}
\includegraphics[height=0.22\textheight]{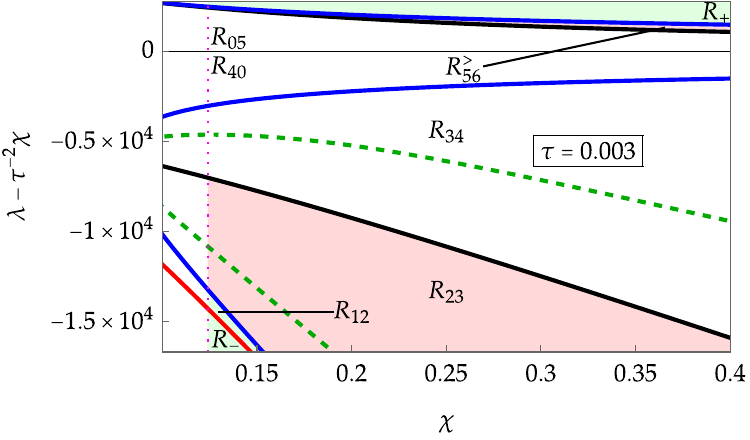}\hfill\includegraphics[height=0.22\textheight]{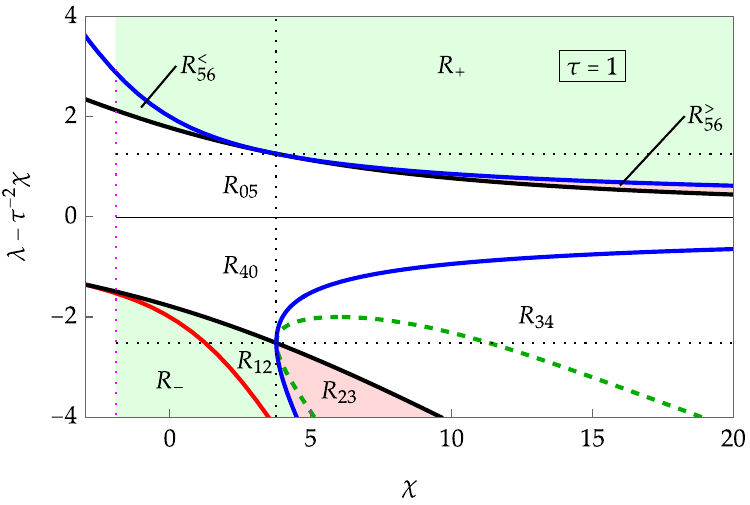}
\end{center}
\caption{The locus $\Qpoly(\lambda;\chi,\tau)=0$ (dashed green curve) and the discriminant locus in the $(\chi,\lambda-\tau^{-2}\chi)$-plane for $\tau=0.003$ (left) and $\tau=1$ (right).  Red curve:  the real root $\lambda=\lambda_1(\chi,\tau)$ of the double factor $\mathscr{D}_2^+(\chi,\tau,\lambda)$.  Blue curves:  roots $\lambda=\lambda_2(\chi,\tau)$, $\lambda=\lambda_4(\chi,\tau)$, and $\lambda=\lambda_6(\chi,\tau)$ of the double factor $\mathscr{D}_2^-(\chi,\tau,\lambda)$.  Black curves:  the real roots $\lambda=\lambda_3(\chi,\tau)$ and $\lambda=\lambda_5(\chi,\tau)$ of the simple factor $\mathscr{D}_1(\chi,\tau,\lambda)$.  The vertical dotted magenta line is $\chi=\chi_\mathrm{c}(\tau)$, the vertical dotted black line is $\chi=\chi_0(\tau)$, and the horizontal dotted lines are $\lambda=18\chi_0(\tau)^{-2}$ and $\lambda=72\chi_0(\tau)^{-2}$.  For $\chi>\chi_\mathrm{c}(\tau)$, green shading indicates the regions (iii) where $R(z)^2$ has distinct complex roots, pink shading indicates the regions (i) where $R(z)^2$ has all real simple roots, and no shading indicates regions (ii) where $R(z)^2$ has two real roots and a complex-conjugate pair of roots.}
\label{fig:DiscriminantLocus}
\end{figure}
For all $(\chi,\tau,\lambda)$ in any given region, $R(z)^2$ is in exactly one of three distinct cases:  (i) four simple real roots, (ii) two simple real roots and two simple non-real roots, or (iii) four simple non-real roots.  To determine the case for a given region it therefore suffices to do so for any convenient point $(\chi,\tau,\lambda)$ contained therein.  Since all of the regions except for $R_{56}^<$ contain arbitrarily large $\chi>0$, for these regions it suffices to determine the root configuration for any chosen $\lambda$ asymptotically confined as $\chi\to+\infty$ in the eight disjoint intervals $\lambda<\lambda_1(\chi,\tau)$, $\lambda_1(\chi,\tau)<\lambda<\lambda_2(\chi,\tau)$, $\lambda_2(\chi,\tau)<\lambda<\lambda_3(\chi,\tau)$, $\lambda_3(\chi,\tau)<\lambda<\lambda_4(\chi,\tau)$, $\lambda_4(\chi,\tau)<\lambda<\tau^{-2}\chi$, $\tau^{-2}\chi<\lambda<\lambda_5(\chi,\tau)$, $\lambda_5(\chi,\tau)<\lambda<\lambda_6(\chi,\tau)$, and $\lambda>\lambda_6(\chi,\tau)$.  To determine the root configuration on $R_{56}^<$ it suffices to consider $\chi=0$ which yields $\lambda_5(0,\tau)=32^\frac{1}{6}|\tau|^{-\frac{4}{3}}$ and $\lambda_6(0,\tau)=64^\frac{1}{6}|\tau|^{-\frac{4}{3}}$; then we may set $\lambda=48^\frac{1}{6}|\tau|^{-\frac{4}{3}}$ and verify that $\tau$ scales out of the root-finding problem for $R(z)^2$.  For each region the root configuration  is thusly determined by explicit calculations.  The results are:
\begin{itemize}
\item $R(z)^2$ has four simple real roots if $(\chi,\tau,\lambda)\in R_{23}\cup R_{56}^>$.
\item $R(z)^2$ has two simple real roots and two simple non-real roots if $(\chi,\tau,\lambda)\in R_{34}\cup R_{40}\cup R_{05}$.
\item $R(z)^2$ has four simple non-real roots if $(\chi,\tau,\lambda)\in R_-\cup R_{12}\cup R_{56}^<\cup R_+$.
\end{itemize}
Moreover, it can be shown that among all of the arcs making up the discriminant locus, those for which there are no real roots of $R(z)^2$ (i.e., a purely complex double root configuration) are precisely $\lambda=\lambda_1(\chi,\tau)$ for $\chi>-\chi_0(\tau)$ (a condition implied by $\chi>\chi_\mathrm{c}(\tau)$) and $\lambda=\lambda_6(\chi,\tau)$ for $\chi<\chi_0(\tau)$.  In other words, these are the arcs of the locus that separate two regions on each of which $R(z)^2$ has four simple non-real roots.  On all other arcs of the discriminant locus, $R(z)^2$ has at least one double real root.

In the case that $\tau=0$ and $\chi>\chi_\mathrm{c}(0)=\frac{1}{8}$, it is clear that $\mathscr{D}_1(\chi,\tau,\lambda)$ has no roots $\lambda$ whatsoever, while $\mathscr{D}_2^+(\chi,\tau,\lambda)=0$ has only the root $\lambda_1(\chi,\tau)=-8\chi^{-2}$ and $\mathscr{D}_2^-(\chi,\tau,\lambda)=0$ has only the root $\lambda_2(\chi,\tau)=8\chi^{-2}$.  Therefore, only three regions survive:  $R_-$, $R_{12}$, and $R_{23}$, with the upper bound on $\lambda$ for $R_{23}$ replaced by $+\infty$.  

\begin{lemma}
On the region 
$R_{12}$
in the $(\chi,\lambda)$-plane, the inequality $\Qpoly(\lambda;\chi,\tau)>0$ holds.
\label{lem:Qpos}
\end{lemma}
\begin{proof}
Since $\Qpoly(\lambda;\chi,0)=\chi^6$ and for $\tau=0$, $\chi>\chi_\mathrm{c}(0)=\frac{1}{8}>0$ on $R_{12}$, it suffices to fix $\tau\neq 0$ for the rest of the proof.
By an asymptotic analysis of $\lambda\mapsto\Qpoly(\lambda;\chi,\tau)$ as $\chi\to+\infty$ analogous to that already conducted above for $\lambda\mapsto \mathscr{D}_2^\pm(\chi,\tau,\lambda)$ and $\lambda\mapsto\mathscr{D}_1(\chi,\tau,\lambda)$, one can check that $\Qpoly(\lambda;\chi,\tau)=0$ has two real solutions $\lambda=\lambda_\Qpoly^\pm(\chi,\tau)$ with asymptotic behavior $\lambda_\Qpoly^\pm(\chi,\tau)=(\frac{1}{2}\pm\frac{1}{6}\sqrt{3})\tau^{-2}\chi + O(\chi^{-5})$ as $\chi\to+\infty$ while the other four roots are non-real for $\chi>0$ sufficiently large.  Hence $\lambda_1(\chi,\tau)<\lambda_2(\chi,\tau)<\lambda_\Qpoly^-(\chi,\tau)<\lambda_3(\chi,\tau)<\lambda_\Qpoly^+(\chi,\tau)<\lambda_4(\chi,\tau)<\lambda_5(\chi,\tau)<\lambda_6(\chi,\tau)$ holds for sufficiently large $\chi>0$.

For $\tau\neq 0$, the discriminant of $\Qpoly(\lambda;\chi,\tau)$ with respect to $\lambda$ is proportional by a positive factor to the product $-(54\tau^2-\chi^3)(54\tau^2+\chi^3)(1492992\tau^4+\chi^6)$, which for $\chi>\chi_\mathrm{c}(\tau)$ vanishes if and only if $\chi=(54\tau^2)^\frac{1}{3}$.  This occurs for no $\chi>\chi_\mathrm{c}(\tau)$ if $|\tau|<|\hat{\tau}|$ and at exactly the value $\chi=\chi_0(\tau)\ge\chi_\mathrm{c}(\tau)$ otherwise (the dotted black vertical line in the right-hand pane of Figure~\ref{fig:DiscriminantLocus}).  However, $\Qpoly(\lambda;\chi_0(\tau),\tau)=(\chi_0(\tau)^2\lambda-18)^2p(\lambda;\tau)$ where $\lambda\mapsto p(\lambda;\tau)$ is a quartic polynomial with discriminant proportional via a positive factor to $\chi_\mathrm{0}(\tau)^{60}>0$.  Hence there is one real double root $\lambda=\lambda_0(\tau):=18\chi_0(\tau)^{-2}$ and four simple roots (of $p$). These four roots are also non-real because there are only two real roots for large $\chi>0$ and the discriminant of $\Qpoly$ is nonzero for $\chi>\chi_0(\tau)$.  Expanding $\Qpoly(\lambda;\chi,\tau)=0$ about $\chi=\chi_0(\tau)$ and $\lambda=\lambda_0(\tau)$ shows that $\chi-\chi_0(\tau)\approx\frac{9}{4}\tau^4(2\tau^2)^{-\frac{1}{3}}(\lambda-\lambda_0(\tau))^2$, so that $\lambda\mapsto\Qpoly(\lambda;\chi,\tau)$ has exactly two real roots $\lambda=\lambda_\Qpoly^\pm(\chi,\tau)$ for $\chi>\chi_0(\tau)$, which coalesce and disappear as $\chi$ decreases through $\chi_0(\tau)$.  Since the discriminant is nonzero for $\chi_\mathrm{c}(\tau)<\chi<\chi_0(\tau)$, there are no real roots at all in this interval.

Next we show that the two roots $\lambda_\Qpoly^-(\chi,\tau)$ and $\lambda_\Qpoly^+(\chi,\tau)$ are confined to the regions $R_{23}$ and $R_{34}$ respectively.  This holds for sufficiently large $\chi>0$ due to the inequalities $\lambda_2(\chi,\tau)<\lambda_\Qpoly^-(\chi,\tau)<\lambda_3(\chi,\tau)<\lambda_\Qpoly^+(\chi,\tau)<\lambda_4(\chi,\tau)$.  It suffices therefore to show that these inequalities persist as $\chi$ decreases to $\chi_0(\tau)$ (if $|\tau|>|\hat{\tau}|$) or to $\chi_\mathrm{c}(\tau)>\chi_0(\tau)$ (if $|\tau|<|\hat{\tau}|)$.
But $\lambda_2(\chi,\tau)$ and $\lambda_4(\chi,\tau)$ are roots of $\lambda\mapsto \mathscr{D}_2^-(\chi,\tau,\lambda)$ defined for $\chi>\chi_0(\tau)$, and 
the resultant of $\Qpoly(\lambda;\chi,\tau)$ with $\mathscr{D}_2^-(\chi,\tau,\lambda)$ as polynomials in $\lambda$ is $65536\tau^{32}(54\tau^2-\chi^3)^2$ which is nonzero for $\chi>\chi_0(\tau)$.  Hence the inequalities $\lambda_2(\chi,\tau)<\lambda_\Qpoly^-(\chi,\tau)$ and $\lambda_\Qpoly^+(\chi,\tau)<\lambda_4(\chi,\tau)$ both persist for $\chi>\chi_0(\tau)$.  Similarly, $\lambda_3(\chi,\tau)$ is a root of $\lambda\mapsto \mathscr{D}_1(\chi,\tau,\lambda)$, and 
the resultant of $\Qpoly(\lambda;\chi,\tau)$ with $\mathscr{D}_1(\chi,\tau,\lambda)$ as polynomials in $\lambda$ is $191102976\tau^{84}(54\tau^2-\chi^3)(54\tau^2+\chi^3)(1492992\tau^4+\chi^6)^2$ which again is nonzero for $\chi>\chi_0(\tau)$ and hence $\lambda_\Qpoly^-(\chi,\tau)<\lambda_3(\chi,\tau)<\lambda_\Qpoly^+(\chi,\tau)$ holds in the same interval.  As the region $R_{23}$ is defined by $\lambda_2(\chi,\tau)<\lambda<\lambda_3(\chi,\tau)$ and $R_{34}$ is defined by $\lambda_3(\chi,\tau)<\lambda<\lambda_4(\chi,\tau)$, the claim is proved.  The roots $\lambda_\Qpoly^\pm(\chi,\tau)$ are illustrated in the plots in Figure~\ref{fig:DiscriminantLocus} with dashed green curves.

Since 
$R_-\cup R_{12}\cup\{\lambda=\lambda_1(\chi,\tau)\}$
is disjoint from $R_{23}$ and from $R_{34}$, $\Qpoly(\lambda;\chi,\tau)$ is of one sign on the whole region of interest.  To determine its sign, it suffices to let $\lambda\to-\infty$ for fixed $(\chi,\tau)$ yielding that $\Qpoly(\lambda;\chi,\tau)>0$ holds on $R_{12}$ in particular.
\end{proof}

\begin{lemma}
If $(\chi,\tau,\lambda)$ are such that $R(z)^2$ has any real repeated roots, then $f(\lambda;\chi,\tau)=-2\neq 0$.
\label{l:DS-no-real-double-roots-under-continuation}
\end{lemma}
\begin{proof}
If $R(z)^2$ has a repeated real root, say $\alpha=\alpha^*$, the integral in \eqref{eq:integralcondition-rewrite} vanishes.
\end{proof}

\begin{proposition}
For each $\chi>\chi_\mathrm{c}(\tau)$, there exists a unique solution $\lambda=\lambda(\chi,\tau)$ of $f(\lambda;\chi,\tau)=0$ that is real analytic as a function of $(\chi,\tau)$ and such that the graph of $\chi\mapsto\lambda(\chi,\tau)$ lies in the region $R_{12}$.
\label{prop:gamma-xt}
\end{proposition}

\begin{proof}
Given $\chi>\chi_\mathrm{c}(\tau)$, the intersection of $R_{12}$ with the fixed-$\chi$ vertical line in the plots shown in Figure~\ref{fig:DiscriminantLocus} is a $\lambda$-interval $\lambda_1(\chi,\tau)<\lambda<\lambda_\mathrm{max}(\chi,\tau)$, where $\lambda_\mathrm{max}(\chi,\tau)=\lambda_2(\chi,\tau)$ or (if $|\tau|>|\hat{\tau}|$ and $\chi_\mathrm{c}(\tau)<\chi<\chi_0(\tau)$) $\lambda_\mathrm{max}(\chi,\tau)=\lambda_3(\chi,\tau)$.  When $\lambda=\lambda_{\mathrm{max}}(\chi,\tau)$ there is a repeated real root of $R(z)^2$ implying via Lemma~\ref{l:DS-no-real-double-roots-under-continuation} that $f(\lambda;\chi,\tau)=-2$.  On the other hand, when $\lambda=\lambda_1(\chi,\tau)$, $R(z)^2=(z-\critpt)^2(z-\critpt^*)^2$ for some $\sigma\in\mathbb{C}_+$, implying that the polynomial $P(z)$ in \eqref{eq:P-sextic} is a perfect square; therefore $g(z)$ vanishes identically and hence $h(z;\chi,\tau)=\phase(z;\chi,\tau)$.
This means that $f(\lambda_1(\chi,\tau);\chi,\tau)=k(\chi,\tau)-2$, where $k(\chi,\tau)$ is the left-hand side of \eqref{eq:BreakingCurve}.  Since by definition of $\chi_\mathrm{c}(\tau)$ we have $k(\chi,\tau)-2=0$ for $\chi=\chi_\mathrm{c}(\tau)$, according to \eqref{eq:k-chi-positive} $f(\lambda_1(\chi,\tau);\chi,\tau)=k(\chi,\tau)-2>0$ holds because $\chi>\chi_\mathrm{c}(\tau)$.  Existence of a unique solution  $\lambda=\lambda(\chi,\tau)\in (\lambda_1(\chi,\tau),\lambda_\mathrm{max}(\chi,\tau))$ of $f(\lambda;\chi,\tau)=0$ then follows from the intermediate value theorem, \eqref{eq:gamma-deriv-3}, and Lemma~\ref{lem:Qpos}.  Real analyticity of $\chi\mapsto\lambda(\chi,\tau)$ follows from the implicit function theorem and analyticity of $f$ with respect to $\chi>\chi_\mathrm{c}(\tau)$ and $\tau$.
\end{proof}

\subsection{Construction of $g(z)$}
\label{s:DS-Constructing-g}
We now show how, given $\chi>\chi_\mathrm{c}(\tau)$, branch cuts for $h(z)$ can be chosen so that $\mathrm{Im}(h(z)\pm\ii)$ has a sign chart that is suitable for subsequent steepest-descent analysis.
Level curves of $\mathrm{Im}(h(z))$ are arcs of horizontal trajectories of the quadratic differential $h'(z)^2\,\dd z^2$, i.e., curves with tangent $\dd z$ satisfying at every point $z$ the condition $h'(z)^2\,\dd z^2>0$.  Since $h'(z)$ is real-valued for $z\in\mathbb{R}$, the real line (omitting the pole $z=0$ and, if $\tau\neq 0$, the double zero at $z=\gamma/(2\tau)=-\chi/(2\tau) +\tau\lambda/2$) is one such trajectory.  As trajectories cannot intersect, all remaining trajectories are therefore confined to the open half-planes $\mathbb{C}_\pm$; by Schwarz symmetry of $h'(z)$ those in $\mathbb{C}_-$ are obtained from those in $\mathbb{C}_+$ by reflection through the real line.  On the Riemann sphere, the quadratic differential $h'(z)^2\,\dd z^2$ has just two poles ($z=0,\infty$), so applying Jenkins' three-pole theorem \cite[Theorem 3.6]{Jenkins58} shows that there can be no recurrent trajectories and therefore also no divergent trajectories (see \cite[\S 10.2 and 11.1]{Strebel84}).  It follows that each of the three trajectories emanating from any simple zero of $h'(z)^2$ tends in the other direction toward a pole or zero of $h'(z)^2$.  For all $\chi>\chi_\mathrm{c}(\tau)$ we have exactly two simple zeros $z=\alpha,\beta$ in the open upper half-plane $\mathbb{C}_+$, and if $\tau\neq 0$ no trajectory from either of them can terminate in the other direction at the real double zero $z=\gamma/(2\tau)$ because this is forbidden by the integral condition \eqref{eq:integralcondition}.  Teichm\"uller's Lemma \cite[Theorem 14.1]{Strebel84} can then be used to show that exactly one of the trajectories from $z=\alpha$ (resp.,  $z=\beta$) tends to $z=0$, exactly one other tends to $z=\infty$, and the third tends to the other zero in $\mathbb{C}_+$, $z=\beta$ (resp., $z=\alpha$).

Given this trajectory structure, we now construct the function $g(z)$ so that $\mathrm{Im}(h(z)\pm\ii)$ has the sign chart necessary to admit steepest-descent analysis.  Indeed, let $\Sigma_g$ denote the Schwarz-symmetric contour consisting of the trajectory $h'(z)^2\,\dd z^2>0$ joining the complex roots $z=\alpha,\beta$ in the upper half-plane, its Schwarz reflection in the real line, and a Schwarz-symmetric arc connecting $z=\beta,\beta^*$ that lies in the sector at $z=\beta$ bounded by the trajectories emanating from this point and tending to $z=0,\infty$ (we assume that $\mathrm{Re}(\alpha)\le\mathrm{Re}(\beta)$, and if the real parts are equal, $\mathrm{Im}(\alpha)>\mathrm{Im}(\beta)$).  In particular, the latter arc of $\Sigma_g$ crosses the real line at a positive value.  Taking the branch cuts of $h'(z)$ to be the arcs of $\Sigma_g$ joining $\alpha,\beta$ and joining $\alpha^*,\beta^*$ and accounting for the fact that $g(\infty)=0$,  $g(z)$ is defined by integration of $g'(z)=h'(z)-\phase'(z)$:
\begin{equation}
g(z)=\int_\infty^z\left[h'(Z)-\phase'(Z)\right]\,\dd Z,\quad z\in\mathbb{C}\setminus\Sigma_g,
\label{eq:DS-g-formula}
\end{equation}
where the path of integration is arbitrary in $\mathbb{C}\setminus\Sigma_g$ because all singularities are removable.  It follows that $g(z)$ and $h(z)=\phase(z)+g(z)$ are both real-valued for $z\in\mathbb{R}$ and that $g(z^*)=g(z)^*$ and $h(z^*)=h(z)^*$.  On the arc of $\Sigma_g$ connecting $z=\alpha,\beta$, it follows from \eqref{eq:integralcondition} and $h'_+(z)+h'_-(z)=0$ that $h_+(z)+h_-(z)$ is independent of $z$ and that $\mathrm{Im}(h_+(z)+h_-(z))=2$.  We may write $h_+(z)+h_-(z)=2\ii + \phi$, where $\phi\in\mathbb{R}$ is independent of $z$.  Likewise, since $h'(z)$ is analytic on the arc of $\Sigma_g$ joining $\beta,\beta^*$, on this arc (taken with downwards orientation) the difference of boundary values is independent of $z$ and real:  $h_+(z)-h_-(z)=\Delta\in\mathbb{R}$ (as in \eqref{eq:Int-Delta-def});  meanwhile given the placement of this arc relative to the critical trajectories, Schwarz symmetry of $h$ and \eqref{eq:integralcondition} imply that $\mathrm{Im}(h_+(z)+h_-(z))\in (-2,2)$.  This in turn implies that $\mathrm{Im}(h(z))>1$ holds on either side of the arc of $\Sigma_g$ joining $z=\alpha,\beta$, and that $\mathrm{Im}(h(z))<1$ holds in the sector at $z=\alpha$ opposite this arc.
ee Figure~\ref{f:arcs-h} for these arcs and the regions where $\mathrm{Im}(h(z))<1$ and $\mathrm{Im}(h(z))> -1$.
\begin{figure}[h]
\includegraphics[width=0.32\textwidth]{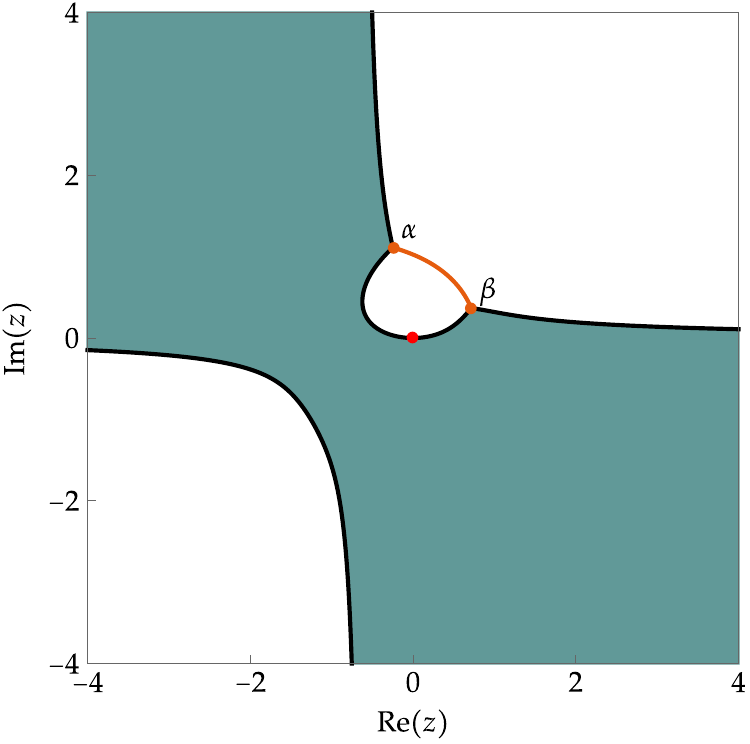}
\includegraphics[width=0.32\textwidth]{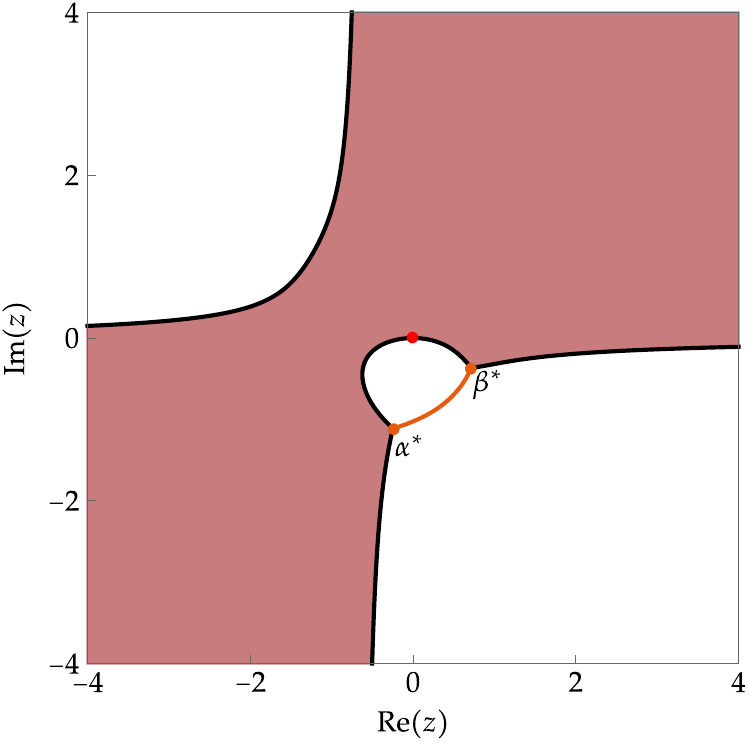}
\includegraphics[width=0.32\textwidth]{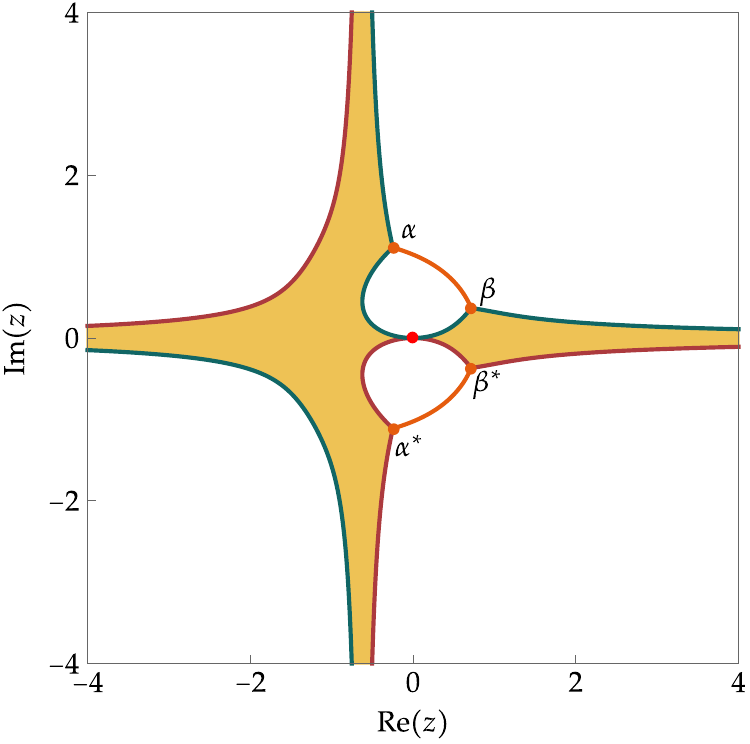}
\caption{Left-pane: the region where $\Im(h(z))<1$ and the arc of $\Sigma_g$ connecting the points $z=\alpha,\beta$. The factor $\ee^{-2\ii M(h(z)-\ii)}$ decays exponentially as $M\to+\infty$ for $z$ in the shaded region.
Center-pane: the region where $\Im(h(z))> -1$ and the arc of $\Sigma_g$ connecting the points $z=\beta^*,\alpha^*$. The factor $\ee^{2\ii M(h(z)+\ii)}$ decays exponentially as $M\to+\infty$ for $z$ in the shaded region.
Right-pane: the region where the composite inequality $-1<\Im(h(z)) < 1$ holds, and the arcs of $\Sigma_g$ connecting $z=\alpha,\beta$ and $z=\beta^*,\alpha^*$. Both of the factors $\ee^{-2\ii M(h(z)-\ii)}$ and $\ee^{2\ii M(h(z)+\ii)}$ decay exponentially as $M\to+\infty$ for $z$ in the shaded region.}
\label{f:arcs-h}
\end{figure}

We denote the three arcs of $\Sigma_g$ as $\Gamma_{\alpha\to\beta}$, $\Gamma_{\beta\to\beta^*}$, and $\Gamma_{\beta^*\to\alpha^*}$ with the subscript indicating the endpoints and orientation.  Including an additional oriented Schwarz-symmetric arc $\Gamma_{\alpha^*\to\alpha}$ lying in the sector bounded by the trajectories emanating from $\alpha$ tending to $z=0,\infty$ (crossing the real line at a negative value) we obtain a closed contour $\Gamma:=\Gamma_{\alpha\to\beta}\cup\Gamma_{\beta\to\beta^*}\cup\Gamma_{\beta^*\to\alpha^*}\cup\Gamma_{\alpha^*\to\alpha}$ with the origin in its interior and clockwise orientation.  Note that $\mathrm{Im}(h(z))\in (-1,1)$ holds for $z\in\Gamma_{\alpha^*\to\alpha}$.

\subsection{Additional properties of the spectral curve for $\chi>\chi_\mathrm{c}(\tau)$}
\subsubsection{Modulation equations}
The four roots $z=\alpha,\beta,\alpha^*,\beta^*$ of $R(z)^2$ are functions of $(\chi,\tau)$ on the domain $\chi>\chi_\mathrm{c}(\tau)$.  Here we derive a quasilinear system of Whitham modulation equations for which these functions are Riemann invariants.  The starting point is the trivial (Clairaut) identity
\begin{equation}
\frac{\partial}{\partial\tau}\frac{\partial h}{\partial\chi}(z;\chi,\tau) +\frac{\partial}{\partial\chi}\left(-\frac{\partial h}{\partial\tau}(z;\chi,\tau)\right)=0,\quad z\in\mathbb{C}\setminus\Sigma_g.
\label{eq:DS-Clairaut}
\end{equation}
The left-hand side of this equation is an analogue in this setting of the canonical differential $\Omega$ in the modulation theory of the Korteweg-de Vries equation as explained by Flaschka, Forest, and McLaughlin \cite{FlaschkaFM80}.  We first express the ``inner'' partial derivatives of $h$ explicitly in terms of $\alpha,\beta,\alpha^*,\beta^*$, and $z$ by differentiating their Riemann-Hilbert jump conditions and solving the differentiated problems.  Using \eqref{eq:DS-tildevartheta} and \eqref{eq:DS-h-g-tildevartheta} shows that the partial derivatives $h_\chi(z;\chi,\tau)$ and $h_\tau(z;\chi,\tau)$ are both analytic functions for $z\in\mathbb{C}\setminus\Sigma_g$, and taking into account further that $g(\infty)=0$ holds for all $(\chi,\tau)$ in question shows that $h_\chi(z;\chi,\tau)=z+O(z^{-1})$ while $h_\tau(z;\chi,\tau)=z^2+O(z^{-1})$ as $z\to\infty$.  Differentiation of the jump conditions for $h(z;\chi,\tau)$ shows that $h_{\chi+}(z;\chi,\tau)+h_{\chi-}(z;\chi,\tau)=\phi_\chi(\chi,\tau)$ and $h_{\tau+}(z;\chi,\tau)+h_{\tau-}(z;\chi,\tau)=\phi_\tau(\chi,\tau)$ both hold for $z\in\Gamma_{\alpha\to\beta}\cup\Gamma_{\beta^*\to\alpha^*}$.  Likewise, $h_{\chi+}(z;\chi,\tau)-h_{\chi-}(z;\chi,\tau)=\Delta_\chi(\chi,\tau)$ and $h_{\tau+}(z;\chi,\tau)-h_{\tau-}(z;\chi,\tau)=\Delta_\tau(\chi,\tau)$ both hold for $z\in\Gamma_{\beta\to\beta^*}$.  It follows from these conditions that $h_\chi(z;\chi,\tau)$ and $h_\tau(z;\chi,\tau)$ necessarily have the form
\begin{equation}
\frac{\partial h}{\partial\chi}(z;\chi,\tau)=R(z)\left[\frac{\phi_\chi(\chi,\tau)}{2\pi\ii}\int_{\Gamma_{\alpha\to\beta}\cup\Gamma_{\beta^*\to\alpha^*}}\frac{\dd w}{R_+(w)(w-z)} + \frac{\Delta_\chi(\chi,\tau)}{2\pi\ii}\int_{\Gamma_{\beta\to\beta^*}}\frac{\dd w}{R(w)(w-z)}\right]
\label{eq:DS-h-chi}
\end{equation}
and
\begin{equation}
\frac{\partial h}{\partial\tau}(z;\chi,\tau)=R(z)\left[1+\frac{\phi_\tau(\chi,\tau)}{2\pi\ii}\int_{\Gamma_{\alpha\to\beta}\cup\Gamma_{\beta^*\to\alpha^*}}\frac{\dd w}{R_+(w)(w-z)} + \frac{\Delta_\tau(\chi,\tau)}{2\pi\ii}\int_{\Gamma_{\beta\to\beta^*}}\frac{\dd w}{R(w)(w-z)}\right].
\label{eq:DS-h-tau}
\end{equation}
But since $R(z)^{-1}=z^{-2}+r_3z^{-3}+r_4z^{-4}+O(z^{-5})$ as $z\to\infty$, where
\begin{equation}
\begin{split}
r_3&:=\frac{1}{2}(\alpha+\beta+\alpha^*+\beta^*),\\ 
r_4&:=\frac{1}{8}\left(3(\alpha+\beta+\alpha^*+\beta^*)^2-4\left(\alpha\alpha^*+\alpha\beta+\alpha^*\beta+\alpha\beta^*+\alpha^*\beta^*+\beta\beta^*\right)\right),
\end{split}
\label{eq:DS-r3-r4}
\end{equation}
upon dividing \eqref{eq:DS-h-chi} by $R(z)$ and using  $h_\chi(z;\chi,\tau)=z+O(z^{-1})$ as $z\to\infty$ we see that
\begin{equation}
\begin{split}
\frac{\phi_\chi(\chi,\tau)}{2\pi\ii}\int_{\Gamma_{\alpha\to\beta}\cup\Gamma_{\beta^*\to\alpha^*}}\frac{\dd w}{R_+(w)} +\frac{\Delta_\chi(\chi,\tau)}{2\pi\ii}\int_{\Gamma_{\beta\to\beta^*}}\frac{\dd w}{R(w)} &= -1\\
\frac{\phi_\chi(\chi,\tau)}{2\pi\ii}\int_{\Gamma_{\alpha\to\beta}\cup\Gamma_{\beta^*\to\alpha^*}}\frac{w\,\dd w}{R_+(w)} +\frac{\Delta_\chi(\chi,\tau)}{2\pi\ii}\int_{\Gamma_{\beta\to\beta^*}}\frac{w\,\dd w}{R(w)} &= -r_3,
\end{split}
\end{equation}
and similarly using $h_\tau(z;\chi,\tau)=z^2+O(z^{-1})$ as $z\to\infty$ in \eqref{eq:DS-h-tau} gives
\begin{equation}
\begin{split}
\frac{\phi_\tau(\chi,\tau)}{2\pi\ii}\int_{\Gamma_{\alpha\to\beta}\cup\Gamma_{\beta^*\to\alpha^*}}\frac{\dd w}{R_+(w)} +\frac{\Delta_\tau(\chi,\tau)}{2\pi\ii}\int_{\Gamma_{\beta\to\beta^*}}\frac{\dd w}{R(w)} &= -r_3\\
\frac{\phi_\tau(\chi,\tau)}{2\pi\ii}\int_{\Gamma_{\alpha\to\beta}\cup\Gamma_{\beta^*\to\alpha^*}}\frac{w\,\dd w}{R_+(w)} +\frac{\Delta_\tau(\chi,\tau)}{2\pi\ii}\int_{\Gamma_{\beta\to\beta^*}}\frac{w\,\dd w}{R(w)} &= -r_4.
\end{split}
\end{equation}
Now a contour deformation shows that 
\begin{equation}
\int_{\Gamma_{\alpha\to\beta}\cup\Gamma_{\beta^*\to\alpha^*}}\frac{\dd w}{R_+(w)} = 0\quad\text{and}\quad
\int_{\Gamma_{\alpha\to\beta}\cup\Gamma_{\beta^*\to\alpha^*}}\frac{w\,\dd w}{R_+(w)} = -\ii\pi.
\end{equation}
Therefore, defining
\begin{equation}
I_p:=\frac{1}{2\pi\ii}\int_{\Gamma_{\beta\to\beta^*}}\frac{w^p\,\dd w}{R(w)},\quad p\in\mathbb{Z}_{\ge 0},
\label{eq:DS-Ip}
\end{equation}
and noting that $I_0\neq 0$ as a complete elliptic integral of the first kind, we may solve explicitly for $\phi_\chi(\chi,\tau)$, $\Delta_\chi(\chi,\tau)$, $\phi_\tau(\chi,\tau)$, and $\Delta_\tau(\chi,\tau)$:
\begin{equation}
\Delta_\chi(\chi,\tau)=-\frac{1}{I_0}\quad\text{and}\quad
\phi_\chi(\chi,\tau)=2r_3-2\frac{I_1}{I_0}
\label{eq:DS-chi-derivs}
\end{equation}
and
\begin{equation}
\Delta_\tau(\chi,\tau)=-\frac{r_3}{I_0}\quad\text{and}\quad \phi_\tau(\chi,\tau)=2r_4-2\frac{r_3I_1}{I_0}.
\label{eq:DS-tau-derivs}
\end{equation}
Using these in \eqref{eq:DS-h-chi}--\eqref{eq:DS-h-tau} explicitly presents the ``inner'' partial derivatives of $h$ as functions of $z$ depending parametrically on $\chi,\tau$ \emph{only via the four points $z=\alpha,\beta,\alpha^*,\beta^*$}.  

The expressions \eqref{eq:DS-h-chi}--\eqref{eq:DS-h-tau} can be further simplified by invoking another contour integration argument to show that
\begin{equation}
\int_{\Gamma_{\alpha\to\beta}\cup\Gamma_{\beta^*\to\alpha^*}}\frac{\dd w}{R_+(w)(w-z)} = \frac{\ii\pi}{R(z)},\quad z\in\mathbb{C}\setminus\Sigma_g.
\end{equation}
We thus
define another branch of the square root by setting
\begin{equation}
\widetilde{R}(z):=\begin{cases}-R(z),&z\in \mathrm{int}(\Gamma),\\R(z),& 
z\in\mathrm{ext}(\Gamma),
\end{cases}
\label{eq:widetilde-R}
\end{equation}
and we see that if $C$ is a clockwise-oriented loop surrounding the branch cut $\Gamma_{\beta\to\beta^*}$ of $\widetilde{R}(z)$ but not enclosing its other branch cut $\Gamma_{\alpha^*\to\alpha}$,
\begin{equation}
\int_{\Gamma_{\beta\to\beta^*}}\frac{\dd w}{R(w)(w-z)} = \begin{cases}
\displaystyle \frac{1}{2}\oint_C\frac{\dd w}{\widetilde{R}(w)(w-z)},&z\in\mathrm{ext}(C)\\
\displaystyle \frac{\ii\pi}{\widetilde{R}(z)} + \frac{1}{2}\oint_C\frac{\dd w}{\widetilde{R}(w)(w-z)},&z\in\mathrm{int}(C)\setminus\Gamma_{\beta\to\beta^*}.\end{cases}
\end{equation} 
Therefore, \eqref{eq:DS-h-chi}--\eqref{eq:DS-h-tau} become
\begin{equation}
\frac{\partial h}{\partial\chi}(z;\chi,\tau)=\begin{cases}
\displaystyle \frac{1}{2}\phi_\chi(\chi,\tau) + \frac{\Delta_\chi(\chi,\tau)}{4\pi\ii}R(z)\oint_C\frac{\dd w}{\widetilde{R}(w)(w-z)},&\text{for $z$ near $\alpha$ or $\alpha^*$}\\
\displaystyle \frac{1}{2}(\phi_\chi(\chi,\tau)\pm\Delta_\chi(\chi,\tau)) + \frac{\Delta_\chi(\chi,\tau)}{4\pi\ii}R(z)\oint_C\frac{\dd w}{\widetilde{R}(w)(w-z)},&\text{for $z$ near $\beta$ or $\beta^*$,}
\end{cases}
\end{equation}
\begin{multline}
\frac{\partial h}{\partial\tau}(z;\chi,\tau)\\=\begin{cases}
\displaystyle R(z)+\frac{1}{2}\phi_\tau(\chi,\tau) + \frac{\Delta_\tau(\chi,\tau)}{4\pi\ii}R(z)\oint_C\frac{\dd w}{\widetilde{R}(w)(w-z)},&\text{for $z$ near $\alpha$ or $\alpha^*$}\\
\displaystyle R(z)+\frac{1}{2}(\phi_\tau(\chi,\tau)\pm\Delta_\tau(\chi,\tau)) + \frac{\Delta_\tau(\chi,\tau)}{4\pi\ii}R(z)\oint_C\frac{\dd w}{\widetilde{R}(w)(w-z)},&\text{for $z$ near $\beta$ or $\beta^*$,}
\end{cases}
\end{multline}
where the $\pm$ sign stands for the fraction $R(z)/\widetilde{R}(z)$.
 Also, $I_p$ defined by \eqref{eq:DS-Ip} can be rewritten as
\begin{equation}
I_p=\frac{1}{4\pi\ii}\oint_C\frac{w^p\,\dd w}{\widetilde{R}(w)},\quad p\in\mathbb{Z}_{\ge 0}.
\label{eq:DS-Ip-Rtilde}
\end{equation}

Next, we differentiate $h_\chi$ and $h_\tau$ with respect to $\tau$ and $\chi$ respectively, thinking of $z$ near $z_1:=\alpha$, $z_2:=\beta$, $z_3:=\alpha^*$, or $z_4:=\beta^*$.    The most singular terms in any of these limits arise from differentiation of $R(z)$ via its branch points, since
\begin{equation}
\frac{\partial R}{\partial\chi,\tau}(z)=-\frac{1}{2}R(z)\sum_{i=1}^4\frac{1}{z-z_i}\frac{\partial z_i}{\partial\chi,\tau}.  
\end{equation}
Hence 
\begin{equation}
\frac{\partial}{\partial\tau}\left(\frac{\partial h}{\partial\chi}(z;\chi,\tau)\right)=-\frac{\Delta_\chi(\chi,\tau)}{8\pi\ii}R(z)\left[\sum_{i=1}^4\frac{1}{z-z_i}\frac{\partial z_i}{\partial \tau}\right]\oint_C\frac{\dd w}{\widetilde{R}(w)(w-z_\ell)} + O(1),\quad z\to z_\ell
\end{equation}
and
\begin{equation}
\frac{\partial}{\partial\chi}\left(-\frac{\partial h}{\partial\tau}(z;\chi,\tau)\right)=R(z)\left[\sum_{i=1}^4\frac{1}{z-z_i}\frac{\partial z_i}{\partial\chi}\right]\left(\frac{1}{2}+\frac{\Delta_\tau(\chi,\tau)}{8\pi\ii}\oint_C\frac{\dd w}{\widetilde{R}(w)(w-z_\ell)}\right)+O(1),\quad z\to z_\ell.
\end{equation}
Therefore, keeping only the coefficients of the most singular terms corresponding to $i=\ell$, the identity \eqref{eq:DS-Clairaut} implies that
\begin{equation}
-\frac{\Delta_\chi(\chi,\tau)}{8\pi\ii}\oint_C\frac{\dd w}{\widetilde{R}(w)(w-z_\ell)}\cdot\frac{\partial z_\ell}{\partial\tau} +\left(\frac{1}{2}+\frac{\Delta_\tau(\chi,\tau)}{8\pi\ii}\oint_C\frac{\dd w}{\widetilde{R}(w)(w-z_\ell)}\right)\frac{\partial z_\ell}{\partial \chi}=0,\quad \ell=1,2,3,4.
\end{equation}
Here, the partial derivatives $\Delta_\chi(\chi,\tau)$ and $\Delta_\tau(\chi,\tau)$ are explicitly expressed in terms of $z_1,\dots,z_4$ using \eqref{eq:DS-chi-derivs}--\eqref{eq:DS-tau-derivs}.

Therefore, the quantities $z_\ell=z_\ell(\chi,\tau)$, $\ell=1,\dots,4$ satisfy a quasilinear first-order system of partial differential equations in diagonal (Riemann-invariant) form:
\begin{equation}
\frac{\partial z_\ell}{\partial\tau} + s_\ell\frac{\partial z_\ell}{\partial \chi}=0,\quad \ell=1,\dots,4,\quad s_\ell:=-\frac{\Delta_\tau}{\Delta_\chi}-\left[\frac{\Delta_\chi}{4\pi\ii}\oint_C\frac{\dd w}{\widetilde{R}(w)(w-z_\ell)}\right]^{-1}.
\label{eq:DS-Whitham}
\end{equation}
Using \eqref{eq:DS-r3-r4} and  \eqref{eq:DS-chi-derivs}--\eqref{eq:DS-tau-derivs}, the characteristic velocities $s_\ell$ can be written in the form
\begin{equation}
s_\ell=-\frac{1}{2}\Sigma_1
+ 4\pi\ii I_0\left[\oint_C\frac{\dd w}{\widetilde{R}(w)(w-z_\ell)}\right]^{-1},\quad \ell=1,\dots,4,
\label{eq:DS-sk-BM}
\end{equation}
in which we use the notation
\begin{equation}
\Sigma_1:=\sum_{i=1}^4 z_i,\quad \Sigma_2:=\mathop{\sum_{i,j=1}^4}_{i\neq j}z_i z_j,\quad \Sigma_3:=\mathop{\sum_{i,j,k=1}^4}_{i\neq j\neq k}z_i z_j z_k.
\label{eq:DS-sums-notation}
\end{equation}
The quantities $s_\ell$ defined by \eqref{eq:DS-sk-BM} agree precisely with the characteristic velocities $S^{(\ell)}$ defined in the well-known paper of Forest and Lee \cite{ForestL86} on the Whitham modulation theory of the focusing nonlinear Schr\"odinger equation after accounting for some typos\footnote{Some typos in \cite{ForestL86} include:  in (II.2b), the first and third equations should read $f_t=(\ii r_x-2Er)g+(\ii q_x+2Eq)h$ and $h_t=2(\ii r_x-2Er)f-2\ii (qr-2E^2)h$ respectively; in the displayed formula for $D_j^{(1)}$ for $0\le j\le N-2$ on page 53, an overall minus sign is missing (at least in the genus-one $N=2$ case); in (III.6) a minus sign is missing on the derivative with respect to $X$ for consistency with (III.2); in (III.10) the characteristic speeds $S^{(k)}$ are off by a sign and the summation in the numerator should begin with the index $j=0$.} and rescaling the time by a factor of $2$.  To see this, we integrate the identity 
\begin{equation}
\frac{\dd}{\dd w}\frac{\widetilde{R}(w)}{w-z_\ell}=\frac{w^2-\frac{1}{2}\Sigma_1w+z_\ell(\frac{1}{2}\Sigma_1-z_\ell)}{\widetilde{R}(w)}
-\frac{1}{2}\mathop{\prod_{i=1}^4}_{i\neq \ell}(z_\ell-z_i)\cdot\frac{1}{\widetilde{R}(w)(w-z_\ell)}
\end{equation}
around the cycle $C$ to obtain
\begin{equation}
\begin{split}
\frac{1}{2}\mathop{\prod_{i=1}^4}_{i\neq \ell}(z_\ell-z_i)\oint_C\frac{\dd w}{\widetilde{R}(w)(w-z_\ell)}&=\oint_C\frac{w^2-\frac{1}{2}\Sigma_1w+z_\ell(\frac{1}{2}\Sigma_1-z_\ell)}{\widetilde{R}(w)}\dd w\\
&=4\pi\ii I_2-2\pi\ii\Sigma_1 I_1+2\pi\ii z_\ell(\Sigma_1-2z_\ell)I_0.
\end{split}
\end{equation}
It follows that $s_\ell$ given by \eqref{eq:DS-sk-BM} can be written as
\begin{equation}
s_\ell=\frac{\displaystyle -\tfrac{1}{2}\Sigma_1I_2 +\tfrac{1}{4}\Sigma_1^{2}I_1 +\left[\tfrac{1}{2}\mathop{\prod_{i=1}^4}_{i\neq \ell}(z_\ell-z_i)-\tfrac{1}{2}\Sigma_1z_\ell(\tfrac{1}{2}\Sigma_1-z_\ell)\right]I_0}{\displaystyle I_2 - \tfrac{1}{2}\Sigma_1I_1 +z_\ell(\tfrac{1}{2}\Sigma_1-z_\ell)I_0}.
\end{equation}
Next, recalling \eqref{eq:DS-sums-notation}, we have 
\begin{equation}
\mathop{\prod_{i=1}^4}_{i\neq \ell}(z_\ell-z_i)=4z_\ell^3-3\Sigma_1z_\ell^2+2\Sigma_2z_\ell -\Sigma_3,
\end{equation}
so $s_\ell$ becomes a ratio of cubic and quadratic polynomials in $z_\ell$ with coefficients that are symmetric functions of all four points $z_i$, $i=1,\dots,4$:
\begin{equation}
s_\ell=\frac{2I_0z_\ell^3-\Sigma_1I_0z_\ell^2 + (\Sigma_2-\tfrac{1}{4}\Sigma_1^2)I_0z_\ell -\tfrac{1}{2}\Sigma_1 I_2 +\tfrac{1}{4}\Sigma_1^2I_1-\tfrac{1}{2}\Sigma_3I_0}{-I_0z_\ell^2+\tfrac{1}{2}\Sigma_1I_0z_\ell + I_2-\tfrac{1}{2}\Sigma_1I_1}.
\label{eq:DS-s-ell-close}
\end{equation}
Finally, since 
\begin{equation}
\frac{\dd}{\dd w}\widetilde{R}(w) = \frac{ 4w^3-3\Sigma_1w^2+2\Sigma_2w-\Sigma_3}{2\widetilde{R}(w)},
\end{equation}
by integration about the loop $C$ we obtain the identity $2I_3-\frac{3}{2}\Sigma_1I_2+\Sigma_2I_1-\frac{1}{2}\Sigma_3I_0=0$.  Using this we eliminate $I_0$ from the constant term in the numerator on the right-hand side of \eqref{eq:DS-s-ell-close} and obtain
\begin{equation}
s_\ell=\frac{2I_0z_\ell^3-\Sigma_1I_0z_\ell^2 + (\Sigma_2-\tfrac{1}{4}\Sigma_1^2)I_0z_\ell -2I_3+\Sigma_1I_2+(\tfrac{1}{4}\Sigma_1^2-\Sigma_2)I_1}{-I_0z_\ell^2+\tfrac{1}{2}\Sigma_1I_0z_\ell + I_2-\tfrac{1}{2}\Sigma_1I_1}.
\end{equation}
Upon dividing numerator and denominator through by $I_0\neq 0$, the resulting ratio of polynomials $D_3^{(2)}z_\ell^3+D_2^{(2)}z_\ell^2+D_1^{(2)}z_\ell + D_0^{(2)}$ and $D_2^{(1)}z_\ell^2 + D_1^{(1)}z_\ell + D_0^{(1)}$ matches those defined in \cite{ForestL86} modulo the aforementioned time rescaling and typos.

\subsubsection{Absolute integrability of derivatives}

We now prove the following.
\begin{lemma}
The partial derivatives $\alpha_\chi(\chi,\tau)$ and $\beta_\chi(\chi,\tau)$ are absolutely integrable on 
$(\chi_\mathrm{c}(\tau),+\infty)$.
\label{lem:absolute-integrability}
\end{lemma}
\begin{proof}
Let $\tau$ be fixed.  By symmetry it suffices to consider prove the integrability of $\alpha_\chi$, and since $\chi\mapsto \alpha(\chi,\tau)$ is continuously differentiable with respect to $\chi$ on the open interval $(\chi_\mathrm{c}(\tau),+\infty)$ it is sufficient to examine $\alpha_\chi$ in the limits $\chi\uparrow+\infty$ and $\chi\downarrow \chi_\mathrm{c}(\tau)$.  

\subsubsection*{Absolute integrability of $\alpha_\chi$ at $\chi=+\infty$}
We begin by analyzing $\lambda(\chi,\tau)$ as $\chi\to+\infty$.  Since $\lambda_1(\chi,\tau)=-8\chi^{-2}+O(\chi^{-5})<\lambda(\chi,\tau)<\lambda_2(\chi,\tau)=8\chi^{-2}+O(\chi^{-5})$ for large positive $\chi>0$, we may write $\lambda(\chi,\tau)=2(4-\delta^2)\chi^{-2}$ for $0<\delta^2<8$.  A calculation then shows that, after rescaling by $z=\chi^{-\frac{1}{2}}w$,
\begin{equation}
\chi^2R(\chi^{-\frac{1}{2}}w)^2 = (w^2-2)^2 + \delta^2w^2 + O(\chi^{-\frac{3}{2}})
\label{eq:Rsquared-large-chi}
\end{equation}
in the limit $\chi\to+\infty$, uniformly for bounded $w$.   Next, we rewrite the integral condition \eqref{eq:integralcondition} with the help of the alternate branch $\widetilde{R}(z)$ of the square root of $R(z)^2$ defined in \eqref{eq:widetilde-R}.  $\widetilde{R}(z)$ is analytic except on $\Gamma_{\beta\to\beta^*}\cup\Gamma_{\alpha^*\to\alpha}$ and satisfies $\widetilde{R}(z)=z^2+O(z)$ as $z\to\infty$.  Using $\widetilde{R}(z)$ in place of $R(z)$ gives a corresponding modification $\widetilde{h}'(z)$ of $h'(z)$.  Then the condition \eqref{eq:integralcondition} can be written as
\begin{equation}
\frac{1}{2\ii}\oint_{L_0} \widetilde{h}'(z)\,\dd z=2
\label{eq:integralcondition-loop}
\end{equation}
where $L_0$ is a clockwise-oriented loop surrounding $\Gamma_{\alpha^*\to\alpha}$.  Then, in the scaling $z=\chi^{-\frac{1}{2}}w$,  $\widetilde{h}'(z)=(2\tau z+\chi-\tau^2\lambda)\widetilde{R}(z)z^{-2}$ has the expansion
\begin{equation}
\begin{split}
\widetilde{h}'(\chi^{-\frac{1}{2}}w)&=\chi^2\widetilde{R}(\chi^{-\frac{1}{2}}w)w^{-2}(1+O(\chi^{-\frac{3}{2}}))\\ &=\chi\left[\frac{w^2-2}{w^2}+\frac{\delta^2}{2(w^2-2)} + O\left(\frac{\delta^4}{(w^2-2)^3}\right)+O(\chi^{-\frac{3}{2}})\right].
\end{split}
\end{equation}
Here the error terms are uniform for $w$ bounded and bounded away from $w=\pm\sqrt{2}$.  Because $\dd z = \chi^{-\frac{1}{2}}\,\dd w$, for \eqref{eq:integralcondition-loop} to be satisfied it is necessary that the rescaled branch cuts are very near to $w=\pm\sqrt{2}$, so we can fix a rescaled contour $L:=\chi^\frac{1}{2}L_0$ enclosing the point $w=-\sqrt{2}$ but excluding $w=0,\sqrt{2}$ and then expand \eqref{eq:integralcondition-loop} as
\begin{equation}
\begin{split}
2&=\frac{\chi^{\frac{1}{2}}}{2\ii}\oint_L\left[\frac{w^2-2}{w^2}+\frac{\delta^2}{2(w^2-2)} + O\left(\frac{\delta^4}{(w^2-2)^3}\right)+O(\chi^{-\frac{3}{2}})\right]\,\dd w \\ &= \chi^\frac{1}{2}\left(\frac{\pi\delta^2}{\sqrt{32}}+O(\delta^4)+O(\chi^{-\frac{3}{2}})\right).
\end{split}
\end{equation}
This shows that $\delta^2$ is necessarily large compared with $\chi^{-\frac{3}{2}}$ as $\chi\to+\infty$, and in fact
\begin{equation}
\delta^2 =\frac{\sqrt{128}}{\pi}\chi^{-\frac{1}{2}} + O(\delta^4)+O(\chi^{-\frac{3}{2}}) = \frac{\sqrt{128}}{\pi}\chi^{-\frac{1}{2}} +O(\chi^{-1}).
\label{eq:delta-squared-expansion}
\end{equation}

With $\lambda=\lambda(\chi,\tau)$, combining \eqref{eq:gamma-deriv-2} (replacing $h',R$ with $\widetilde{h}',\widetilde{R}$) with 
\begin{equation}
\frac{\partial\widetilde{h}'}{\partial\chi}(z)=\left(z^2+\frac{1}{2}\tau\lambda z + \tau^2\lambda^2-\frac{3}{4}\chi\lambda+48(\tau^2\lambda-\chi)^{-4}\tau^2\right)\frac{1}{\widetilde{R}(z)},
\end{equation}
implicit differentiation of \eqref{eq:integralcondition-loop} yields
\begin{equation}
\frac{\partial\lambda}{\partial\chi}=\frac{4(\tau^2\lambda-\chi)^4}{\Qpoly(\lambda;\chi,\tau)}\frac{\displaystyle\oint_{L_0}\left(z^2+\frac{1}{2}\tau\lambda z + \tau^2\lambda^2-\frac{3}{4}\chi\lambda+48(\tau^2\lambda-\chi)^{-4}\tau^2\right)\frac{\dd z}{\widetilde{R}(z)}}{\displaystyle\oint_{L_0}\frac{\dd z}{\widetilde{R}(z)}}.
\label{eq:DS-gamma-chi-loop}
\end{equation}
Using $\lambda=2(4-\delta^2)\chi^{-2}$ with $\delta^2=O(\chi^{-\frac{1}{2}})$, it is straightforward to obtain
\begin{equation}
\frac{4(\tau^2\lambda-\chi)^4}{\Qpoly(\lambda;\chi,\tau)}=4\chi^{-2}+192\tau^2\chi^{-5} +O(\chi^{-\frac{11}{2}}),\quad\chi\to+\infty.
\label{eq:Q-prefactor}
\end{equation}
For the ratio of integrals in \eqref{eq:DS-gamma-chi-loop}, the constant terms in the integrand of the numerator have the expansion
\begin{equation}
\tau^2\lambda^2-\frac{3}{4}\chi\lambda+48(\tau^2\lambda-\chi)^{-4}\tau^2 = -6\chi^{-1}+O(\chi^{-\frac{3}{2}}),\quad\chi\to+\infty,
\end{equation}
and for the remaining terms we use the substitution $z=\chi^{-\frac{1}{2}}w$ and the expansion \eqref{eq:Rsquared-large-chi} to obtain
\begin{equation}
\frac{\displaystyle\oint_{L_0}\left(z^2+\frac{1}{2}\tau\lambda z\right)\frac{\dd z}{\widetilde{R}(z)}}{\displaystyle\oint_{L_0}\frac{\dd z}{\widetilde{R}(z)}}=\frac{\displaystyle\oint_L\left(\chi^{-1}w^2+(4-\delta^2)\tau\chi^{-\frac{5}{2}}w\right)\frac{\dd w}{\chi \widetilde{R}(\chi^{-\frac{1}{2}}w)}}{\displaystyle\oint_L\frac{\dd w}{\chi\widetilde{R}(\chi^{-\frac{1}{2}}w)}} = 2\chi^{-1}+O(\chi^{-\frac{3}{2}})
\end{equation}
as $\chi\to+\infty$, where we used the fact that $L$ encloses but is bounded away from the pole at $w=-\sqrt{2}$.  The ratio of integrals is therefore $-4\chi^{-1}+O(\chi^{-\frac{3}{2}})$, so combining with \eqref{eq:Q-prefactor} gives 
\begin{equation}
\frac{\partial\lambda}{\partial \chi} = -16\chi^{-3}+O(\chi^{-\frac{7}{2}}),\quad\chi\to+\infty.
\label{eq:DS-gamma-prime-estimate}
\end{equation}

Now differentiation with respect to $\chi$ for $\alpha=\alpha(\chi,\tau)$ of the monic quartic equation $R(\alpha)^2=\alpha^4-S_1\alpha^3+S_2\alpha^2-S_3\alpha+S_4=0$ yields
\begin{equation}
\frac{\partial\alpha}{\partial\chi}=\frac{\alpha^3b_1-\alpha^2b_2+\alpha b_3-b_4}{(\alpha-\alpha^*)(\alpha-\beta)(\alpha-\beta^*)},
\label{eq:DS-alpha-chi-1-A}
\end{equation}
in which $b_j=S_{j,\chi}$ for $j=1,\dots,4$ are given by
\begin{equation}
\begin{gathered}
b_1=-\tau\frac{\partial\lambda}{\partial\chi},\quad b_2=-\frac{1}{2}\lambda +\frac{1}{2}(3\tau^2\lambda-\chi)\frac{\partial\lambda}{\partial\chi},\\
b_3=-\frac{48\tau}{(\tau^2\lambda-\chi)^4}\left(1-\tau^2\frac{\partial\lambda}{\partial\chi}\right),\quad
b_4=\frac{8}{(\tau^2\lambda-\chi)^3}\left(1-\tau^2\frac{\partial\lambda}{\partial\chi}\right).
\end{gathered}
\label{eq:DS-b-vec-A}
\end{equation}
Since the roots of $R(\chi^{-\frac{1}{2}}w)^2$ converge to $w=\pm\sqrt{2}$ as $\chi\to+\infty$, it is immediate that $|\alpha-\beta|\gtrsim\chi^{-\frac{1}{2}}$ and $|\alpha-\beta^*|\gtrsim\chi^{-\frac{1}{2}}$.  Expanding  \eqref{eq:Rsquared-large-chi} shows that the roots near $w=-\sqrt{2}$ satisfy $w=-\sqrt{2}\pm\frac{1}{2}\ii\delta +O(\delta^2)$ which implies that $|\alpha-\alpha^*|\gtrsim \chi^{-\frac{1}{2}}\delta\gtrsim\chi^{-\frac{3}{4}}$.  For the numerator in \eqref{eq:DS-alpha-chi-1-A}, from $|\alpha|\lesssim\chi^{-\frac{1}{2}}$ we then obtain $\alpha^3b_1=O(\chi^{-\frac{3}{2}}\chi^{-3})=O(\chi^{-\frac{9}{2}})$ and  $\alpha b_3 = O(\chi^{-\frac{1}{2}}\chi^{-4})=O(\chi^{-\frac{9}{2}})$.  Some useful cancellation then occurs in the combination $\alpha^2b_2+b_4$ via $\alpha^2=2\chi^{-1}+O(\chi^{-\frac{5}{4}})$, $b_2=4\chi^{-2}+O(\chi^{-\frac{5}{2}})$, and $b_4=-8\chi^{-3}+O(\chi^{-6})$, hence $\alpha^2b_2+b_4=O(\chi^{-\frac{13}{4}})$.  Using these estimates in \eqref{eq:DS-alpha-chi-1-A} yields
$|\alpha_\chi|\lesssim\chi^{-\frac{3}{2}}$ as $\chi\to+\infty$, which is integrable.

\subsubsection*{Absolute integrability of $\alpha_\chi$ at $\chi=\chi_\mathrm{c}(\tau)$.}
Fix $\tau\in\mathbb{R}$.  For the other limit $\chi\downarrow\chi_\mathrm{c}:=\chi_\mathrm{c}(\tau)$, the roots $z=\alpha,\beta$ of the quartic $R(z)^2$ coalesce at $z=\critpt$, $\mathrm{Im}(\critpt)>0$, which leads to divergence of the contour integrals in \eqref{eq:DS-gamma-chi-loop} and logarithmic corrections in the sub-leading order terms in the expansion of $f(\lambda;\chi,\tau)$ defined in \eqref{eq:integralcondition-rewrite}.  To study this behavior, we first recall that since $h(z)$ and $\phase(z;\chi,\tau)$ agree when $\chi=\chi_\mathrm{c}$, using \eqref{eq:chi-tau-sigma}--\eqref{eq:gamma-negative} we may parametrize the limiting values $(\chi_\mathrm{c},\tau,\lambda_\mathrm{c})$ of $(\chi,\tau,\lambda)$ by the common limiting value $\critpt=u+\ii v$ (with $v>0$) of $\alpha,\beta$. With this notation, from \eqref{eq:DS-gamma-chi-loop} one can observe that the polynomial in the integrand of the numerator has the same limiting value at the two points $z=\critpt,\critpt^*$ at which the denominator becomes small:
\begin{equation}
\left.\left(z^2+\frac{1}{2}\tau\lambda_\mathrm{c} z\right)\right|_{z=u\pm\ii v}=-(u^2+v^2).
\end{equation}
Therefore, the most singular contribution of the numerator integral in the limit $\chi\downarrow\chi_\mathrm{c}$ is explicitly proportional to that of the denominator integral, so the ratio of these integrals tends to the limit $18u^2+2v^2$.  Evaluating the remaining factors in \eqref{eq:DS-gamma-chi-loop} in the limit, we therefore deduce that the partial derivative $\lambda_\chi$ has a well-defined finite limiting value as $\chi\downarrow\chi_\mathrm{c}$ that we denote by $\lambda_{\chi,\mathrm{c}}$:
\begin{equation}
\lambda_{\chi,\mathrm{c}}:=\lim_{\chi\downarrow\chi_\mathrm{c}}\frac{\partial\lambda}{\partial\chi}(\chi,\tau)=\frac{2(u^2+v^2)^4}{9u^2+v^2}.
\end{equation}
Since its derivative is continuous for $\chi\ge \chi_\mathrm{c}$, it follows that we may represent $\lambda(\chi,\tau)$ in the form
\begin{equation}
\lambda(\chi,\tau)=\lambda_\mathrm{c}+\lambda_{\chi,\mathrm{c}}\Delta\chi + \Delta\lambda,\quad\Delta\chi:=\chi-\chi_\mathrm{c},
\end{equation}
where $\Delta\lambda=o(\Delta\chi)$ as $\Delta\chi\downarrow 0$.

We now determine the leading term of $\Delta\lambda$.  Starting from \eqref{eq:hprime-Rsquared}, with $\gamma=-\chi+\tau^2\lambda$, we write $h'(z)$ in the form 
\begin{equation}
h'(z)=\frac{1}{z^2}L(z)r(z)r(z^*)^*,
\end{equation}
where $L(z)$ is the linear function
\begin{equation}
L(z):=2\tau z+\chi_\mathrm{c}-\tau^2\lambda_\mathrm{c}+(1-\tau^2\lambda_{\chi,c})\Delta\chi - \tau^2\Delta\lambda,
\end{equation}
and where $r(z)^2=(z-\alpha)(z-\beta)$ and $r(z)=z+O(1)$ as $z\to\infty$ with $r$ having a straight-line branch cut joining $\alpha,\beta$.  Let $\widehat{\critpt}=\frac{1}{2}(\alpha+\beta)$ denote the midpoint of the branch cut.  Then since $r(z)$ has opposite signs on opposite sides of its branch cut and elsewhere is analytic, in the definition \eqref{eq:integralcondition-rewrite} we may replace the integral over a path from $\alpha^*$ to $\alpha$ by the average of two integrals over Schwarz-symmetric paths denoted $\widehat{A}_\pm$, each of which goes from $\widehat{\critpt}_\pm^*$ to $\widehat{\critpt}_\pm$ in the domain of analyticity of $h'(z)$, where $\widehat{\critpt}_\pm$ are points at $z=\widehat{\critpt}$ on opposite sides of the branch cut for $r$.  It is not important where $\widehat{A}_\pm$ cross the real line other than the pole $z=0$ of $h'(z)$, because the residue of $h'(z)$ at $z=0$ vanishes.  Concretely, we assume that $\widehat{A}_\pm$ are as follows.  Let $C$ be the circle centered at $\widehat{\critpt}$ with radius $\frac{1}{2}\mathrm{Im}(\widehat{\critpt})$; let $J$ denote one of the two points on $C$ with $\mathrm{Im}(J)=\mathrm{Im}(\widehat{\critpt})$ for which $|\Re(J)|$ is largest.   Thus $\mathrm{Re}(J)\neq 0$.  Finally, let $A_\pm^\mathrm{in}$ denote the two segments from antipodal points $J_\pm$ of $C$ to $\widehat{\critpt}$ that are perpendicular to the branch cut between $\alpha$ and $\beta$.  Then $\widehat{A}_\pm\cap\mathbb{C}$ consists of the vertical line from $z=\mathrm{Re}(J)\neq 0$ to $J$ followed by a shortest arc of $C$ terminating at $J_\pm$ followed by $\widehat{A}_\pm^\mathrm{in}$.  Thus we arrive at the rewritten integral condition
\begin{equation}
\begin{split}
0=f(\lambda;\chi,\tau):=&\frac{1}{2\ii}\int_{\widehat{A}_+\cup \widehat{A}_-} L(z)r(z)r(z^*)^*\frac{\dd z}{z^2} -2
\\=&\mathrm{Im}\left(\int_{(\widehat{A}_+\cup \widehat{A}_-)\cap\mathbb{C}_+} L(z)r(z)r(z^*)^*\frac{\dd z}{z^2}\right)-2,
\end{split}
\end{equation}
where the second formula follows by Schwarz symmetry.
Now let $\rho:=|\alpha-\widehat{\critpt}|=|\beta-\widehat{\critpt}|=\frac{1}{2}|\beta-\alpha|$.  We use the following representation of $r(z^*)^*$:
\begin{equation}
r(z^*)^*=(z-\widehat{\critpt}^*)\left(1-\frac{(\beta^*-\alpha^*)^2}{4(z-\widehat{\critpt}^*)^2}\right)^\frac{1}{2}=z-\widehat{\critpt}^* + m(z)
\end{equation}
where by Taylor expansion it follows that $m(z)=O(\rho^2)$ holds uniformly for $z$ in bounded subsets of $\mathbb{C}_+$ since $\widehat{\critpt}^*\in\mathbb{C}_-$.  We also use
a corrected representation of $r(z)$:
\begin{equation}
r(z)=z-\widehat{\critpt} -\frac{1}{8}(\beta-\alpha)^2\left(\frac{1}{z-\widehat{\critpt}}\right)_\rho + e(z),
\end{equation}
where $(1/\zeta)_\rho:=\chi_{|\zeta|>\rho}(\zeta)/\zeta$ is a cutoff version of $1/\zeta$.  By Taylor expansion in $\frac{1}{4}(\beta-\alpha)/(z-\widehat{\critpt})$, it is easy to see that $e(z)=O(\rho^3/(z-\widehat{\critpt})^2)$ for $|z-\widehat{\critpt}|>\rho$;  also since $e(z)=r(z)-(z-\widehat{\critpt})$ for $|z-\widehat{\critpt}|\le \rho$, using $|r(z)|\le (|z-\widehat{\critpt}|+\rho)$ it follows that $e(z)=O(\rho)$ for $|z-\widehat{\critpt}|\le\rho$.  These estimates imply that $e\in L^1((\widehat{A}_+\cup\widehat{A}_-)\cap\mathbb{C}_+)$ with norm
\begin{equation}
\|e\|_{L^1((\widehat{A}_+\cup\widehat{A}_-)\cap\mathbb{C}_+)}\lesssim
\mathop{\int_{(\widehat{A}_+\cup \widehat{A}_-)\cap\mathbb{C}_+}}_{|z-\widehat{\critpt}|>\rho} \frac{\rho^3\,|\dd z|}{|z-\widehat{\critpt}|^2} +\mathop{\int_{(\widehat{A}_+\cup \widehat{A}_-)\cap\mathbb{C}_+}}_{|z-\widehat{\critpt}|\le\rho} \rho\,|\dd z|\lesssim \rho^2.
\end{equation}
Since $L(z)/z^2$ is uniformly bounded on the union of contours $(\widehat{A}_+\cup\widehat{A}_-)\cap\mathbb{C}_+$ as $\widehat{\critpt}\to\critpt$,  it then follows that 
\begin{equation}
f(\lambda;\chi,\tau)=\mathrm{Im}\left(\int_{(\widehat{A}_+\cup \widehat{A}_-)\cap\mathbb{C}_+}L(z)\left[z-\widehat{\critpt}-\frac{1}{8}(\beta-\alpha)^2\left(\frac{1}{z-\widehat{\critpt}}\right)_\rho\right](z-\widehat{\critpt}^*)\,\frac{\dd z}{z^2}\right)-2 + O(\rho^2).
\label{eq:f-expand-crit-1}
\end{equation}
We explicitly evaluate the terms not involving the cutoff function defined as
\begin{equation}
f_0(\lambda;\chi,\tau):=\mathrm{Im}\left(\int_{(\widehat{A}_+\cup \widehat{A}_-)\cap\mathbb{C}_+}L(z)(z-\widehat{\critpt})(z-\widehat{\critpt}^*)\frac{\dd z}{z^2}\right)-2
\label{eq:f0-definition}
\end{equation}
by first noticing that since the integrand vanishes to first order at the common endpoint $\widehat{\critpt}$ of $\widehat{A}_\pm$, and since $\widehat{\critpt}-\critpt=O(\rho)$, we may write
\begin{equation}
f_0(\lambda;\chi,\tau)=\mathrm{Im}\left(\int_{(A_+\cup A_-)\cap\mathbb{C}_+}L(z)(z-\widehat{\critpt})(z-\widehat{\critpt}^*)\frac{\dd z}{z^2}\right)-2+O(\rho^2),
\end{equation}
where $A_\pm$ now are contours terminating at $\critpt$ instead of $\widehat{\critpt}$.  
We next observe that, because the integral condition is definitely satisfied when $\Delta\chi=0$ and $\Delta\lambda=0$ and $\alpha=\beta=\critpt$ (so that also $r(z)r(z^*)^*=(z-\critpt)(z-\critpt^*)$), we have the identity
\begin{equation}
\mathrm{Im}\left(\int_{(A_+\cup A_-)\cap\mathbb{C}_+}L_0(z)(z-\critpt)(z-\critpt^*)\frac{\dd z}{z^2}\right)=2,\quad L_0(z):=2\tau z+\chi_\mathrm{c}-\tau^2\lambda_\mathrm{c}.
\end{equation}
Therefore, since $L(z)-L_0(z)=(1-\tau^2\lambda_{\chi,\mathrm{c}})\Delta\chi -\tau^2\Delta\lambda=(1-\tau^2\lambda_{\chi,\mathrm{c}})\Delta\chi + o(\Delta\chi)$,  and
\begin{equation}
\begin{split}
(z-\widehat{\critpt})(z-\widehat{\critpt}^*)-(z-\critpt)(z-\critpt^*)&=-2\mathrm{Re}(\widehat{\critpt}-\critpt)z+|\widehat{\critpt}|^2-|\critpt|^2\\
&=-2\mathrm{Re}(\widehat{\critpt}-\critpt)z +2\mathrm{Re}((\widehat{\critpt}-\critpt)\critpt^*) + O(\rho^2),
\end{split}
\end{equation}
which in particular is $o(1)$ as $\Delta\chi\to 0$, we get
\begin{multline}
f_0(\lambda;\chi,\tau)=(1-\tau^2\lambda_{\chi,\mathrm{c}})\Delta\chi\mathrm{Im}\left(\int_{(A_+\cup A_-)\cap\mathbb{C}_+}(z-\critpt)(z-\critpt^*)\frac{\dd z}{z^2}\right) \\+
\mathrm{Im}\left(\int_{(A_+\cup A_-)\cap\mathbb{C}_+}L_0(z)\left[-2\mathrm{Re}(\widehat{\critpt}-\critpt)z+2\mathrm{Re}((\widehat{\critpt}-\critpt)\critpt^*)\right]\frac{\dd z}{z^2}\right) + o(\Delta\chi)+O(\rho^2).
\label{eq:f0-almost-there}
\end{multline}
The last step in evaluating $f_0(\lambda;\chi,\tau)$ is to determine the leading asymptotic behavior of $\widehat{\critpt}-\critpt$, which turns out to be proportional to $\Delta\chi$.

For this purpose, we now study the asymptotic behavior of $z=\alpha,\beta$ by seeking small roots $w$ of the quartic $R(\critpt+w)^2$.  To simplify the use of the expansions of $\chi$ and $\lambda$, we first multiply through by $(\tau^2\lambda-\chi)^3\neq 0$ (see Lemma~\ref{lem:gamma-nonzero}).  Then $(\tau^2\lambda-\chi)^3R(\critpt+w)^2$ has coefficients that are all polynomials in $\Delta\chi$ and $\Delta\lambda$ and that are rational in $(u,v)$ with common denominator $2(u^2+v^2)^{16}(9u^2+v^2)^5$.  Therefore clearing that denominator:
\begin{equation}
(\tau^2\lambda-\chi)^3R(\critpt+w)^2=\frac{p(w)}{2(u^2+v^2)^{16}(9u^2+v^2)^5},\quad p(w):=\sum_{j=0}^4p_jw^j,
\end{equation}
where the coefficients $p_j$ are polynomials in $u,v,\Delta\chi,\Delta\lambda$ with expansions with respect to the small parameters $\Delta\chi$ and $\Delta\lambda$ of the form:
\begin{equation}
\begin{gathered}
p_4=O(1),
\quad p_3=O(1),\quad p_2=64v^2(u^2+v^2)^{13}(9u^2+v^2)^5 + O(\Delta\chi),\\
p_1=-32\ii v(u+\ii v)(3u+\ii v)(u^2+v^2)^{15}(9u^2+v^2)^4\Delta\chi + O(\Delta\lambda),\\
p_0=16(u+\ii v)(3u+\ii v)(u^2+v^2)^{11}(9u^2+v^2)^6\Delta\lambda + O(\Delta\chi^2).
\end{gathered}
\end{equation}
Here we used the relations $\Delta\lambda=o(\Delta\chi)$ and $\Delta\chi=o(1)$ only to express the error terms.  Setting $p(w)=0$ and seeking dominant balances for small $w$, since $p_2$ has a positive limit as $\Delta\chi\to 0$, the terms $p_4w^4+p_3w^3$ are definitely negligible compared to $p_2w^2$.  From the remaining quadratic, linear, and constant terms, we can determine the leading-order asymptotic behavior of the sum and product of the roots:
\begin{equation}
w_1+w_2=\alpha+\beta-2\critpt=-\frac{p_1}{p_2}(1+o(1))=\frac{\ii (u+\ii v)(3u+\ii v)(u^2+v^2)^2}{2v(9u^2+v^2)}\Delta\chi(1+o(1)),
\label{eq:sum-of-roots}
\end{equation}
\begin{equation}
w_1w_2=(\alpha-\critpt)(\beta-\critpt)=\frac{p_0}{p_2}(1+o(1))=\frac{(u+\ii v)(3u+\ii v)(9u^2+v^2)}{4v^2(u^2+v^2)^2}\Delta\lambda(1+o(1)).
\label{eq:product-of-roots}
\end{equation}
In particular since $\alpha+\beta=2\widehat{\critpt}$, we can combine \eqref{eq:sum-of-roots} with \eqref{eq:f0-almost-there} and use $v=\mathrm{Im}(\critpt)$ to give a remarkably simple formula:
\begin{equation}
f_0(\lambda;\chi,\tau)=\mathrm{Im}\left(\int_{(A_+\cup A_-)\cap\mathbb{C}_+}\,\dd z\right)\Delta\chi + o(\Delta\chi)+O(\rho^2) =2v\Delta\chi + o(\Delta\chi)+O(\rho^2).
\end{equation}

Therefore, going back to $f(\lambda;\chi,\tau)$ using \eqref{eq:f-expand-crit-1}--\eqref{eq:f0-definition} we have
\begin{multline}
f(\lambda;\chi,\tau)=2v\Delta\chi -\frac{1}{8} \mathrm{Im}\left((\beta-\alpha)^2\int_{(\widehat{A}_+\cup\widehat{A}_-)\cap\mathbb{C}_+}L(z)\left(\frac{1}{z-\widehat{\critpt}}\right)_\rho (z-\widehat{\critpt}^*)\frac{\dd z}{z^2}\right)\\+o(\Delta\chi)+O(\rho^2).
\end{multline}
Since $(\beta-\alpha)^2=O(\rho^2)$, all bounded contributions from the integral can be absorbed into the error terms.  Due to the cutoff $\chi_{|z-\widehat{\critpt}|>\rho}(z)$ implicit in the integrand, we are essentially integrating in $z$ up to a pair of antipodal points a distance $\rho$ from $\widehat{\critpt}$, and the only contribution to the integral that grows as $\rho\to 0$ comes from a logarithmic singularity at $z=\widehat{\critpt}$.  Thus,
\begin{equation}
\int_{(\widehat{A}_+\cup\widehat{A}_-)\cap\mathbb{C}_+}L(z)\left(\frac{1}{z-\widehat{\critpt}}\right)_\rho (z-\widehat{\critpt}^*)\frac{\dd z}{z^2}=-2\frac{L(\widehat{\critpt})(\widehat{\critpt}-\widehat{\critpt}^*)}{\widehat{\critpt}^2}\ln\left(\frac{1}{\rho}\right) + O(1),\quad\rho\to 0.
\end{equation}
Since $\widehat{\critpt}=\critpt+O(\rho)$ and $\rho\ln(1/\rho)$ is bounded as $\rho\to 0$, we can replace $\widehat{\critpt}$ with the limiting value $\critpt$ and absorb the difference into the error terms.  Therefore,
\begin{equation}
\begin{split}
f(\lambda;\chi,\tau)&=2v\Delta\chi +\frac{1}{4}\mathrm{Im}\left(\frac{L(\critpt)(\critpt-\critpt^*)}{\critpt^2}(\beta-\alpha)^2\right)\ln\left(\frac{1}{\rho}\right) + o(\Delta\chi)+O(\rho^2)\\
&=2v\Delta\chi +\frac{1}{4}\mathrm{Im}\left(\frac{L_0(\critpt)(\critpt-\critpt^*)}{\critpt^2}(\beta-\alpha)^2\right)\ln\left(\frac{1}{\rho}\right) + o(\Delta\chi)+O(\rho^2),
\end{split}
\label{eq:f-lambda-final}
\end{equation}
where on the second line we used the fact that $\rho^2\ln(1/\rho)=o(1)$ as $\Delta\chi\to 0$ and hence also $\rho\to 0$ to replace $L(\critpt)$ with $L_0(\critpt)$.  Note that in terms of $\critpt=u+\ii v$, we have
\begin{equation}
\frac{L_0(\critpt)(\critpt-\critpt^*)}{\critpt^2}=\frac{4\ii v(3u-\ii v)(u-\ii v)}{(u^2+v^2)^3}.
\end{equation}

Now, $(\beta-\alpha)^2=((\beta-\critpt)-(\alpha-\critpt))^2 = (\alpha+\beta-2\critpt))^2-4(\alpha-\critpt)(\beta-\critpt)=(w_1+w_2)^2-4w_1w_2$, so this quantity can be approximated using \eqref{eq:sum-of-roots} and \eqref{eq:product-of-roots}, and it is a sum of terms proportional to $\Delta\chi^2$ and $\Delta\lambda$.  A priori, it is not clear which of these terms is dominant, although it is known that $\Delta\lambda=o(\Delta\chi)$.  To answer this question, suppose that $\Delta\lambda=O(\Delta\chi^2)$.  Then also $(\beta-\alpha)^2=O(\Delta\chi^2)$ and hence $\rho=O(\Delta\chi)$, so using \eqref{eq:f-lambda-final} the integral condition on $\lambda$ becomes
\begin{equation}
0=f(\lambda;\chi,\tau)=2v\Delta\chi +o(\Delta\chi)\quad\text{if $\Delta\lambda=O(\Delta\chi^2)$,}
\end{equation}
because $\rho\mapsto\rho^2\ln(1/\rho)$ is monotone increasing for small $\rho>0$.
Since $v=\mathrm{Im}(\critpt)>0$, this is clearly a contradiction because $\Delta\chi>0$ holds for $\chi>\chi_\mathrm{c}$.  Therefore, in fact $\Delta\lambda$ is large compared with $\Delta\chi^2$, and so
\begin{equation}
(\beta-\alpha)^2=-4w_1w_2(1+o(1))=-\frac{(u+\ii v)(3u+\ii v)(9u^2+v^2)}{v^2(u^2+v^2)^2}\Delta\lambda (1+o(1)).
\end{equation}
So, multiplying by $L_0(\critpt)(\critpt-\critpt^*)/\critpt^2$ gives a product that is purely imaginary to leading order and therefore the integral condition on $\lambda$ actually reads
\begin{equation}
0=f(\lambda;\chi,\tau)=2v\Delta\chi -\frac{(9u^2+v^2)^2}{2v(u^2+v^2)^4}\Delta\lambda\ln\left(\frac{1}{|\Delta\lambda|}\right) + o(\Delta\chi),
\label{eq:Delta-lambda-asymptotic-c}
\end{equation}
where we used the fact that $\ln(1/\rho)=\frac{1}{2}\ln(1/|\Delta\lambda|)+O(1)$.

Finally, we return to \eqref{eq:DS-alpha-chi-1-A} and \eqref{eq:DS-b-vec-A} and note that since $\alpha\to\critpt=u+\ii v$, $\lambda\to\lambda_\mathrm{c}$ and $\lambda_\chi\to\lambda_{\chi,\mathrm{c}}$,
\begin{equation}
\lim_{\chi\downarrow\chi_\mathrm{c}} \alpha^3b_1-\alpha^2b_2+\alpha b_3-b_4 = 0.
\end{equation}
Also, since $(\alpha-\alpha^*)(\alpha-\beta^*)\to -4\mathrm{Im}(\critpt)^2\neq 0$ we clearly have
\begin{equation}
\left|\frac{\partial\alpha}{\partial\chi}\right|=o\left(\frac{1}{|\beta-\alpha|}\right)=o\left(\frac{1}{|\Delta\lambda|^{\frac{1}{2}}}\right)=o\left(\left(|\Delta\lambda|\ln\left(\frac{1}{|\Delta\lambda|}\right)\right)^{-\frac{1}{2}-\epsilon}\right),\quad\chi\downarrow\chi_\mathrm{c}(\tau)
\end{equation}
holds for every $\epsilon>0$.  Therefore, using \eqref{eq:Delta-lambda-asymptotic-c} gives $|\alpha_\chi|=o(\Delta\chi^{-\frac{1}{2}-\epsilon})$ which is integrable at $\Delta\chi=\chi-\chi_\mathrm{c}(\tau)=0$ by taking $\epsilon<\frac{1}{2}$.

This completes the proof.
\end{proof}

\subsection{Use of $g(z)$}
Assuming $\chi>\chi_\mathrm{c}(\tau)$,
we take the jump contour $\Gamma$ for Riemann-Hilbert Problem~\ref{rhp:S} to be that described at the end of Section~\ref{s:DS-Constructing-g}. 
We introduce a new
unknown $\mathbf{T}(z;\chi,\tau,M)$ by setting
\begin{equation}
\mathbf{T}(z;\chi,\tau,M):=\mathbf{S}(z;\chi,\tau,M)\ee^{\ii M g(z)\sigma_3},\quad z\in\mathbb{C}\setminus\Gamma.
\label{eq:DS-T-from-S}
\end{equation}
Then, from the conditions of Riemann-Hilbert Problem~\ref{rhp:S} we see that $\mathbf{T}(\cdot;\chi,\tau,M)$ is analytic in its domain of definition, takes continuous boundary values on $\Gamma$ from each side, and tends to $\mathbb{I}$ as $z\to\infty$.  Moreover, since $g(z)=g(z^*)^*$, the matrix $\mathbf{T}(z;\chi,\tau,M)$ inherits from $\mathbf{S}(z;\chi,\tau,M)$ the Schwarz symmetry \eqref{eq:S-Schwarz}.
The jump conditions satisfied by $\mathbf{T}(z;\chi,\tau,M)$ across the arcs of $\Gamma$ read as follows:
\begin{equation}
\mathbf{T}_+(z;\chi,\tau,M)=\mathbf{T}_-(z;\chi,\tau,M)\begin{bmatrix}\sqrt{1-\ee^{-4M}} & -\ee^{-2\ii M(h(z)-\ii)}\\\ee^{2\ii M(h(z)+\ii)} &\sqrt{1-\ee^{-4M}}\end{bmatrix},\quad z\in\Gamma_{\alpha^*\to\alpha},
\end{equation}
\begin{multline}
\mathbf{T}_+(z;\chi,\tau,M)\begin{bmatrix}1 & 0\\\sqrt{1-\ee^{-4M}}\ee^{2\ii M(h_+(z)-\ii)} & 1\end{bmatrix}\\
=\mathbf{T}_-(z;\chi,\tau,M)\begin{bmatrix}1 & 0\\-\sqrt{1-\ee^{-4M}}\ee^{2\ii M(h_-(z)-\ii)} & 1\end{bmatrix}
\begin{bmatrix}0 & -\ee^{-\ii M\phi}\\\ee^{\ii M\phi} & 0\end{bmatrix},\quad z\in\Gamma_{\alpha\to\beta},
\end{multline}
\begin{equation}
\mathbf{T}_+(z;\chi,\tau,M)=\mathbf{T}_-(z;\chi,\tau,M)\begin{bmatrix}\sqrt{1-\ee^{-4M}}\ee^{\ii M\Delta} & -\ee^{-\ii M(h_+(z)+h_-(z)-2\ii)}\\
\ee^{\ii M(h_+(z)+h_-(z)+2\ii)} & \sqrt{1-\ee^{-4M}}\ee^{-\ii M\Delta}\end{bmatrix},\quad z\in\Gamma_{\beta\to\beta^*}, 
\end{equation}
and there is a jump condition on $\Gamma_{\beta^*\to\alpha^*}$ that follows from that on $\Gamma_{\alpha\to\beta}$ by Schwarz symmetry.
We have written the jump condition across $\Gamma_{\alpha\to\beta}$ in factorized form to facilitate the opening of small lenses surrounding this arc and its reflection $\Gamma_{\beta^*\to\alpha^*}$.  Letting $\Omega^+_{\alpha\to\beta}$ and $\Omega^-_{\alpha\to\beta}$ denote lens-shaped regions to the left and right respectively of $\Gamma_{\alpha\to\beta}$,
we define
\begin{equation}
\mathbf{O}(z;\chi,\tau,M):=\mathbf{T}(z;\chi,\tau,M)\begin{bmatrix}1&0\\\pm \sqrt{1-\ee^{-4M}}\ee^{2\ii M(h(z)-\ii)} & 1\end{bmatrix},\quad z\in\Omega^\pm_{\alpha\to\beta}.
\end{equation}
To preserve Schwarz symmetry in the form \eqref{eq:S-Schwarz} for $\mathbf{O}(z;\chi,\tau,M)$ we make corresponding substitutions in the reflected domains $\Omega^{\pm*}_{\alpha\to\beta}$
and elsewhere we set $\mathbf{O}(z;\chi,\tau,M):=\mathbf{T}(z;\chi,\tau,M)$.  The jump conditions for $\mathbf{O}(z;\chi,\tau,M)$ read as follows.  Let $\Lambda^\pm_{\alpha\to\beta}$ denote the outer boundary of the lens domain $\Omega^\pm_{\alpha\to\beta}$ with the orientation from $\alpha$ to $\beta$.  
Then,
\begin{equation}
\mathbf{O}_+(z;\chi,\tau,M)=\mathbf{O}_-(z;\chi,\tau,M)\begin{bmatrix}\sqrt{1-\ee^{-4M}} & -\ee^{-2\ii M(h(z)-\ii)}\\\ee^{2\ii M(h(z)+\ii)} & \sqrt{1-\ee^{-4M}}\end{bmatrix},\quad z\in\Gamma_{\alpha^*\to\alpha},
\end{equation}
\begin{equation}
\mathbf{O}_+(z;\chi,\tau,M)=\mathbf{O}_-(z;\chi,\tau,M)\begin{bmatrix}0 & -\ee^{-\ii M\phi}\\\ee^{\ii M\phi} & 0\end{bmatrix},\quad z\in\Gamma_{\alpha\to\beta},
\end{equation}
\begin{equation}
\mathbf{O}_+(z;\chi,\tau,M)=\mathbf{O}_-(z;\chi,\tau,M)\begin{bmatrix}1 & 0\\-\sqrt{1-\ee^{-4M}}\ee^{2\ii M(h(z)-\ii)} & 1\end{bmatrix},\quad z\in\Lambda^\pm_{\alpha\to\beta},
\end{equation} 
\begin{equation}
\mathbf{O}_+(z;\chi,\tau,M)=\mathbf{O}_-(z;\chi,\tau,M)\begin{bmatrix}\sqrt{1-\ee^{-4M}}\ee^{\ii M\Delta} & -\ee^{-\ii M(h_+(z)+h_-(z)-2\ii)}\\
\ee^{\ii M(h_+(z)+h_-(z)+2\ii)} & \sqrt{1-\ee^{-4M}}\ee^{-\ii M\Delta} \end{bmatrix},\quad z\in\Gamma_{\beta\to\beta^*},
\end{equation}
and on the reflected contours $\Gamma_{\beta^*\to\alpha^*}$ and $\Lambda^{\pm*}_{\alpha\to\beta}$ there are corresponding jump conditions induced by Schwarz symmetry.
Since $\mathbf{O}(z;\chi,\tau,M)=\mathbf{T}(z;\chi,\tau,M)$ holds for $|z|$ sufficiently large, while $\mathbf{T}(z;\chi,\tau,M)$ is related to $\mathbf{S}(z;\chi,\tau,M)$ by \eqref{eq:DS-T-from-S} where $g(\infty)=0$, it follows from \eqref{eq:DS-Psi-from-S} that also
\begin{equation}
M\Psi(M^2\chi,M^3\tau;\mathbf{G}(\ee^{-2M},\sqrt{1-\ee^{-4M}}))=2\ii\lim_{z\to\infty}zO_{12}(z;\chi,\tau,M).
\label{eq:DS-Psi-from-O}
\end{equation}

A useful formula for the constant $\Delta\in\mathbb{R}$ can be found as follows.  Let $z$ denote the real point on the arc $\Gamma_{\beta\to\beta^*}$.  Then, since $\phase(z)$ is single-valued,
\begin{equation}
\Delta=h_+(z)-h_-(z) = (g_+(z)+\phase_+(z))-(g_-(z)+\phase_-(z)) = g_+(z)-g_-(z)
\end{equation}
assuming also that $z\neq 0$ (we can deform $\Gamma_{\beta\to\beta^*}$ if necessary to ensure this).
Using \eqref{eq:DS-g-formula} and integrating on the real line from $\infty=\pm\infty$ for calculating $g_\pm(z)$,
\begin{equation}
\Delta = \int_{+\infty}^z\left[h'(Z)-\phase'(Z)\right]\,\dd Z - \int_{-\infty}^z\left[h'(Z)-\phase'(Z)\right]\,\dd Z = -\int_\mathbb{R}\left[h'(Z)-\phase'(Z)\right]\,\dd Z.
\end{equation}
Clearly the result is independent of $z$.  Using \eqref{eq:hprime-Rsquared} and the definition of $\phase(z)$ in \eqref{eq:DS-tildevartheta} gives
\begin{equation}
\Delta = -\int_\mathbb{R}\left[\frac{2\tau z+\chi-\tau^2\lambda}{z^2}R(z)-\chi-2\tau z-\frac{2}{z^2}\right]\,\dd z.
\end{equation}
The apparent singularities of the integrand at $z=0$ and $z=\infty$ cancel by choice of the branch points $\alpha=\alpha(\chi,\tau)$ and $\beta=\beta(\chi,\tau)$, so the integral is absolutely convergent and since the integrand is analytic except for the branch cuts $\Gamma_{\alpha\to\beta}$ and $\Gamma_{\beta^*\to\alpha^*}$, Cauchy's theorem can be applied to replace $\mathbb{R}$ with a counterclockwise-oriented loop surrounding the cut $\Gamma_{\alpha\to\beta}$.  Then, with this replacement, the contribution to the integrand from $\phase'(z)$ vanishes as it is analytic inside the loop, so ultimately the result is that
\begin{equation}
\Delta = 2\int_{\Gamma_{\alpha\to\beta}}\frac{2\tau z+\chi-\tau^2\lambda}{z^2}R_+(z)\,\dd z.
\label{eq:DS-Delta-equation}
\end{equation}
Combining the first identities in \eqref{eq:Ss-gamma-chi-tau} and \eqref{eq:SymmetricPolynomials-lambda} gives $-\tau^2\lambda=2\tau\mathrm{Re}(\alpha+\beta)$, so to implement this formula requires only determining $\alpha$ and $\beta$ as functions of $(\chi,\tau)\in\mathbb{R}^2$ with $\chi>\chi_\mathrm{c}(\tau)$.

\subsection{Outer parametrix}
\label{sec:DS-outer}
The properties of $h(z)$ now imply that all of the exponential factors involving $h(z)$ decay rapidly to zero as $M\to+\infty$, although the decay is not uniform in neighborhoods of the four points $\alpha,\beta,\alpha^*,\beta^*$.  The outer parametrix captures the pointwise asymptotic behavior of the jump matrices.
\begin{rhp}[Outer parametrix]
Let $\phi$ and $\Delta$ be given real numbers.  Seek a $2\times 2$ matrix-valued function $\breve{\mathbf{O}}^\mathrm{out}(z)$ with the following properties:
\begin{itemize}
\item[]\textbf{Analyticity:} $\breve{\mathbf{O}}^\mathrm{out}(z)$ is analytic in $z$ for $z\in\mathbb{C}\setminus\Gamma_{\alpha\to\beta}\cup\Gamma_{\beta\to\beta^*}\cup\Gamma_{\beta^*\to\alpha^*}$, and it takes continuous boundary values except near the four points $\alpha,\beta,\alpha^*,\beta^*$ at which negative one-fourth power singularities are admissible in all four matrix elements.
\item[]\textbf{Jump conditions:}  The boundary values on the jump contour are related as follows:
\begin{equation}
\breve{\mathbf{O}}^\mathrm{out}_+(z)=\breve{\mathbf{O}}^\mathrm{out}_-(z)\begin{bmatrix}0 & -\ee^{-\ii M\phi}\\\ee^{\ii M\phi} & 0\end{bmatrix},\quad z\in\Gamma_{\alpha\to\beta}\cup\Gamma_{\beta^*\to\alpha^*},
\end{equation}
\begin{equation}
\breve{\mathbf{O}}^\mathrm{out}_+(z)=\breve{\mathbf{O}}^\mathrm{out}_-(z)\begin{bmatrix}\ee^{\ii M\Delta}&0\\0 & \ee^{-\ii M\Delta}\end{bmatrix},\quad z\in\Gamma_{\beta\to\beta^*}.
\end{equation}
\item[]\textbf{Normalization:}  $\breve{\mathbf{O}}^\mathrm{out}(z)\to\mathbb{I}$ as $z\to\infty$.
\end{itemize}
\label{rhp:O-out}
\end{rhp}

To solve this problem, first note that the matrix $\mathbf{F}(z):=\ee^{-\ii\pi\sigma_3/4}\ee^{\ii M\phi\sigma_3/2}\breve{\mathbf{O}}^\mathrm{out}(z)\ee^{-\ii M\phi\sigma_3/2}\ee^{\ii\pi\sigma_3/4}$ satisfies exactly the same conditions as does $\breve{\mathbf{O}}^\mathrm{out}(z)$ except that 
the jump conditions on $\Gamma_{\alpha\to\beta}\cup\Gamma_{\beta^*\to\alpha^*}$ are reduced to the simple form $\mathbf{F}_+(z)=\mathbf{F}_-(z)\ii\sigma_1$.

Let $j(z)$ denote the function analytic for $z\in\mathbb{C}\setminus(\Gamma_{\alpha\to\beta}\cup\Gamma_{\beta^*\to\alpha^*})$ with asymptotic behavior $j(z)\to 1$ as $z\to\infty$ and that satisfies 
\begin{equation}
j(z)^4=\frac{(z-\alpha)(z-\beta^*)}{(z-\alpha^*)(z-\beta)}.
\end{equation}
This function satisfies the scalar jump condition $j_+(z)=-\ii j_-(z)$ on both arcs of 
its jump contour:  $\Gamma_{\alpha\to\beta}\cup\Gamma_{\beta^*\to\alpha^*}$.
From $j(z)$, we define two related functions:
\begin{equation}
F^\mathrm{D}(z):=\frac{1}{2}\left(j(z)+j(z)^{-1}\right),\quad F^\mathrm{OD}(z):=\frac{1}{2\ii}\left(j(z)-j(z)^{-1}\right).
\end{equation}
Note that 
\begin{equation}
\begin{split}
F^\mathrm{D}(z)F^\mathrm{OD}(z)&=\frac{1}{4\ii j(z)^2}\left(j(z)^4-1\right)\\
&=\frac{(z-\alpha)(z-\beta^*)-(z-\alpha^*)(z-\beta)}{4\ii j(z)^2(z-\alpha^*)(z-\beta)}\\
&=\frac{(\alpha^*+\beta-\alpha-\beta^*)z+\alpha\beta^*-\alpha^*\beta}{4\ii R(z)}.
\end{split}
\label{eq:FD-FOD}
\end{equation}
Unless $\mathrm{Im}(\beta)=\mathrm{Im}(\alpha)$, this product has a single real-valued root:
\begin{equation}
F^\mathrm{D}(z)F^\mathrm{OD}(z)=0\iff z=z_0:=\frac{\mathrm{Im}(\alpha^*\beta)}{\mathrm{Im}(\beta)-\mathrm{Im}(\alpha)}.
\label{eq:z0}
\end{equation}
We notice that $j(z)$ is a well-defined quantity of unit modulus for all $z\in\mathbb{R}$, so we can write it as $j(z)=\ee^{\ii\theta}$, in which case $F^\mathrm{D}(z)=\cos(\theta)$ and $F^\mathrm{OD}(z)=\sin(\theta)$.  By the argument principle, using the fact that $j(z)^4$ has one pole and one zero on either side of the real line, $j(z)^4$ has zero winding number as $z$ traverses the real line.  This implies that $\theta$ is a real analytic function of $z\in\mathbb{R}$ that tends to zero in both limits $z\to\pm\infty$.  Since the roots of $F^\mathrm{D}(z)$ correspond to values $\theta\in\pi(\mathbb{Z}+\tfrac{1}{2})$ (and hence $\theta\neq 0$), they must therefore be even in number (counted with multiplicity).  Because $z_0$ is the unique simple root of $F^\mathrm{D}(z)F^\mathrm{OD}(z)$, it follows that $F^\mathrm{OD}(z_0)=0$.  Note that $j_+(z)=-\ii j_-(z)$ implies that 
\begin{equation}
F^\mathrm{D}_+(z)=F^\mathrm{OD}_-(z)\quad\text{and}\quad
F^\mathrm{OD}_+(z)=-F^\mathrm{D}_-(z),\quad z\in\Gamma_{\alpha\to\beta}\cup\Gamma_{\beta^*\to\alpha^*}.
\label{eq:f-jumps}
\end{equation}
Also,
\begin{equation}
F^\mathrm{D}(z)\to 1\quad\text{and}\quad F^\mathrm{OD}(z)=\frac{\mathrm{Im}(\beta)-\mathrm{Im}(\alpha)}{2z}+O(z^{-2}),\quad z\to\infty.
\label{eq:f-asymp}
\end{equation}
Finally, observe that according to the jump conditions \eqref{eq:f-jumps}, the function $k$ defined on a two-sheeted cover $\mathcal{R}$ of the $z$-plane joined at the cuts $\Gamma_{\alpha\to\beta}\cup\Gamma_{\beta^*\to\alpha^*}$ by setting $k(z):=F^\mathrm{OD}(z)^2$ on one sheet and $k(z):=F^\mathrm{D}(z)^2$ on the other sheet is a meromorphic function on $\mathcal{R}$ with simple poles at the branch points $\alpha,\beta,\alpha^*,\beta^*$ and double zeros at $z=z_0$ and $z=\infty$ on the sheet where $k(z)=F^\mathrm{OD}(z)^2$ (and no other poles or zeros).

Let $L$ denote a clockwise-oriented loop surrounding the arc $\Gamma_{\alpha\to\beta}$, and consider the integrals
\begin{equation}
I_\mathcal{A}:=2\int_{\alpha^*}^\alpha\frac{\dd z}{R(z)},\quad I_\mathcal{B}:=\oint_L\frac{\dd z}{R(z)},
\label{eq:DS-I-AB}
\end{equation}
where the path of integration in $I_\mathcal{A}$ is a straight vertical line in the domain of analyticity of $R(z)$.  Since $R(z^*)=R(z)^*$, it follows that $I_\mathcal{A}$ is purely imaginary;  in fact, since $R(z)>0$ at the midpoint of the integration contour, it can be shown that $I_\mathcal{A}$ is strictly positive imaginary.  
Also by Cauchy's theorem, the contour of integration in $I_\mathcal{B}$ can be replaced by $\mathbb{R}$ with right-to-left orientation and therefore $I_\mathcal{B}<0$ because $R(z)>0$ for $z\in\mathbb{R}$.  Passing to the Riemann surface of $y^2=R(z)^2$, $I_\mathcal{A}$ and $I_\mathcal{B}$ are integrals of the same holomorphic differential over a canonical basis $(\mathcal{A},\mathcal{B})$ of homology.  We can simplify these integrals as follows.

Provided that $\mathrm{Re}(\alpha)<\mathrm{Re}(\beta)$, the points $\alpha,\beta,\alpha^*,\beta^*$ all lie on a circle with center 
$x\in\mathbb{R}$ and radius $\rho>0$ given by
\begin{equation}
x := \frac{|\beta|^2-|\alpha|^2}{2\mathrm{Re}(\beta)-2\mathrm{Re}(\alpha)}\quad\text{and}\quad
\rho :=\frac{|\alpha-\beta| |\alpha-\beta^*|}{2|\mathrm{Re}(\beta)-\mathrm{Re}(\alpha)|}.
\end{equation}
By the affine transformation $w=(z-x)/\rho$, we get
\begin{equation}
I_\mathcal{A}=\frac{2}{\rho}\int_{\ee^{-\ii\theta_\alpha}}^{\ee^{\ii\theta_\alpha}}\frac{\dd w}{S(w)},\quad 
I_\mathcal{B}=\frac{1}{\rho}\oint_{\tilde{L}}\frac{\dd w}{S(w)}
\end{equation}
where $S(w)$ is a corresponding square root of the monic quartic in $w$ with roots $w=\ee^{\pm\ii\theta_\alpha}$ and $w=\ee^{\pm\ii\theta_\beta}$, in which $\theta_\alpha=\arg(\alpha-x)\in (0,\pi)$ and $\theta_\beta=\arg(\beta-x)\in (0,\pi)$.  Note that $S(w)=w^2+O(w)$ as $w\to\infty$, and we may assume that $S(w)$ is cut on the upper unit semicircle between $w=\ee^{\ii\theta_\alpha}$ and $w=\ee^{\ii\theta_\beta}$ as well as its Schwarz reflection.  We assume that the points $\alpha,\beta$ are labeled such that $0<\theta_\beta<\theta_\alpha<\pi$.  A fractional linear transformation taking $w=\ee^{\pm\ii\theta_\alpha}$ to $W=\pm 1$ respectively is
\begin{equation}
W=\ii\tan(\tfrac{1}{2}\theta_\alpha)\frac{w+1}{w-1}.
\label{eq:DS-FLM}
\end{equation}
The same transformation maps $w=\ee^{\pm\ii\theta_\beta}$ to $W=\pm m^{-1/2}=\pm\tan(\tfrac{1}{2}\theta_\alpha)\cot(\tfrac{1}{2}\theta_\beta)$.  Since $\theta\mapsto\cot(\tfrac{1}{2}\theta)$ is positive and monotone decreasing on $0<\theta<\pi$, it follows that $0<m<1$.  Therefore,
\begin{equation}
I_\mathcal{A}=\frac{\ii}{\rho\sin(\tfrac{1}{2}\theta_\alpha)\cos(\tfrac{1}{2}\theta_\beta)}\int_{-1}^1\frac{\dd W}{\sqrt{1-W^2}\sqrt{1-mW^2}},\quad m=\cot^2(\tfrac{1}{2}\theta_\alpha)\tan^2(\tfrac{1}{2}\theta_\beta).
\end{equation}
Similarly, by collapsing the image of $\tilde{L}$ in the $W$-plane to the opposite sides of the real interval $[1,m^{-\frac{1}{2}}]$ we obtain
\begin{equation}
I_\mathcal{B}=-\frac{1}{\rho\sin(\tfrac{1}{2}\theta_\alpha)\cos(\tfrac{1}{2}\theta_\beta)}\int_1^{m^{-\frac{1}{2}}}\frac{\dd W}{\sqrt{W^2-1}\sqrt{1-mW^2}},\quad m=\cot^2(\tfrac{1}{2}\theta_\alpha)\tan^2(\tfrac{1}{2}\theta_\beta).
\end{equation}
By definition of the complete elliptic integral $\mathbb{K}(m)$ of the first kind
\begin{equation}
\mathbb{K}(m):=\int_0^1\frac{\dd x}{\sqrt{1-x^2}\sqrt{1-mx^2}},\quad 0<m<1,
\end{equation}
one easily sees that
\begin{equation}
I_\mathcal{A}=\frac{2\ii\mathbb{K}(m)}{\rho\sin(\tfrac{1}{2}\theta_\alpha)\cos(\tfrac{1}{2}\theta_\beta)},\quad m=\cot^2(\tfrac{1}{2}\theta_\alpha)\tan^2(\tfrac{1}{2}\theta_\beta).
\end{equation}
One can write $I_\mathcal{B}$ in a similar form by means of 
the substitution $mW^2=1-(1-m)Z^2$ mapping $W\in (1,m^{-\frac{1}{2}})$ onto $Z\in (0,1)$:
\begin{equation}
I_\mathcal{B}=-\frac{1}{\rho\sin(\frac{1}{2}\theta_\alpha)\cos(\frac{1}{2}\theta_\beta)}\int_0^1\frac{\dd Z}{\sqrt{1-Z^2}\sqrt{1-(1-m)Z^2}} = -\frac{\mathbb{K}(1-m)}{\rho\sin(\frac{1}{2}\theta_\alpha)\cos(\frac{1}{2}\theta_\beta)}.
\end{equation}
It follows that 
\begin{equation}
H:=2\pi\ii\frac{I_\mathcal{B}}{I_\mathcal{A}}=-\pi\frac{\mathbb{K}(1-m)}{\mathbb{K}(m)}<0,\quad m=\cot^2(\tfrac{1}{2}\theta_\alpha)\tan^2(\tfrac{1}{2}\theta_\beta).
\label{eq:DS-H-def}
\end{equation}
In the situation that $\mathrm{Re}(\beta)-\mathrm{Re}(\alpha)$ tends to zero while $\mathrm{Im}(\alpha)>\mathrm{Im}(\beta)>0$, the elliptic parameter $m$ has limiting value $[\mathrm{Im}(\beta)/\mathrm{Im}(\alpha)]^2\in (0,1)$, so $H$ makes sense as well.

The Abel map is defined as follows:
\begin{equation}
A(z):=\frac{2\pi\ii}{I_\mathcal{A}}\int_\alpha^z\frac{\dd z'}{R(z')},\quad z\in\mathbb{C}\setminus(\Gamma_{\alpha\to\beta}\cup\Gamma_{\beta\to\beta^*}\cup\Gamma_{\beta^*\to\alpha^*}),
\label{eq:Abel-define}
\end{equation}
and it is single-valued and analytic in its domain of definition.  Its boundary values are related by:
\begin{equation}
A_+(z)+A_-(z)=0,\quad z\in\Gamma_{\alpha\to\beta},
\label{eq:Abel-alpha-beta}
\end{equation}
\begin{equation}
A_+(z)-A_-(z)=H,\quad z\in\Gamma_{\beta\to\beta^*},
\label{eq:Abel-beta-betastar}
\end{equation}
\begin{equation}
A_+(z)+A_-(z)=-2\pi\ii,\quad z\in\Gamma_{\beta^*\to\alpha^*}.
\label{eq:Abel-betastar-alphastar}
\end{equation}
Note that the Abel map can be extended from its domain of definition to the whole Riemann surface $\mathcal{R}$ by using the definition \eqref{eq:Abel-define} on the sheet of $\mathcal{R}$ where $k(z)=F^\mathrm{OD}(z)^2$ and changing the sign on the other sheet; then allowing the path of integration to be arbitrary on $\mathcal{R}$, the lift of $A(z)$ to $\mathcal{R}$ becomes well-defined modulo integer multiples of $2\pi\ii$ and $H$.  According to the Abel-Jacobi theorem, this extended Abel map acting linearly on divisors of meromorphic functions on $\mathcal{R}$ yields zero (modulo the period lattice).  Applying this result to the meromorphic function $k$ whose divisor is $(k)=2z_0^\mathrm{OD}+2\infty^\mathrm{OD}-\alpha-\beta-\alpha^*-\beta^*$ (the superscript denotes the sheet where $k(z)=F^\mathrm{OD}(z)^2$) and noticing that the sum of $A$ applied to the four branch points yields a period, we learn that the function defined precisely by \eqref{eq:Abel-define} satisfies
\begin{equation}
2A(z_0)+2A(\infty)=2\pi \ii n_1 + Hn_2
\label{eq:Az0-identity}
\end{equation}
for some particular $n_1,n_2\in\mathbb{Z} \pmod 2$ (each of $A(z_0)$ and $A(\infty)$ is well-defined modulo integer multiples of $2\pi\ii$ and $H$).  In fact\footnote{Since $n_1$ and $n_2$ are integers and the left-hand side depends continuously on $\alpha$ and $\beta$, we can calculate them from a limiting configuration in which $\alpha=\ee^{3\pi\ii/4}$ and $\beta=\ee^{\ii\pi/4}$.  In this limit, $z_0\to\infty$ so we should evaluate $4A(\infty)$ in the limiting configuration.  To do this, we take a path of integration from $z=\alpha=\ee^{3\pi\ii/4}$ along the unit circle in the counterclockwise direction to $z=-1$ and then a real path from $z=-1$ to $z=-\infty$.  The leg of the path on the unit circle gives a purely imaginary contribution to $4A(\infty)$ equal to $-2\pi\ii$ while the symmetries $R(z)=R(-z)=z^2R(z^{-1})$ valid for $z\in\mathbb{R}$ show that the real leg of the path gives a real contribution to $4A(\infty)$ equal to $H$.}, we may take $n_1=-1$ and $n_2=1$.  

The Riemann theta function is defined for $w\in\mathbb{C}$ and $\mathrm{Re}(H)<0$ by
\begin{equation}
\Theta(w;H):=\sum_{n\in\mathbb{Z}}\ee^{\frac{1}{2}n^2H}\ee^{n w},
\label{eq:DS-Theta-define}
\end{equation}
and it is an entire function of $w$ satisfying the identities
\begin{equation}
\Theta(-w;H)=\Theta(w;H),\quad\Theta(w+2\pi\ii;H)=\Theta(w;H),\quad\text{and}\quad\Theta(w\pm H;H)=\ee^{-\frac{1}{2}H}\ee^{\mp w}\Theta(w;H).
\label{eq:automorphic}
\end{equation}
It has only simple zeros, and they are located at the lattice points $w=(j+\tfrac{1}{2})2\pi\ii + (k+\tfrac{1}{2})H$ for $(j,k)\in\mathbb{Z}^2$.  We denote the zero for $j=k=0$ by $\mathcal{K}$:
\begin{equation}
\mathcal{K}:=\ii\pi +\frac{1}{2}H.
\label{eq:RiemannConstant}
\end{equation}
Taking an arbitrary complex shift $s\in\mathbb{C}$ and a point $z_0\in\mathbb{C}\setminus(\Gamma_{\alpha\to\beta}\cup\Gamma_{\beta\to\beta^*}\cup\Gamma_{\beta^*\to\alpha^*})$, we define two functions of $z$ by
\begin{equation}
q^\pm(z;z_0,s):=\frac{\Theta(A(z)\pm A(z_0)\pm\mathcal{K}-s;H)}{\Theta(A(z)\pm A(z_0)\pm\mathcal{K};H)},\quad
z\in\mathbb{C}\setminus(\Gamma_{\alpha\to\beta}\cup\Gamma_{\beta\to\beta^*}\cup\Gamma_{\beta^*\to\alpha^*}).
\label{eq:qpm-def}
\end{equation}
One can check that $q^+$ is analytic in its domain of definition.  On the other hand $q^-$ has a simple pole at $z=z_0$, and this is its only singularity (unless $s$ is an integer linear combination of $2\pi\ii$ and $H$ in which case the singularity is cancelled and $q^-$ becomes analytic).  Taking boundary values, it follows from $\Theta(-w;H)=\Theta(w;H)$, $\Theta(w+2\pi\ii;H)=\Theta(w;H)$,  \eqref{eq:Abel-alpha-beta}, and \eqref{eq:Abel-betastar-alphastar} that 
\begin{equation}
q^\pm_+(z;z_0,s)=q^\mp_-(z;z_0,-s),\quad z\in\Gamma_{\alpha\to\beta}\cup\Gamma_{\beta^*\to\alpha^*}.
\end{equation}
Similarly, it follows from $\Theta(w+H;H)=\ee^{-\frac{1}{2}H}\ee^{-w}\Theta(w;H)$ and \eqref{eq:Abel-beta-betastar} that
\begin{equation}
q^\pm_+(z;z_0,s)=\ee^{s}q^\pm_-(z;z_0,s),\quad z\in\Gamma_{\beta\to\beta^*}.
\end{equation}
Therefore, the matrix $\mathbf{Q}(z)$ defined by
\begin{equation}
\mathbf{Q}(z):=\begin{bmatrix}q^+(z;z_0,s) & -\ii q^-(z;z_0,-s)\\\ii q^-(z;z_0,s) & q^+(z;z_0,-s)\end{bmatrix}
\end{equation}
satisfies the following jump conditions
\begin{equation}
\mathbf{Q}_+(z)=\mathbf{Q}_-(z)\begin{bmatrix}0&-\ii \\\ii & 0\end{bmatrix},\quad z\in\Gamma_{\alpha\to\beta}\cup\Gamma_{\beta^*\to\alpha^*},
\end{equation}
and
\begin{equation}
\mathbf{Q}_+(z)=\mathbf{Q}_-(z)\ee^{s\sigma_3},\quad z\in\Gamma_{\beta\to\beta^*}.
\end{equation}
Identifying $z_0$ with the value given in \eqref{eq:z0}, which is the simple root of $F^{\mathrm{OD}}(z)$, we can modify $\mathbf{Q}(z)$ as follows:
\begin{equation}
\widetilde{\mathbf{Q}}(z):=\begin{bmatrix}F^\mathrm{D}(z)q^+(z;z_0,s) & -\ii F^\mathrm{OD}(z)q^-(z;z_0,-s)\\
-\ii F^\mathrm{OD}(z)q^-(z;z_0,s) & F^\mathrm{D}(z)q^+(z;z_0,-s)\end{bmatrix}.
\end{equation}
With this modification, the pole at $z=z_0$ in $q^-(z;z_0,\pm s)$ is removed, and so $\widetilde{\mathbf{Q}}(z)$ is analytic in the complement of the jump contour.  
Using the jump conditions \eqref{eq:f-jumps} one has
\begin{equation}
\widetilde{\mathbf{Q}}_+(z)=\widetilde{\mathbf{Q}}_-(z)\ii\sigma_1,\quad z\in\Gamma_{\alpha\to\beta}\cup\Gamma_{\beta^*\to\alpha^*}.
\end{equation}
Using the fact that $F^\mathrm{D}(z)$ and $F^\mathrm{OD}(z)$ are analytic on $\Gamma_{\beta\to\beta^*}$, one has
\begin{equation}
\widetilde{\mathbf{Q}}_+(z)=\widetilde{\mathbf{Q}}_-(z)\ee^{s\sigma_3},\quad z\in\Gamma_{\beta\to\beta^*}.
\end{equation}
To match the desired jump conditions of $\mathbf{F}(z)$ it therefore only remains to choose $s=\ii M\Delta$.  Finally, using \eqref{eq:f-asymp}, we can normalize at infinity by multiplication on the left by a suitable constant diagonal matrix to obtain
\begin{equation}
\mathbf{F}(z)=\begin{bmatrix}\displaystyle F^\mathrm{D}(z)\frac{q^+(z;z_0,\ii M\Delta)}{q^+(\infty;z_0,\ii M\Delta)} &\displaystyle  -\ii F^\mathrm{OD}(z)\frac{q^-(z;z_0,-\ii M\Delta)}{q^+(\infty;z_0,\ii M\Delta)}\\
\displaystyle -\ii F^\mathrm{OD}(z)\frac{q^-(z;z_0,\ii M\Delta)}{q^+(\infty;z_0,-\ii M\Delta)} & \displaystyle F^\mathrm{D}(z)\frac{q^+(z;z_0,-\ii M\Delta)}{q^+(\infty;z_0,-\ii M\Delta)}\end{bmatrix}.
\end{equation}
Going back to $\breve{\mathbf{O}}^\mathrm{out}(z)$ by a constant diagonal conjugation yields
\begin{equation}
\breve{\mathbf{O}}^\mathrm{out}(z):=\begin{bmatrix}
\displaystyle F^\mathrm{D}(z)\frac{q^+(z;z_0,\ii M\Delta)}{q^+(\infty;z_0,\ii M\Delta)}  & \displaystyle \ee^{-\ii M\phi}F^\mathrm{OD}(z)\frac{q^-(z;z_0,-\ii M\Delta)}{q^+(\infty;z_0,\ii M\Delta)}\\
\displaystyle -\ee^{\ii M\phi}F^\mathrm{OD}(z)\frac{q^-(z;z_0,\ii M\Delta)}{q^+(\infty;z_0,-\ii M\Delta)} & 
\displaystyle F^\mathrm{D}(z)\frac{q^+(z;z_0,-\ii M\Delta)}{q^+(\infty;z_0,-\ii M\Delta)}\end{bmatrix}.
\label{eq:DS-outer-parametrix-formula}
\end{equation}
One can verify that this solution of \rhref{rhp:O-out} is unique and therefore, like $\mathbf{O}(z;\chi,\tau,M)$, $\mathbf{T}(z;\chi,\tau,M)$, and $\mathbf{S}(z;\chi,\tau,M)$, it has Schwarz symmetry of the form \eqref{eq:S-Schwarz}.

\subsection{Inner parametrices}
We define two conformal mappings, one on a neighborhood of each of the points $z=\alpha,\beta$ as follows.
First, note that $h(z)$ is well defined at $z=\alpha$ and that $h(\alpha)=\ii+\frac{1}{2}\phi$.  At $z=\beta$, the sum and difference of boundary values of $h(z)$ are well defined, and $h_+(\beta)+h_-(\beta)=2\ii +\phi$ while $h_+(\beta)-h_-(\beta)=\Delta$, implying that $h_+(\beta)=\ii+\frac{1}{2}\phi+\frac{1}{2}\Delta$.  Let $D_p$, $p=\alpha,\beta$ denote disks of radius $\delta>0$ fixed but sufficiently small centered at $p$.  
\begin{itemize}
\item On $D_\alpha$ we define a conformal coordinate $\varphi_\alpha(z):=(2\ii (h(z)- h(\alpha)))^{\frac{2}{3}}$ by analytic continuation of the positive $\frac{2}{3}$ power from the arc emanating from $z=\alpha$ along which $h(z)-h(\alpha)$ is negative imaginary, and within $D_\alpha$ we choose $\Gamma_{\alpha^*\to\alpha}$ to agree with that arc.  We also choose the image under $\varphi=\varphi_\alpha$ of $\Gamma_{\alpha\to\beta}\cap D_\alpha$ to lie on the negative real axis, and that of $\Lambda_{\alpha\to\beta}^\pm\cap D_\alpha$ to lie on the ray $\arg(\varphi)=\mp\frac{2}{3}\pi$.
\item On $D_\beta$ we define a conformal coordinate $\varphi_\beta(z):=(2\ii (h_+(z)-h_+(\beta)))^{\frac{2}{3}}$ by analytic continuation of the positive $\frac{2}{3}$ power from the arc emanating from $z=\beta$ along which $\tilde{h}(z)-h_+(\beta)$ is negative imaginary, where $\tilde{h}(z)$ denotes the analytic continuation of $h_+(z)$ from $\Gamma_{\beta\to\beta^*}$ to $D_\beta\setminus\Gamma_{\alpha\to\beta}$, and within $D_\beta$ we choose $\Gamma_{\beta\to\beta^*}$ to agree with that arc.  We also choose the image under $\varphi=\varphi_\beta$ of $\Gamma_{\alpha\to\beta}\cap D_\beta$ to lie on the negative real axis, and that of $\Lambda_{\alpha\to\beta}^\pm\cap D_\beta$ to lie on the ray $\arg(\varphi)=\pm\frac{2}{3}\pi$.
\end{itemize}
The two maps $z\mapsto \varphi_p(z)$ for $p=\alpha,\beta$ are conformal near $z=p$ because these points are simple roots of $h'(z)^2$. 
If we define 
\begin{equation}
\mathbf{P}(z):=\mathbf{O}(z;\chi,\tau,M)\ee^{-\frac{1}{2}\ii M\phi\sigma_3},\quad z\in D_\alpha
\label{eq:P-in-Dalpha}
\end{equation}
and
\begin{equation}
\mathbf{P}(z):=\mathbf{O}(z;\chi,\tau,M)\begin{cases}\ee^{-\frac{1}{2}\ii M(\phi+\Delta)\sigma_3}\ii^{\sigma_3},& z\in D_\beta,\quad \mathrm{Im}(\varphi_\beta(z))>0,\\
\ee^{-\frac{1}{2}\ii M(\phi-\Delta)\sigma_3}\ii^{\sigma_3},&z\in D_\beta,\quad\mathrm{Im}(\varphi_\beta(z))<0,
\end{cases}
\label{eq:P-in-Dbeta}
\end{equation}
then the jump conditions satisfied by $\mathbf{P}(z)$ in both disks can be approximated universally by the same formul\ae; namely we have
\begin{equation}
\mathbf{P}_+(z)=\mathbf{P}_-(z)\left(\mathbb{I}+O(\ee^{-4M})\right)\begin{bmatrix}1 & \ee^{-\zeta^{3/2}}\\0 & 1\end{bmatrix},\quad \arg(\zeta)=0,
\label{eq:P-jump-pos}
\end{equation}
\begin{equation}
\mathbf{P}_+(z)=\mathbf{P}_-(z)\left(\mathbb{I}+O(\ee^{-4M})\right)\begin{bmatrix}1 & 0\\\ee^{\zeta^{3/2}} & 1\end{bmatrix},\quad\arg(\zeta)=\pm\frac{2\pi}{3},
\label{eq:P-jump-lenses}
\end{equation}
and
\begin{equation}
\mathbf{P}_+(z)=\mathbf{P}_-(z)\begin{bmatrix} 0 & 1\\-1 & 0\end{bmatrix},\quad \arg(-\zeta)=0.
\label{eq:P-jump-twist}
\end{equation}
Here $\zeta=M^{\frac{2}{3}}\varphi_p(z)$ for $z\in D_p$, $p=\alpha,\beta$, and to define the boundary values all rays are taken with orientation in the direction of increasing real part of $\zeta$ (this matches the original orientation within $D_\beta$ but reverses the orientation within $D_\alpha$).  The error terms are uniform for $z\in D_p$ and they vanish along the indicated ray in the limit $\zeta\to 0$.  Defining a matrix function $\breve{\mathbf{P}}^\mathrm{out}(z)$ for $z\in D_\alpha$ (resp., for $z\in D_\beta$) by an analogue of the formula \eqref{eq:P-in-Dalpha} (resp., of the formula \eqref{eq:P-in-Dbeta}) in which $\mathbf{O}(z;\chi,\tau,M)$ is replaced with $\breve{\mathbf{O}}^\mathrm{out}(z)$, we see that $\breve{\mathbf{P}}^\mathrm{out}(z)$ is analytic within $D_p$ except along the arc where $\varphi_p(z)\le 0$, and where $\breve{\mathbf{P}}^\mathrm{out}(z)$ satisfies exactly the jump condition \eqref{eq:P-jump-twist}, and that $\breve{\mathbf{P}}^\mathrm{out}(z)$ blows up like a negative one-fourth power at $z=p$.  Therefore, the matrix function
\begin{equation}
\mathbf{\holomat}_p(z):=\breve{\mathbf{P}}^\mathrm{out}(z)\mathbf{V}^{-1}\varphi_p(z)^{-\frac{1}{4}\sigma_3},\quad z\in D_p,\quad\mathbf{V}:=\frac{1}{\sqrt{2}}\begin{bmatrix}1 & -\ii\\-\ii & 1\end{bmatrix}
\label{eq:H-define}
\end{equation}
has a removable singularity along the arc $\varphi_p(z)\le 0$ and hence is analytic in $D_p$.  Letting $\mathbf{A}(\zeta)$ denote the standard Airy parametrix analytic for $\mathrm{Im}(\zeta)\neq 0$ except across the rays $\arg(\zeta)=\pm\frac{2}{3}\pi$, satisfying jump conditions given in \eqref{eq:P-jump-pos}--\eqref{eq:P-jump-twist} with the error terms neglected, and satisfying the normalization condition 
\begin{equation}
\mathbf{A}(\zeta)\mathbf{V}^{-1}\zeta^{-\frac{1}{4}\sigma_3}=\mathbb{I} + \begin{bmatrix}O(\zeta^{-3}) & O(\zeta^{-1})\\O(\zeta^{-2}) & O(\zeta^{-3})\end{bmatrix},\quad \zeta\to\infty,
\label{eq:A-norm}
\end{equation}
(i.e., $\mathbf{A}(\zeta)$ is the unique solution of Riemann-Hilbert Problem 4 of \cite{BothnerM19}, for instance --- see \cite[Appendix B]{BothnerM19} for full details), we then define a parametrix for $\mathbf{O}(z;\chi,\tau,M)$ within $D_p$, $p=\alpha,\beta$ by setting 
\begin{equation}
\breve{\mathbf{O}}^\alpha(z):=\mathbf{\holomat}_\alpha(z)M^{-\frac{1}{6}\sigma_3}\mathbf{A}(M^{\frac{2}{3}}\varphi_\alpha(z))\ee^{\frac{1}{2}\ii M\phi\sigma_3},\quad z\in D_\alpha,
\label{eq:O-parametrix-Dalpha}
\end{equation}
and
\begin{equation}
\breve{\mathbf{O}}^\beta(z):=\mathbf{\holomat}_\beta(z)M^{-\frac{1}{6}\sigma_3}\mathbf{A}(M^\frac{2}{3}\varphi_\beta(z))\begin{cases}\ee^{\frac{1}{2}\ii M(\phi+\Delta)\sigma_3}\ii^{-\sigma_3},& z\in D_\beta,\quad \mathrm{Im}(\varphi_\beta(z))>0,\\
\ee^{\frac{1}{2}\ii M(\phi-\Delta)\sigma_3}\ii^{-\sigma_3},&z\in D_\beta,\quad \mathrm{Im}(\varphi_\beta(z))<0.
\end{cases}
\label{eq:O-parametrix-Dbeta}
\end{equation}

\subsection{Global parametrix and error estimation}
Let $D_p^*$ denote the Schwarz reflection in the real axis of the disk $D_p$, $p=\alpha,\beta$.  We define a global parametrix for $\mathbf{O}(z;\chi,\tau,M)$ by setting
\begin{equation}
\breve{\mathbf{O}}(z):=\begin{cases}\breve{\mathbf{O}}^\alpha(z),&z\in D_\alpha,\\
\breve{\mathbf{O}}^\beta(z),&z\in D_\beta,\\
\sigma_2\breve{\mathbf{O}}^\alpha(z^*)^*\sigma_2,&z\in D_\alpha^*,\\
\sigma_2\breve{\mathbf{O}}^\beta(z^*)^*\sigma_2,&z\in D_\beta^*,\\
\breve{\mathbf{O}}^\mathrm{out}(z),&z\in\mathbb{C}\setminus (\overline{D_\alpha}\cup\overline{D_\beta}\cup\overline{D_\alpha^*}\cup\overline{D_\beta^*}),
\end{cases}
\end{equation}
which globally satisfies Schwarz symmetry in the form \eqref{eq:S-Schwarz}.
The corresponding error is defined as $\mathbf{E}(z):=\mathbf{O}(z;\chi,\tau,M)\breve{\mathbf{O}}(z)^{-1}$ wherever both factors make sense, and it also is Schwarz symmetric.  Since $\mathbf{O}(z;\chi,\tau,M)$ and $\breve{\mathbf{O}}(z)$ satisfy the same jump conditions across the arcs $\Gamma_{\alpha\to\beta}$ and $\Gamma_{\beta^*\to\alpha^*}$ both within and exterior to the disks $D_p$ and $D_p^*$ for $p=\alpha,\beta$, $\mathbf{E}(z)$ can be defined on these arcs so as to be analytic there.  On the parts of the arcs $\Gamma_{\alpha^*\to\alpha}$, $\Gamma_{\beta\to\beta^*}$, $\Lambda^\pm_{\alpha\to\beta}$, and $\Lambda^{\pm*}_{\alpha\to\beta}$  lying outside of all four disks, $\breve{\mathbf{O}}(z)=\breve{\mathbf{O}}^\mathrm{out}(z)$ is analytic and bounded, and $\mathrm{Im}(h(z))\in (-1+\epsilon,1-\epsilon)$ holds on these arcs for some $\epsilon>0$ independent of $M$.  It follows that on these arcs the boundary values of $\mathbf{E}(z)$ are related by $\mathbf{E}_+(z)=\mathbf{E}_-(z)(\mathbb{I}+O(\ee^{-M\epsilon}))$ as $M\to+\infty$.  On the arcs of $\Gamma_{\alpha^*\to\alpha}$, $\Gamma_{\beta\to\beta^*}$, $\Lambda^\pm_{\alpha\to\beta}$, and $\Lambda^{\pm*}_{\alpha\to\beta}$ within the four disks, both $\mathbf{O}(z;\chi,\tau,M)$ and $\breve{\mathbf{O}}(z)$ have jump discontinuities.  Using \eqref{eq:P-in-Dalpha}--\eqref{eq:P-in-Dbeta}, \eqref{eq:P-jump-pos}--\eqref{eq:P-jump-lenses},  and \eqref{eq:O-parametrix-Dalpha}--\eqref{eq:O-parametrix-Dbeta} one can check that on these arcs within $D_p$, $p=\alpha,\beta$, we have
\begin{equation}
\mathbf{E}_+(z)=\mathbf{E}_-(z)\mathbf{\holomat}_p(z)M^{-\frac{1}{6}\sigma_3}\mathbf{A}(M^\frac{2}{3}\varphi_p(z))\left(\mathbb{I}+O(\ee^{-4M})\right)\mathbf{A}(M^\frac{2}{3}\varphi_p(z))^{-1}M^{\frac{1}{6}\sigma_3}\mathbf{\holomat}_p(z)^{-1}.
\end{equation}
But since the holomorphic factors $\mathbf{\holomat}_p(z)$ have unit determinant and are bounded independent of $M$ in $D_p$, and since $\mathbf{A}(\zeta)$ has unit determinant and satisfies $\mathbf{A}(\zeta)=O(|\zeta|^{\frac{1}{4}})$ as $\zeta\to\infty$ according to \eqref{eq:A-norm}, it follows that $\mathbf{E}_+(z)=\mathbf{E}_-(z)(\mathbb{I}+O(M^\frac{2}{3}\ee^{-4M}))$ holds uniformly on these arcs in $D_\alpha$ and $D_\beta$.  By Schwarz reflection symmetry of $\mathbf{E}(z)$, the same holds in $D_\alpha^*$ and $D_\beta^*$.  Finally, on the boundaries of all four disks, taken with clockwise orientation, it follows from the definitions \eqref{eq:H-define}, \eqref{eq:O-parametrix-Dalpha}, and \eqref{eq:O-parametrix-Dbeta}, and from the large-$\zeta$ asymptotic property of $\mathbf{A}(\zeta)$ given in \eqref{eq:A-norm} that $\mathbf{E}_+(z)=\mathbf{E}_-(z)(\mathbb{I}+O(M^{-1}))$ holds on these circles.

From these arguments, it follows that if the jump matrix for $\mathbf{E}(z)$ on its jump contour $\Sigma_\mathbf{E}$ is denoted by $\mathbf{V}^\mathbf{E}(z)$, so that $\mathbf{E}_+(z)=\mathbf{E}_-(z)\mathbf{V}^\mathbf{E}(z)$ for $z$ any non-self-intersection point of $\Sigma_\mathbf{E}$ (so that the boundary values are well-defined), then
\begin{equation}
\sup_{z\in\Sigma_\mathbf{E}}\|\mathbf{V}^\mathbf{E}(z)-\mathbb{I}\|=O(M^{-1}),\quad M\to+\infty.
\label{eq:VE-sup}
\end{equation}
Since $\mathbf{E}(z)$ is analytic for $z\in\mathbb{C}\setminus\Sigma_\mathbf{E}$ and $\mathbf{E}(z)\to\mathbb{I}$ as $z\to\infty$, appealing to standard small-norm theory for Riemann-Hilbert problems then shows that $\mathbf{E}_-(\diamond)-\mathbb{I}=O(M^{-1})$ holds in the $L^2(\Sigma_\mathbf{E})$ sense, and therefore also that
\begin{equation}
\begin{split}
\lim_{z\to\infty}zE_{12}(z)&=-\frac{1}{2\pi\ii}\int_{\Sigma_\mathbf{E}}E_{-,11}(\zeta)V^\mathbf{E}_{12}(\zeta)\,\dd\zeta -\frac{1}{2\pi\ii}\int_{\Sigma_\mathbf{E}}E_{-,12}(\zeta)(V^\mathbf{E}_{22}(\zeta)-1)\,\dd\zeta \\
&= -\frac{1}{2\pi\ii}\int_{\Sigma_\mathbf{E}}V^\mathbf{E}_{12}(\zeta)\,\dd\zeta-\frac{1}{2\pi\ii}\int_{\Sigma_\mathbf{E}}(E_{-,11}(\zeta)-1)V^\mathbf{E}_{12}(\zeta)\,\dd\zeta\\
&\qquad\qquad\qquad{} -\frac{1}{2\pi\ii}\int_{\Sigma_\mathbf{E}}E_{-,12}(\zeta)(V^\mathbf{E}_{22}(\zeta)-1)\,\dd\zeta\\
&= O(M^{-1}),\quad M\to+\infty,
\end{split}
\label{eq:E-moment-estimate}
\end{equation}
where we used the Cauchy-Schwarz inequality and the fact that \eqref{eq:VE-sup} implies that also $\mathbf{V}_\mathbf{E}(\diamond)-\mathbb{I}=O(M^{-1})$ in $L^2(\Sigma_\mathbf{E})$ because $\Sigma_\mathbf{E}$ is compact.  

Recalling \eqref{eq:DS-Psi-from-O} and using the fact that $O_{12}(z;\chi,\tau,M)=E_{11}(z)\breve{O}_{12}(z)+E_{12}(z)\breve{O}_{22}(z)$ and that $\breve{\mathbf{O}}(z)=\breve{\mathbf{O}}^\mathrm{out}(z)$ holds for $|z|$ sufficiently large, we appeal to \eqref{eq:E-moment-estimate} and the limits $E_{11}(z)\to 1$ and $\breve{O}^\mathrm{out}_{22}(z)\to 1$ as $z\to\infty$ to obtain
\begin{equation}
\begin{split}
M\Psi(M^2\chi,M^3\tau;\mathbf{G}(\ee^{-2M},\sqrt{1-\ee^{-4M}}))&=2\ii\lim_{z\to\infty} z\left[E_{11}(z)\breve{O}^\mathrm{out}_{12}(z)+E_{12}(z)\breve{O}^\mathrm{out}_{22}(z)\right] \\
&= \breve{\Psi}(\chi,\tau;M) + O(M^{-1}),\quad M\to +\infty,
\end{split}
\end{equation}
where $\breve{\Psi}(\chi,\tau;M)$ is defined in terms of the outer parametrix by
\begin{equation}
\breve{\Psi}(\chi,\tau;M):=2\ii\lim_{z\to\infty}z\breve{O}^\mathrm{out}_{12}(z).
\label{eq:Psi-breve-def}
\end{equation}

\subsection{Properties of $\breve{\Psi}(\chi,\tau;M)$}
\label{s:DS-Interpretation}

In this section we obtain differential equations 
for $\breve{\Psi}(\chi,\tau;M)$ (Section~\ref{s:Lax}), estimate some of the coefficients in these equations (Section~\ref{s:estimate-coeffs}), use these results to compute the $L^2$-norm of $\breve{\Psi}(\chi,\tau;M)$ (Section~\ref{sec:L2-norm}), and derive the explicit formula \eqref{eq:Intro-square-modulus} for $|\breve{\Psi}(\chi,\tau;M)|^2$ (Section~\ref{s:explicit-formulae}).

\subsubsection{Lax equations}
\label{s:Lax}
Consider the matrix 
\begin{equation}
\mathbf{L}(z;\chi,\tau,M):=\breve{\mathbf{O}}^\mathrm{out}(z)\ee^{-\ii M(h(z)+2z^{-1})\sigma_3}.
\label{eq:DS-U-Oout}
\end{equation}
Note that $h(z)+2z^{-1}$ is analytic for $z\in\mathbb{C}\setminus(\Gamma_{\alpha\to\beta}\cup\Gamma_{\beta\to\beta^*}\cup\Gamma_{\beta^*\to\alpha^*})$ and satisfies $h(z)+2z^{-1}=\chi z + \tau z^2+O(z^{-1})$ as $z\to\infty$.  Since $h_+(z)+h_-(z)=2\ii +\phi$ for $z\in\Gamma_{\alpha\to\beta}$, $h_+(z)+h_-(z)=-2\ii +\phi$ for $z\in \Gamma_{\beta^*\to\alpha^*}$, and $h_+(z)-h_-(z)=\Delta$ for $z\in\Gamma_{\beta\to\beta^*}$, it follows that $\mathbf{L}(z;\chi,\tau,M)$ is analytic for $z\in\mathbb{C}\setminus(\Gamma_{\alpha\to\beta}\cup\Gamma_{\beta^*\to\alpha^*})$ (one checks that $\mathbf{L}_+(z;\chi,\tau,M)=\mathbf{L}_-(z;\chi,\tau,M)$ for $z\in\Gamma_{\beta\to\beta^*}$ and applies Morera's theorem to deduce analyticity on the interior of $\Gamma_{\beta\to\beta^*}$), and that the jump conditions on $\Gamma_{\alpha\to\beta}$ and $\Gamma_{\beta^*\to\alpha^*}$ are independent of $(\chi,\tau)\in\mathbb{R}^2$:
\begin{equation}
\mathbf{L}_+(z;\chi,\tau,M)=\mathbf{L}_-(z;\chi,\tau,M)\begin{bmatrix}0 & -\ee^{-2M}\ee^{4\ii Mz^{-1}}\\\ee^{2M}\ee^{-4\ii Mz^{-1}} & 0\end{bmatrix},\quad z\in\Gamma_{\alpha\to\beta},
\end{equation}
and
\begin{equation}
\mathbf{L}_+(z;\chi,\tau,M)=\mathbf{L}_-(z;\chi,\tau,M)\begin{bmatrix}0 & -\ee^{2M}\ee^{4\ii Mz^{-1}}\\\ee^{-2M}\ee^{-4\ii Mz^{-1}} & 0\end{bmatrix},\quad z\in\Gamma_{\beta^*\to\alpha^*}.
\end{equation}
Also, since $h(z)$ is bounded at the four points $\alpha,\beta,\alpha^*,\beta^*$, $\mathbf{L}(z;\chi,\tau,M)$ inherits from $\breve{\mathbf{O}}^\mathrm{out}(z)$ the property that it has at worst negative one-fourth root singularities near each of these points.  Finally, since $g(z)=O(z^{-1})$ as $z\to\infty$ and $h(z)=\phase(z;\chi,\tau)+g(z)$, from \eqref{eq:DS-tildevartheta} it follows that 
\begin{equation}
\mathbf{L}(z;\chi,\tau,M)=\left(\mathbb{I}+\sum_{n=1}^\infty \mathbf{L}^{[n]}(\chi,\tau;M)z^{-n}\right)\ee^{-\ii M(\chi z +\tau z^2)\sigma_3}
\label{eq:DS-Lax-expansion}
\end{equation}
holds for $|z|$ sufficiently large, where the series is convergent as well as asymptotic as $z\to\infty$.  One can also check that $\mathbf{L}(z^*;\chi,\tau,M)=\sigma_2\mathbf{L}(z;\chi,\tau,M)^*\sigma_2$.
Now, $\breve{\Psi}(\chi,\tau;M)$ defined by \eqref{eq:Psi-breve-def} can also be represented in the form
\begin{equation}
\breve{\Psi}(\chi,\tau;M)=2\ii L^{[1]}_{12}(\chi,\tau;M)=-2\ii L^{[1]}_{21}(\chi,\tau;M)^*,
\label{eq:DS-Psi-breve-U}
\end{equation}
where the second equality comes via the Schwarz symmetry of $\mathbf{L}(z;\chi,\tau,M)$.
These properties of $\mathbf{L}(z;\chi,\tau,M)$ show that 
\begin{equation}
\mathbf{X}:=\frac{\partial\mathbf{L}}{\partial \chi}\mathbf{L}^{-1}\quad\text{and}\quad
\mathbf{T}:=\frac{\partial\mathbf{L}}{\partial \tau}\mathbf{L}^{-1}
\label{eq:DS-XT-define}
\end{equation}
are analytic functions of $z$ for $z\in\mathbb{C}\setminus\{\alpha,\beta,\alpha^*,\beta^*\}$ with asymptotic expansions
\begin{equation}
\begin{split}
\mathbf{X}&=-\ii Mz\sigma_3 -\ii M[\mathbf{L}^{[1]},\sigma_3] + O(z^{-1})\\
\mathbf{T}&=-\ii M z^2\sigma_3 -\ii Mz[\mathbf{L}^{[1]},\sigma_3] -\ii M ([\mathbf{L}^{[2]},\sigma_3] - [\mathbf{L}^{[1]},\sigma_3\mathbf{L}^{[1]}]) + O(z^{-1})
\end{split}
\label{eq:DS-XT-asymp}
\end{equation}
as $z\to\infty$.  
If $\alpha,\beta,\alpha^*,\beta^*$ were independent of $(\chi,\tau)$, then $\mathbf{X}$ and $\mathbf{T}$ would be entire and hence by Liouville's Theorem they would be polynomials in $z$.
However the dependence of these quantities on $(\chi,\tau)$ via the Whitham equations \eqref{eq:DS-Whitham} implies that $\mathbf{X}$ and $\mathbf{T}$ have simple poles at all four branch points, so for certain residue matrices $\mathbf{X}^{(p)}$ and $\mathbf{T}^{(p)}$ depending on $(\chi,\tau;M)$ we can write
\begin{equation}
\begin{split}
\mathbf{X}&=-\ii Mz\sigma_3 -\ii M[\mathbf{L}^{[1]},\sigma_3] + \sum_{p=\alpha,\beta,\alpha^*,\beta^*}\frac{\mathbf{X}^{(p)}}{z-p}\\
\mathbf{T}&=-\ii M z^2\sigma_3 -\ii Mz[\mathbf{L}^{[1]},\sigma_3] -\ii M ([\mathbf{L}^{[2]},\sigma_3] - [\mathbf{L}^{[1]},\sigma_3\mathbf{L}^{[1]}]) + \sum_{p=\alpha,\beta,\alpha^*,\beta^*}\frac{\mathbf{T}^{(p)}}{z-p}.
\end{split}
\label{eq:DS-X-T-poles}
\end{equation}
To determine the residues, note that for $z$ in a neighborhood of $\Gamma_{\alpha\to\beta}$ we may express $\mathbf{L}(z;\chi,\tau,M)$ in the form
\begin{equation}
\mathbf{L}(z;\chi,\tau,M)=\mathbf{\holomat}(z;\chi,\tau,M)\left(\frac{z-\alpha}{z-\beta}\right)^{-\frac{1}{4}\sigma_3}\frac{1}{\sqrt{2}}\begin{bmatrix}1 & \ii \ee^{-2M}\ee^{4\ii Mz^{-1}}\\ \ii \ee^{2M}\ee^{-4\ii Mz^{-1}} & 1\end{bmatrix},
\label{eq:DS-L-H-diag-ev}
\end{equation}
where $\mathbf{\holomat}(z;\chi,\tau,M)$ is analytic for $z$ in a neighborhood of $\Gamma_{\alpha\to\beta}$ and where the power function is cut on $\Gamma_{\alpha\to\beta}$ and tends to $1$ as $z\to\infty$.  Now, $\alpha(\chi,\tau)$, $\beta(\chi,\tau)$, and $\mathbf{\holomat}(z;\chi,\tau,M)$ are differentiable with respect to $(\chi,\tau)$ on the region $\chi>\chi_\mathrm{c}(\tau)$, and in particular derivatives of $\mathbf{\holomat}(z;\chi,\tau,M)$ with respect to $(\chi,\tau)$ are analytic in $z$ near $\Gamma_{\alpha\to\beta}$ as well.  It then follows from \eqref{eq:DS-XT-define} that
\begin{equation}
\begin{split}
\mathbf{X}^{(\alpha)}&=\frac{1}{4}\frac{\partial\alpha}{\partial\chi}(\chi,\tau)\mathbf{\holomat}(\alpha;\chi,\tau,M)\sigma_3\mathbf{\holomat}(\alpha;\chi,\tau,M)^{-1}\\
\mathbf{T}^{(\alpha)}&=\frac{1}{4}\frac{\partial\alpha}{\partial\tau}(\chi,\tau)\mathbf{\holomat}(\alpha;\chi,\tau,M)\sigma_3\mathbf{\holomat}(\alpha;\chi,\tau,M)^{-1}\\
\mathbf{X}^{(\beta)}&=-\frac{1}{4}\frac{\partial\beta}{\partial\chi}(\chi,\tau)\mathbf{\holomat}(\beta;\chi,\tau,M)\sigma_3\mathbf{\holomat}(\beta;\chi,\tau,M)^{-1}\\
\mathbf{T}^{(\beta)}&=-\frac{1}{4}\frac{\partial\beta}{\partial\tau}(\chi,\tau)\mathbf{\holomat}(\beta;\chi,\tau,M)\sigma_3\mathbf{\holomat}(\beta;\chi,\tau,M)^{-1},
\end{split}
\end{equation}
and by Schwarz symmetry (i.e., from $\breve{\mathbf{O}}^\mathrm{out}(z^*)^*=\sigma_2\breve{\mathbf{O}}^\mathrm{out}(z)\sigma_2$) we have $\mathbf{X}^{(p^*)}=\sigma_2\mathbf{X}^{(p)*}\sigma_2$ and $\mathbf{T}^{(p^*)}=\sigma_2\mathbf{T}^{(p)*}\sigma_2$ for $p=\alpha,\beta$.
With $\mathbf{X}$ and $\mathbf{T}$ determined from $\mathbf{L}$ in this way, rearranging the definitions \eqref{eq:DS-XT-define} yields a Lax pair of differential equations for which $\mathbf{L}$ is a fundamental (as $\det(\mathbf{L})=1$) solution matrix:
\begin{equation}
\frac{\partial\mathbf{L}}{\partial\chi}=\mathbf{XL}\quad\text{and}\quad\frac{\partial\mathbf{L}}{\partial\tau}=\mathbf{TL},
\label{eq:DS-Lax-equations}
\end{equation}
and hence
the zero-curvature condition
\begin{equation}
\frac{\partial\mathbf{X}}{\partial\tau} -\frac{\partial\mathbf{T}}{\partial\chi} + [\mathbf{X},\mathbf{T}]=\mathbf{0}
\label{eq:DS-ZCC}
\end{equation}
holds.
If $\alpha,\beta,\alpha^*,\beta^*$ are fixed, then the pole contributions vanish from $\mathbf{X}$ and $\mathbf{T}$, and \eqref{eq:DS-ZCC} becomes equivalent to the scaled focusing nonlinear Schr\"odinger equation \eqref{eq:semiclassicalNLS} on $q=\breve{\Psi}(\chi,\tau;M)$.  Since these points are not fixed, $\breve{\Psi}(\chi,\tau;M)$ satisfies instead more complicated equations, one of which we use in Section~\ref{sec:L2-norm} below to calculate the $L^2$-norm.

Combining \eqref{eq:DS-U-Oout} and \eqref{eq:DS-L-H-diag-ev} allows the residues to be expressed directly in terms of $\breve{\mathbf{O}}^\mathrm{out}(z)$ and $h(z)$, although since $\breve{\mathbf{O}}^\mathrm{out}(z)$ is undefined at the branch points, the evaluations at $z=\alpha,\beta$ must be replaced by limits:
\begin{equation}
\mathbf{X}^{(\alpha)}=\frac{1}{4}\frac{\partial\alpha}{\partial\chi}(\chi,\tau)\lim_{z\to\alpha}\mathbf{N}(z;\chi,\tau,M),\quad
\mathbf{X}^{(\beta)}=-\frac{1}{4}\frac{\partial\beta}{\partial\chi}(\chi,\tau)\lim_{z\to\beta}\mathbf{N}(z;\chi,\tau,M),
\label{eq:DS-X-alpha-beta-N}
\end{equation}
and similarly for $\mathbf{T}^{(\alpha,\beta)}$, wherein
\begin{equation}
\begin{split}
\mathbf{N}(z;\chi,\tau,M):&=\mathbf{\holomat}(z;\chi,\tau,M)\sigma_3\mathbf{\holomat}(z;\chi,\tau,M)^{-1}\\
& =\breve{\mathbf{O}}^\mathrm{out}(z)\begin{bmatrix}0 & \ii\ee^{-2M-2\ii Mh(z)}\\-\ii\ee^{2M+2\ii Mh(z)} & 0\end{bmatrix}\breve{\mathbf{O}}^\mathrm{out}(z)^{-1}.
\end{split}
\end{equation}
Although it must be true from the first formula for $\mathbf{N}(z;\chi,\tau,M)$ in terms of $\mathbf{\holomat}(z;\chi,\tau,M)$, which has unit determinant and is analytic near $\Gamma_{\alpha\to\beta}$, one can directly confirm from the second formula that $\mathbf{N}(z;\chi,\tau,M)$ has no jump across either $\Gamma_{\alpha\to\beta}$ or $\Gamma_{\beta\to\beta^*}$.  Now notice that, since $h(\alpha)=\ii +\frac{1}{2}\phi$ is well-defined and $h'(z)=O((z-\alpha)^\frac{1}{2})$ as $z\to\alpha$, reality of $\phi=\phi(\chi,\tau)$ yields $\ee^{\pm(2M+2\ii Mh(z))}=\ee^{\pm\ii M\phi}+O((z-\alpha)^\frac{3}{2})$.  Then, since $\breve{\mathbf{O}}^\mathrm{out}(z)$ has unit determinant and satisfies $\breve{\mathbf{O}}^\mathrm{out}(z)=O((z-\alpha)^{-\frac{1}{4}})$ as $z\to\alpha$, we get $\mathbf{N}(z;\chi,\tau,M)=\widetilde{\mathbf{N}}^{(\alpha)}(z;\chi,\tau,M)+O(z-\alpha)$, where
\begin{equation}
\widetilde{\mathbf{N}}^{(\alpha)}(z;\chi,\tau,M):=\breve{\mathbf{O}}^\mathrm{out}(z)\begin{bmatrix}0 & \ii\ee^{-\ii M\phi}\\-\ii\ee^{\ii M\phi} & 0\end{bmatrix}\breve{\mathbf{O}}^\mathrm{out}(z)^{-1}.
\end{equation}
One can check directly that $\widetilde{\mathbf{N}}^{(\alpha)}(z;\chi,\tau,M)$ is analytic for $z$ in a neighborhood of $\alpha$.  Hence  in the expression for $\mathbf{X}^{(\alpha)}$ in \eqref{eq:DS-X-alpha-beta-N}, $\mathbf{N}(z;\chi,\tau,M)$ can be replaced with $\widetilde{\mathbf{N}}^{(\alpha)}(z;\chi,\tau,M)$.  Similarly, using the fact that the boundary values taken by $h(z)$ at $z=\beta$ on $\Gamma_{\alpha\to\beta}\cup\Gamma_{\beta\to\beta^*}$ are $h_\pm(\beta)=\ii+\frac{1}{2}\phi\pm\frac{1}{2}\Delta$, we get $\mathbf{N}(z;\chi,\tau,M)=\widetilde{\mathbf{N}}^{(\beta)}(z;\chi,\tau,M)+O(z-\beta)$ as $z\to\beta$, where
\begin{equation}
\widetilde{\mathbf{N}}^{(\beta)}(z;\chi,\tau,M):=\breve{\mathbf{O}}^\mathrm{out}(z)\begin{bmatrix}0 & \ii\ee^{-\ii M\phi\mp\ii M\Delta}\\
-\ii\ee^{\ii M\phi\pm\ii M\Delta} & 0\end{bmatrix}\breve{\mathbf{O}}^\mathrm{out}(z)^{-1},
\end{equation}
where the top/bottom sign indicates that $z$ lies on the left/right side of $\Gamma_{\alpha\to\beta}\cup\Gamma_{\beta\to\beta^*}$ near $z=\beta$.   One can also check that $\widetilde{\mathbf{N}}^{(\beta)}(z;\chi,\tau,M)$ extends to $\Gamma_{\alpha\to\beta}\cup\Gamma_{\beta\to\beta^*}$ as an analytic function of $z$ near $z=\beta$.   In the expression for $\mathbf{X}^{(\beta)}$ in \eqref{eq:DS-X-alpha-beta-N}, $\mathbf{N}(z;\chi,\tau,M)$ can then be replaced with $\widetilde{\mathbf{N}}^{(\beta)}(z;\chi,\tau,M)$ provided the limit $z\to\beta$ is taken from the correct side of the jump contour corresponding to the top/bottom sign. 

\subsubsection{Estimates of the coefficients}
\label{s:estimate-coeffs}
In Section~\ref{sec:L2-norm} below we will need information about the (opposite) diagonal elements of the matrices $\mathbf{X}^{(\alpha)}$ and $\mathbf{X}^{(\beta)}$, and in particular we need to estimate $\widetilde{N}_{22}^{(p)}$ for $p=\alpha,\beta$.  In the notation of Section~\ref{sec:DS-outer}, we have
\begin{multline}
\widetilde{N}^{(\alpha)}_{22}(z;\chi,\tau,M)=\frac{\ii F^\mathrm{D}(z)F^{\mathrm{OD}}(z)}{q^+(\infty;z_0,-\ii M\Delta)q^+(\infty;z_0,\ii M\Delta)}\\
{}\cdot\left(
q^+(z;z_0,-\ii M\Delta)q^-(z;z_0,-\ii M\Delta)-q^+(z;z_0,\ii M\Delta)q^-(z;z_0,\ii M\Delta)\right)
\label{eq:tilde-N-alpha}
\end{multline}
and
\begin{multline}
\widetilde{N}^{(\beta)}_{22}(z;\chi,\tau,M)=\frac{\ii F^\mathrm{D}(z)F^\mathrm{OD}(z)}{q^+(\infty;z_0,-\ii M\Delta)q^+(\infty;z_0,\ii M\Delta)}\\{}\cdot\left(
\ee^{\pm\ii M\Delta}q^+(z;z_0,-\ii M\Delta)q^-(z;z_0,-\ii M\Delta)-\ee^{\mp \ii M\Delta}q^+(z;z_0,\ii M\Delta)q^-(z;z_0,\ii M\Delta)\right),
\label{eq:tilde-N-beta}
\end{multline}
where again the top/bottom sign indicates that $z$ lies on the left/right of the jump contour $\Gamma_{\alpha\to\beta}\cup\Gamma_{\beta\to\beta^*}$ near $z=\beta$.  The common $z$-independent denominator can be simplified using \eqref{eq:Az0-identity} with $n_1=-1$ and $n_2=1$ and \eqref{eq:RiemannConstant} in \eqref{eq:qpm-def}:
\begin{equation}
q^+(\infty;z_0,-\ii M\Delta)q^+(\infty;z_0,\ii M\Delta)=\frac{\Theta(H+\ii M\Delta;H)\Theta(H-\ii M\Delta;H)}{\Theta(H;H)^2}=\frac{\Theta(\ii M\Delta;H)^2}{\Theta(0;H)^2},
\end{equation}
where in the second equality we used the identities \eqref{eq:automorphic}.
To evaluate \eqref{eq:tilde-N-alpha} in the limit $z\to\alpha$, we may first use \eqref{eq:FD-FOD} to obtain
\begin{equation}
\lim_{z\to\alpha}R(z)\cdot \ii F^\mathrm{D}(z)F^\mathrm{OD}(z)=\frac{1}{4}(\alpha-\alpha^*)(\beta-\alpha).
\label{eq:FD-FOD-limit}
\end{equation}
Next, using $A(\alpha)=0$ in \eqref{eq:qpm-def} and Taylor expansion about $z=\alpha$ gives
\begin{multline}
q^+(z;z_0,-\ii M\Delta)q^-(z;z_0,-\ii M\Delta)-q^+(z;z_0,\ii M\Delta)q^-(z;z_0,\ii M\Delta)\\
{}=2\frac{\Theta(x_\alpha-\ii M\Delta;H)\Theta'(x_\alpha+\ii M\Delta;H)-
\Theta'(x_\alpha-\ii M\Delta;H)\Theta(x_\alpha+\ii M\Delta;H)}{\Theta(x_\alpha;H)^2}A(z)+O(A(z)^3)
\label{eq:qpm-diff-alpha}
\end{multline}
as $z\to\alpha$, where $x_\alpha:=A(z_0)+\mathcal{K}$ and the first identity in \eqref{eq:automorphic} was also used.  From \eqref{eq:Abel-define} we then obtain
\begin{equation}
\lim_{z\to\alpha}\frac{A(z)}{R(z)}=\frac{\pi\ii}{I_\mathcal{A}}\lim_{z\to\alpha}\frac{\dd}{\dd z}\left(\int_\alpha^z\frac{\dd z'}{R(z')}\right)^2 =\frac{4\pi\ii}{I_\mathcal{A}(\alpha-\alpha^*)(\alpha-\beta)(\alpha-\beta^*)}.
\label{eq:A-over-R-alpha}
\end{equation}
Therefore, combining \eqref{eq:FD-FOD-limit}, \eqref{eq:qpm-diff-alpha}, and \eqref{eq:A-over-R-alpha} shows that if $x_\alpha:=A(z_0)+\mathcal{K}$, then
\begin{multline}
\lim_{z\to\alpha}\widetilde{N}_{22}^{(\alpha)}(z;\chi,\tau,M)=-\frac{2\pi\ii\Theta(0;H)^2}{I_\mathcal{A}(\alpha-\beta^*)\Theta(\ii M\Delta;H)^2}\\
{}\cdot\frac{\Theta(x_\alpha-\ii M\Delta;H)\Theta'(x_\alpha+\ii M\Delta;H)-
\Theta'(x_\alpha-\ii M\Delta;H)\Theta(x_\alpha+\ii M\Delta;H)}{\Theta(x_\alpha;H)^2}.
\label{eq:N22-alpha-limit}
\end{multline}
Likewise, to evaluate \eqref{eq:tilde-N-beta} in the limit $z\to\beta$, we start from the analogue of \eqref{eq:FD-FOD-limit}:
\begin{equation}
\lim_{z\to\beta} R(z)\cdot\ii F^\mathrm{D}(z)F^\mathrm{OD}(z)=\frac{1}{4}(\beta-\beta^*)(\beta-\alpha).
\end{equation}
Next, from \eqref{eq:DS-H-def} and \eqref{eq:Abel-define} we find that $A_\pm(\beta)=\pm\frac{1}{2}H$, where the subscript denotes taking a limit from the left ($+$) or right ($-$) side of $\Gamma_{\alpha\to\beta}\cup\Gamma_{\beta\to\beta^*}$.  Since we know that $\widetilde{N}_{22}^{(\beta)}(z;\chi,\tau,M)$ is analytic in $z$ at $z=\beta$, without loss of generality we may agree to take the limit $z\to\beta$ in \eqref{eq:tilde-N-beta} from the left side, taking the top sign therein.  Hence, using \eqref{eq:qpm-def} and Taylor expansion about $z=\beta$ recalling the automorphic identities \eqref{eq:automorphic}, 
\begin{multline}
\ee^{\ii M\Delta}q^+(z;z_0,-\ii M\Delta)q^-(z;z_0,-\ii M\Delta)-\ee^{-\ii M\Delta}q^+(z;z_0,\ii M\Delta)q^-(z;z_0,\ii M\Delta) \\
{}=2\frac{\Theta'(\widetilde{x}_\beta+\ii M\Delta;H)\Theta(\widetilde{x}_\beta-\ii M\Delta;H)-\Theta'(\widetilde{x}_\beta-\ii M\Delta;H)\Theta(\widetilde{x}_\beta+\ii M\Delta;H)}{\Theta(\widetilde{x}_\beta;H)^2}(A(z)-A(\beta))\\
{}+O((A(z)-A(\beta))^3),
\end{multline}
as $z\to\beta$, where $\widetilde{x}_\beta:=A(z_0)+\mathcal{K}+\frac{1}{2}H$.  By the third identity in \eqref{eq:automorphic}, the same formula holds if $\widetilde{x}_\beta$ is replaced with $x_\beta:=\widetilde{x}_\beta-H=A(z_0)+\mathcal{K}-\frac{1}{2}H$.  
Then again taking the limit from the left,
\begin{equation}
\lim_{z\to\beta}\frac{A(z)-A(\beta)}{R(z)} =\frac{\pi\ii}{I_\mathcal{A}}\frac{\dd}{\dd z}\left(\int_\beta^z\frac{\dd z'}{R(z')}\right)^2 = \frac{4\pi\ii}{I_\mathcal{A}(\beta-\alpha)(\beta-\alpha^*)(\beta-\beta^*)}.
\end{equation}
Combining these results then shows that if $x_\beta:=A(z_0)+\mathcal{K}-\tfrac{1}{2}H$, then
\begin{multline}
\lim_{z\to\beta}\widetilde{N}_{22}^{(\beta)}(z;\chi,\tau,M)=\frac{2\pi\ii\Theta(0;H)^2}{I_\mathcal{A}(\beta-\alpha^*)\Theta(\ii M\Delta;H)^2}\\
{}\cdot\frac{\Theta'(x_\beta+\ii M\Delta;H)\Theta(x_\beta-\ii M\Delta;H)-\Theta'(x_\beta-\ii M\Delta;H)\Theta(x_\beta+\ii M\Delta;H)}{\Theta(x_\beta;H)^2}.
\label{eq:N22-beta-limit}
\end{multline}
We notice that the denominators in \eqref{eq:N22-alpha-limit} and \eqref{eq:N22-beta-limit} involve related products $(\alpha-\beta^*)I_\mathcal{A}$ and $(\beta-\alpha^*)I_\mathcal{A}$, where $I_\mathcal{A}$ was defined in \eqref{eq:DS-I-AB}.
Since $I_\mathcal{A}$ is positive imaginary, these are related by complex conjugation.  Moreover, recalling \eqref{eq:DS-H-def}, we have the following result:
\begin{lemma}
In the limit $\chi\downarrow\chi_\mathrm{c}(\tau)$, we have $H\uparrow 0$ and
\begin{equation}
\frac{1}{(\alpha-\beta^*)I_\mathcal{A}}=\left[\frac{1}{(\beta-\alpha^*)I_\mathcal{A}}\right]^*=O(H),
\label{eq:frac-bound-H-small}
\end{equation}
while in the limit $\chi\uparrow +\infty$ we have $H\downarrow -\infty$ and
\begin{equation}
\frac{1}{(\alpha-\beta^*)I_\mathcal{A}}=\left[\frac{1}{(\beta-\alpha^*)I_\mathcal{A}}\right]^*=O(1).
\label{eq:frac-bound-H-large}
\end{equation}
\label{lem:DS-IA-factors-bounds}
\end{lemma}

\begin{proof}
For the limit as $\chi\downarrow\chi_\mathrm{c}(\tau)$, we note that $I_\mathcal{B}$ has a finite strictly negative real limit, while $\alpha-\beta^*\to \critpt-\critpt^*=2\ii\mathrm{Im}(\critpt)\neq 0$.  Therefore multiplying by $1=I_\mathcal{B}/I_\mathcal{B}$ and using \eqref{eq:DS-H-def} (which also shows that $H\uparrow 0$ in the limit because $I_\mathcal{A}$ blows up logarithmically in $|\beta-\alpha|$ as in the proof of Lemma~\ref{lem:absolute-integrability}) proves \eqref{eq:frac-bound-H-small}.

For the limit as $\chi\uparrow+\infty$, we recall that $\alpha$ and $\beta$ both tend to zero in this limit, and in fact $\chi^\frac{1}{2}\alpha\to -\sqrt{2}$ while $\chi^\frac{1}{2}\beta\to \sqrt{2}$.  In particular, $\alpha-\beta^* =2\sqrt{2}\chi^{-\frac{1}{2}}(1+o(1))$.  By the scaling $z=\chi^{-\frac{1}{2}}w$ in the definition \eqref{eq:DS-I-AB}, one then also sees that $\chi^{-\frac{1}{2}}I_\mathcal{A}$ has a finite nonzero limit while $\chi^{-\frac{1}{2}}I_\mathcal{B}$ blows up as $\chi\uparrow+\infty$ (and hence $H\downarrow -\infty$).  Again, see the proof of Lemma~\ref{lem:absolute-integrability} for further details.  This proves \eqref{eq:frac-bound-H-large}.
\end{proof}

Since $z_0\in\mathbb{R}$, it is not difficult to show that $\mathrm{Im}(A(z_0))=-\frac{1}{2}\pi$, which in view of \eqref{eq:RiemannConstant} and $H<0$ implies that $x_\alpha$ and $x_\beta$ are both real.  Also, since to compute $\mathrm{Re}(A(z_0))$ from \eqref{eq:Abel-define} we may integrate $R(z)^{-1}$ between two points on the real line, the corresponding integral is bounded in absolute value by $|I_\mathcal{B}|$, so using \eqref{eq:DS-H-def} we see that $|\mathrm{Re}(A(z_0))|\le -H>0$.  Hence using \eqref{eq:RiemannConstant} shows that $\frac{3}{2}H\le x_\alpha\le -\frac{1}{2}H$ and $H\le x_\beta\le -H$.  Using this information, we now wish to estimate, for $x=x_\alpha,x_\beta\in\mathbb{R}$, $y=M\Delta\in\mathbb{R}$, and $H<0$:
\begin{equation}
T(H):=\mathop{\sup_{\frac{3}{2}H\le x\le-\frac{3}{2}H}}_{y\in\mathbb{R}}\left|\frac{\Theta(0;H)^2}{\Theta(\ii y;H)^2}\frac{\Theta'(x+\ii y;H)\Theta(x-\ii y;H)-\Theta'(x-\ii y;H)\Theta(x+\ii y;H)}{\Theta(x;H)^2}\right|
\end{equation}
\begin{lemma}
$T(H)$ is a continuous function of $H<0$ that satisfies
\begin{equation}
T(H)=O(1),\quad H\downarrow -\infty
\label{eq:DS-T-final-bound-H-large}
\end{equation}
and
\begin{equation}
T(H)=O(H^{-1}),\quad H\uparrow 0.
\label{eq:DS-T-final-bound-H-small}
\end{equation}
\label{lem:DS-T-bounds}
\end{lemma}
\begin{proof}
We give a proof in Appendix~\ref{a:Lemma}.
\end{proof}

Combining Lemma~\ref{lem:DS-IA-factors-bounds} and Lemma \ref{lem:DS-T-bounds} we obtain the following result:
\begin{proposition}
For each $\tau\in\mathbb{R}$, 
$\lim_{z\to\alpha}\widetilde{N}^{(\alpha)}_{22}(z;\chi,\tau,M)$ and $\lim_{z\to\beta}\widetilde{N}^{(\beta)}_{22}(z;\chi,\tau,M)$
are bounded uniformly with respect to $M>0$ and $\chi>\chi_\mathrm{c}(\tau)$.
\label{prop:N-alpha-beta-bound}
\end{proposition}
\begin{proof}
It only remains to explain that the bounds on $x$ in the definition of $T(H)$ are sufficient to guarantee that $\frac{3}{2}H\le x_\alpha\le-\frac{1}{2}H$ and $H\le x_\beta\le -H$ as needed for consistency with the value of $A(z_0)$.
\end{proof}

 In Section~\ref{sec:L2-norm} below we will also need to estimate the quantity $\breve{O}^{[1]}_{22}(\chi,\tau;M)$ defined by
\begin{equation}
\breve{O}^{[1]}_{22}(\chi,\tau;M):=\lim_{z\to\infty} z\left(\breve{O}^{\mathrm{out}}_{22}(z)-1\right).
\end{equation}
Since $F^\mathrm{D}(z)=1+O(z^{-2})$ as $z\to\infty$, and since using $R(z)=z^2+O(z)$ in the same limit gives
\begin{equation}
A(z)-A(\infty)=\frac{2\pi\ii}{I_\mathcal{A}}\int_\infty^z\frac{\dd z'}{R(z')} = -\frac{2\pi\ii}{I_\mathcal{A}}\frac{1}{z}+O(z^{-2}),\quad z\to\infty,
\end{equation}
combining \eqref{eq:qpm-def} with \eqref{eq:DS-outer-parametrix-formula} gives
\begin{equation}
\breve{O}^{[1]}_{22}(\chi,\tau;M)=\frac{2\pi\ii}{I_\mathcal{A}}\left[\frac{\Theta'(A(\infty)+A(z_0)+\mathcal{K};H)}{\Theta(A(\infty)+A(z_0)+\mathcal{K};H)}-\frac{\Theta'(A(\infty)+A(z_0)+\mathcal{K}+\ii M\Delta;H)}{\Theta(A(\infty)+A(z_0)+\mathcal{K}+\ii M\Delta;H)}\right].
\end{equation}
Using $A(\infty)+A(z_0)=\frac{1}{2}H-\ii\pi$ together with \eqref{eq:RiemannConstant}, and applying the first and third identities in \eqref{eq:automorphic} we arrive at
\begin{equation}
\breve{O}^{[1]}_{22}(\chi,\tau;M)=-\frac{2\pi\ii}{I_\mathcal{A}}\frac{\Theta'(\ii M\Delta;H)}{\Theta(\ii M\Delta;H)}.
\label{eq:breve-O-1-22-simple}
\end{equation}
We then have the following:
\begin{proposition}
For each $\tau\in\mathbb{R}$, $\breve{O}^{[1]}_{22}(\chi,\tau;M)$ is bounded uniformly with respect to $M>0$ and $\chi>\chi_\mathrm{c}(\tau)$.
\label{prop:breve-O-1-bound}
\end{proposition}
\begin{proof}
Since the denominator is nonzero for each $\chi>\chi_\mathrm{c}(\tau)$ and the expression \eqref{eq:breve-O-1-22-simple} is $2\pi$-periodic in $y:=M\Delta\in\mathbb{R}$, it suffices to examine it as a function of $y\in [-\pi,\pi]$ and determine its behavior in the limits $\chi\downarrow\chi_\mathrm{c}(\tau)$ and $\chi\uparrow +\infty$.

As in the proof of Lemma~\ref{lem:DS-IA-factors-bounds}, one sees that $H\uparrow 0$ and $I_\mathcal{A}^{-1}=O(H)$ as $\chi\downarrow\chi_\mathrm{c}(\tau)$, while $H\downarrow -\infty$ and $I_\mathcal{A}^{-1}\to 0$ proportional to $\chi^{-\frac{1}{2}}$ as $\chi\uparrow +\infty$.  Then, as in the proof of Lemma~\ref{lem:DS-T-bounds}, one checks using dominated convergence applied to \eqref{eq:DS-Theta-define} that $\Theta'(\ii y;H)/\Theta(\ii y;H)\to 0$ as $H\downarrow -\infty$ uniformly for $y\in\mathbb{R}$, and that \eqref{eq:DS-Theta-frac-estimate} implies in particular that $\Theta'(\ii y;H)/\Theta(\ii y;H)=O(H^{-1})$ as $H\uparrow 0$ uniformly for $y\in\mathbb{R}$.  
\end{proof}

\subsubsection{$L^2$-norm}
\label{sec:L2-norm}
A consequence of the expansion \eqref{eq:DS-Lax-expansion} and the Lax equations \eqref{eq:DS-Lax-equations} is the pair of identities
\begin{equation}
\begin{split}
\frac{\partial L^{[1]}_{11}}{\partial\chi}(\chi,\tau;M)-2\ii M L^{[1]}_{12}(\chi,\tau;M)L^{[1]}_{21}(\chi,\tau;M) &=X_{11}^{[-1]}(\chi,\tau;M)\\
\frac{\partial L^{[1]}_{22}}{\partial\chi}(\chi,\tau;M)+2\ii M L^{[1]}_{21}(\chi,\tau;M)L^{[1]}_{12}(\chi,\tau;M)&=X_{22}^{[-1]}(\chi,\tau;M),
\end{split}
\end{equation}
where $\mathbf{X}^{[-1]}(\chi,\tau;M)$ denotes the coefficient of $z^{-1}$ in the Laurent expansion of $\mathbf{X}$ as $z\to\infty$.  From \eqref{eq:DS-X-T-poles} we see that 
\begin{equation}
\mathbf{X}^{[-1]}(\chi,\tau;M)=\mathbf{X}^{(\alpha)}+\mathbf{X}^{(\beta)}+\mathbf{X}^{(\alpha^*)}+\mathbf{X}^{(\beta^*)} = \mathbf{X}^{(\alpha)}+\mathbf{X}^{(\beta)}+\sigma_2\mathbf{X}^{(\alpha)*}\sigma_2 +\sigma_2\mathbf{X}^{(\beta)*}\sigma_2.
\end{equation}
Using \eqref{eq:DS-Psi-breve-U} and the fact that the traces of $\mathbf{L}^{[1]}$ and of the four residue matrices vanish, we deduce the two equivalent formul\ae\
\begin{equation}
\begin{split}
|\breve{\Psi}(\chi,\tau;M)|^2 &= \frac{2\ii}{M}\left[\frac{\partial L^{[1]}_{11}}{\partial\chi}(\chi,\tau;M)-2\ii\mathrm{Im}(X^{(\alpha)}_{11}+X^{(\beta)}_{11})\right]\\
&=-\frac{2\ii}{M}\left[\frac{\partial L^{[1]}_{22}}{\partial\chi}(\chi,\tau;M)-2\ii\mathrm{Im}(X^{(\alpha)}_{22}+X^{(\beta)}_{22})\right].
\end{split}
\label{eq:DS-mod-psibreve-squared-Lax}
\end{equation}
Now combining \eqref{eq:DS-tildevartheta}, \eqref{eq:DS-h-g-tildevartheta}, \eqref{eq:DS-U-Oout}, and \eqref{eq:DS-Lax-expansion} shows that
\begin{equation}
\begin{split}
\mathbf{L}^{[1]}(\chi,\tau;M)&=\lim_{z\to\infty} z\left(\mathbf{L}(z;\chi,\tau,M)\ee^{\ii M(\chi z + \tau z^2)\sigma_3}-\mathbb{I}\right)\\
&=\lim_{z\to\infty}z\left(\breve{\mathbf{O}}^\mathrm{out}(z)\ee^{-\ii M g(z)\sigma_3}-\mathbb{I}\right)\\
&=\lim_{z\to\infty}z\left(\left(\mathbb{I}+z^{-1}\breve{\mathbf{O}}^{[1]}(\chi,\tau;M) + O(z^{-2})\right)\left(\mathbb{I} -\ii M g_1(\chi,\tau)z^{-1}\sigma_3 + O(z^{-2})\right)-\mathbb{I}\right)\\
&=\breve{\mathbf{O}}^{[1]}(\chi,\tau;M) -\ii Mg_1(\chi,\tau)\sigma_3,
\end{split}
\end{equation}
where $\breve{\mathbf{O}}^{[1]}(\chi,\tau;M)$ is the coefficient of $z^{-1}$ in the Laurent expansion of $\breve{\mathbf{O}}^{\mathrm{out}}(z)$ as $z\to\infty$, and
\begin{equation}
g_1(\chi,\tau):=\lim_{z\to\infty}zg(z)
\end{equation}
which is well defined because $g$ is analytic for large $z$ and $g(z)\to 0$ as $z\to\infty$.  Integrating \eqref{eq:DS-mod-psibreve-squared-Lax} in $\chi$ from $\chi_\mathrm{c}(\tau)$ to $+\infty$ for fixed $\tau$ therefore gives
\begin{multline}
\int_{\chi_\mathrm{c}(\tau)}^{+\infty}|\breve{\Psi}(\chi,\tau;M)|^2\,\dd\chi = 2g_1(+\infty,\tau)-2g_1(\chi_\mathrm{c}(\tau),\tau) -\frac{2\ii}{M}\left(\breve{O}^{[1]}_{22}(+\infty,\tau;M)-\breve{O}^{[1]}_{22}(\chi_\mathrm{c}(\tau),\tau;M)\right)\\
{}-\frac{4}{M}\int_{\chi_\mathrm{c}(\tau)}^{+\infty}\mathrm{Im}(X_{22}^{(\alpha)}+X_{22}^{(\beta)})\,\dd\chi.
\end{multline}
Here, the terms on the first line of the right-hand side are understood to refer to the limits $\chi\downarrow \chi_\mathrm{c}(\tau)$ and $\chi\uparrow+\infty$.

First, we apply Proposition~\ref{prop:breve-O-1-bound} and obtain in the limit $M\to\infty$,
\begin{equation}
\int_{\chi_\mathrm{c}(\tau)}^{+\infty}|\breve{\Psi}(\chi,\tau;M)|^2\,\dd\chi = 2g_1(+\infty,\tau)-2g_1(\chi_\mathrm{c}(\tau),\tau) -\frac{4}{M}\int_{\chi_\mathrm{c}(\tau)}^{+\infty}
\mathrm{Im}(X_{22}^{(\alpha)}+X_{22}^{(\beta)})\,\dd\chi + O(M^{-1}).
\label{eq:DS-L2-norm-1}
\end{equation}
Next, we use \eqref{eq:DS-X-alpha-beta-N} and $\mathbf{N}(z;\chi,\tau,M)=\widetilde{\mathbf{N}}^{(\alpha)}(z;\chi,\tau,M) + O(z-\alpha)=\widetilde{\mathbf{N}}^{(\beta)}(z;\chi,\tau,M)+O(z-\beta)$ and apply Proposition~\ref{prop:N-alpha-beta-bound} to get
\begin{equation}
\left|\int_{\chi_\mathrm{c}(\tau)}^{+\infty}\mathrm{Im}(X_{22}^{(\alpha)}+X_{22}^{(\beta)})\,\dd\chi\right| = O\left(\int_{\chi_\mathrm{c}(\tau)}^{+\infty}\left|\frac{\partial\alpha}{\partial\chi}(\chi,\tau)\right|\,\dd\chi +\int_{\chi_\mathrm{c}(\tau)}^{+\infty}\left|\frac{\partial\beta}{\partial\chi}(\chi,\tau)\right|\,\dd\chi \right),\quad M\to\infty.
\label{eq:DS-integral-bound}
\end{equation}

From Lemma~\ref{lem:absolute-integrability} it follows that the bound on the right-hand side of \eqref{eq:DS-integral-bound} is finite, and it is independent of $M$.  Consequently we obtain from \eqref{eq:DS-L2-norm-1} that 
\begin{equation}
\int_{\chi_\mathrm{c}(\tau)}^{+\infty}|\breve{\Psi}(\chi,\tau;M)|^2\,\dd\chi = 2g_1(+\infty,\tau)-2g_1(\chi_\mathrm{c}(\tau),\tau)  + O(M^{-1}),\quad M\to\infty.
\end{equation}

To finish the calculation, we will compute $g_1(\chi,\tau)$, by first noting that 
\begin{equation}
g(z)=g_1(\chi,\tau)z^{-1} + O(z^{-2}) \implies g'(z)=-g_1(\chi,\tau)z^{-2} + O(z^{-3}),\quad z\to\infty.
\end{equation}
Now combining \eqref{eq:DS-tildevartheta}, \eqref{eq:DS-h-g-tildevartheta}, and \eqref{eq:hprime-Rsquared}, we get
\begin{equation}
g_1(\chi,\tau)=-\lim_{z\to\infty}z^2g'(z) = -\lim_{z\to\infty}z^2\left[\frac{2\tau z+\chi-\tau^2\lambda}{z^2}R(z)-(\chi+2\tau z + 2z^{-2})\right],
\label{eq:DS-g1}
\end{equation}
where we recall that $R(z)=z^2+O(z)$ as $z\to\infty$ and $R(z)^2$ is given by \eqref{eq:Intro-Rsquared}.  It follows that also
\begin{equation}
R(z)=z^2 + \frac{1}{2}\tau\lambda z+\frac{1}{4}(\tau^2\lambda-\chi)\lambda +
\frac{(8+(\tau^2\lambda-\chi)^2\lambda)(8-(\tau^2\lambda-\chi)^2\lambda)}{8(\tau^2\lambda-\chi)^3}\tau
z^{-1}+O(z^{-2}),\quad z\to\infty.
\end{equation}
Using this in \eqref{eq:DS-g1} gives the explicit (in terms of the solution $\lambda=\lambda(\chi,\tau)$ of the integral condition $f(\lambda;\chi,\tau)=0$) formula
\begin{equation}
g_1(\chi,\tau) = \frac{(2(\tau^2\lambda-\chi)^3+\chi(\tau^2\lambda-\chi)^2-8\tau^2)(8+(\tau^2\lambda-\chi)^2\lambda)}{4(\tau^2\lambda-\chi)^3}.
\label{eq:DS-g1-formula}
\end{equation}
For the limit $\chi\downarrow\chi_\mathrm{c}(\tau)$, we the parametrization of $(\chi,\tau)$ by \eqref{eq:chi-tau-sigma} and of $\lambda$ by \eqref{eq:gamma-negative} to get
\begin{equation}
g_1(\chi_\mathrm{c}(\tau),\tau)=\lim_{\chi\downarrow\chi_\mathrm{c}(\tau)}g_1(\chi,\tau)=0.
\end{equation}
To calculate $g_1(\chi,\tau)$ in the limit $\chi\to+\infty$ we use instead the representation $\lambda=2(4-\delta^2)\chi^{-2}$ with $\delta^2=O(\chi^{-\frac{1}{2}})$ according to \eqref{eq:delta-squared-expansion} in the proof of Lemma~\ref{lem:absolute-integrability}.  Substituting into \eqref{eq:DS-g1-formula} then gives 
\begin{equation}
g_1(+\infty,\tau)=\lim_{\chi\to+\infty}g_1(\chi,\tau)=4.
\end{equation}
Therefore, we have shown that 
\begin{equation}
\int_{\chi_\mathrm{c}(\tau)}^{+\infty}|\breve{\Psi}(\chi,\tau;M)|^2\,\dd\chi = 2g_1(+\infty,\tau) -2g_1(\chi_\mathrm{c}(\tau),\tau)+ O(M^{-1}) =8 + O(M^{-1}),\quad M\to+\infty.
\end{equation}
This completes the proof of \eqref{eq:Intro-approx-L2} in Theorem~\ref{t:DS}.

\begin{remark}
This approach would appear to be an alternative to the more direct one based on integration of the formula for $|\breve{\Psi}(\chi,\tau;M)|^2$ given in \eqref{eq:Intro-square-modulus}.  Noting that the dependence on $M$ in \eqref{eq:Intro-square-modulus} enters through the argument of the Jacobi elliptic function $\mathrm{sn}^2(\diamond;m)$ which is rapidly varying for large $M$, one could perhaps 
pass to the limit $M\to+\infty$ by replacing it with its period average, which is $\langle \mathrm{sn}^2(\diamond;m)\rangle = (\mathbb{K}(m)-\mathbb{E}(m))/(m\mathbb{K}(m))$.  Then, to establish the same result it would remain to prove that 
\begin{multline}
\int_{\chi_\mathrm{c}(\tau)}^{+\infty} \left[(\mathrm{Im}(\alpha(\chi,\tau))+\mathrm{Im}(\beta(\chi,\tau)))^2\vphantom{\frac{\mathbb{K}(m_1(\chi,\tau))-\mathbb{E}(m_1(\chi,\tau))}{m_1(\chi,\tau)\mathbb{K}(m_1(\chi,\tau))}}\right.\\
{}\left.-4\mathrm{Im}(\alpha(\chi,\tau))\mathrm{Im}(\beta(\chi,\tau))\frac{\mathbb{K}(m_1(\chi,\tau))-\mathbb{E}(m_1(\chi,\tau))}{m_1(\chi,\tau)\mathbb{K}(m_1(\chi,\tau))}\right]\dd\chi = 8
\end{multline}
where $m_1(\chi,\tau)$ is defined in terms of $\alpha$ and $\beta$ by \eqref{eq:Intro-m1}, given the rather implicit characterization of $\alpha$ and $\beta$ as functions of $(\chi,\tau)$.  In fact, we may regard the above arguments as an indirect proof of this identity.
\end{remark}

\subsubsection{Explicit formul\ae}
\label{s:explicit-formulae}
Using the definition of the outer parametrix obtained in Section~\ref{sec:DS-outer}, we have
\begin{equation}
\begin{split}
\breve{\Psi}(\chi,\tau;M)&=2\ii\ee^{-\ii M\phi}\frac{q^-(\infty;z_0,-\ii M\Delta)}{q^+(\infty;z_0,\ii M\Delta)}\lim_{z\to\infty}zF^\mathrm{OD}(z)\\
&=\ii(\mathrm{Im}(\beta)-\mathrm{Im}(\alpha))\ee^{-\ii M\phi}\frac{q^-(\infty;z_0,-\ii M\Delta)}{q^+(\infty;z_0,\ii M\Delta)}
\end{split}
\label{eq:DS-Psi-breve-define}
\end{equation}
in which the final ratio is given explicitly by
\begin{equation}
\frac{q^-(\infty;z_0,-\ii M\Delta)}{q^+(\infty;z_0,\ii M\Delta)}
=\frac{\Theta(A(\infty)-A(z_0)-\mathcal{K}+\ii M\Delta;H)\Theta(A(\infty)+A(z_0)+\mathcal{K};H)}{\Theta(A(\infty)-A(z_0)-\mathcal{K};H)\Theta(A(\infty)+A(z_0)+\mathcal{K}-\ii M\Delta;H)},
\label{eq:DS-q-ratio}
\end{equation}
where $z_0$ is defined by \eqref{eq:z0}.

Our aim in this section is to express $|\breve{\Psi}(\chi,\tau;M)|^2$ explicitly in terms of Jacobi elliptic functions with elliptic modulus $m$ defined in Section~\ref{sec:DS-outer}.  Note that the modulus $m$ is the one most naturally associated with the spectral curve underlying the outer parametrix $\breve{\mathbf{O}}^\mathrm{out}(z)$ via the integrals $I_\mathfrak{a}$ and $I_\mathfrak{b}$ and the associated Abel mapping $A(z)$.  In order to achieve this goal, it is first necessary to 
express the formula in \eqref{eq:DS-q-ratio} in terms of theta functions for the doubled parameter $H_0:=2H$ using the identity (see \cite[Eqn.\@ 20.7.14]{DLMF} and \cite[Eqns.\@ 20.2.10 and 20.2.12]{DLMF})
\begin{multline}
\Theta(w_1;H)\Theta(w_2;H)=\Theta(w_1+w_2;2H)\Theta(w_1-w_2;2H)\\
{}+\ee^{\frac{1}{2}H}\ee^{w_1}\Theta(w_1+w_2+H;2H)\Theta(w_1-w_2+H;2H).
\end{multline}
Using this result as well as \eqref{eq:RiemannConstant} and $\Theta(z+2\pi\ii;\cdot)=\Theta(z;\cdot)$, the numerator of \eqref{eq:DS-q-ratio} becomes
\begin{multline}
\Theta(A(\infty)-A(z_0)-\mathcal{K}+\ii M\Delta;H)\Theta(A(\infty)+A(z_0)+\mathcal{K};H)\\
{}=\Theta(2A(\infty)+\ii M\Delta;2H)\Theta(-2A(z_0)-H+\ii M\Delta;2H)\\
{}-\ee^{A(\infty)-A(z_0)+\ii M\Delta}\Theta(2A(\infty)+\ii M\Delta + H;2H)\Theta(-2A(z_0)+\ii M\Delta;2H).
\label{eq:DS-numerator}
\end{multline}
Likewise, the denominator of \eqref{eq:DS-q-ratio} becomes
\begin{multline}
\Theta(A(\infty)-A(z_0)-\mathcal{K};H)\Theta(A(\infty)+A(z_0)+\mathcal{K}-\ii M\Delta;H)\\
{}=\Theta(2A(\infty)-\ii M\Delta;2H)\Theta(-2A(z_0)-H+\ii M\Delta;2H) \\
{}-\ee^{A(\infty)-A(z_0)}\Theta(2A(\infty)-\ii M\Delta+H;2H)\Theta(-2A(z_0)+\ii M\Delta;2H).
\label{eq:DS-denominator}
\end{multline}
According to \eqref{eq:Az0-identity} with $n_1=-1\pmod{2}$ and $n_2=1\pmod{2}$,  we may replace 
$A(z_0)$ with $-A(\infty)-\ii\pi+\tfrac{1}{2}H$ in the right-hand sides of \eqref{eq:DS-numerator}--\eqref{eq:DS-denominator} up to an ambiguity that cancels between the numerator and denominator due to the identities \eqref{eq:automorphic}.  Therefore, with the shorthand $w:=2A(\infty)$ and $\Theta_0(\cdot):=\Theta(\cdot;2H)$,
\begin{multline}
\frac{q^-(\infty;z_0,-\ii M\Delta)}{q^+(\infty;z_0,\ii M\Delta)}\\
{}=\frac{\Theta_0(w+\ii M\Delta)\Theta_0(w+\ii M\Delta-2H)+\ee^{w+\ii M\Delta-\frac{1}{2}H}\Theta_0(w+\ii M\Delta + H)\Theta_0(w+\ii M\Delta-H)}{\Theta_0(w-\ii M\Delta)\Theta_0(w+\ii M\Delta -2H)+\ee^{w-\frac{1}{2}H}\Theta_0(w-\ii M\Delta+H)\Theta_0(w+\ii M\Delta -H)}.
\end{multline}
Then using \eqref{eq:automorphic} to shift arguments of four factors by $-2H$,
\begin{multline}
\frac{q^-(\infty;z_0,-\ii M\Delta)}{q^+(\infty;z_0,\ii M\Delta)}\\
{}=\frac{\Theta_0(w+\ii M\Delta)^2+\ee^{w+\ii M\Delta +\frac{1}{2}H}\Theta_0(w+\ii M\Delta+H)^2}{\Theta_0(w-\ii M\Delta)\Theta_0(w+\ii M\Delta)+\ee^{w+\frac{1}{2}H}\Theta_0(w-\ii M\Delta + H)\Theta_0(w+\ii M\Delta + H)}.
\label{eq:q-fraction}
\end{multline}
Now observe that for $w=2A(\infty)$ we have $\mathrm{Im}(w)=-\ii\pi$.  Therefore because $H$ is real, $\Theta_0(w-\ii M\Delta)=\Theta_0(w^*+\ii M\Delta)^* = \Theta_0(w+2\pi\ii +\ii M\Delta)^*=\Theta_0(w+\ii M\Delta)^*$.  Similarly, $\Theta_0(w-\ii M\Delta+H)=\Theta_0(w+\ii M\Delta+H)^*$.  For the same reason, $\ee^{w+\frac{1}{2}H}<0$.  It follows that the denominator is real.

Therefore, the square modulus of the quantity in \eqref{eq:q-fraction} is (squaring the denominator and multiplying the numerator by its complex conjugate $\Theta_0(w-\ii M\Delta)^2-\ee^{w-\ii M\Delta+\frac{1}{2}H}\Theta_0(w-\ii M\Delta+H)^2$)
\begin{equation}
\left|\frac{q^-(\infty;z_0,-\ii M\Delta)}{q^+(\infty;z_0,\ii M\Delta)}\right|^2 =\frac{1+\ee^{w+\ii M\Delta +\frac{1}{2}H}\phi_+^2+\ee^{w-\ii M\Delta+\frac{1}{2}H}\phi_-^2+\ee^{2w+H}\phi_+^2\phi_-^2}{1+2\ee^{w+\frac{1}{2}H}\phi_+\phi_-+\ee^{2w+H}\phi_+^2\phi_-^2}
\end{equation}
where
\begin{equation}
\phi_\pm:=\frac{\Theta_0(w\pm\ii M\Delta+H)}{\Theta_0(w\pm\ii M\Delta)}.
\end{equation}
Now, since the modular parameter in $\Theta_0$ is $2H=-2\pi \mathbb{K}(1-m)/\mathbb{K}(m)$, the theta ratios $\phi_\pm$ can be expressed in terms of the Jacobi elliptic function $\mathrm{sn}(\cdot;m)$ by the following identity:
\begin{equation}
\frac{\Theta(W+H;2H)}{\Theta(W;2H)}=-\ee^{-\frac{1}{2}W}\frac{\Theta(-H;2H)}{\Theta(0;2H)}\mathrm{sn}\left(\frac{\mathbb{K}(m)}{\ii\pi}(W-\ii\pi);m\right).
\label{eq:DS-sn}
\end{equation}
Therefore, defining a constant $\JacobiSigma>0$ by
\begin{equation}
\JacobiSigma:=\ee^{\frac{1}{2}H}\frac{\Theta(-H;2H)^2}{\Theta(0;2H)^2},
\label{eq:sigma-define}
\end{equation}
we have 
\begin{multline}
\left|\frac{q^-(\infty;z_0,-\ii M\Delta)}{q^+(\infty;z_0,\ii M\Delta)}\right|^2\\
{}=\frac{\displaystyle \left(1+\JacobiSigma\mathrm{sn}^2\left(\frac{\mathbb{K}(m)}{\ii\pi}(w+\ii M\Delta-\ii\pi);m\right)\right)\left(1+\JacobiSigma\mathrm{sn}^2\left(\frac{\mathbb{K}(m)}{\ii\pi}(w-\ii M\Delta-\ii\pi);m\right)\right)}{\displaystyle \left(1+\JacobiSigma\mathrm{sn}\left(\frac{\mathbb{K}(m)}{\ii\pi}(w+\ii M\Delta-\ii\pi);m\right)\mathrm{sn}\left(\frac{\mathbb{K}(m)}{\ii\pi}(w-\ii M\Delta-\ii\pi);m\right)\right)^2}.
\end{multline}
Now we may write $w=2A(\infty)=\eta-\ii\pi$ with $\eta\in\mathbb{R}$.  Using the identity $\mathrm{sn}(u-2\mathbb{K}(m);m)=-\mathrm{sn}(u;m)$ gives 
\begin{multline}
\left|\frac{q^-(\infty;z_0,-\ii M\Delta)}{q^+(\infty;z_0,\ii M\Delta)}\right|^2\\
{}=\frac{\displaystyle \left(1+\JacobiSigma\mathrm{sn}^2\left(\frac{\mathbb{K}(m)}{\ii\pi}(\eta+\ii M\Delta);m\right)\right)\left(1+\JacobiSigma\mathrm{sn}^2\left(\frac{\mathbb{K}(m)}{\ii\pi}(\eta-\ii M\Delta);m\right)\right)}{\displaystyle \left(1+\JacobiSigma\mathrm{sn}\left(\frac{\mathbb{K}(m)}{\ii\pi}(\eta+\ii M\Delta);m\right)\mathrm{sn}\left(\frac{\mathbb{K}(m)}{\ii\pi}(\eta-\ii M\Delta);m\right)\right)^2}.
\end{multline}
Combining the addition formula \cite[Eqn.\@ 22.8.1]{DLMF} with Jacobi's imaginary transformation (see \cite[\S 22.6(iv)]{DLMF}), we have
\begin{equation}
\mathrm{sn}(-\ii \Jacobiu\pm \Jacobiv;m)=\frac{-\ii\mathrm{sc}(\Jacobiu;1-m)\mathrm{cn}(\Jacobiv;m)\mathrm{dn}(\Jacobiv;m)\pm\mathrm{sn}(\Jacobiv;m)\mathrm{nc}(\Jacobiu;1-m)\mathrm{dc}(\Jacobiu;1-m)}{1+m\mathrm{sc}^2(\Jacobiu;1-m)\mathrm{sn}^2(\Jacobiv;m)},
\end{equation}
and for real $\Jacobiu,m$ and $0<m<1$ the denominator is strictly positive.  Using also the trigonometric identities
\begin{equation}
\begin{split}
\mathrm{cn}^2(\Jacobiv;m)\mathrm{dn}^2(\Jacobiv;m)&=1-(1+m)\mathrm{sn}^2(\Jacobiv;m)+m\mathrm{sn}^4(\Jacobiv;m)\\
\mathrm{nc}^2(\Jacobiu;1-m)\mathrm{dc}^2(\Jacobiu;1-m)&=1+(1+m)\mathrm{sc}^2(\Jacobiu;1-m)+m\mathrm{sc}^4(\Jacobiu;1-m),
\end{split}
\end{equation}
we obtain
\begin{equation}
\left|\frac{q^-(\infty;z_0,-\ii M\Delta)}{q^+(\infty;z_0,\ii M\Delta)}\right|^2=\frac{A\mathrm{sn}^4(\Jacobiv;m) +B\mathrm{sn}^2(\Jacobiv;m)+C}
{(D\mathrm{sn}^2(\Jacobiv;m) + E)^2},
\end{equation}
where
\begin{equation}
\begin{gathered}
A:= (-\JacobiSigma+m\mathrm{sc}^2(\Jacobiu;1-m))^2,\quad
C:=(1-\JacobiSigma\mathrm{sc}^2(\Jacobiu;1-m))^2,\\
B:=2m\JacobiSigma\mathrm{sc}^4(\Jacobiu;1-m)+2(m+2\JacobiSigma+2m\JacobiSigma+\JacobiSigma^2)\mathrm{sc}^2(\Jacobiu;1-m)+2\JacobiSigma,\\
D:=-\JacobiSigma+m\mathrm{sc}^2(\Jacobiu;1-m),\quad
E:=1-\JacobiSigma\mathrm{sc}^2(\Jacobiu;1-m)
\end{gathered}
\end{equation}
and
\begin{equation}
\Jacobiu:=\frac{\mathbb{K}(m)}{\pi}\eta,\quad\eta:=\mathrm{Re}(2A(\infty)),\quad \Jacobiv:=\frac{\mathbb{K}(m)}{\pi}M\Delta.
\label{eq:DS-u-eta-v}
\end{equation}
Note that $A=D^2$ and $C=E^2$.  Therefore, adding and subtracting $2DE\mathrm{sn}^2(\Jacobiv;m)$ in the numerator gives
\begin{equation}
\left|\frac{q^-(\infty;z_0,-\ii M\Delta)}{q^+(\infty;z_0,\ii M\Delta)}\right|^2=1+\frac{4\JacobiSigma(1+\mathrm{sc}^2(\Jacobiu;1-m))(1+m\mathrm{sc}^2(\Jacobiu;1-m))\mathrm{sn}^2(\Jacobiv;m)}{(D\mathrm{sn}^2(\Jacobiv;m)+E)^2}.
\end{equation}
To further simplify, we can use \cite[Eqn.\@ 20.9.1]{DLMF} and \eqref{eq:sigma-define} to obtain simply $\JacobiSigma=\sqrt{m}$.  Thus, $D=-\sqrt{m}E$, and hence 
\begin{equation}
\left|\frac{q^-(\infty;z_0,-\ii M\Delta)}{q^+(\infty;z_0,\ii M\Delta)}\right|^2=1+\frac{4\sqrt{m}(1+\mathrm{sc}^2(\Jacobiu;1-m))(1+m\mathrm{sc}^2(\Jacobiu;1-m))}{(1-\sqrt{m}\mathrm{sc}^2(\Jacobiu;1-m))^2}\cdot
\frac{\mathrm{sn}^2(\Jacobiv;m)}{(1-\sqrt{m}\mathrm{sn}^2(\Jacobiv;m))^2}.
\end{equation}
Using $\mathrm{sc}(\Jacobiu;1-m)=\mathrm{sn}(\Jacobiu;1-m)/\mathrm{cn}(\Jacobiu;1-m)$ and $\mathrm{sn}^2(\Jacobiu;1-m)+\mathrm{cn}^2(\Jacobiu;1-m)=1$, this can be rewritten as
\begin{equation}
\left|\frac{q^-(\infty;z_0,-\ii M\Delta)}{q^+(\infty;z_0,\ii M\Delta)}\right|^2=1+\frac{4\sqrt{m}(1-(1-m)\mathrm{sn}^2(\Jacobiu;1-m))}{(1-(1+\sqrt{m})\mathrm{sn}^2(\Jacobiu;1-m))^2}\cdot\frac{\mathrm{sn}^2(\Jacobiv;m)}{(1-\sqrt{m}\mathrm{sn}^2(\Jacobiv;m))^2}.
\label{eq:DS-squared-modulus-u-v}
\end{equation}

To further simplify, observe that upon combining the definition of $\Jacobiu$ and $\eta$ given in \eqref{eq:DS-u-eta-v} with the definition of $A(\infty)$ given in \eqref{eq:Abel-define} and using the same affine transformation $w=(z-x)/\rho$ and fractional linear mapping \eqref{eq:DS-FLM} used to simplify $I_\mathcal{A}$, one can write $\Jacobiu$ in the form $\Jacobiu=U(\ii\tan(\frac{1}{2}\theta_\alpha))$, where
\begin{equation}
U(\zeta):=\ii\int_0^{\zeta}\frac{\dd W}{\sqrt{1-W^2}\sqrt{1-mW^2}}.
\label{eq:DS-U-of-zeta}
\end{equation}
For evaluation at $\zeta=\ii\tan(\frac{1}{2}\theta_\alpha)$ we can integrate along the positive imaginary axis taking the positive square roots of $1-W^2>1$ and $1-mW^2>1$.   More generally, $U(\zeta)$ admits analytic continuation from the positive imaginary axis to the domain $\zeta\in\mathbb{C}\setminus[-1/\sqrt{m},1/\sqrt{m}]$.  Even more generally, we may take $\zeta$ on a sheet of the Riemann surface $\mathcal{R}$ of $y^2=(1-\zeta^2)(1-m\zeta^2)$ and admit arbitrary paths of integration on $\mathcal{R}$ in which case $U(\zeta)$ is a multi-valued function of $\zeta\in\mathcal{R}$ that is well-defined modulo integer linear combinations of the periods $4\ii \mathbb{K}(m)$ and $2 \mathbb{K}(1-m)$.  Since $\mathrm{sn}^2(\Jacobiu;1-m)$ is doubly periodic in $\Jacobiu$ with periods $2\mathbb{K}(1-m)$ (because $\mathrm{sn}(\Jacobiu;1-m)$ changes sign upon $\Jacobiu\mapsto \Jacobiu+2\mathbb{K}(1-m)$) and $2\ii \mathbb{K}(m)$, it follows that upon composing $\Jacobiu\mapsto \mathrm{sn}^2(\Jacobiu;1-m)$ with $\zeta\mapsto \Jacobiu=U(\zeta)$ one obtains a single-valued function on $\mathcal{R}$, analytic except from isolated singularities coming from the poles of $\mathrm{sn}^2(\Jacobiu;1-m)$.  In other words, $\zeta\mapsto\mathrm{sn}^2(U(\zeta);1-m)$ is a meromorphic function on $\mathcal{R}$ for each $m\in (0,1)$.  

In fact, the restriction of this composition to one sheet of $\mathcal{R}$ defined by the analytic continuation of the integral formula \eqref{eq:DS-U-of-zeta} from positive imaginary $\zeta$ to the domain $\zeta\in\mathbb{C}\setminus[-1/\sqrt{m},1/\sqrt{m}]$ is meromorphic on the $\zeta$-sphere and hence a rational function of $\zeta\in\mathbb{C}$.  To see this, we let $U_\pm(\zeta)$ denote the boundary values on $\zeta\in\mathbb{R}$ from the half planes $\mathbb{C}_\pm$.  Then, since the integrand is integrable at $\zeta=\infty$, we have $U_+(\zeta)=U_-(\zeta)$ for $\zeta<-1/\sqrt{m}$ and $\zeta>1/\sqrt{m}$.  If instead $-1<\zeta<1$, then $U_+(\zeta)$ is calculated directly using the formula \eqref{eq:DS-U-of-zeta} while to compute $U_-(\zeta)$ we have to integrate from $W=\ii 0$ to $W=\zeta-\ii 0$ along a path in the domain $\mathbb{C}\setminus [-1/\sqrt{m},1/\sqrt{m}]$ yielding the identity
\begin{equation}
U_+(\zeta)-U_-(\zeta)=-2\int_1^{1/\sqrt{m}}\frac{\dd W}{\sqrt{W^2-1}\sqrt{1-mW^2}}=-2\mathbb{K}(1-m),\quad -1<\zeta<1.
\end{equation}
Therefore, if $r(\zeta):=\mathrm{sn}^2(U(\zeta);1-m)$, then
\begin{equation}
\begin{split}
r_+(\zeta)=\mathrm{sn}^2(U_+(\zeta);1-m)&=\mathrm{sn}^2(U_-(\zeta)-2\mathbb{K}(1-m);1-m)\\ &=\mathrm{sn}^2(U_-(\zeta);1-m)=r_-(\zeta),\quad
\zeta\in (-1,1).
\end{split}
\end{equation}
By similar calculations, one finds that
\begin{equation}
\begin{split}
U_+(\zeta)+U_-(\zeta)&=2\ii\int_0^1\frac{\dd W}{\sqrt{1-W^2}\sqrt{1-mW^2}} + 2\int_1^{1/\sqrt{m}}\frac{\dd W}{\sqrt{W^2-1}\sqrt{1-mW^2}} \\ &= 2\ii \mathbb{K}(m)+2\mathbb{K}(1-m),\quad 1<\zeta<1/\sqrt{m},
\end{split}
\end{equation}
and therefore, using evenness of $\mathrm{sn}^2(\cdot;1-m)$,
\begin{equation}
\begin{split}
r_+(\zeta)=\mathrm{sn}^2(U_+(\zeta);1-m)&=\mathrm{sn}^2(-U_-(\zeta)+2\ii\mathbb{K}(m)+2\mathbb{K}(1-m);1-m) \\
&= \mathrm{sn}^2(-U_-(\zeta);1-m)\\
&=\mathrm{sn}^2(U_-(\zeta);1-m)=r_-(\zeta),\quad \zeta\in (1,1/\sqrt{m}).
\end{split}
\end{equation}
Likewise,
\begin{equation}
U_+(\zeta)+U_-(\zeta)=-2\ii\mathbb{K}(m)+2\mathbb{K}(1-m),\quad -1/\sqrt{m}<\zeta<-1
\end{equation}
implying that
\begin{equation}
r_+(\zeta)=r_-(\zeta),\quad \zeta\in (-1/\sqrt{m},-1).
\end{equation}

To fully characterize the rational function $r(\zeta)$ it remains to determine its partial fraction expansion.  The poles of $
\Jacobiu\mapsto \mathrm{sn}^2(\Jacobiu;1-m)$ are at the translations by lattice periods $2\mathbb{K}(1-m)$ and $2\ii\mathbb{K}(m)$ of the point $u=\ii\mathbb{K}(m)$, where we have the Laurent expansion
\begin{equation}
\mathrm{sn}^2(\Jacobiu;1-m)=\frac{1}{1-m}\cdot\frac{1}{(\Jacobiu-\ii\mathbb{K}(m))^2} + O(1),\quad \Jacobiu\to\ii\mathbb{K}(m).
\end{equation}
The only points mapped by $\zeta\mapsto \Jacobiu=U(\zeta)$ to poles of $\mathrm{sn}^2(\Jacobiu;1-m)$ are the points $\zeta=\pm 1$.  Since $r(-\zeta)=r(\zeta)$ it is enough to work out the Laurent expansion of $r(\zeta)$ near $\zeta=1$.  Expanding $U_+(\zeta)$ for $\zeta<1$ with $1-\zeta$ small yields
\begin{equation}
U_+(\zeta)=\ii\mathbb{K}(m)-\ii\sqrt{\frac{2}{1-m}}\cdot\sqrt{1-\zeta}+O(1-\zeta).
\end{equation}
Therefore, since $r(\zeta):=\mathrm{sn}^2(U(\zeta);1-m)$ is single-valued near $\zeta=1$,
\begin{equation}
r(\zeta)=\frac{1}{2(\zeta-1)} +O(1),\quad\zeta\to 1.
\end{equation}
Since $r(-\zeta)=r(\zeta)$, we then obtain the partial fraction expansion
\begin{equation}
r(\zeta)=\frac{1}{2(\zeta-1)} + \frac{1}{2(-\zeta-1)} + C
\end{equation}
for some constant $C$.  Using the obvious identity $r(0)=0$ then shows that $C=1$.  Consequently
the rational function at hand is exactly
\begin{equation}
r(\zeta):=\mathrm{sn}^2(U(\zeta);1-m) = \frac{\zeta^2}{\zeta^2-1}.
\end{equation}
(The result does not depend on the elliptic parameter $1-m$.)  Since the value of $\Jacobiu$ defined in \eqref{eq:DS-u-eta-v} requires taking $\zeta=\ii\tan(\frac{1}{2}\theta_\alpha)$, we have proved the identity
\begin{equation}
\mathrm{sn}^2(\Jacobiu;1-m)=\frac{\tan^2(\frac{1}{2}\theta_\alpha)}{\tan^2(\frac{1}{2}\theta_\alpha)+1}=\frac{t_\alpha^2}{1+t_\alpha^2},
\end{equation}
where we have introduced the shorthand notation
\begin{equation}
t_\alpha:=\tan(\tfrac{1}{2}\theta_\alpha)\quad\text{and}\quad t_\beta:=\tan(\tfrac{1}{2}\theta_\beta).
\end{equation}
Using this result in \eqref{eq:DS-squared-modulus-u-v} along with $m=\cot^2(\frac{1}{2}\theta_\alpha)\tan^2(\frac{1}{2}\theta_\beta)=t_\alpha^{-2}t_\beta^2$ gives
\begin{equation}
\frac{4\sqrt{m}(1-(1-m)\mathrm{sn}^2(\Jacobiu;1-m))}{(1-(1+\sqrt{m})\mathrm{sn}^2(\Jacobiu;1-m))^2}=
\frac{4t_\alpha^{-1}t_\beta(1+t_\alpha^2)(1+t_\beta^2)}{(1-t_\alpha t_\beta)^2}.
\end{equation}
Comparing with the trigonometric identity
\begin{equation}
\frac{4(1-\sqrt{m})^2\mathrm{Im}(\alpha)\mathrm{Im}(\beta)}{(\mathrm{Im}(\alpha)-\mathrm{Im}(\beta))^2}=\frac{4t_\alpha^{-1}t_\beta(1+t_\alpha^2)(1+t_\beta^2)}{(1-t_\alpha t_\beta)^2},
\end{equation}
we can write \eqref{eq:DS-squared-modulus-u-v} in the form
\begin{equation}
\left|\frac{q^-(\infty;z_0,-\ii M\Delta)}{q^+(\infty;z_0,\ii M\Delta)}\right|^2=
1+\frac{4\mathrm{Im}(\alpha)\mathrm{Im}(\beta)}{(\mathrm{Im}(\alpha)-\mathrm{Im}(\beta))^2}\cdot\frac{(1-\sqrt{m})^2\mathrm{sn}^2(\Jacobiv;m)}{(1-\sqrt{m}\mathrm{sn}^2(\Jacobiv;m))^2}.
\end{equation}
Going back to \eqref{eq:DS-Psi-breve-define}, we have shown that
\begin{equation}
|\breve{\Psi}(\chi,\tau;M)|^2=(\mathrm{Im}(\alpha)-\mathrm{Im}(\beta))^2+4\mathrm{Im}(\alpha)\mathrm{Im}(\beta)\cdot\frac{(1-\sqrt{m})^2\mathrm{sn}^2(\Jacobiv;m)}{(1-\sqrt{m}\mathrm{sn}^2(\Jacobiv;m))^2}.
\end{equation}

Some calculus shows that as $\mathrm{sn}^2(\Jacobiv;m)$ varies in $[0,1]$, so also
\begin{equation}
\frac{(1-\sqrt{m})^2\mathrm{sn}^2(\Jacobiv;m)}{(1-\sqrt{m}\mathrm{sn}^2(\Jacobiv;m))^2}\in [0,1],\quad 0<m<1.
\end{equation}
Therefore, as $\Jacobiv$ varies for fixed $(\alpha,\beta)$,
\begin{equation}
(\mathrm{Im}(\alpha)-\mathrm{Im}(\beta))^2\le |\breve{\Psi}(\chi,\tau;M)|^2\le (\mathrm{Im}(\alpha)+\mathrm{Im}(\beta))^2.
\label{eq:DS-range-of-brevePsi}
\end{equation}
Actually, a simpler formula for $|\breve{\Psi}(\chi,\tau;M)|^2$ can be obtained via a Landen transformation in terms of a less-natural elliptic parameter 
\begin{equation}
m_1:=\frac{4\sqrt{m}}{(1+\sqrt{m})^2}=\frac{4\cot(\frac{1}{2}\theta_\alpha)\tan(\frac{1}{2}\theta_\beta)}{(1+\cot(\frac{1}{2}\theta_\alpha)\tan(\frac{1}{2}\theta_\beta))^2}=\frac{\sin(\theta_\alpha)\sin(\theta_\beta)}{\sin^2(\frac{1}{2}(\theta_\alpha+\theta_\beta))},
\label{eq:DS-m1-define}
\end{equation}
and variable
\begin{equation}
\Jacobiv_1:=\frac{\mathbb{K}(m_1)}{\pi}(M\Delta+\pi).
\label{eq:DS-v1-define}
\end{equation}
Note that, like $m$, $m_1$ varies in $[0,1]$ from $m_1=1$ at $\chi=\chi_\mathrm{c}(\tau)$ (because there $\alpha=\beta=\critpt$ with $\mathrm{Im}(\critpt)>0$) to $m_1=0$ in the limit $\chi\to+\infty$ (because $\theta_\alpha\to 0$ while $\theta_\beta\to\pi$).
The simpler formula for $|\breve{\Psi}(\chi,\tau;M)|^2$ reads
\begin{equation}
|\breve{\Psi}(\chi,\tau;M)|^2=(\mathrm{Im}(\alpha)+\mathrm{Im}(\beta))^2-4\mathrm{Im}(\alpha)\mathrm{Im}(\beta)\mathrm{sn}^2(\Jacobiv_1;m_1),
\label{eq:DS-Breve-Psi-alt}
\end{equation}
which again obviously varies in the range indicated in \eqref{eq:DS-range-of-brevePsi}.  We provide an alternative derivation of this formula directly from \eqref{eq:DS-q-ratio} in Appendix~\ref{a:Landen}.

\subsection{Proof of Corollary~\ref{cor:L2-convergence}}
\label{sec:L2-convergence}
\begin{proof}
We compute
\begin{multline}
\|M\Psi(M^2\diamond,M^3\tau;\mathbf{G}(a,b))-\ee^{-\ii\arg(ab)}\breve{\Psi}(\diamond,\tau;M)\|_{L^2(\mathbb{R})}^2 = \|M\Psi(M^2\diamond,M^3\tau;\mathbf{G}(a,b)\|_{L^2(\mathbb{R})}^2 \\{}+ \|\breve{\Psi}(\diamond,\tau;M)\|_{L^2(\mathbb{R})}^2 - 2\mathrm{Re}\left(\ee^{-\ii\arg(ab)}\int_\mathbb{R}M\Psi(M^2\chi,M^3\tau;\mathbf{G}(a,b))^*\breve{\Psi}(\chi,\tau;M)\,\dd\chi\right).
\end{multline}
But, by Theorem~\ref{t:L2-norm} and \eqref{eq:Intro-approx-L2}, we can write this in the form
\begin{multline}
\|M\Psi(M^2\diamond,M^3\tau;\mathbf{G}(a,b))-\ee^{-\ii\arg(ab)}\breve{\Psi}(\diamond,\tau;M)\|_{L^2(\mathbb{R})}^2 =\\2\mathrm{Re}\left(\ee^{-\ii\arg(ab)}\int_\mathbb{R}\left(\ee^{-\ii\arg(ab)}\breve{\Psi}(\chi,\tau;M)-M\Psi(M^2\chi,M^3\tau;\mathbf{G}(a,b))\right)^*\breve{\Psi}(\chi,\tau;M)\,\dd\chi\right) \\
{}+O(M^{-1}).
\end{multline}
Given $\delta>0$, there exists $M_1(\delta)$ such that $M>M_1(\delta)$ implies that the error term is less than $\frac{1}{3}\delta^2$.  
Also, by \eqref{eq:Intro-square-modulus}, we have $|\breve{\Psi}(\chi,\tau;M)|\le \mathrm{Im}(\alpha(\chi,\tau))+\mathrm{Im}(\beta(\chi,\tau))$ for $\chi>\chi_\mathrm{c}(\tau)$, an upper bound that is independent of $M$ and that lies in $L^2(\chi_\mathrm{c}(\tau),+\infty)$ because $\alpha$ and $\beta$ are continuous down to $\chi=\chi_\mathrm{c}(\tau)$ with common value $\critpt$, and as shown in the proof of Lemma~\ref{lem:absolute-integrability}, $\mathrm{Im}(\alpha(\chi,\tau))=O(\chi^{-\frac{3}{4}})$ and $\mathrm{Im}(\beta(\chi,\tau))=O(\chi^{-\frac{3}{4}})$ as $\chi\to+\infty$.  Therefore, when $M>M_1(\delta)$,
\begin{multline}
\|M\Psi(M^2\diamond,M^3\tau;\mathbf{G}(a,b))-\ee^{-\ii\arg(ab)}\breve{\Psi}(\diamond,\tau;M)\|_{L^2(\mathbb{R})}^2\le\\
2\int_{\chi_\mathrm{c}(\tau)}^{+\infty}\left|M\Psi(M^2\chi,M^3\tau;\mathbf{G}(a,b))-\ee^{-\ii\arg(ab)}\breve{\Psi}(\chi,\tau;M)\right|\left(\mathrm{Im}(\alpha(\chi,\tau))+\mathrm{Im}(\beta(\chi,\tau))\right)\,\dd\chi + \frac{1}{3}\delta^2.
\end{multline}
Now let $f(\diamond,\tau)\in C_0^\infty(\chi_\mathrm{c}(\tau),+\infty)$ be a test function with compact support in $\chi>\chi_\mathrm{c}(\tau)$.  By the Cauchy-Schwarz and Minkowski inequalities,
\begin{multline}
2\int_{\chi_\mathrm{c}(\tau)}^{+\infty}\left|M\Psi(M^2\chi,M^3\tau;\mathbf{G}(a,b))-\ee^{-\ii\arg(ab)}\breve{\Psi}(\chi,\tau;M)\right|\left(\mathrm{Im}(\alpha(\chi,\tau))+\mathrm{Im}(\beta(\chi,\tau))\right)\,\dd\chi\le\\
2\int_{\mathrm{spt}(f)}\left|M\Psi(M^2\chi,M^3\tau;\mathbf{G}(a,b))-\ee^{-\ii\arg(ab)}\breve{\Psi}(\chi,\tau;M)\right| |f(\chi,\tau)|\,\dd\chi \\
{}+2\left(\|M\Psi(M^2\diamond,M^3\tau;\mathbf{G}(a,b))\|_{L^2(\mathbb{R})}+\|\breve{\Psi}(\diamond,\tau;M)\|_{L^2(\chi_\mathrm{c}(\tau),+\infty)}\right)\\
\cdot\|\mathrm{Im}(\alpha(\diamond,\tau))+\mathrm{Im}(\beta(\diamond,\tau))-f(\diamond,\tau)\|_{L^2(\chi_\mathrm{c}(\tau),+\infty)}.
\end{multline}
The quantity in 
parentheses
is equal to $2\sqrt{8} + O(M^{-1})$ according to Theorem~\ref{t:L2-norm} and \eqref{eq:Intro-approx-L2}.  
Furthermore,
by density, for each $\tau\in\mathbb{R}$, the test function $f(\chi,\tau)$ can be chosen to approximate the $M$-independent quantity $\mathrm{Im}(\alpha(\diamond,\tau))+\mathrm{Im}(\beta(\diamond,\tau))$ to arbitrary accuracy in $L^2(\chi_\mathrm{c}(\tau),+\infty)$.  Therefore, there exists some $M_2(\delta)$ such that $M>M_2(\delta)$ and suitable $f_\delta(\diamond,\tau)\in C_0^\infty(\chi_\mathrm{c}(\tau),+\infty)$ implies that    
\begin{multline}
2\int_{\chi_\mathrm{c}(\tau)}^{+\infty}\left|M\Psi(M^2\chi,M^3\tau;\mathbf{G}(a,b))-\ee^{-\ii\arg(ab)}\breve{\Psi}(\chi,\tau;M)\right|\left(\mathrm{Im}(\alpha(\chi,\tau))+\mathrm{Im}(\beta(\chi,\tau))\right)\,\dd\chi\le\\
2\int_{\mathrm{spt}(f_\delta)}\left|M\Psi(M^2\chi,M^3\tau;\mathbf{G}(a,b))-\ee^{-\ii\arg(ab)}\breve{\Psi}(\chi,\tau;M)\right| |f_\delta(\chi,\tau)|\,\dd\chi +\frac{1}{3}\delta^2.
\end{multline}
Finally, using the locally-uniform convergence on compact subsets of $\chi>\chi_\mathrm{c}(\tau)$ such as $\mathrm{spt}(f_\delta)$ afforded by Theorem~\ref{t:DS} and the fact that as an element of $C_0^\infty(\chi_\mathrm{c}(\tau),+\infty)$, $f_\delta(\diamond,\tau)\in L^1(\chi_\mathrm{c}(\tau),+\infty)$ with norm independent of $M$, there exists $M_3(\delta)$ such that $M>M_3(\delta)$ implies that
\begin{equation}
2\int_{\mathrm{spt}(f_\delta)}\left|M\Psi(M^2\chi,M^3\tau;\mathbf{G}(a,b))-\ee^{-\ii\arg(ab)}\breve{\Psi}(\chi,\tau;M)\right| |f_\delta(\chi,\tau)|\,\dd\chi \le\frac{1}{3}\delta^2.
\end{equation}
Combining the results shows that $M>\max\{M_1(\delta),M_2(\delta),M_3(\delta)\}$ implies that 
\begin{equation}
\|M\Psi(M^2\diamond,M^3\tau;\mathbf{G}(a,b))-\ee^{-\ii\arg(ab)}\breve{\Psi}(\diamond,\tau;M)\|_{L^2(\mathbb{R})}\le\delta, 
\end{equation}
so the proof is finished.
\end{proof}

\section{Asymptotic behavior of $\Psi(X,T;\mathbf{G}(a,b))$ in the limit $a\to 0$:  the case $\chi\approx\chi_\mathrm{c}(\tau)$}
\label{s:edge}

In this section we reconsider \rhref{rhp:S} in the limit $M\to+\infty$, treating the case $\chi\approx \chi_\mathrm{c}(\tau)$ for given $\tau$, i.e., we study the transition between the $\chi <\chi_\mathrm{c}(\tau)$ and $\chi > \chi_\mathrm{c}(\tau)$ regimes.
Recalling the jump matrix $\mathbf{V}^\mathbf{S}(z;\chi,\tau,M)$ from \rhref{rhp:S} defined in \eqref{eq:S-jump}, we are now interested in the case that the region of the $z$-plane on which the inequalities $-1<\mathrm{Re}(\ii\phase(z;\chi,\tau))<1$ hold resembles the third pane in Figure~\ref{f:vartheta-composite}.  In this situation, $\mathrm{Re}(\ii\phase(z;\chi,\tau))\approx -1$   for $z\approx \critpt(\chi,\tau)$, the complex critical point of $z\mapsto\phase(z;\chi,\tau)$ in the upper half-plane near which the contour $\Gamma$ must pass.  Therefore, when $\chi\approx \chi_\mathrm{c}(\tau)$ and $z\approx \critpt(\chi,\tau)$, the exponential factor $\ee^{-2\ii M(\phase(z;\chi,\tau)-\ii)}$ in the $(1,2)$-element of $\mathbf{V}^\mathbf{S}(z;\chi,\tau,M)$ fails to be negligible; however the factor $\ee^{2\ii M(\phase(z;\chi,\tau)+\ii)}$ present in the other off-diagonal element is then automatically exponentially small near the same point $z$ in the same situation, because $\mathrm{Re}(\ii(\phase(z;\chi,\tau)+\ii)=\mathrm{Re}(\ii(\phase(z;\chi,\tau)-\ii)-2\approx -2<0$.

For the remainder of this section we assume that the Jordan curve $\Gamma$ passes over the complex-conjugate critical points $z=\critpt,\critpt^*$ and study how to account for the fact that the $(1,2)$-entry of the jump matrix $\mathbf{V}^\mathbf{S}(z;\chi,\tau,M)$ may become large for $z\approx\critpt(\chi,\tau)$ in $\Gamma$ as $\chi$ increases beyond $\chi_\mathrm{c}(\tau)$.  To measure the size,   
recall the quantity $d(\chi,\tau)$ defined in \eqref{d-def-intro} 
so that $2Md(\chi,\tau)$ is the value of the exponent in the $(1,2)$-entry of \eqref{eq:S-jump} at the critical point $z=\critpt(\chi,\tau)$.
Then, the boundary curve $\chi=\chi_\mathrm{c}(\tau)$ is defined by $\Re(2d(\chi,\tau))=0$, and the region $\chi>\chi_\mathrm{c}(\tau)$ corresponds to $\Re(2d(\chi,\tau))>0$.  We will show that it is possible to analyze \rhref{rhp:S} in the limit $M\to+\infty$ without the use of any $g$-function provided that for some $K_\pm>0$, $(\chi,\tau)$ lies in the region $\mathcal{S}$ defined in \eqref{d-size}.
The constant $K_->0$ in \eqref{d-size} should be taken sufficiently small to guarantee existence of the complex critical points $\critpt,\critpt^*$, and we also assume (see Section~\ref{s:transition-error}) that $K_-<2$.
To carry out this analysis, we will build a suitable inner parametrix in a small disk $D_\critpt$ of fixed sufficiently small radius (see Section~\ref{s:transition-error}) centered at $z=\critpt$, and deal with the neighborhood of $z=\critpt^*$ using Schwarz reflection symmetry \eqref{eq:S-Schwarz}.  Since \eqref{d-size} implies that $\Re(2d(\chi,\tau))\to 0$ as $M\to +\infty$ if $(\chi,\tau)\in\mathcal{S}$ with $\Re(2d(\chi,\tau))>0$, then such points $(\chi,\tau)$ must tend to the boundary curve.  Therefore, there exists a constant $c>0$ such that $(\chi,\tau)\in \mathcal{S}$ defined in \eqref{d-size} implies the uniform estimate ($\|\diamond\|$ denotes an arbitrary norm on $2\times 2$ matrices)
\begin{equation}
\sup_{z\in\Gamma\setminus \overline{D_\critpt\cup D_\critpt^*}}\left\|\mathbf{V}^\mathbf{S}(z;\chi,\tau,M)-\mathbb{I}\right\|=O(\ee^{-cM}),\quad M\to+\infty,
\label{eq:transitional-uniform-outside}
\end{equation}
but due to contributions from $\Gamma\cap (D_\critpt\cup D_\critpt^*)$ we expect to obtain a new asymptotic description of $\Psi$ in this transitional regime.

\subsection{Inner parametrices}
\label{s:local-param}
Under the condition $(\chi,\tau)\in\mathcal{S}$ (see \eqref{d-size}) we cannot fully extend the estimate \eqref{eq:transitional-uniform-outside} to $z\in\Gamma\cap D_\critpt$, however we may write
\begin{equation}
\mathbf{V}^{\mathbf{S}}(z ;\chi,\tau,M)
=\mathbb{I} + \begin{bmatrix} O(\ee^{-c M}) & -\ee^{-2\ii M(\phase(z;\chi,\tau)-\ii)}\\
O(\ee^{-c M})  & O(\ee^{-c M}) \end{bmatrix},\quad z\in\Gamma\cap D_\critpt,
\label{eq:VS-in-1}
\end{equation}
where the error terms are uniform with respect to $z\in\Gamma\cap D_\critpt$ and $(\chi,\tau)\in\mathcal{S}$.  We do not further estimate $V_{12}^\mathbf{S}(z;\chi,\tau,M)$ because it may be large for $z\in\Gamma\cap D_\critpt$ under the condition $(\chi,\tau)\in\mathcal{S}$, so it should be retained and accounted for.

For $z\in D_\critpt$ we write
\begin{equation}
- 2 \ii \left( \phase(z;\chi,\tau) - \ii \right) = 2 d(\chi,\tau) - \varphi(z;\chi,\tau)^2.
\label{conformal-phi-def}
\end{equation}
This defines two opposite $M$-independent conformal maps $\varphi(z)=\varphi(z;\chi,\tau)$ from $D_\critpt$ to a neighborhood of the origin $\varphi=0$ since $z=\critpt(\chi,\tau)$ is a simple critical point of $\phase(z;\chi,\tau)$ and since $d(\chi,\tau)$ is defined by \eqref{d-def-intro}.   Note that $\varphi(\critpt(\chi,\tau);\chi,\tau)=0$, and denoting $\conformalprime(\chi,\tau):=\varphi'(\critpt(\chi,\tau);\chi,\tau)$, differentiating \eqref{conformal-phi-def} twice at $z=\critpt(\chi,\tau)$ gives 
\begin{equation}
\conformalprime(\chi,\tau):=\varphi'(\critpt(\chi,\tau);\chi,\tau)\implies \conformalprime(\chi,\tau)^2=\ii\phase''(\critpt(\chi,\tau);\chi,\tau)\neq 0.
\label{eq:conf-map-deriv}
\end{equation}
We assume that the smooth oriented contour $\Gamma$ passes over $z=\critpt$ in such a way that $\varphi$ maps $\Gamma\cap D_\critpt$ to a real interval containing $\varphi=0$, and then we resolve the sign ambiguity by insisting that $z\mapsto \varphi(z)$ is increasing along $\Gamma\cap D_\critpt$. 
This makes $\conformalprime$ a well-defined continuous complex-valued function of $(\chi,\tau)$.  In the rescaled coordinate
\begin{equation}
\zeta:= M^{\frac{1}{2}} \varphi(z;\chi,\tau),
\label{zeta-def}
\end{equation}
the jump matrix is expressed in the limit $M\to+\infty$ as
\begin{equation}
\mathbf{V}^{\mathbf{S}}(z ;\chi,\tau,M)=
\ee^{Md(\chi,\tau)\sigma_3} \ii^{\sigma_3} \begin{bmatrix}1 + O(\ee^{-c M}) & \ee^{-\zeta^2}\\
O(\ee^{-c M}M^{K_+})  & 1 + O(\ee^{-c M}) \end{bmatrix} \ii^{-\sigma_3}\ee^{-Md(\chi,\tau)\sigma_3},
\quad z\in\Gamma\cap D_\critpt.
\label{jump-zeta}
\end{equation}
The condition $z\in\Gamma\cap D_\critpt$ implies that $\zeta$ lies in a large interval containing $\zeta=0$ with endpoints proportional to $M^\frac{1}{2}$.  Neglecting the errors, we see that the central factor is the jump matrix for the following 
Riemann-Hilbert problem due to Fokas, Its, and Kitaev \cite{FokasIK1991} 
(see also \cite[Section 3]{DeiftKMVZ1999} for a context more relevant to our work).
\begin{rhp}[Hermite polynomials]
\label{rhp:Hermite}
Given $n\in\mathbb{Z}_{\ge 0}$, find a $2\times 2$ matrix valued function $\mathbf{H}^{[n]}(\zeta)$ with the following properties:
\begin{itemize}
\item\textbf{Analyticity:} $\mathbf{H}^{[n]}(\zeta)$ is analytic for $\zeta\in\mathbb{C}\setminus\mathbb{R}$, taking continuous boundary values on $\mathbb{R}$.
\item\textbf{Jump conditions:}  The boundary values on the jump contour $\mathbb{R}$ oriented from $\zeta=-\infty$ to $\zeta=+\infty$ are related as follows:
\begin{equation}
\mathbf{H}^{[n]}_+(\zeta)=\mathbf{H}^{[n]}_-(\zeta)\begin{bmatrix}1 &\ee^{-\zeta^2}\\ 0 & 1\end{bmatrix},\quad \zeta\in\mathbb{R}.
\end{equation}
\item\textbf{Normalization:}  $\mathbf{H}^{[n]}(\zeta) \zeta^{-n \sigma_3} = \mathbb{I} + O(\zeta^{-1})$ as $\zeta\to\infty$.
\end{itemize}
\end{rhp}
This Riemann-Hilbert problem has a unique solution given in terms of the Hermite polynomials $\{P_{n}(\zeta)\}_{n=0}^{\infty}$, which are
determined by the following properties: $P_n(\zeta)$ is a polynomial of degree $n$ with a positive leading coefficient $\gamma_n>0$ as defined in \eqref{gamma-n-intro} and $\{P_{n}(\zeta)\}_{n=0}^{\infty}$ are orthonormal with respect to the measure $\ee^{-\zeta^2}\dd \zeta$, $\zeta\in\mathbb{R}$, that is,
\begin{equation}
\int_{-\infty}^{+\infty} P_j(\zeta) P_k(\zeta) \ee^{-\zeta^2}\dd \zeta = \delta_{jk},
\end{equation}
where $\delta_{jk}$ is the Kronecker delta.
Let $\pi_n(\zeta):= \gamma_{n}^{-1} P_{n}(\zeta)=\zeta^n+\cdots$ denote the monic Hermite polynomial of degree $n$.
Then with the convention that $\gamma_{-1}=0$ the formula
\begin{equation}
\mathbf{H}^{[n]}(\zeta) = \begin{bmatrix} \pi_n(\zeta) &\displaystyle \frac{1}{2\pi \ii} \int_{\mathbb{R}} \frac{\pi_n(s) \ee^{-s^2} \dd s}{s-\zeta} \\ -2\pi \ii \gamma_{n-1}^{2}\pi_{n-1}(\zeta) & -\gamma_{n-1}^{2}  \displaystyle \int_{\mathbb{R}} \frac{\pi_{n-1}(s) \ee^{-s^2} \dd s}{s-\zeta}  \end{bmatrix},\quad n=0,1,2,3,\ldots,
\end{equation}
provides the solution of \rhref{rhp:Hermite}.
Moreover, $\det(\mathbf{H}^{[n]}(\zeta))\equiv 1$.
See \cite[Chapter 3]{Deift2008} or \cite[Theorem 3.1]{DeiftKMVZ1999} for more details.
The matrix $\mathbf{H}^{[n]}(\zeta)$ also has the Laurent expansion
\begin{equation}
\mathbf{H}^{[n]}(\zeta) \zeta^{-n \sigma_3} = \mathbb{I} + \mathbf{H}_{-1}^{[n]}\zeta^{-1} + \mathbf{H}_{-2}^{[n]}\zeta^{-2} + O(\zeta^{-3}),\qquad \zeta\to\infty,
\label{H-expansion}
\end{equation}
where 
\begin{align}
\mathbf{H}_{-1}^{[n]} := 
\begin{bmatrix}
0 & \dfrac{-1}{2\pi \ii \gamma_n^2} \\ {-2\pi \ii \gamma_{n-1}^2} & 0
 \end{bmatrix},\quad n\in\mathbb{Z}_{\ge 0}.
\label{H-error}
\end{align}
The coefficient $\mathbf{H}_{-2}^{[n]}$ in \eqref{H-expansion} is a diagonal matrix due to the property
$\pi_n(-\zeta) = (-1)^n \pi_n(\zeta)$ for $n\in\mathbb{Z}_{\ge 0}$,
and in fact $\mathbf{H}_{-2}^{[0]}=\mathbf{0}$, see \cite[Table 18.6.1]{DLMF}.
From this point on, we indicate the dependence of various quantities on the chosen integer $n$ via subscripts.

For a given value of $n\in\mathbb{Z}_{\ge 0}$ and holomorphic matrix-valued function $\mathbf{Y}_n(z)$ both to be determined, we define an inner parametrix for $z\in D_\critpt$ by
\begin{equation}
\breve{\mathbf{S}}_n^{\mathrm{in}}(z;\chi,\tau,M) := \mathbf{Y}_n(z) M^{-\frac{1}{2}n \sigma_3}\ee^{M d(\chi,\tau)\sigma_3} \ii^{\sigma_3} \mathbf{H}^{[n]}(M^{\frac{1}{2}}\varphi(z;\chi,\tau))  \ii^{-\sigma_3}  \ee^{ - M d(\chi,\tau)\sigma_3}.
\label{S-xi}
\end{equation}
This is analytic for $z\in D_\critpt\setminus\Gamma$ with jump matrix
\begin{equation}
\mathbf{V}^{\breve{\mathbf{S}}_n^{\mathrm{in}}}(z;\chi,\tau,M):= \begin{bmatrix} 1 & - \ee^{2M d(\chi,\tau) -M\varphi(z;\chi,\tau)^2} \\ 0 & 1 \end{bmatrix},\quad z\in\Gamma\cap D_\critpt,
\label{S-xi-jump}
\end{equation}
which coincides with the matrix in \eqref{jump-zeta} after neglecting the error terms.

A second inner parametrix for $\mathbf{S}(z;\chi,\tau,M)$ is then defined in the reflected disk $D_\critpt^*$ centered at $z=\critpt^*$ by Schwarz reflection, see \eqref{S-global} below.  
See \cite{BertolaT2010}, \cite{BuckinghamM2015}, and \cite{ClaeysG2010} for other other applications of \rhref{rhp:Hermite} to the solution of nonlinear equations in transitional regimes.

\subsection{Outer parametrix}
Given $n\in\mathbb{Z}_{\ge 0}$, we need an outer parametrix defined for $z\in\mathbb{C}\setminus\overline{D_\critpt \cup D_{\critpt}^*}$ that matches with the $\zeta^{n\sigma_3}$ behavior of the local parametrices.
Thanks to the estimate \eqref{eq:transitional-uniform-outside},
the outer parametrix need not satisfy any sort of nontrivial jump condition on $\Gamma$ outside the disks $D_\critpt\cup D_\critpt^*$. Therefore, the only flexibility we have for an outer parametrix is to allow it to have poles in the disks. We introduce an $M$-independent outer parametrix by
\begin{equation}
\breve{\mathbf{S}}_n^{\mathrm{out}}(z;\chi,\tau):=\left( \frac{z-\critpt(\chi,\tau)}{z-\critpt(\chi,\tau)^*}\right)^{n\sigma_3},
\label{S-dot-out}
\end{equation}
which satisfies the Schwarz symmetry $\breve{\mathbf{S}}_n^{\mathrm{out}}(z;\chi,\tau)=\sigma_2\breve{\mathbf{S}}_n^{\mathrm{out}}(z^*;\chi,\tau)^*\sigma_2$.  Here, $n\in\mathbb{Z}_{\ge 0}$ has the same value as in \eqref{S-xi} --- the degree of the Hermite polynomial.
For $z\in D_\critpt$, the formula \eqref{S-dot-out} can be expressed as $\breve{\mathbf{S}}_n^{\mathrm{out}}(z;\chi,\tau)=\mathbf{Y}_n(z)\varphi(z;\chi,\tau)^{n\sigma_3}$ with
\begin{equation}
\mathbf{Y}_n(z):=y(z;\chi,\tau)^{n\sigma_3},\quad y(z;\chi,\tau):=\frac{z-\critpt(\chi,\tau)}{\varphi(z;\chi,\tau)(z-\critpt(\chi,\tau)^*)}.
\label{f-def}
\end{equation}
Note that $z\mapsto y(z;\chi,\tau)$ is analytic and nonvanishing for $z\in \overline{D_{\critpt}}$.  In particular $z=\critpt(\chi,\tau)$ is a removable singularity:
\begin{equation}
y(\critpt(\chi,\tau);\chi,\tau)=\frac{1}{2\ii\Im(\critpt(\chi,\tau))\conformalprime(\chi,\tau)}.
\label{eq:y-critpt}
\end{equation}
This defines the holomorphic prefactor $\mathbf{Y}_n(z)$ in \eqref{S-xi}.  With this definition, the mismatch of the inner and outer parametrices 
for $z\in \partial D_\critpt$ can be written in the form
\begin{equation}
\begin{aligned}
\breve{\mathbf{S}}_n^{\mathrm{in}}(z;\chi,\tau,M) \breve{\mathbf{S}}_n^{\mathrm{out}}(z;\chi,\tau)^{-1} 
&=y(z;\chi,\tau)^{n\sigma_3} \ii^{\sigma_3}  M^{-\frac{1}{2}n \sigma_3}\ee^{M d(\chi,\tau)\sigma_3} \\
&\quad\cdot \mathbf{H}^{[n]}(M^{\frac{1}{2}}\varphi(z;\chi,\tau)) \left(M^{\frac{1}{2}} \varphi(z;\chi,\tau)\right)^{-n\sigma_3} \\
&\quad\cdot \ee^{ - M d(\chi,\tau)\sigma_3} M^{\frac{1}{2}n\sigma_3}   \ii^{-\sigma_3}  y(z;\chi,\tau)^{-n\sigma_3},\quad z\in\partial D_\critpt.\\
\end{aligned}
\end{equation}
Using \eqref{zeta-def} and \eqref{H-expansion}, 
we obtain for $z\in \partial D_\critpt$ (on which $\varphi(z;\chi,\tau)$ is bounded away from zero) that
\begin{equation}
\begin{aligned}
\breve{\mathbf{S}}_n^{\mathrm{in}}(z;\chi,\tau,M) \breve{\mathbf{S}}_n^{\mathrm{out}}(z;\chi,\tau)^{-1} &=y(z;\chi,\tau)^{n\sigma_3} \ii^{\sigma_3}  M^{-\frac{1}{2}n \sigma_3}\ee^{M d(\chi,\tau)\sigma_3}\\
&\quad\cdot \left( \mathbb{I} + O(M^{-\frac{1}{2}}) \right)\\
&\quad\cdot \ee^{ - M d(\chi,\tau)\sigma_3} M^{\frac{1}{2}n\sigma_3}   \ii^{-\sigma_3}  y(z;\chi,\tau)^{-n\sigma_3},\quad M\to +\infty.
\end{aligned}
\label{Sin-Sout-comparison}
\end{equation}

To prevent the conjugating factors from contaminating the error term we now tie $n$ to $(\chi,\tau)$ to ensure that $| M^{-\frac{1}{2}n} \ee^{M d(\chi,\tau)} |$ is as close as possible to $1$.
Since
\begin{equation}
\begin{aligned}
A_n(\chi,\tau;M):=M^{-\frac{1}{2}n} \ee^{M d(\chi,\tau)} 
&=\ee^{\ii M \Im(d(\chi,\tau))} \exp\left(\ln(M)\frac{1}{2}\left(\frac{\Re(2 d(\chi,\tau))}{\ln(M) M^{-1}} -n \right)\right),
\end{aligned}
\label{A-quantity}
\end{equation}
we assume that $n=N(\chi,\tau;M)$, where
\begin{equation}
N(\chi,\tau;M):= \left\lfloor \frac{\Re(2 d(\chi,\tau))}{\ln(M) M^{-1}} \right\rceil \geq 0,\qquad \text{whenever}\quad \frac{\Re(2 d(\chi,\tau))}{\ln(M) M^{-1}}\geq -\frac{1}{2},
\label{n-def}
\end{equation}
where $\lfloor x \rceil$ denotes the integer closest to a real number $x$ with the convention that $\lfloor n-0.5 \rceil =n$. 
We truncate the choice of $n$ to $0$ otherwise: 
\begin{equation}
N(\chi,\tau;M):=0, \qquad \text{whenever}\quad \frac{\Re(2 d(\chi,\tau))}{\ln(M) M^{-1}}< -\frac{1}{2}.
\label{n-def-0}
\end{equation}
Thus, the condition $(\chi,\tau)\in\mathcal{S}$ implies that $n=N(\chi,\tau;M)$ lies in the finite range $\{0,1,2,\dots,\lfloor K_+\rceil\}$.

With $n$ related to $(\chi,\tau)\in\mathcal{S}$, using \eqref{n-def} in \eqref{A-quantity} shows that
$A_n(\chi,\tau;M)$
satisfies the estimate
\begin{equation}
M^{-\frac{1}{4}} \leq | A_n(\chi,\tau;M)| < M^{\frac{1}{4}},\quad n=N(\chi,\tau;M)\in\mathbb{Z}_{\geq 1}.
\label{A-bound}
\end{equation}
In the case $n=0$ (resulting from either \eqref{n-def} or from the truncation \eqref{n-def-0}), the lower bound in \eqref{A-bound} 
is replaced with an exponentially small quantity:
\begin{equation}
\ee^{-K_-M}\le | A_0(\chi,\tau;M)| < M^{\frac{1}{4}},\quad N(\chi,\tau;M)=0.
\label{A-bound-0}
\end{equation}

We finally define the global parametrix $\breve{\mathbf{S}}(z;\chi,\tau,M)$ by
\begin{equation}
\breve{\mathbf{S}}(z;\chi,\tau,M):=
\begin{cases}
\breve{\mathbf{S}}_n^{\mathrm{in}}(z;\chi,\tau,M),&\quad z\in D_{\critpt},\\
\sigma_2\breve{\mathbf{S}}_n^{\mathrm{in}}(z^*;\chi,\tau,M)^*\sigma_2,&\quad z\in D_{\critpt}^*,\\
\breve{\mathbf{S}}_n^{\mathrm{out}}(z;\chi,\tau),&\quad z\in \mathbb{C}\setminus \overline{D_{\critpt}\cup D_{\critpt}^*},
\end{cases}
\label{S-global}
\end{equation}
wherein $n=N(\chi,\tau;M)\in\{0,1,2,\dots,\lfloor K_+\rceil\}$.

\subsection{Initial error analysis}
\label{s:transition-error}
The accuracy of approximating $\mathbf{S}(z;\chi,\tau,M)$ with $\breve{\mathbf{S}}(z;\chi,\tau,M)$ is measured by $\mathbf{F}(z;\chi,\tau,M)-\mathbb{I}$, where
\begin{equation}
\mathbf{F}(z;\chi,\tau,M):= \mathbf{S}(z;\chi,\tau,M)\breve{\mathbf{S}}(z;\chi,\tau,M)^{-1},\quad z\in\mathbb{C}\setminus \Sigma^{\mathbf{F}}.
\label{S-to-F}
\end{equation}
Here $\Sigma^\mathbf{F}:=\Gamma\cup\partial D_\critpt\cup\partial D_\critpt^*$, where the circles $\partial D_\critpt$ and $\partial D_\critpt^*$ both have clockwise orientation.
Note that $\mathbf{F}(z;\chi,\tau,M)$ has a Schwarz-symmetric property analogous to \eqref{eq:S-Schwarz} and is analytic in its domain of definition $\mathbb{C}\setminus \Sigma^{\mathbf{F}}$, and its boundary values on the jump contour $\Sigma^\mathbf{F}$ are related by
$\mathbf{F}_{+}(z;\chi,\tau,M) = \mathbf{F}_{-}(z;\chi,\tau,M)\mathbf{V}^{\mathbf{F}}(z;\chi,\tau,M)$, where $\mathbf{V}^\mathbf{F}(z)=\mathbf{V}^\mathbf{F}(z;\chi,\tau,M)$ is defined by
\begin{equation}
\mathbf{V}^{\mathbf{F}}(z) :=
\begin{cases}
\breve{\mathbf{S}}^{\mathrm{out}}_n(z;\chi,\tau)\mathbf{V}^{\mathbf{S}}(z;\chi,\tau,M)\breve{\mathbf{S}}^{\mathrm{out}}_n(z;\chi,\tau)^{-1},&\quad z\in \Gamma\setminus\overline{D_\critpt\cup D_\critpt^*},\\
\breve{\mathbf{S}}^{\mathrm{in}}_{n -}(z;\chi,\tau,M) \mathbf{V}^{\mathbf{S}}(z;\chi,\tau,M)\mathbf{V}^{\breve{\mathbf{S}}_n^{\mathrm{in}}}(z;\chi,\tau,M)^{-1} \breve{\mathbf{S}}^{\mathrm{in}}_{n -}(z;\chi,\tau,M)^{-1},&\quad z\in \Gamma\cap D_\critpt,\\
\breve{\mathbf{S}}_n^{\mathrm{in}}(z;\chi,\tau,M)\breve{\mathbf{S}}_n^{\mathrm{out}}(z;\chi,\tau)^{-1},&\quad z\in \partial D_{\critpt},
\end{cases}
\label{F-jump}
\end{equation}
where $n=N(\chi,\tau;M)$ and the value of the jump matrix on $\Gamma\cap D_\critpt^*$ and on $\partial D_\critpt^*$ follows by Schwarz symmetry.
Note also that $\mathbf{F}(z;\chi,\tau,M) \to \mathbb{I}$ as $z\to\infty$.

Since for $n=N(\chi,\tau;M)\in \{0,1,2,\dots,\lfloor K_+\rceil\}$ the outer parametrix $\breve{\mathbf{S}}_n^\mathrm{out}(z;\chi,\tau)$ is unimodular, independent of $M$, and bounded on $\Gamma\setminus\overline{D_\critpt\cup D_\critpt^*}$, using the estimate \eqref{eq:transitional-uniform-outside} in \eqref{F-jump} implies that
\begin{equation}
\sup_{z\in\Gamma\setminus\overline{D_\critpt\cup D_\critpt^*}}\| \mathbf{V}^{\mathbf{F}}(z;\chi,\tau,M)-\mathbb{I} \| = O (\ee^{-c M}),\quad M\to +\infty,
\label{VF-Sigma-infty}
\end{equation}
for some constant $c>0$.

To analyze $\mathbf{V}^{\mathbf{F}}(z;\chi,\tau,M)-\mathbb{I}$
on the arc $\Gamma\cap D_\critpt$, we begin by noting that the inner parametrix $\breve{\mathbf{S}}_n^{\mathrm{in}}(z;\chi,\tau,M)$ does not satisfy exactly the same jump conditions as $\mathbf{S}(z;\chi,\tau,M)$ on that arc. Indeed, recalling \eqref{eq:S-jump} and using \eqref{conformal-phi-def} and \eqref{S-xi-jump}, we see that
a more precise version of \eqref{eq:VS-in-1} is 
\begin{equation}
\mathbf{V}^{\mathbf{S}}(z;\chi,\tau,M) 
=
\mathbf{V}^{\breve{\mathbf{S}}_n^\mathrm{in}}(z;\chi,\tau,M)
+
\begin{bmatrix}
O(\ee^{-4M}) & 0\\
O(\ee^{M[\varphi(z;\chi,\tau)^2-2d(\chi,\tau)-4]})
& O(\ee^{-4M})
\end{bmatrix},
\quad z\in \Gamma\cap D_\critpt,
\label{VS-inside}
\end{equation}
in the limit $M\to+\infty$.
Hence using \eqref{S-xi-jump} again,
\begin{multline}
\mathbf{V}^\mathbf{S}(z;\chi,\tau,M)\mathbf{V}^{\breve{\mathbf{S}}_n^\mathrm{in}}(z;\chi,\tau,M)^{-1}\\=\mathbb{I}+\begin{bmatrix} 
O(\ee^{-4M}) & O(\ee^{M[2d(\chi,\tau) -\varphi(z;\chi,\tau)^2-4]}) \\ 
O(\ee^{M[\varphi(z;\chi,\tau)^2 -2 d(\chi,\tau)-4]}) & O(\ee^{-4M})
 \end{bmatrix},\quad z\in\Gamma\cap D_\critpt,\quad M\to+\infty.
 \end{multline}
Thus, with \eqref{S-xi} taking $\mathbf{Y}_n(z)$ as in \eqref{f-def} and using \eqref{A-quantity}, we find that with $n=N(\chi,\tau;M)$,
\begin{equation}
\begin{aligned}
\mathbf{V}^{\mathbf{F}}(z;\chi,\tau,M)  - \mathbb{I}
&= y(z;\chi,\tau)^{n\sigma_3} A_n(\chi,\tau;M)^{\sigma_3}\ii^{\sigma_3}  \mathbf{H}^{[n]}_-(M^{\frac{1}{2}}\varphi(z;\chi,\tau))\ii^{-\sigma_3}A_n(\chi,\tau;M)^{-\sigma_3}\\
 &\quad \cdot
\begin{bmatrix}
O(\ee^{-4M}) & O(A_n(\chi,\tau;M)^2 \ee^{ -M[\varphi(z;\chi,\tau)^2+4]})\\
O(A_n(\chi,\tau;M)^{-2}\ee^{M[\varphi(z;\chi,\tau)^2-4]}) & O(\ee^{-4M})
 \end{bmatrix}
 \\
&\quad \cdot
A_n(\chi,\tau;M)^{\sigma_3}\ii^{\sigma_3}\mathbf{H}^{[n]}_-(M^{\frac{1}{2}}\varphi(z;\chi,\tau))^{-1}
 \ii^{- \sigma_3} 
A_n(\chi,\tau;M)^{-\sigma_3} 
 y(z;\chi,\tau)^{n\sigma_3}
\end{aligned}
\label{VF-inside}
\end{equation}
 for $z\in\Gamma\cap D_\critpt$.  Now, since $\mathbf{H}^{[n]}(\zeta)=O(\langle\zeta\rangle^n)$, $\mathbf{H}^{[n]}(M^\frac{1}{2}\varphi(z;\chi,\tau))=O(M^{\frac{1}{2}n})$ for $z\in D_\critpt$ because $\varphi(z;\chi,\tau)$ is bounded there.  The estimates \eqref{A-bound}--\eqref{A-bound-0} on $|A_n(\chi,\tau;M)|$ then show that the matrix product on the first line of the right-hand side of \eqref{VF-inside} is $O(M^{\frac{1}{2}(n+1)})$ (in the case $n=0$, the exponential lower bound in \eqref{A-bound-0} is harmless because $\mathbf{H}^{[0]}(\zeta)$ is upper triangular).  By unimodularity, the same estimate holds for the inverse matrix on the third line of the right-hand side of \eqref{VF-inside}.  
Because $0<K_-<2$, the central factor is uniformly $O(\ee^{-M\delta})$ on $\Gamma\cap D_\critpt$ for some $\delta>0$, provided that the $M$-independent radius of $D_\critpt$ is sufficiently small.  Indeed, this is obvious for all but the $2,1$-entry because $\varphi(z;\chi,\tau)^2\ge 0$ holds on the jump contour.  Now, since $K_-<2$, there exists $\delta>0$ such that $4-2K_--2\delta>0$, so using the exponential lower bound for $|A_0(\chi,\tau;M)|$ from \eqref{A-bound-0} shows that when $\max_{z\in\Gamma\cap D_\critpt}\varphi(z;\chi,\tau)^2\le 4-2K_--2\delta$ we have
for $n=0$ that $O(A_0(\chi,\tau;M)^{-2}\ee^{M[\varphi(z;\chi,\tau)^2-4]})=O(\ee^{-2M\delta})$.  Under the same condition for $n>0$ we have $O(A_n(\chi,\tau;M)^{-2}\ee^{M[\varphi(z;\chi,\tau)^2-4]})=O(M^\frac{1}{2}\ee^{-M[2K_-+2\delta]})=O(\ee^{-2M\delta})$.  Since the conjugating factors contribute $O(M^{n+1})=O(M^{\lfloor K_+\rceil+1})=O(\ee^{M\delta})$, combining the results shows that
\begin{equation}
\sup_{z\in\Gamma\cap D_\critpt}\|\mathbf{V}^\mathbf{F}(z;\chi,\tau,M)-\mathbb{I}\|=O(\ee^{-M\delta}),\quad M\to+\infty.
\label{eq:VF-inside-disk}
\end{equation}

It now remains to estimate the jump matrix $\mathbf{V}^{\mathbf{F}}(z;\chi,\tau,M)$ on the disk boundaries $\partial D_\critpt$ and $\partial D_{\critpt}^*$. 
For $z\in\partial D_\critpt$, we recall that $\varphi(z;\chi,\tau)\neq 0$ and first use \eqref{H-expansion} and \eqref{A-quantity} in \eqref{Sin-Sout-comparison} to write
\begin{equation}
\begin{aligned}
\mathbf{V}^{\mathbf{F}}(z;\chi,\tau,M) - \mathbb{I} & =y(z;\chi,\tau)^{n\sigma_3} A_n(\chi,\tau;M)^{\sigma_3} \ii^{\sigma_3}  \\
&\quad\cdot \left(  \mathbf{H}^{[n]}_{-1}\frac{1}{\varphi(z;\chi,\tau)M^{\frac{1}{2}}} + \mathbf{H}^{[n]}_{-2}\frac{1}{\varphi(z;\chi,\tau)^2 M } + O(M^{-\frac{3}{2}}) \right)\\
&\quad\cdot   \ii^{-\sigma_3}  A_n(\chi,\tau;M)^{-\sigma_3}  y(z;\chi,\tau)^{-n\sigma_3}, \qquad z\in\partial D_\critpt,\quad M\to +\infty,
\end{aligned}
\label{VF-disk-expansion}
\end{equation}
where $n=N(\chi,\tau;M)$.  Next, incorporating the bounds \eqref{A-bound}--\eqref{A-bound-0} in the above expression and recalling the definition \eqref{H-error} together with the facts that $\mathbf{H}_{-2}^{[n]}$ is a diagonal matrix, that the error term is an upper-triangular matrix if $n=N(\chi,\tau;M)=0$, and that $y(z;\chi,\tau)^{n\sigma_3}$ is bounded for $z\in\partial D_{\critpt}$ independently of $M$, we find that
\begin{equation}
\begin{aligned}
\mathbf{V}^{\mathbf{F}}(z;\chi,\tau,M)  - \mathbb{I} &= 
\begin{bmatrix}
0 & \dfrac{y(z;\chi,\tau)^{2n}  A_n(\chi,\tau;M)^2}{2\pi \ii \gamma_n^2} \\ \dfrac{2\pi \ii \gamma_{n-1}^2}{y(z;\chi,\tau)^{2n}  A_n(\chi,\tau;M)^{2}} & 0
\end{bmatrix}\frac{1}{\varphi(z;\chi,\tau) M^{\frac{1}{2}} }\\
&\quad + O(M^{-1}),\quad z\in\partial D_\critpt,\quad n=N(\chi,\tau;M),\quad M\to +\infty.
\end{aligned}
\label{VF-disk-estimate}
\end{equation}
\begin{remark}
This estimate shows that $\mathbf{V}^{\mathbf{F}}(z;\chi,\tau,M) - \mathbb{I}$ fails to be small for $z\in \partial D_{\critpt}$ whenever $|A_n(\chi,\tau;M)|$ approaches the endpoints of the range \eqref{A-bound} (or the upper endpoint of the range \eqref{A-bound-0} for $n=0$ --- the exponentially small lower bound at the lower endpoint is irrelevant because $\gamma_{n-1}=0$ by convention) and in this case $\mathbf{F}(z;\chi,\tau,M)$ does not satisfy a small-norm Riemann-Hilbert problem in the limit as $M\to+\infty$.
The same problem arises also on $\partial D_{\critpt}^*$ by symmetry.
\label{r:no-small-norm}
\end{remark}

In light of this observation, to be able to cover the whole range of $(\chi,\tau)\in\mathcal{S}$ defined in \eqref{d-size}, it is necessary
to mitigate the 
difficulty
by considering an explicitly solvable model Riemann-Hilbert problem whose jump condition involves the components of $\mathbf{V}^{\mathbf{F}}(z;\chi,\tau,M) - \mathbb{I}$ on $\partial D_{\critpt}\cup \partial D_{\critpt}^*$ that are not small.
Put differently, we will now build a parametrix for $\mathbf{F}(z;\chi,\tau,M)$.

\subsection{A soliton parametrix for $\mathbf{F}(z;\chi,\tau,M)$}
A closer inspection of \eqref{VF-disk-estimate} shows that the largest terms are different depending on whether $|A_n(\chi,\tau;M)|< 1$ or $|A_n(\chi,\tau;M)|\ge 1$.  Indeed, if $n=N(\chi,\tau;M)\in \{1,2,\dots,\lfloor K_+\rceil$\} and $|A_n(\chi,\tau;M)|\ge 1$, or if $n=N(\chi,\tau;M)=0$ without any further condition on $|A_0(\chi,\tau;M)|$, then 
\begin{equation}
\mathbf{V}^{\mathbf{F}}(z;\chi,\tau,M) = \mathbb{I} + 
\begin{bmatrix}
0 & \dfrac{y(z;\chi,\tau)^{2n}  A_n(\chi,\tau;M)^2}{2\pi \ii \gamma_n^2 \varphi(z;\chi,\tau)M^{\frac{1}{2}}} \\ 0 & 0
\end{bmatrix}+ O(M^{-\frac{1}{2}}),\quad z\in\partial D_\critpt, \quad M\to +\infty.
\label{VF-plus}
\end{equation}
In the $n=0$ case one gets an improvement of the error term to $O(M^{-1})$ for the same reason no condition on $|A_0(\chi,\tau;M)|$ is needed, namely that $\gamma_{-1}=0$.
On the other hand, if $n=N(\chi,\tau;M)\in \{1,2,\dots,\lfloor K_+\rceil\}$ and $|A_n(\chi,\tau;M)|< 1$, then
\begin{equation}
\mathbf{V}^{\mathbf{F}}(z;\chi,\tau,M) = \mathbb{I} + 
\begin{bmatrix}
0 & 0 \\ \dfrac{2\pi \ii \gamma_{n-1}^2y(z;\chi,\tau)^{-2n}}{ A_n(\chi,\tau;M)^{2}\varphi(z;\chi,\tau)M^{\frac{1}{2}}} & 0
\end{bmatrix}+ O(M^{-\frac{1}{2}}),\quad z\in \partial D_\critpt, \quad M\to +\infty.
\label{VF-minus}
\end{equation}
In both \eqref{VF-plus} and \eqref{VF-minus}, the error terms are in the $L^\infty(\partial D_\critpt)$ sense.
We introduce the following quantities in order to treat both cases simultaneously. Let
\begin{equation}
L_n^{+}(\chi,\tau;M):=\frac{A_n(\chi,\tau;M)^2}{2\pi \ii \gamma_n^2 M^{\frac{1}{2}}},\qquad
L_n^{-}(\chi,\tau;M):=\frac{2\pi \ii \gamma_{n-1}^2 }{A_n(\chi,\tau;M)^2 M^{\frac{1}{2}}},
\label{Ln-def}
\end{equation}
and define the $2$-nilpotent matrices
\begin{equation}
\mathbf{N}^+ := \begin{bmatrix} 0 & 1 \\ 0 & 0 \end{bmatrix},\qquad \mathbf{N}^- := \begin{bmatrix} 0 & 0 \\ 1 & 0 \end{bmatrix}.
\label{2-nilpotents}
\end{equation}
Also define $\mathfrak{s}=\mathfrak{S}(\chi,\tau;M)$ to be the sign determined from $\chi,\tau,M$ in the following way:
\begin{equation}
\mathfrak{S}(\chi,\tau;M):=
\begin{cases}
+,&\quad \text{$n\in\{1,\dots,\lfloor K_+\rceil\}$ and $|A_n(\chi,\tau;M)| \geq 1$, or $n=0$}\\
-,&\quad \text{$n\in\{1,\dots,\lfloor K_+\rceil\}$ and $|A_n(\chi,\tau;M)| < 1$,}
\end{cases}
\label{s-def-1}
\end{equation}
where $n=N(\chi,\tau;M)$.
Then \eqref{VF-plus} and \eqref{VF-minus} can be written in a common form as 
\begin{equation}
\mathbf{V}^\mathbf{F}(z;\chi,\tau,M)=\mathbb{I} + L_n^\mathfrak{s}(\chi,\tau;M)\frac{y(z;\chi,\tau)^{\mathfrak{s}2n}}{\varphi(z;\chi,\tau)}\mathbf{N}^\mathfrak{s} + O(M^{-\frac{1}{2}}),\quad z\in\partial D_\critpt,\quad M\to+\infty.
\label{eq:VF-plus-minus-together}
\end{equation}

Neglecting the error terms in \eqref{VF-plus}--\eqref{VF-minus} leads to the following 
Riemann-Hilbert problem 
characterizing
a parametrix for $\mathbf{F}(z;\chi,\tau,M)$.  
\begin{rhp}[Soliton parametrix]
Given an arbitrary sign $\mathfrak{s}=\pm$ and a function $p(z)$ analytic on a punctured closed disk $\overline{D}_\critpt\setminus\{\critpt\}$ with a simple pole at the center $z=\critpt$, find a $2\times 2$ matrix-valued function $\mathbf{\solmat}^\mathfrak{s}(z;p)$ with the following properties:
\begin{itemize}
\item[]\textbf{Analyticity:} $\mathbf{\solmat}^\mathfrak{s}(z;p)$ is analytic for $z\in\mathbb{C}\setminus \left(\partial D_\critpt\cup\partial D_\critpt^*\right)$, and it takes continuous boundary values on the clockwise-oriented disk boundaries.
\item[]\textbf{Jump conditions:}  The boundary values on the jump contour $\partial D_\critpt\cup\partial D^*_\critpt$ are related by 
\begin{equation}
\mathbf{\solmat}^\mathfrak{s}_+(z;p)=\mathbf{\solmat}^\mathfrak{s}_-(z;p)\mathbf{V}^{\mathbf{\solmat}^\mathfrak{s}}(z;p),\quad z\in \partial D_\critpt\cup\partial D_\critpt^*,
\end{equation}
where 
\begin{equation}
\mathbf{V}^{\mathbf{\solmat}^\mathfrak{s}}(z;p):=\begin{cases} \mathbb{I}+p(z)\mathbf{N}^\mathfrak{s},&\quad z\in\partial D_\critpt,\\
\sigma_2(\mathbb{I}+p(z^*)\mathbf{N}^\mathfrak{s})^*\sigma_2,&\quad z\in\partial D_\critpt^*.
\end{cases}
\label{eq:VJ-jump}
\end{equation}
\item[]\textbf{Normalization:} $\mathbf{\solmat}^\mathfrak{s}(z;p)=\mathbb{I}+O(z^{-1})$ as $z\to\infty$.
\end{itemize}
\label{rhp:generic-soliton-parametrix}
\end{rhp}

\begin{proposition}\label{p:generic-soliton-parametrix}
 \rhref{rhp:generic-soliton-parametrix} has a unique solution given by
\begin{equation}
\mathbf{\solmat}^\mathfrak{s}(z;p) = 
\begin{cases}
\mathbf{R}^{\mathfrak{s}}(z;p) \mathbf{V}^{\mathbf{\solmat}^\mathfrak{s}}(z;p)^{-1},&\quad z\in D_\critpt\cup D_{\critpt^*},\\
\mathbf{R}^{\mathfrak{s}}(z;p),&\quad z\in \mathbb{C}\setminus {\overline{D_\critpt\cup D_{\critpt^*}}},
\end{cases}
\label{G-s-def}
\end{equation}
where $\mathbf{R}^\mathfrak{s}(z;p)$ is the rational function
\begin{equation}
\mathbf{R}^{\mathfrak{s}}(z;p) :=\mathbb{I} + \frac{\mathbf{X}^{\mathfrak{s}}(p)}{z-\critpt}+\frac{\mathbf{Z}^{\mathfrak{s}}(p)}{z-\critpt^*}.
\label{G-sharp-form}
\end{equation}
Here
\begin{equation}
\mathbf{X}^-(p):=\frac{1}{4\Im(\critpt)^2+|C|^2}\begin{bmatrix}2\ii\Im(\critpt)|C|^2 & 0\\4\Im(\critpt)^2C & 0\end{bmatrix},\quad 
\mathbf{Z}^{-}(p)
:=
\frac{1}{4\Im(\critpt)^2 + |C|^2}
\begin{bmatrix}  0 & - 4\Im(\critpt)^2 C^*  \\  0& - 2\ii \Im(\critpt) |C|^2 
\end{bmatrix},
\label{Z-s-minus}
\end{equation}
and
\begin{equation}
\mathbf{X}^+(p)
:=
\frac{1}{4\Im(\critpt)^2 + |C|^2}
\begin{bmatrix} 0 & 4\Im(\critpt)^2 C  \\ 0& 2\ii \Im(\critpt) |C|^2 
\end{bmatrix},\quad
\mathbf{Z}^+(p):=\frac{1}{4\Im(\critpt)^2+|C|^2}\begin{bmatrix}-2\ii\Im(\critpt)|C|^2 & 0\\-4\Im(\critpt)^2C^* & 0\end{bmatrix},
\label{X-s-plus}
\end{equation}
where
\begin{equation}
C := \mathop{\mathrm{Res}}_{z=\critpt}p(z).
\label{C-residue-def}
\end{equation}
The first line of \eqref{G-s-def} has only removable singularities at $z=\critpt, \critpt^*$ and it extends as a function analytic in $D_{\critpt} \cup D_{\critpt}^*$.
\end{proposition}
\begin{remark}
When the residue $C$ has the form $C_0\ee^{-\mathfrak{s}2\ii M(\critpt \chi+\critpt^2\tau)}$ for $C_0\neq 0$, \rhref{rhp:generic-soliton-parametrix} characterizes a fundamental solution matrix of the Zakharov-Shabat eigenvalue problem with spectral parameter $z$ and simple eigenvalue $z=\critpt\in\mathbb{C}_+$ corresponding to the soliton solution of the focusing nonlinear Schr\"odinger equation in the form \eqref{eq:semiclassicalNLS}.  This solution is a reflectionless potential for which the only nontrivial scattering coefficient is a simple Blaschke factor $a(z)=(z-\critpt)/(z-\critpt^*)$.  Actually, defining the scattering coefficient requires selecting a direction $\mathfrak{s}$ of normalization so that $\mathbf{\solmat}\to\mathbb{I}$ as $\chi\to-\mathfrak{s}\infty$, and it is well known that one may change the direction of normalization by conjugating by the scattering matrix $a(z)^{\sigma_3}$.
Indeed, the solutions of \rhref{rhp:generic-soliton-parametrix} are explicitly related for opposite signs $\mathfrak{s}$ and corresponding residues $C,\widetilde{C}$.  To see this, it suffices to work outside of the two disks and consider the rational matrix $\mathbf{R}^+(z;p)$ where $p$ has residue $C$ at $z=\critpt$.  A simple calculation starting from \eqref{G-sharp-form} and \eqref{X-s-plus} shows that $\mathbf{R}^+(z;p)a(z)^{-\sigma_3}$ is a rational function of the form $\mathbf{R}^-(z;\widetilde{p})$, where the residue $\widetilde{C}$ of $\widetilde{p}$ at $z=\critpt$ is given by $\widetilde{C}=-4\Im(\critpt)^2/C$.
\label{r:Jost-Swap}
\end{remark}
We prove Proposition~\ref{p:generic-soliton-parametrix} in Appendix~\ref{a:soliton-RHP}. Taking the  function $p(z)$ to be given in terms of $z\in\overline{D}_{\critpt(\chi,\tau)}\setminus\{\critpt(\chi,\tau)\}$ and parameters $\chi,\tau,M$ by
\begin{equation}
p(z)=p_n^\mathfrak{s}(z;\chi,\tau,M):=L_n^\mathfrak{s}(\chi,\tau;M)\frac{y(z;\chi,\tau)^{\mathfrak{s}2n}}{\varphi(z;\chi,\tau)},\quad n=N(\chi,\tau;M),\quad\mathfrak{s}=\mathfrak{S}(\chi,\tau;M),
\end{equation}
we define a parametrix for $\mathbf{F}(z;\chi,\tau,M)$ by $\breve{\mathbf{F}}(z;\chi,\tau,M):=\mathbf{\solmat}^\mathfrak{s}(z;p_n^\mathfrak{s}(\diamond;\chi,\tau,M))$.

Next, we proceed with analyzing the accuracy of approximating $\mathbf{F}(z;\chi,\tau,M)$ with $\breve{\mathbf{F}}(z;\chi,\tau,M)$.

\subsection{Improved error analysis}
Given $M>0$ and $(\chi,\tau)\in\mathcal{S}$ (see \eqref{d-size}), we define the (improved) error matrix 
\begin{equation}
\mathbf{E}(z;\chi,\tau,M):=\mathbf{F}(z;\chi,\tau,M)\breve{\mathbf{F}}(z;\chi,\tau,M)^{-1},\quad z\in\mathbb{C}\setminus\Sigma^\mathbf{E},
\end{equation}
where $\Sigma^\mathbf{E}=\Sigma^\mathbf{F}=\Gamma\cup\partial D_\critpt\cup\partial D_\critpt^*$.  It is easy to see that $\mathbf{E}(z;\chi,\tau,M)$ satisfies a Riemann-Hilbert problem of small-norm type as $M\to\infty$, uniformly for $(\chi,\tau)\in\mathcal{S}$.  Indeed, $z\mapsto\mathbf{E}(z;\chi,\tau,M)$ is analytic in its domain of definition and tends to the identity as $z\to\infty$.  On the arcs of $\Sigma^\mathbf{E}$ it satisfies jump conditions of the form $\mathbf{E}_+(z;\chi,\tau,M)=\mathbf{E}_-(z;\chi,\tau,M)\mathbf{V}^\mathbf{E}(z;\chi,\tau,M)$ where the jump matrix is given by
\begin{equation}
\mathbf{V}^\mathbf{E}(z;\chi,\tau,M):=\breve{\mathbf{F}}(z;\chi,\tau,M)\mathbf{V}^\mathbf{F}(z;\chi,\tau,M)\breve{\mathbf{F}}(z;\chi,\tau,M)^{-1},\quad z\in \Sigma^\mathbf{E}\setminus(\partial D_\critpt\cup \partial D_\critpt^*),
\label{VE-not-circles-2}
\end{equation}
and
\begin{equation}
\mathbf{V}^\mathbf{E}(z;\chi,\tau,M):=\breve{\mathbf{F}}_-(z;\chi,\tau,M)\mathbf{V}^\mathbf{F}(z;\chi,\tau,M)\mathbf{V}^{\breve{\mathbf{F}}}(z;\chi,\tau,M)^{-1}\breve{\mathbf{F}}_-(z;\chi,\tau,M)^{-1},\quad z\in \partial D_\critpt\cup \partial D_\critpt^*,
\label{VE-on-circles-2}
\end{equation}
where $\mathbf{V}^{\breve{\mathbf{F}}}(z;\chi,\tau,M)$ denotes the jump matrix for the parametrix $\breve{\mathbf{F}}(z;\chi,\tau,M)$.  Since the residue of $p(z)$ at $z=\critpt(\chi,\tau)$ is bounded uniformly for $(\chi,\tau)\in\mathcal{S}$ due to the definitions \eqref{Ln-def} and \eqref{s-def-1} and the bounds \eqref{A-bound}--\eqref{A-bound-0}, it follows from Proposition~\ref{p:generic-soliton-parametrix} that $z\mapsto \breve{\mathbf{F}}(z;\chi,\tau,M)$ is also uniformly bounded.  Since $\breve{\mathbf{F}}(z;\chi,\tau,M)$ has unit determinant, it is then clear that the conjugating factors in \eqref{VE-not-circles-2}--\eqref{VE-on-circles-2} are uniformly bounded.  Then using the estimate \eqref{VF-Sigma-infty} and the estimate \eqref{eq:VF-inside-disk} with its Schwarz-reflection analogue valid on $\Gamma\cap D_\critpt^*$, one sees from \eqref{VE-not-circles-2} that 
\begin{equation}
\sup_{z\in\Sigma^\mathbf{E}\setminus (\partial D_\critpt\cup\partial D_\critpt^*)}\|\mathbf{V}^\mathbf{E}(z;\chi,\tau,M)-\mathbb{I}\|=O(\ee^{-c'M}),\quad M\to+\infty
\end{equation}
holds for $(\chi,\tau)\in\mathcal{S}$, where $c'=\min\{c,\delta\}>0$.  Since  $\mathbf{V}^{\breve{\mathbf{F}}}(z;\chi,\tau,M)$ is exactly given by the explicit terms in $\mathbf{V}^\mathbf{F}(z;\chi,\tau,M)$ in \eqref{eq:VF-plus-minus-together} for $z\in\partial D_\critpt$, and since these terms are bounded on $\partial D_\critpt$, one sees that $\mathbf{V}^\mathbf{F}(z;\chi,\tau,M)\mathbf{V}^{\breve{\mathbf{F}}}(z;\chi,\tau,M)^{-1}=\mathbb{I}+O(M^{-\frac{1}{2}})$ holds uniformly on $\partial D_\critpt$ and for $(\chi,\tau)\in\mathcal{S}$.  Combining this with its Schwarz-reflection analogue and using the results in \eqref{VE-on-circles-2} then shows that
\begin{equation}
\sup_{z\in\partial D_\critpt\cup\partial D_\critpt^*}\|\mathbf{V}^\mathbf{E}(z;\chi,\tau,M)-\mathbb{I}\|=O(M^{-\frac{1}{2}}),\quad M\to\infty,
\end{equation}
an estimate that is also uniform for $(\chi,\tau)\in\mathcal{S}$.  

Since $\mathbf{V}^\mathbf{E}(z;\chi,\tau,M)-\mathbb{I}$ is uniformly $O(M^{-\frac{1}{2}})$ on the compact jump contour $\Sigma^\mathbf{E}$, it follows from standard small-norm theory\footnote{See, for example, \cite[Section 4.1.3]{BilmanLM2020} for an application of this theory in more detail in a similar setting.} that 
\begin{equation}
\mathbf{E}_{-1}(\chi,\tau;M):=\lim_{z\to\infty} z(\mathbf{E}(z;\chi,\tau,M)-\mathbb{I}) = O(M^{-\frac{1}{2}})
\label{eq:E-residue-estimate}
\end{equation}
holds uniformly for $(\chi,\tau)\in\mathcal{S}$.  

Recall now the recovery formula \eqref{eq:DS-Psi-from-S} and note that in a neighborhood of $z=\infty$ we have the identities $\mathbf{S}(z;\chi,\tau,M)=\mathbf{F}(z;\chi,\tau,M)\breve{\mathbf{S}}_n^\mathrm{out}(z;\chi,\tau)=\mathbf{E}(z;\chi,\tau,M)\breve{\mathbf{F}}(z;\chi,\tau,M)\breve{\mathbf{S}}^\mathrm{out}_n(z;\chi,\tau)$.  Therefore, as all three factors tend to the identity as $z\to\infty$ and $\breve{\mathbf{S}}^\mathrm{out}_n(z;\chi,\tau,M)$ is diagonal,
\begin{equation}
\begin{split}
M\Psi(M^2\chi,M^3\tau;\mathbf{G}(\ee^{-2M},\sqrt{1-\ee^{-4M}}))&=2\ii\lim_{z\to\infty}z\breve{F}_{12}(z;\chi,\tau,M) + 2\ii E_{-1,12}(\chi,\tau;M)\\
&=2\ii\lim_{z\to\infty}z\breve{F}_{12}(z;\chi,\tau,M) +O(M^{-\frac{1}{2}}),\quad M\to+\infty,
\end{split}
\end{equation}
where we used \eqref{eq:E-residue-estimate}.  Now we calculate the explicit term using Proposition~\ref{p:generic-soliton-parametrix}.  If $n=N(\chi,\tau;M)$ and $\mathfrak{s}=\mathfrak{S}(\chi,\tau;M)=+$, then $\breve{F}_{12}(z;\chi,\tau,M)=\solmat^+_{12}(z;p_n^+(\diamond;\chi,\tau,M))=R^+_{12}(z;p_n^+(\diamond;\chi,\tau,M))$ holds by definition for large $z$.  Therefore, using \eqref{G-sharp-form} and \eqref{X-s-plus} gives
\begin{equation}
2\ii\lim_{z\to\infty}z\breve{F}_{12}(z;\chi,\tau,M)=2\ii X^+_{12}(p_n^+(\diamond;\chi,\tau,M)) = \frac{8\ii\Im(\critpt)^2C_n^+}{4\Im(\critpt)^2+|C_n^+|^2}=:\psi^+_n(\chi,\tau;M),
\end{equation}
where
\begin{multline}
C_n^+=C_n^+(\chi,\tau;M):=\mathop{\mathrm{Res}}_{z=\critpt(\chi,\tau)}p_n^+(z;\chi,\tau,M)=L_n^+(\chi,\tau;M)\frac{y(\critpt(\chi,\tau);\chi,\tau)^{2n}}{\conformalprime(\chi,\tau)}\\
=L_n^+(\chi,\tau;M)[2\ii\Im(\critpt(\chi,\tau))]^{-2n}\conformalprime(\chi,\tau)^{-2n-1},
\label{eq:C-n-plus}
\end{multline}
and \eqref{eq:y-critpt} was used on the second line.
On the other hand, if $n=N(\chi,\tau;M)$ and $\mathfrak{s}=\mathfrak{S}(\chi,\tau;M)=-$, then for large $z$ we have $\breve{F}_{12}(z;\chi,\tau,M)=\solmat_{12}^-(z;p_n^-(\diamond;\chi,\tau,M))=R_{12}^-(z;p_n^-(\diamond;\chi,\tau,M))$, so using \eqref{G-sharp-form} and \eqref{Z-s-minus} gives
\begin{equation}
2\ii\lim_{z\to\infty}z\breve{F}_{12}(z;\chi,\tau,M)=2\ii Z^-_{12}(p_n^-(\diamond;\chi,\tau,M))=\frac{-8\ii\Im(\critpt)^2C_n^{-*}}{4\Im(\critpt)^2+|C_n^-|^2}=:\psi^-_n(\chi,\tau;M),
\end{equation}
where
\begin{multline}
C_n^-=C_n^-(\chi,\tau;M):=\mathop{\mathrm{Res}}_{z=\critpt(\chi,\tau)}p_n^-(z;\chi,\tau,M)=L_n^-(\chi,\tau;M)\frac{y(\critpt(\chi,\tau);\chi,\tau)^{-2n}}{\conformalprime(\chi,\tau)}\\
=L_n^-(\chi,\tau;M)[2\ii\Im(\critpt(\chi,\tau))]^{2n}\conformalprime(\chi,\tau)^{2n-1}.
\label{eq:C-n-minus}
\end{multline}

Although $\psi_n^\mathfrak{s}$ and $C_n^\mathfrak{s}$ are only used when $n=N(\chi,\tau;M)$ and $\mathfrak{s}=\mathfrak{S}(\chi,\tau;M)$, these expressions have meaning when $n,\mathfrak{s},\chi,\tau,M$ are all independent.  In this setting, we have the following identity, which is also related to Remark~\ref{r:Jost-Swap}:
\begin{lemma}
\label{L:s-relations-general}
$C^+_{n-1}(\chi,\tau;M)C_n^-(\chi,\tau;M)=-4\Im(\critpt(\chi,\tau))^2$ holds for $n\in\{1,2,\dots,\lfloor K_+\rceil\}$.
\end{lemma}
The proof will be given in Appendix~\ref{a:soliton-RHP}.  Lemma~\ref{L:s-relations-general} immediately implies the identity
\begin{equation}
\psi_n^-(\chi,\tau;M)=\psi_{n-1}^+(\chi,\tau;M),\quad n\in \{1,2,\dots,\lfloor K_+\rceil\}.
\end{equation}
Therefore, it is sufficient to work with the functions $\psi^+_n(\chi,\tau;M)$ for $n\in\{0,1,\dots,\lfloor K_+\rceil\}$.  These can be written equivalently in the form
\begin{equation}
\psi_n^+(\chi,\tau;M)=2\Im(\critpt(\chi,\tau))\mathrm{sech}(\Phi_n(\chi,\tau;M))\ee^{\ii\solitonphase_n(\chi,\tau;M)},
\label{eq:psi-n-plus}
\end{equation}
where
\begin{equation}
\Phi_n(\chi,\tau;M):=\ln\left(\frac{|C_n^+(\chi,\tau;M)|}{2\Im(\critpt(\chi,\tau))}\right),\quad \solitonphase_n(\chi,\tau;M):=\frac{\pi}{2}+\arg(C_n^+(\chi,\tau;M)).
\label{eq:Phi-n-theta-n}
\end{equation}
Since $\ii C_n^+(\chi,\tau;M)=2\Im(\critpt(\chi,\tau))D_n(\chi,\tau;M)$, where $D_n(\chi,\tau;M)$ is defined in \eqref{eq:Dn-def}, these expressions agree with those in \eqref{eq:intro-Phi-n-theta-n}, and $\psi_n^+(\chi,\tau;M)$ agrees with $\psi_n(\chi,\tau;M)$ defined in \eqref{eq:modulated-soliton}.

We therefore have shown that with $n=N(\chi,\tau;M)$,
\begin{equation}
M\Psi(M^2\chi,M^3\tau;\mathbf{G}(\ee^{-2M},\sqrt{1-\ee^{-4M}}))=\begin{cases}\psi_{n}^+(\chi,\tau;M)+O(M^{-\frac{1}{2}}),&\quad \mathfrak{s}(\chi,\tau;M)=+,\\
\psi_{n-1}^+(\chi,\tau;M)+O(M^{-\frac{1}{2}}),&\quad\mathfrak{s}(\chi,\tau;M)=-
\end{cases}
\label{eq:Psi-two-approximations}
\end{equation}
in the limit $M\to+\infty$ with the error being uniform on the region $\mathcal{S}$ specified by \eqref{d-size}.  
Using \eqref{eq:y-critpt}, \eqref{A-quantity}, \eqref{Ln-def} and \eqref{eq:C-n-plus} in the definition \eqref{eq:Phi-n-theta-n}, we have
\begin{multline}
\Phi_n(\chi,\tau;M)=M\left(\Re(2d(\chi,\tau))-\left(n+\frac{1}{2}\right)\frac{\ln(M)}{M}\right)\\
{}-\ln(2\pi\gamma_n^2)-(2n+1)\ln(2\Im(\critpt(\chi,\tau))|\conformalprime(\chi,\tau)|).
\label{eq:Phi-n-rewrite}
\end{multline}
Now the conditions $n=N(\chi,\tau;M)$ and $\mathfrak{S}(\chi,\tau;M)=+$ imply that $-\frac{1}{2}\ln(M)\le M\Re(2d(\chi,\tau))-(n+\frac{1}{2})\ln(M)\le 0$, and it follows that $-\frac{1}{2}\ln(M) + O(1)\le\Phi_n(\chi,\tau;M)\le O(1)$ as $M\to+\infty$.  Furthermore, for any integer $k\in\mathbb{Z}$, we have $-(k+\frac{1}{2})\ln(M)+O(1)\le\Phi_{n+k}(\chi,\tau;M)\le -k\ln(M)+O(1)$ in the same limit.  Provided that $k\neq 0$, we then have $|\Phi_{n+k}(\chi,\tau;M)|\gtrsim \frac{1}{2}\ln(M)$ and so $|\psi_{n+k}^+(\chi,\tau;M)|\lesssim M^{-\frac{1}{2}}$.  Therefore,
\begin{equation}
\psi_n^+(\chi,\tau;M)=\sum_{m=0}^{\lfloor K_+\rceil}\psi_m^+(\chi,\tau;M)-\mathop{\sum_{m=0}^{\lfloor K_+\rceil}}_{m\neq n}\psi_m^+(\chi,\tau;M) =     \sum_{m=0}^{\lfloor K_+\rceil}\psi_m^+(\chi,\tau;M) + O(M^{-\frac{1}{2}}).
\end{equation}

Likewise the conditions $n=N(\chi,\tau;M)$ and $\mathfrak{S}(\chi,\tau;M)=-$ imply that $0\le M\Re(2d(\chi,\tau))-((n-1)+\frac{1}{2})\ln(M)\le\frac{1}{2}\ln(M)$, and it follows that $O(1)\le \Phi_{n-1}(\chi,\tau;M)\le \frac{1}{2}\ln(M)+O(1)$, so also for fixed $k\in\mathbb{Z}$, $-k\ln(M)+O(1)\le\Phi_{n-1+k}(\chi,\tau;M)\le -(k-\frac{1}{2})\ln(M)+O(1)$ in the same limit.  Hence for $k\neq 0$ we have $|\psi_{n-1+k}^+(\chi,\tau;M)|\lesssim M^{-\frac{1}{2}}$ and so
\begin{equation}
\psi_{n-1}^+(\chi,\tau;M)=\sum_{m=0}^{\lfloor K_+\rceil}\psi_m^+(\chi,\tau;M)-\mathop{\sum_{m=0}^{\lfloor K_+\rceil}}_{m\neq n-1}\psi_m^+(\chi,\tau;M)=\sum_{m=0}^{\lfloor K_+\rceil}\psi_m^+(\chi,\tau;M)+O(M^{-\frac{1}{2}}).
\end{equation}
Using these results in \eqref{eq:Psi-two-approximations} shows that regardless of the value
of
$N(\chi,\tau;M)\in\{0,1,\dots,\lfloor K_+\rceil\}$ and of the sign $\mathfrak{S}(\chi,\tau;M)\in\{-,+\}$,
\begin{equation}
M\Psi(M^2\chi,M^3\tau;\mathbf{G}(\ee^{-2M},\sqrt{1-\ee^{-4M}}))=\sum_{m=0}^{\lfloor K_+\rceil}\psi_m^+(\chi,\tau;M)+O(M^{-\frac{1}{2}}),\quad M\to+\infty
\label{eq:edge-final-estimate}
\end{equation}
holds uniformly for $(\chi,\tau)\in\mathcal{S}$.  Appealing to the estimate of the left-hand side on arbitrary compacts $\compact$ consisting of points with $\chi<\chi_\mathrm{c}(\tau)$ that is part of Theorem~\ref{t:DS}, and noting that the finite sum on the right-hand side of \eqref{eq:edge-final-estimate} is exponentially small on $\compact$ extends the domain of validity of \eqref{eq:edge-final-estimate} to $\compact$ as well.  Restoring the phase factor $\ee^{-\ii\arg(ab)}$ completes the proof of Theorem~\ref{t:edge}.

\begin{remark}
The condition that the sum in \eqref{eq:edge-final-estimate} is finite is essential for the validity of the argument, to control the $O(1)$ terms in the bounds for $\Phi_{n+k}(\chi,\tau;M)$ and $\Phi_{n-1+k}(\chi,\tau;M)$ and hence guarantee that these phases are logarithmically large as $M\to+\infty$ for $k\neq 0$.  These $O(1)$ terms depend on $k$, and due to the term $-\ln(2\pi\gamma_n^2)$ in \eqref{eq:Phi-n-rewrite}, which grows rapidly with $n$, there is always some value of $n$ for which this $M$-independent becomes competitive with the otherwise dominant terms in \eqref{eq:Phi-n-rewrite} proportional to $M$ and $\ln(M)$.  This subtle point has been missed in some prior works (e.g., \cite{BuckinghamM2015,ClaeysG2010}).
\label{r:finite-sum}
\end{remark}

\subsection{Proof of Corollary~\ref{cor:soliton}}
\label{sec:local-behavior}
\begin{proof}
Let $n\in\{0,1,\dots,\lfloor K_+\rceil\}$ be fixed.  The maximum value of $|\psi_n(\chi,\tau;M)|$ is achieved on a curve depending on $M$ defined by the condition $\Phi_n(\chi,\tau;M)=0$.  Let $(\chi_0,\tau_0)$ denote a point on this curve.  We Taylor-expand $\Phi_n(\chi,\tau;M)$ and $\solitonphase_n(\chi,\tau;M)$ about this point using the fact that for each index pair $(k,l)$ with $k\ge 0$ and $l\ge 0$ but $k+l\ge 1$,
\begin{equation}
\partial_\chi^k\partial_\tau^l\Phi_n(\chi,\tau;M)=2M\Re(\partial_\chi^k\partial_\tau^l d(\chi,\tau)) + O(1)
\end{equation}
and
\begin{equation}
\partial_\chi^k\partial_\tau^l\solitonphase_n(\chi,\tau;M)=2M\Im(\partial_\chi^k\partial_\tau^l d(\chi,\tau)) + O(1).
\end{equation}
Thus, if $\chi-\chi_0=O(M^{-1})$ and $\tau-\tau_0=O(M^{-1})$, then since $\Phi_n(\chi_0,\tau_0;M)=0$,
\begin{equation}
\Phi_n(\chi,\tau;M)=\partial_\chi\Phi_n(\chi_0,\tau_0;M)(\chi-\chi_0) + \partial_\tau\Phi_n(\chi_0,\tau_0;M)(\tau-\tau_0) + O(M^{-1}).
\end{equation}
Likewise, setting $\solitonphase_n^0:=\solitonphase_n(\chi_0,\tau_0;M)$,
\begin{equation}
\solitonphase_n(\chi,\tau;M)=\solitonphase_n^0 + \partial_\chi\solitonphase_n(\chi_0,\tau_0;M)(\chi-\chi_0) + \partial_\tau\solitonphase_n(\chi_0,\tau_0;M)(\tau-\tau_0) + O(M^{-1}).
\end{equation}
Since $z=\critpt(\chi,\tau)$ is a critical point of $z\mapsto \phase(z;\chi,\tau)$, we have
\begin{align}
{\partial_\chi} d(\chi,\tau) &= -\ii  \phase_{\chi}(\critpt(\chi,\tau);\chi,\tau) = -\ii \critpt(\chi,\tau),\\
{\partial_\tau} d(\chi,\tau) &= -\ii  \phase_{\tau}(\critpt(\chi,\tau);\chi,\tau) = -\ii \critpt(\chi,\tau)^2.
\end{align}
Therefore, denoting $\critpt_0:=\critpt(\chi_0,\tau_0)$,
\begin{equation}
\Phi_n(\chi,\tau;M)=2M\Im(\critpt_0)(\chi-\chi_0) +4M\Im(\critpt_0)\Re(\critpt_0)(\tau-\tau_0)+O(M^{-1})
\end{equation}
and
\begin{equation}
\solitonphase_n(\chi,\tau;M)=\solitonphase_n^0-2M\Re(\critpt_0)(\chi-\chi_0)-2M[\Re(\critpt_0)^2-\Im(\critpt_0)^2](\tau-\tau_0) + O(M^{-1}).
\end{equation}
Therefore, replacing $(\chi,\tau)$ with $(\chi_0+\chi,\tau_0+\tau)$ and substituting into \eqref{eq:psi-n-plus} gives $\psi_n(\chi_0+\chi,\tau_0+\tau;M)=q_n(\chi,\tau;M) + O(M^{-1})$ for $\chi=O(M^{-1})$ and $\tau=O(M^{-1})$, where $q_n(\chi,\tau;M)$ is defined in \eqref{eq:q-define}.  This completes the proof.
\end{proof}

\appendix

\section{Proof of Lemma~\ref{lem:DS-T-bounds}}
\label{a:Lemma}
\begin{proof}[Proof of Lemma~\ref{lem:DS-T-bounds}]
By the second identity in \eqref{eq:automorphic}, we can take the supremum over $y\in [-\pi,\pi]$ instead of $y\in\mathbb{R}$.
Since $\Theta(x;H)^2>0$ and $\Theta(\ii y;H)^2>0$ for all variables under consideration, $T(H)$ is a continuous function of $H$. To estimate it in the limits $H\downarrow -\infty$ and $H\uparrow 0$.  As $\Theta(w^*;H)=\Theta(w;H)^*$ for real $H<0$ and complex $w$, we have the simplified estimate
\begin{equation}
T(H)\le 2\mathop{\sup_{\frac{3}{2}H\le x\le-\frac{3}{2}H}}_{-\pi\le y\le\pi}\left|\frac{\Theta(0;H)^2\Theta'(x+\ii y;H)\Theta(x-\ii y;H)}{\Theta(x;H)^2\Theta(\ii y;H)^2}\right|.
\label{eq:DS-T-bound-1}
\end{equation}
Now, for all $x,y\in\mathbb{R}$ and $H<0$, from \eqref{eq:DS-Theta-define} we have
\begin{equation}
|\Theta(x-\ii y;H)| = \left|\sum_{n\in\mathbb{Z}}\ee^{\frac{1}{2}n^2H}\ee^{nx}\ee^{-\ii n y}\right|\le \sum_{n\in\mathbb{Z}}\ee^{\frac{1}{2}n^2H}\ee^{nx} = \Theta(x;H).
\label{eq:DS-Theta-complex-estimate}
\end{equation}
Similarly,
\begin{equation}
|\Theta'(x+\ii y;H)|\le \sum_{n\in\mathbb{Z}}|n|\ee^{\frac{1}{2}n^2H}\ee^{nx},
\end{equation}
and using the inequality $|n|\le \ee^{|n|-1}$ valid for all $n\in\mathbb{Z}$, we get
\begin{equation}
\begin{split}
|\Theta'(x+\ii y;H)|\le \frac{1}{\ee}\sum_{n\in\mathbb{Z}}\ee^{\frac{1}{2}n^2H}\ee^{nx}\ee^{|n|}&\le \frac{1}{\ee}\sum_{n\in\mathbb{Z}}\ee^{\frac{1}{2}n^2H}\ee^{n(x+1)} +\frac{1}{\ee}\sum_{n\in\mathbb{Z}}\ee^{\frac{1}{2}n^2H}\ee^{n(x-1)} \\ &= \frac{1}{\ee}(\Theta(x+1;H)+\Theta(x-1;H)).
\end{split}
\label{eq:DS-Theta-prime-complex-estimate}
\end{equation}
Using \eqref{eq:DS-Theta-complex-estimate} and \eqref{eq:DS-Theta-prime-complex-estimate} in \eqref{eq:DS-T-bound-1} along with the first identity in \eqref{eq:automorphic} gives
\begin{equation}
T(H)\le \frac{4}{\ee}\mathop{\sup_{\frac{3}{2}H\le x\le -\frac{3}{2}H}}_{-\pi\le y\le\pi}\left|\frac{\Theta(0;H)^2\Theta(x+1;H)}{\Theta(x;H)\Theta(\ii y;H)^2}\right|.
\label{eq:DS-T-bound-2}
\end{equation}
Applying the dominated convergence theorem to \eqref{eq:DS-Theta-define} shows that 
\begin{equation}
\lim_{H\to -\infty}\Theta(\ii y;H) = 1,\quad\text{uniformly for $y\in\mathbb{R}$}.
\label{eq:DS-Theta-imaginary-asymptotic}
\end{equation}
And for real arguments $x\in\mathbb{R}$, to study the same limit we complete the square in the exponent of the summand in \eqref{eq:DS-Theta-define} to obtain
\begin{equation}
\ee^{x^2/(2H)}\Theta(x;H)=\sum_{n\in\mathbb{Z}}\ee^{\frac{1}{2}H(n+X)^2},\quad X:=\frac{x}{H}\in\mathbb{R}.
\label{eq:DS-Theta-completed-square}
\end{equation}
This is a periodic function of $X\in\mathbb{R}$ with period $1$.  Assuming without loss of generality that $-\frac{1}{2}<X\le\frac{1}{2}$, in the limit $H\downarrow -\infty$ the largest terms correspond to $n=-1,0,1$; subtracting them off gives:
\begin{equation}
\begin{split}
\sum_{n\in\mathbb{Z}}\ee^{\frac{1}{2}H(n+X)^2} - \ee^{\frac{1}{2}HX^2} -\ee^{\frac{1}{2}H(X-1)^2}-\ee^{\frac{1}{2}H(X+1)^2}&=
\ee^{\frac{1}{2}HX^2}\mathop{\sum_{n\in\mathbb{Z}}}_{|n|\ge 2}\ee^{\frac{1}{2}H(n^2+2nX)}\\
&\le 2\ee^{\frac{1}{2}HX^2}\sum_{n=2}^\infty\ee^{\frac{1}{2}H(n^2-n)},\quad -\tfrac{1}{2}<X\le\tfrac{1}{2}.
\end{split}
\end{equation}
Reindexing by $n=k+1$, using $k^2\ge k$ for $k=1,2,3,\dots$, and summing the resulting geometric series shows that 
\begin{equation}
\sum_{n\in\mathbb{Z}}\ee^{\frac{1}{2}H(n+X)^2}= \ee^{\frac{1}{2}HX^2}\left(1+\ee^{\frac{1}{2}H(1-2X)}+\ee^{\frac{1}{2}H(1+2X)}+O(\ee^{H})\right),\quad H\downarrow -\infty,
\end{equation}
uniformly for $-\frac{1}{2}<X\le\frac{1}{2}$.  We can use this result on the right-hand side of \eqref{eq:DS-Theta-completed-square} if we replace $X$ with the fractional part $\{x/H\}$ defined by $\{\diamond\}:=\diamond-[\diamond]\in (-\frac{1}{2},\frac{1}{2}]$ where $[\diamond]$ denotes the nearest integer function, rounding down at the half-integers:
\begin{equation}
\Theta(x;H)=\ee^{-x^2/(2H)}\ee^{\frac{1}{2}H\{x/H\}^2}\left(1+\ee^{\frac{1}{2}H(1-2\{x/H\})}+\ee^{\frac{1}{2}H(1+2\{x/H\})}+O(\ee^{H})\right),\quad H\downarrow -\infty,
\label{eq:DS-Theta-real-large-negative-H}
\end{equation}
with the result being uniformly valid for $x\in\mathbb{R}$.  Now notice that as $\{x/H\}$ varies between $-\frac{1}{2}$ and $\frac{1}{2}$,
\begin{equation}
1\le 1+\ee^{\frac{1}{2}H(1-2\{x/H\})} +\ee^{\frac{1}{2}H(1+2\{x/H\})}\le 2+\ee^{H}.
\end{equation}
Hence also
\begin{equation}
0<\frac{\Theta(x+1;H)}{\Theta(x;H)}\le(2+O(\ee^H))\ee^{-x/H}\ee^{-1/(2H)}\ee^{\frac{1}{2}H\{x/H+1/H\}^2}\ee^{-\frac{1}{2}H\{x/H\}^2},\quad H\downarrow -\infty.
\label{eq:DS-Shifted-Theta-Ratio-large-H}
\end{equation}
Now, unless $-\frac{1}{2}<\{x/H\}\le-\frac{1}{2}-\frac{1}{H}$, we will have $\{x/H+1/H\}=\{x/H\}+1/H$, and it follows that $\ee^{-1/(2H)}\ee^{\frac{1}{2}H\{x/H+1/H\}^2}\ee^{-\frac{1}{2}H\{x/H\}^2}=\ee^{\{x/H\}}\le \ee^{\frac{1}{2}}$.  Otherwise, $\{x/H\}=-\frac{1}{2}+O(H^{-1})$ and $\{x/H+1/H\}=\frac{1}{2}+O(H^{-1})$, and therefore $\ee^{-1/(2H)}\ee^{\frac{1}{2}H\{x/H+1/H\}^2}\ee^{-\frac{1}{2}H\{x/H\}^2}=O(1)$ as $H\downarrow -\infty$.  Therefore,
\begin{equation}
\frac{\Theta(x+1;H)}{\Theta(x;H)} = O(\ee^{-x/H}),\quad H\downarrow -\infty
\label{eq:DS-Theta-ratio-bound}
\end{equation}
holds uniformly for $x\in\mathbb{R}$.  Using \eqref{eq:DS-Theta-imaginary-asymptotic} and \eqref{eq:DS-Theta-ratio-bound} in \eqref{eq:DS-T-bound-2} proves the estimate \eqref{eq:DS-T-final-bound-H-large}.

To obtain asymptotics of $T(H)$ as $H\uparrow 0$ instead, we use the Poisson summation formula to convert \eqref{eq:DS-Theta-define} into a different series representation of $\Theta(w;H)$:
 \begin{equation}
 \Theta(w;H)=\sqrt{-\frac{2\pi}{H}}\sum_{n\in\mathbb{Z}}\ee^{2\pi^2(n-w/(2\pi\ii))^2/H} = \sqrt{-\frac{2\pi}{H}}\ee^{-w^2/(2H)}\Theta\left(\frac{2\pi\ii w}{H};\frac{4\pi^2}{H}\right),\quad H<0.
 \label{eq:DS-Theta-PS}
 \end{equation}
Combining this formula with the inequality \eqref{eq:DS-Theta-complex-estimate} gives
\begin{equation}
\begin{split}
|\Theta(x-\ii y;H)| &= \sqrt{-\frac{2\pi}{H}}\ee^{-x^2/(2H)}\ee^{y^2/(2H)}\left|\Theta\left(-\frac{2\pi y}{H}+\frac{2\pi\ii x}{H};\frac{4\pi^2}{H}\right)\right|\\
&\le\sqrt{-\frac{2\pi}{H}}\ee^{-x^2/(2H)}\ee^{y^2/(2H)}\Theta\left(-\frac{2\pi y}{H};\frac{4\pi^2}{H}\right)\\
&=\ee^{-x^2/(2H)}\Theta(\ii y;H).
\end{split}
\label{eq:DS-Theta-complex-estimate-2}
\end{equation}
Using this inequality in \eqref{eq:DS-T-bound-1} gives
\begin{equation}
\begin{split}
T(H)&\le 2\mathop{\sup_{\frac{3}{2}H\le x\le -\frac{3}{2}H}}_{-\pi\le y\le\pi}\ee^{-x^2/(2H)}\frac{\Theta(0;H)^2}{\Theta(x;H)^2}\left|\frac{\Theta'(x+\ii y;H)}{\Theta(\ii y;H)}\right|\\
&\le 2\ee^{-\frac{9}{8}H}\mathop{\sup_{\frac{3}{2}H\le x\le -\frac{3}{2}H}}_{-\pi\le y\le\pi}\frac{\Theta(0;H)^2}{\Theta(x;H)^2}\left|\frac{\Theta'(x+\ii y;H)}{\Theta(\ii y;H)}\right|.
\end{split}
\label{eq:DS-T-bound-3}
\end{equation}
Next, differentiation of \eqref{eq:DS-Theta-PS} yields the identity
\begin{equation}
\Theta'(w;H)=-\frac{w}{H}\sqrt{-\frac{2\pi}{H}}\ee^{-w^2/(2H)}\Theta\left(\frac{2\pi\ii w}{H};\frac{4\pi^2}{H}\right) + \frac{2\pi\ii}{H}\sqrt{-\frac{2\pi}{H}}\ee^{-w^2/(2H)}\Theta'\left(\frac{2\pi\ii w}{H};\frac{4\pi^2}{H}\right).
\end{equation}
Therefore, for all $x\in\mathbb{R}$ and $y\in [-\pi,\pi]$,
\begin{multline}
\left|\frac{\Theta'(x+\ii y;H)}{\Theta(\ii y;H)}\right|\le\frac{\sqrt{x^2+y^2}}{(-H)}\sqrt{-\frac{2\pi}{H}}\ee^{(y^2-x^2)/(2H)}\left|\frac{\displaystyle\Theta\left(-\frac{2\pi y}{H}+\frac{2\pi \ii x}{H};\frac{4\pi^2}{H}\right)}{\Theta(\ii y;H)}\right| \\
{}+
\left(-\frac{2\pi}{H}\right)^{\frac{3}{2}}\ee^{(y^2-x^2)/(2H)}\left|\frac{\displaystyle\Theta'\left(-\frac{2\pi y}{H} +\frac{2\pi \ii x}{H};\frac{4\pi^2}{H}\right)}{\Theta(\ii y;H)}\right|,
\end{multline}
and applying \eqref{eq:DS-Theta-complex-estimate-2} to the first term in the upper bound gives
\begin{equation}
\left|\frac{\Theta'(x+\ii y;H)}{\Theta(\ii y;H)}\right|\le\frac{\sqrt{x^2+y^2}}{(-H)}\ee^{-x^2/(2H)}
+
\left(-\frac{2\pi}{H}\right)^{\frac{3}{2}}\ee^{(y^2-x^2)/(2H)}\left|\frac{\displaystyle\Theta'\left(-\frac{2\pi y}{H} +\frac{2\pi \ii x}{H};\frac{4\pi^2}{H}\right)}{\Theta(\ii y;H)}\right|.
\label{eq:DS-dlogTheta-bound-small-H}
\end{equation}
To deal with the second term, we begin by rewriting \eqref{eq:DS-Theta-prime-complex-estimate} with the necessary substitutions and then use \eqref{eq:DS-Theta-PS} once again on the resulting right-hand side to obtain the inequality
\begin{multline}
\left|\Theta'\left(-\frac{2\pi y}{H}+\frac{2\pi\ii x}{H};\frac{4\pi^2}{H}\right)\right|\le\frac{1}{\ee}\left(\Theta\left(-\frac{2\pi y}{H}+1;\frac{4\pi^2}{H}\right) +\Theta\left(-\frac{2\pi y}{H}-1;\frac{4\pi^2}{H}\right)\right)\\
{}=\frac{1}{\ee}\sqrt{-\frac{H}{2\pi}}\ee^{-y^2/(2H)}\ee^{-H/(8\pi^2)}\left(\ee^{y/(2\pi)}\Theta\left(\ii y-\frac{\ii H}{2\pi};H\right) +
\ee^{-y/(2\pi)}\Theta\left(\ii y+\frac{\ii H}{2\pi};H\right)\right).
\label{eq:DS-ThetaPrime-small-H}
\end{multline}
Now, taking $w=\ii y$ for $y\in\mathbb{R}$ and combining \eqref{eq:DS-Theta-real-large-negative-H} and \eqref{eq:DS-Theta-PS} gives
\begin{multline}
\Theta(\ii y;H)=\sqrt{-\frac{2\pi}{H}}\ee^{2\pi^2\{-y/(2\pi)\}^2/H}\\
\cdot\left(1+\ee^{2\pi^2(1-2\{-y/(2\pi)\})/H} +\ee^{2\pi^2(1+2\{-y/(2\pi)\})/H} + O(\ee^{4\pi^2/H})\right),
\label{eq:DS-Theta-imaginary-small-H}
\end{multline}
as $H\uparrow 0$, uniformly for $y\in\mathbb{R}$, and we have the upper and lower bounds
\begin{equation}
1\le 1+\ee^{2\pi^2(1-2\{-y/(2\pi)\})/H}+\ee^{2\pi^2(1+2\{-y/(2\pi)\})/H}\le 2+\ee^{4\pi^2/H}.  
\label{eq:DS-upper-lower-bounds-2}
\end{equation}
Using these shows that, uniformly for $y\in\mathbb{R}$,
\begin{equation}
\frac{\displaystyle\Theta\left(\ii y\mp\frac{\ii H}{2\pi};H\right)}{\Theta(\ii y;H)}\le
\ee^{2\pi^2\{-y/(2\pi)\pm H/(4\pi^2)\}/H}\ee^{-2\pi^2\{-y/(2\pi)\}/H}\left(2+O(\ee^{4\pi^2/H})\right) = O(1)
\label{eq:DS-Theta-ratio-order-one-small-H}
\end{equation}
holds in the limit $H\uparrow 0$ (the second estimate follows from exactly the same reasoning as applied after \eqref{eq:DS-Shifted-Theta-Ratio-large-H} above).
Therefore, combining \eqref{eq:DS-dlogTheta-bound-small-H}, \eqref{eq:DS-ThetaPrime-small-H}, and 
\eqref{eq:DS-Theta-ratio-order-one-small-H} gives
\begin{equation}
\mathop{\sup_{\frac{3}{2}H\le x\le -\frac{3}{2}H}}_{-\pi\le y\le \pi}\left|\frac{\Theta'(x+\ii y;H)}{\Theta(\ii y;H)}\right|=O(H^{-1}),\quad H\uparrow 0.
\label{eq:DS-Theta-frac-estimate}
\end{equation}
Finally, taking $w=x\in\mathbb{R}$ and combining \eqref{eq:DS-Theta-imaginary-asymptotic} with \eqref{eq:DS-Theta-PS} gives
\begin{equation}
\Theta(x;H)=\sqrt{-\frac{2\pi}{H}}\ee^{-x^2/(2H)}\left(1+o(1)\right),\quad H\uparrow 0
\end{equation}
uniformly for $x\in\mathbb{R}$.  
This immediately yields that uniformly for $x\in\mathbb{R}$,
\begin{equation}
\frac{\Theta(0;H)^2}{\Theta(x;H)^2} = \ee^{x^2/H}(1+o(1)),\quad H\uparrow 0.
\end{equation}
Using this along with \eqref{eq:DS-Theta-frac-estimate} in \eqref{eq:DS-T-bound-3} proves the estimate \eqref{eq:DS-T-final-bound-H-small} and finishes the proof.
\end{proof}

\section{Alternative {formula} for $|\breve{\Psi}(\chi,\tau;M)|^2$}
\label{a:Landen}
Although the alternate elliptic parameter $m_1$ defined in \eqref{eq:DS-m1-define} is apparently not naturally associated to the underlying elliptic curve $\mathcal{R}$ whose theta function theory is used to solve Riemann-Hilbert Problem~\ref{rhp:O-out}, it is the parameter most naturally linking the theta functions of $\mathcal{R}$ to Jacobi elliptic functions.  This is best seen using the \emph{descending Landen transformation} \cite[Eqn.\@ 19.8.12]{DLMF}:
\begin{equation}
\mathbb{K}(m)=\frac{2}{1+\sqrt{1-m}}\mathbb{K}\left(\left(\frac{1-\sqrt{1-m}}{1+\sqrt{1-m}}\right)^2\right),\quad 0<m<1.
\end{equation}
Replacing $m$ in this identity with $1-((1-\sqrt{1-m})/(1+\sqrt{1-m}))^2$ yields
\begin{equation}
\mathbb{K}\left(1-\left(\frac{1-\sqrt{1-m}}{1+\sqrt{1-m}}\right)^2\right)=(1+\sqrt{1-m})\mathbb{K}(1-m),\quad 0<m<1.
\end{equation}
Therefore,
\begin{equation}
\frac{\mathbb{K}(1-m)}{\mathbb{K}(m)}=\frac{1}{2}\frac{\displaystyle\mathbb{K}\left(1-\left(\frac{1-\sqrt{1-m}}{1+\sqrt{1-m}}\right)^2\right)}{\displaystyle\mathbb{K}\left(\left(\frac{1-\sqrt{1-m}}{1+\sqrt{1-m}}\right)^2\right)},\quad 0<m<1.
\end{equation}
With another substitution, this can be written as
\begin{equation}
\frac{\mathbb{K}(1-m)}{\mathbb{K}(m)}=2\frac{\displaystyle\mathbb{K}\left(1-\frac{4\sqrt{m}}{(1+\sqrt{m})^2}\right)}{\displaystyle\mathbb{K}\left(\frac{4\sqrt{m}}{(1+\sqrt{m})^2}\right)},\quad 0<m<1.
\end{equation} 
Consequently, according to \eqref{eq:DS-H-def}, the period ratio appearing in the solution of Riemann-Hilbert Problem~\ref{rhp:O-out} is
\begin{equation}
H=-\pi\frac{\mathbb{K}(1-m)}{\mathbb{K}(m)} = -2\pi\frac{\displaystyle\mathbb{K}\left(1-\frac{4\sqrt{m}}{(1+\sqrt{m})^2}\right)}{\displaystyle\mathbb{K}\left(\frac{4\sqrt{m}}{(1+\sqrt{m})^2}\right)}.
\end{equation}
Taking $m=\cot^2(\frac{1}{2}\theta_\alpha)\tan^2(\frac{1}{2}\theta_\beta)$ as indicated in \eqref{eq:DS-H-def} then implies that 
\begin{equation}
H=-2\pi\frac{\mathbb{K}(1-m_1)}{\mathbb{K}(m_1)},\quad m_1 := \frac{4\sqrt{m}}{(1+\sqrt{m})^2}=\frac{4\cot(\frac{1}{2}\theta_\alpha)\tan(\frac{1}{2}\theta_\beta)}{(1+\cot(\frac{1}{2}\theta_\alpha)\tan(\frac{1}{2}\theta_\beta))^2}=\frac{\sin(\theta_\alpha)\sin(\theta_\beta)}{\sin^2(\frac{1}{2}(\theta_\alpha+\theta_\beta))}.
\label{eq:DS-m1}
\end{equation}
With this in hand, we can return to \eqref{eq:DS-q-ratio} and directly try to write the ratios of theta functions in terms of Jacobi elliptic functions with parameter $m_1$.  In this Appendix, we will use this approach to prove the alternate formula \eqref{eq:DS-m1-define}--\eqref{eq:DS-Breve-Psi-alt} for $|\breve{\Psi}(\chi,\tau;M)|^2$.

To this end, first write \eqref{eq:DS-Psi-breve-define}--\eqref{eq:DS-q-ratio} as follows:
$\breve{\Psi}(\chi,\tau;M)= \ee^{-\ii M\phi}FG$,
where
\begin{equation}
F:=\ii (\mathrm{Im}(\beta)-\mathrm{Im}(\alpha))\frac{\Theta(A(\infty)+A(z_0)+\mathcal{K};H)}{\Theta(A(\infty)-A(z_0)-\mathcal{K};H)},\quad
G:=\frac{\Theta(A(\infty)-A(z_0)-\mathcal{K}+\ii M\Delta;H)}{\Theta(A(\infty)+A(z_0)+\mathcal{K}-\ii M\Delta;H)}.
\end{equation}
Then, using the fact that $\mathrm{Im}(A(\infty))=\mathrm{Im}(A(z_0))=-\frac{1}{2}\pi$ we can write $A(\infty)=\frac{1}{2}\eta-\frac{1}{2}\ii\pi$ and $A(z_0)=\frac{1}{2}\gamma-\frac{1}{2}\ii\pi$ where $\eta$ and $\gamma$ are real.  The latter are related by \eqref{eq:Az0-identity} in which $n_1=-1+2\mu$ and $n_2=1+2\nu$ for $(\mu,\nu)\in\mathbb{Z}^2$.  Therefore,
\begin{equation}
A(\infty)+A(z_0)=(2\mu-1)\ii\pi + (\tfrac{1}{2}+\nu)H\quad\text{and}\quad
A(\infty)-A(z_0)=\eta -2\pi\ii\mu -(\tfrac{1}{2}+\nu)H.
\end{equation}
Using also \eqref{eq:RiemannConstant} and $\Theta(w+2\pi\ii;H)=\Theta(w;H)$ gives
\begin{equation}
F=\ii (\mathrm{Im}(\beta)-\mathrm{Im}(\alpha))\frac{\Theta((\nu+1)H;H)}{\Theta(\eta-\ii\pi -(\nu+1)H;H)},\quad
G=\frac{\Theta(\eta-\ii\pi+\ii M\Delta-(\nu+1)H;H)}{\Theta(-\ii M\Delta +(\nu+1)H;H)}.
\end{equation}
Next, using the ``iterated'' identity $\Theta(w+nH;H)=\ee^{-\frac{1}{2}n^2H-nw}\Theta(w;H)$ for $n\in\mathbb{Z}$ gives
\begin{equation}
\begin{split}
F&=\ii(\mathrm{Im}(\beta)-\mathrm{Im}(\alpha))\ee^{-(\nu+1)(\eta-\ii\pi)}\frac{\Theta(0;H)}{\Theta(\eta-\ii\pi;H)}\\
G&=\ee^{(\nu+1)(\eta-\ii\pi)}\frac{\Theta(\eta-\ii\pi +\ii M\Delta;H)}{\Theta(-\ii M\Delta;H)}.
\end{split}
\end{equation}
Hence also 
$\breve{\Psi}(\chi,\tau;M)=\ee^{-\ii M\phi}\widetilde{F}\widetilde{G}$
where
\begin{equation}
\widetilde{F}:= \ii(\mathrm{Im}(\beta)-\mathrm{Im}(\alpha))\frac{\Theta(0;H)}{\Theta(\eta-\ii\pi;H)}\quad\text{and}\quad
\widetilde{G}:=\frac{\Theta(\eta-\ii\pi +\ii M\Delta;H)}{\Theta(-\ii M\Delta;H)}.
\end{equation}
Note that $\widetilde{F}$ is purely imaginary.  Therefore
\begin{equation}
|\widetilde{F}|^2 = (\mathrm{Im}(\beta)-\mathrm{Im}(\alpha))^2\frac{\Theta(0;H)^2}{\Theta(\eta-\ii\pi;H)^2}.
\end{equation}
On the other hand, $\widetilde{G}$ has a variable phase depending on $M\Delta$, so from \eqref{eq:DS-Theta-define} and $\Theta(-w;H)=\Theta(w;H)$,
\begin{equation}
|\widetilde{G}|^2 = \frac{\Theta(\eta-\ii\pi+\ii M\Delta;H)\Theta(\eta+\ii\pi-\ii M\Delta;H)}{\Theta(\ii M\Delta;H)^2}.
\end{equation}
Next, using the addition formula (deduced from \cite[Eqn.\@ 20.7.8]{DLMF})
\begin{multline}
\Theta(w+z;H)\Theta(w-z;H)\\{}=\frac{\Theta(w+\ii\pi;H)^2\Theta(z;H)^2 + \ee^{\frac{1}{2}H+w+z}\Theta(w+\ii\pi +\frac{1}{2}H;H)^2\Theta(z+\frac{1}{2}H;H)^2}{\Theta(\ii\pi;H)^2}
\end{multline}
with $w=\eta$ and $z=\ii M\Delta-\ii\pi$ gives
\begin{multline}
|\widetilde{G}|^2=\frac{\Theta(\eta+\ii\pi;H)^2\Theta(\ii M\Delta-\ii\pi;H)^2}{\Theta(\ii\pi;H)^2\Theta(\ii M\Delta;H)^2} \\{}+ \ee^{\frac{1}{2}H +\eta-\ii\pi+\ii M\Delta}\frac{\Theta(\eta+\ii\pi +\frac{1}{2}H;H)^2\Theta(\ii M\Delta-\ii\pi+\frac{1}{2}H;H)^2}{\Theta(\ii\pi;H)^2\Theta(\ii M\Delta;H)^2}.
\end{multline}
Now we use the identities
\begin{equation}
\frac{\Theta(W+\ii\pi;H)}{\Theta(W;H)}=\frac{\Theta(0;H)}{\Theta(\ii\pi;H)}\mathrm{dn}\left(\frac{\mathbb{K}(m_1)}{\ii\pi}(W+\ii\pi);m_1\right)
\label{eq:DS-dn}
\end{equation}
and
\begin{equation}
\frac{\Theta(W+\frac{1}{2}H+\ii\pi;H)}{\Theta(W;H)} =- \ii\ee^{-\frac{1}{2}W}\frac{\Theta(-\frac{1}{2}H;H)}{\Theta(\ii\pi;H)}\mathrm{cn}\left(\frac{\mathbb{K}(m_1)}{\ii\pi}(W+\ii\pi);m_1\right),
\label{eq:DS-cn}
\end{equation}
along with $\Theta(W+2\pi\ii;H)=\Theta(W;H)$ to give
\begin{multline}
|\widetilde{G}|^2=\frac{\Theta(\eta+\ii\pi;H)^2\Theta(0;H)^2}{\Theta(\ii\pi;H)^4}\mathrm{dn}^2\left(\frac{\mathbb{K}(m_1)}{\pi}(M\Delta+\pi);m_1\right)\\
{}+\ee^{\frac{1}{2}H+\eta}\frac{\Theta(\eta+\ii\pi+\frac{1}{2}H;H)^2\Theta(\frac{1}{2}H;H)^2}{\Theta(\ii\pi;H)^4}\mathrm{cn}^2\left(\frac{\mathbb{K}(m_1)}{\pi}(M\Delta+\pi);m_1\right).
\end{multline}
Next, using \eqref{eq:DS-sn} with $H$ replaced by $H/2$ gives 
\begin{multline}
|\widetilde{F}\widetilde{G}|^2 = (\mathrm{Im}(\beta)-\mathrm{Im}(\alpha))^2\frac{\Theta(0;H)^4}{\Theta(\ii\pi;H)^4}\mathrm{dn}^2\left(\frac{\mathbb{K}(m_1)}{\pi}(M\Delta+\pi);m_1\right)\\
-\ee^{\frac{1}{2}H}(\mathrm{Im}(\beta)-\mathrm{Im}(\alpha))^2\frac{\Theta(\frac{1}{2}H;H)^4}{\Theta(\ii\pi;H)^4}\mathrm{sn}^2\left(\frac{\mathbb{K}(m_1)}{\ii\pi}\eta;m_1\right)\mathrm{cn}^2\left(\frac{\mathbb{K}(m_1)}{\pi}(M\Delta+\pi);m_1\right).
\end{multline}
Using \cite[Eqn.\@ 20.9.1]{DLMF} we have
\begin{equation}
m_1=\ee^{\frac{1}{2}H}\frac{\Theta(\frac{1}{2}H;H)^4}{\Theta(0;H)^4},
\label{eq:DS-m1-identity}
\end{equation}
so 
\begin{multline}
|\widetilde{F}\widetilde{G}|^2=(\mathrm{Im}(\beta)-\mathrm{Im}(\alpha))^2\frac{\Theta(0;H)^4}{\Theta(\ii\pi;H)^4}\left[\mathrm{dn}^2\left(\frac{\mathbb{K}(m_1)}{\pi}(M\Delta+\pi);m_1\right) \right.\\\left.{}-m_1\mathrm{sn}^2\left(\frac{\mathbb{K}(m_1)}{\ii\pi}\eta;m_1\right)\mathrm{cn}^2\left(\frac{\mathbb{K}(m_1)}{\pi}(M\Delta+\pi);m_1\right)\right].
\end{multline}
Combining \eqref{eq:DS-m1-identity} with \cite[Eqn.\@ 20.7.5]{DLMF} shows that also
\begin{equation}
\frac{\Theta(\ii\pi;H)^4}{\Theta(0;H)^4}=1-m_1.
\end{equation}
Using this, as well as $\mathrm{dn}^2(\cdot;m_1)=1-m_1\mathrm{sn}^2(\cdot;m_1)$ and $\mathrm{cn}^2(\cdot;m_1)=1-\mathrm{sn}^2(\cdot;m_1)$,
together with $\mathrm{sn}^2(-\ii \Jacobiu;m_1)=-\mathrm{sc}^2(\Jacobiu;1-m_1)$ (see \cite[\S22.6(iv)]{DLMF}),
\begin{multline}
|\widetilde{F}\widetilde{G}|^2=\frac{(\mathrm{Im}(\beta)-\mathrm{Im}(\alpha))^2}{1-m_1}\left[
\left(1+m_1\mathrm{sc}^2(\Jacobiu_1;1-m_1)\right)\right.\\
\left.{}-m_1\left(1+\mathrm{sc}^2(\Jacobiu_1;1-m_1)\right)\mathrm{sn}^2(\Jacobiv_1;m_1)\right],
\label{eq:DS-Abs-Squared-REDO}
\end{multline}
wherein 
\begin{equation}
\Jacobiu_1:=\frac{\mathbb{K}(m_1)}{\pi}\eta,\quad\eta=\mathrm{Re}(2A(\infty)),\quad \Jacobiv_1:=\frac{\mathbb{K}(m_1)}{\pi}(M\Delta+\pi).
\end{equation}

Comparing with \eqref{eq:DS-u-eta-v}, we have
\begin{equation}
\Jacobiu_1=\frac{\mathbb{K}(m_1)}{\mathbb{K}(m)}\Jacobiu = (1+\sqrt{m})\Jacobiu.
\end{equation}
The substitution behind the descending Landen transformation in the form
\begin{equation}
\mathbb{K}(m)=\int_0^1\frac{\dd W}{\sqrt{1-W^2}\sqrt{1-mW^2}}=\frac{1}{1+\sqrt{m}}\int_0^1\frac{\dd t}{\sqrt{1-t^2}\sqrt{1-m_1t^2}}=\frac{\mathbb{K}(m_1)}{1+\sqrt{m}}
\end{equation}
is 
\begin{equation}
W=\frac{1+\sqrt{m}-\sqrt{(1+\sqrt{m})^2-4\sqrt{m}t^2}}{2\sqrt{m}t}\implies t=\frac{(1+\sqrt{m})W}{1+\sqrt{m}W^2}.
\end{equation}
Using this same transformation in the formula $\Jacobiu=U(\ii\tan(\frac{1}{2}\theta_\alpha))$ where $U(\zeta)$ is defined by \eqref{eq:DS-U-of-zeta} gives $\Jacobiu_1=U_1(\ii\tan(\frac{1}{2}(\theta_\alpha+\theta_\beta)))$ wherein
\begin{equation}
U_1(\zeta):=\ii\int_0^\zeta\frac{\dd t}{\sqrt{1-t^2}\sqrt{1-m_1t^2}}.
\end{equation}
Note that in computing $\Jacobiu_1$, the upper limit of integration can take any purely imaginary value; hence to make the integral well-defined we can assume that the path of integration lies on the imaginary axis, proceeding upwards from the origin and possibly continuing through the integrable singularity at $t=\infty$ should the upper limit be negative imaginary.  Then, by exactly the same arguments as were used to analyze $r(\zeta):=\mathrm{sn}^2(U(\zeta);1-m)$ in Section~\ref{s:DS-Interpretation} we have the analogous result:
\begin{equation}
r_1(\zeta):=\mathrm{sn}^2(U_1(\zeta);1-m_1)=\frac{\zeta^2}{\zeta^2-1}.
\end{equation}
Taking $\zeta=\ii\tan(\frac{1}{2}(\theta_\alpha+\theta_\beta))$ yields
\begin{equation}
\mathrm{sc}^2(\Jacobiu_1;1-m_1)=\frac{\mathrm{sn}^2(\Jacobiu_1;1-m_1)}{1-\mathrm{sn}^2(\Jacobiu_1;1-m_1)}=\frac{r_1(\ii\tan(\frac{1}{2}(\theta_\alpha+\theta_\beta)))}{1-r_1(\ii\tan(\frac{1}{2}(\theta_\alpha+\theta_\beta)))} =\mathrm{tan}^2(\tfrac{1}{2}(\theta_\alpha+\theta_\beta)).
\end{equation}
Using the notation $t_\alpha:=\tan(\frac{1}{2}\theta_\alpha)$ and $t_\beta:=\tan(\frac{1}{2}\theta_\beta)$ we then have by an addition formula
\begin{equation}
\mathrm{sc}^2(\Jacobiu_1;1-m_1)=\tan^2(\tfrac{1}{2}(\theta_\alpha+\theta_\beta))=\frac{(t_\alpha+t_\beta)^2}{(1-t_\alpha t_\beta)^2}.
\end{equation}
Then using the definition of $m_1$ in \eqref{eq:DS-m1} gives
\begin{equation}
m_1=\frac{4t_\alpha^{-1}t_\beta}{(1+t_\alpha^{-1} t_\beta)^2} = \frac{4t_\alpha t_\beta}{(t_\alpha+t_\beta)^2}\implies \frac{1}{1-m_1}=\frac{(t_\alpha+t_\beta)^2}{(t_\alpha-t_\beta)^2}.
\end{equation}
Hence,
\begin{equation}
\frac{1+m_1\mathrm{sc}^2(\Jacobiu_1;1-m_1)}{1-m_1} = \frac{(t_\alpha+t_\beta)^2}{(t_\alpha-t_\beta)^2}\left(1+\frac{4t_\alpha t_\beta}{(1-t_\alpha t_\beta)^2} \right)= 
\frac{(t_\alpha+t_\beta)^2(1+t_\alpha t_\beta)^2}{(t_\alpha-t_\beta)^2(1-t_\alpha t_\beta)^2}
\end{equation}
and
\begin{equation}
\frac{m_1\left(1+\mathrm{sc}^2(\Jacobiu_1;1-m_1)\right)}{1-m_1}=\frac{4t_\alpha t_\beta}{(t_\alpha-t_\beta)^2}\left(1+\frac{(t_\alpha+t_\beta)^2}{(1-t_\alpha t_\beta)^2}\right)=\frac{4t_\alpha t_\beta (1+t_\alpha^2)(1+t_\beta^2)}{(t_\alpha-t_\beta)^2(1-t_\alpha t_\beta)^2}.
\end{equation}
Then also,
\begin{equation}
(\mathrm{Im}(\beta)-\mathrm{Im}(\alpha))^2 = \rho^2(\sin(\theta_\beta)-\sin(\theta_\alpha))^2
=4\rho^2\frac{(t_\beta-t_\alpha)^2(1-t_\alpha t_\beta)^2}{(1+t_\alpha^2)^2(1+t_\beta^2)^2}.
\end{equation}
Using these results in \eqref{eq:DS-Abs-Squared-REDO} gives
\begin{equation}
|\widetilde{F}\widetilde{G}|^2=\frac{4\rho^2(t_\alpha+t_\beta)^2(1+t_\alpha t_\beta)^2}{(1+t_\alpha^2)^2(1+t_\beta^2)^2} - \frac{16\rho^2 t_\alpha t_\beta}{(1+t_\alpha^2)(1+t_\beta^2)} \mathrm{sn}^2(\Jacobiv_1;m_1).
\end{equation}
Then observe that
\begin{equation}
\frac{4\rho^2(t_\alpha+t_\beta)^2(1+t_\alpha t_\beta)^2}{(1+t_\alpha^2)^2(1+t_\beta^2)^2}
=\left(\rho\sin(\theta_\alpha)+\rho\sin(\theta_\beta)\right)^2
= \left(\mathrm{Im}(\alpha)+\mathrm{Im}(\beta)\right)^2,
\end{equation}
and 
\begin{equation}
\frac{16\rho^2t_\alpha t_\beta}{(1+t_\alpha^2)(1+t_\beta^2)} 
=4\rho^2\sin(\theta_\alpha)\sin(\theta_\beta)
=4\mathrm{Im}(\alpha)\mathrm{Im}(\beta).
\end{equation}
Since $|\breve{\Psi}(\chi,\tau;M)|^2 = |\widetilde{F}\widetilde{G}|^2$, this proves \eqref{eq:DS-Breve-Psi-alt}.

\section{Solution of \rhref{rhp:generic-soliton-parametrix} and proofs of related properties}
\label{a:soliton-RHP}
\begin{proof}[Proof of Proposition~\ref{p:generic-soliton-parametrix}]
We first produce a solution of \rhref{rhp:generic-soliton-parametrix}. The fact that this solution is unique follows by a straightforward application of Liouville's theorem using the fact that the jump matrix has unit determinant.

Observe that
\begin{equation}
\mathbf{Q}^{\mathfrak{s}} (z):=
\begin{cases}
\mathbf{\solmat}^{\mathfrak{s}}(z;p)\mathbf{V}^{\mathbf{\solmat}^\mathfrak{s}}(z;p),&\quad z\in (D_\critpt\cup D_{\critpt}^*)\setminus\{\critpt,\critpt^*\},\\
\mathbf{\solmat}^{\mathfrak{s}}(z;p), &\quad z\in \mathbb{C}\setminus {\overline{D_\critpt\cup D_\critpt^*}},
\end{cases}
\label{Fdot-to-Q}
\end{equation}
extends to $\mathbb{C}\setminus\{\critpt,\critpt^*\}$ as a meromorphic matrix function with simple poles at $z=\critpt$ and $z=\critpt^*$ only, thanks to the analyticity properties of ${\mathbf{\solmat}}^\mathfrak{s}(z;p)$ and the fact that the jump matrix ${\mathbf{V}}^{\mathbf{\solmat}^\mathfrak{s}}(z;p)$ defined in \eqref{eq:VJ-jump} has simple poles at $z=\critpt,\critpt^*$ and is otherwise analytic in the disks.
Therefore, $\mathbf{Q}^{\mathfrak{s}}(z)$ necessarily has the form
\begin{equation}
\mathbf{Q}^{\mathfrak{s}}(z) = \mathbb{I} + \frac{\mathbf{X}^{\mathfrak{s}}}{z-\critpt}+\frac{\mathbf{Z}^{\mathfrak{s}}}{z-\critpt^*},
\label{Q-form}
\end{equation}
where 
$\mathbf{X}^\mathfrak{s}$ and $\mathbf{Z}^\mathfrak{s}$ 
are constant matrices to be determined.  Because the conditions of \rhref{rhp:generic-soliton-parametrix}  are symmetric with respect to Schwarz reflection, $\mathbf{X}^\mathfrak{s}$ and $\mathbf{Z}^\mathfrak{s}$ are related by 
\begin{equation}
\mathbf{Z}^{\mathfrak{s}}=\sigma_2\mathbf{X}^{\mathfrak{s}*}\sigma_2.
\label{Z-X-symmetry}
\end{equation}
As $\mathbf{\solmat}^{\mathfrak{s}}(z;p)$ is analytic at $z=\critpt,\critpt^*$, the condition
\begin{align}
\mathbf{Q}^{\mathfrak{s}}(z) \mathbf{V}^{\mathbf{\solmat}^\mathfrak{s}} (z;p)^{-1} = O(1),\quad z\to \critpt,\critpt^*,
\label{error-jump-expand-xi}
\end{align}
must hold. In what follows we focus on the limit $z\to\critpt$ as what happens as $z\to\critpt^*$ follows from the symmetry in the problem.
Since $\mathbf{N}^\mathfrak{s}$ is $2$-nilpotent, from \eqref{eq:VJ-jump} we have
\begin{equation}
\mathbf{V}^{\mathbf{\solmat}^\mathfrak{s}}(z;p)^{-1}=\mathbb{I}-p(z)\mathbf{N}^\mathfrak{s} = -C\mathbf{N}^\mathfrak{s}(z-\critpt)^{-1} + \mathbb{I}-p_0\mathbf{N}^\mathfrak{s} + O(z-\critpt),\quad z\to\critpt,
\label{V-F-dot-xi-expand}
\end{equation}
where $p(z)$ has the Laurent expansion $p(z)=C(z-\critpt)^{-1}+p_0 + O(z-\critpt)$ (see \eqref{C-residue-def}).
Then, expanding \eqref{Q-form} as $z\to\critpt$, we see that
\begin{equation}
\begin{aligned}
\mathbf{Q}^{\mathfrak{s}}(z)\mathbf{V}^{\mathbf{\solmat}^\mathfrak{s}}(z;p)^{-1}
&=-C\mathbf{X}^{\mathfrak{s}} \mathbf{N}^{\mathfrak{s}}(z-\critpt)^{-2} +\left[ \mathbf{X}^{\mathfrak{s}} (\mathbb{I}-p_0\mathbf{N}^\mathfrak{s}) -C\left( \mathbb{I} + \frac{\mathbf{Z}^{\mathfrak{s}} }{\critpt-\critpt^*}\right)\mathbf{N}^{\mathfrak{s}} \right](z-\critpt)^{-1} + O(1).
\end{aligned}
\label{analytic-expand-xi}
\end{equation}
Comparing with \eqref{error-jump-expand-xi}, we arrive at the following linear system of equations for the unknown matrices 
$\mathbf{X}^{\mathfrak{s}}$ and $\mathbf{Z}^{\mathfrak{s}}$:
\begin{align}
C\mathbf{X}^{\mathfrak{s}} \mathbf{N}^{\mathfrak{s}} &= \mathbf{0},\label{neg2-term}\\
 \mathbf{X}^{\mathfrak{s}}(\mathbb{I}-p_0\mathbf{N}^\mathfrak{s})-C\left( \mathbb{I} + \frac{\mathbf{Z}^{\mathfrak{s}}}{\critpt-\critpt^*}\right)\mathbf{N}^{\mathfrak{s}} &= \mathbf{0}.
\label{neg1-term}
\end{align}
Because $\mathbf{N}^\mathfrak{s}$ is $2$-nilpotent, multiplying \eqref{neg1-term} on the right by $(\mathbb{I}-p_0\mathbf{N}^\mathfrak{s})^{-1}=\mathbb{I}+p_0\mathbf{N}^\mathfrak{s}$ allows $\mathbf{X}^\mathfrak{s}$ to be explicitly expressed in terms of $\mathbf{Z}^\mathfrak{s}$ as:
\begin{equation}
\mathbf{X}^\mathfrak{s}=C\left(\mathbb{I}+\frac{\mathbf{Z}^\mathfrak{s}}{\critpt-\critpt^*}\right)\mathbf{N}^\mathfrak{s}.
\end{equation}
The equation \eqref{neg2-term} gives no additional information, but using \eqref{Z-X-symmetry} to eliminate $\mathbf{Z}^\mathfrak{s}$ gives 
\begin{equation}
\mathbf{X}^\mathfrak{s}=C\left(\mathbb{I}+\frac{\sigma_2\mathbf{X}^{\mathfrak{s}*}\sigma_2}{\critpt-\critpt^*}\right)\mathbf{N}^\mathfrak{s}.
\end{equation}
In the case $\mathfrak{s}=-$, using \eqref{2-nilpotents} shows that the second column of $\mathbf{X}^-$ vanishes, and the first column then reads
\begin{equation}
\begin{bmatrix}X^-_{11}\\X^-_{21}\end{bmatrix}=C\begin{bmatrix}\displaystyle -\frac{X_{21}^{-*}}{2\ii\Im(\critpt)}\\\displaystyle 1+\frac{X_{11}^{-*}}{2\ii\Im(\critpt)}\end{bmatrix}\implies
X_{11}^-=\frac{2\ii\Im(\critpt)|C|^2}{4\Im(\critpt)^2+|C|^2},\quad X_{21}^-=\frac{4\Im(\critpt)^2C}{4\Im(\critpt)^2+|C|^2},
\end{equation}
which together with \eqref{Z-X-symmetry} then proves \eqref{Z-s-minus}.  If instead $\mathfrak{s}=+$, then the first column of $\mathbf{X}^+$ vanishes and the second column then reads
\begin{equation}
\begin{bmatrix}X^+_{12}\\X^+_{22}\end{bmatrix}=C\begin{bmatrix}\displaystyle 1+\frac{X_{22}^{+*}}{2\ii\Im(\critpt)}\\\displaystyle -\frac{X_{12}^{+*}}{2\ii\Im(\critpt)}\end{bmatrix}\implies X_{12}^-=\frac{4\Im(\critpt)^2C}{4\Im(\critpt)^2+|C|^2},\quad X_{22}^-=\frac{2\ii\Im(\critpt)|C|^2}{4\Im(\critpt)^2+|C|^2},
\end{equation}
which together with \eqref{Z-X-symmetry} proves \eqref{X-s-plus}.
\end{proof}

\begin{proof}[Proof of Lemma~\ref{L:s-relations-general}]
From the definitions \eqref{eq:C-n-plus} and \eqref{eq:C-n-minus} we get
\begin{equation}
C_{n-1}^+(\chi,\tau;M)C_n^-(\chi,\tau;M)=\frac{L_{n-1}^+(\chi,\tau;M)L_{n}^-(\chi,\tau;M)}{y(\critpt(\chi,\tau);\chi,\tau)^2\conformalprime(\chi,\tau)^2}.
\end{equation}
Using \eqref{eq:y-critpt}, this can be written in the form
\begin{equation}
C_{n-1}^+(\chi,\tau;M)C_n^-(\chi,\tau;M)=-4\Im(\critpt(\chi,\tau))^2L_{n-1}^+(\chi,\tau;M)L_n^-(\chi,\tau;M).
\end{equation}
But, recalling \eqref{A-quantity} and \eqref{Ln-def}, we have
\begin{equation}
L_{n-1}^+(\chi,\tau;M)L_n^-(\chi,\tau;M)=\frac{A_{n-1}(\chi,\tau;M)^2}{A_n(\chi,\tau;M)^2M} = 1,
\end{equation} 
so the proof is complete.
\end{proof}

\end{document}